\documentclass[letterpaper,11pt]{amsart}

\usepackage{amssymb}
\usepackage{amsmath,amssymb,amsthm,amsfonts}
\usepackage{hyperref}
\usepackage{graphicx,color,colortbl}
\usepackage{upgreek}
\usepackage[mathscr]{euscript}
\usepackage{hhline}
\DeclareMathAlphabet{\mathpzc}{OT1}{pzc}{m}{it}
\usepackage{mathbbol,mathtools,calc}
\usepackage[labelfont={rm}]{subfig}
\usepackage{bm}
\usepackage{framed}
\usepackage{musicography}
\usepackage{tikz-cd} 

\usepackage{blkarray}

\DeclareMathAlphabet{\mathpzc}{OT1}{pzc}{m}{it}
\DeclareMathAlphabet{\mathdutchcal}{U}{dutchcal}{m}{n}
\SetMathAlphabet{\mathdutchcal}{bold}{U}{dutchcal}{b}{n}

\newcommand{\ts}[1]{{\color{red} #1}}
\newcommand{\ti}[1]{{\color{blue} #1}}
\newcommand{\matteo}[1]{{\color{brown} #1}}

\allowdisplaybreaks
\numberwithin{equation}{section}
\usepackage{ytableau}
\usepackage{mathdots}
\usepackage{tikz}
\usetikzlibrary{
	shapes,
	arrows,
	positioning,
	decorations.markings,
	decorations.pathmorphing,
	circuits.logic.US,
	circuits.logic.IEC,
	fit,
	calc,
	plotmarks,
	matrix
}

\tikzset{
>=stealth',
help lines/.style={dashed, thick},
axis/.style={<->},
important line/.style={thick},
connection/.style={thick, dotted},
punkt/.style={
rectangle,
rounded corners,
draw=black, thick,
text width=4.5em,
minimum height=2em,
text centered,
},
pil/.style={
->,
thick,
gray,
shorten <=2pt,
shorten >=2pt,}
}

\usepackage{array}
\usepackage{adjustbox}
\usepackage{cleveref}
\usepackage{enumitem}

\usepackage[tmargin=3cm,lmargin=3cm,rmargin=3cm,bmargin=3cm,headheight=1cm,footskip=1.5cm]{geometry}

\usepackage{fancyhdr}

\newtheorem{proposition}{Proposition}[section]
\newtheorem{lemma}[proposition]{Lemma}
\newtheorem{corollary}[proposition]{Corollary}
\newtheorem{theorem}[proposition]{Theorem}
\newtheorem*{theorem*}{Theorem}

\theoremstyle{definition}
\newtheorem{definition}[proposition]{Definition}

\newtheorem{remark}[proposition]{Remark}
\newtheorem*{remark*}{Remark}

\newcommand{\Qbinomial}[3]{\binom{#1}{#2}_{\hspace{-2pt} #3}}

\renewcommand{\a}{\alpha}
\renewcommand{\i}{\infty}
\newcommand{\E}[1]{\widetilde{e}_{#1}}

\newcommand{\Ai}{\rm{Ai}}
\newcommand{\Pf}{\mathrm{Pf}}

\newcommand{\arccosh}{\mathrm{arccosh}}
\newcommand{\R}{\mathbb{R}}

\newcommand{\Z}{\mathbb{Z}}

\newcommand*\diff{\mathop{}\!\mathrm{d}}

\begin{document}
\title{Solvable models in the KPZ class: approach through periodic and free boundary Schur measures}

\author[T. Imamura]{Takashi Imamura}\address{T. Imamura, 
Department of Mathematics and Informatics, Chiba University, Chiba, 263-8522, Japan}\email{imamura@math.s.chiba-u.ac.jp}

\author[M. Mucciconi]{Matteo Mucciconi}\address{M. Mucciconi, 
Department of Mathematics,
University of Warwick, Coventry, CV4 7HP,  United Kingdom}\email{matteomucciconi@gmail.com}

\author[T. Sasamoto]{Tomohiro Sasamoto}\address{T. Sasamoto, 
Department of Physics,
Tokyo Institute of Technology, Tokyo, 152-8551, Japan}\email{sasamoto@phys.titech.ac.jp}

\date{}

\maketitle

\begin{abstract}
    We explore probabilistic consequences of correspondences between $q$-Whittaker measures and periodic and free boundary Schur measures established by the authors in the recent paper [arXiv:2106.11922]. The result is a comprehensive theory of solvability of stochastic models in the KPZ class where exact formulas descend from mapping to explicit determinantal and pfaffian point processes. We discover new variants of known results as determinantal formulas for the current distribution of the ASEP on the line and new results such as Fredholm pfaffian formulas for the distribution of the point-to-point partition function of the Log Gamma polymer model in half space. In the latter case, scaling limits and asymptotic analysis allow to establish Baik-Rains phase transition for height function of the KPZ equation on the half line at the origin.
\end{abstract}

\tableofcontents

\section{Introduction}
\label{sec:Intro}

\subsection{One dimensional KPZ equation and stochastic heat equation}

    The Kardar-Parisi-Zhang (KPZ) equation \cite{KPZ1986} is the default model of growth of random surfaces and in one dimension it reads
    \begin{equation} \label{eq:KPZ}
        \partial_t \mathcal{H} = \frac{1}{2} \partial_x^2 \mathcal{H} + \frac{1}{2} \left( \partial_x \mathcal{H}  \right)^2 + \dot{\mathscr{W}},
    \end{equation}
    where $\mathcal{H}=\mathcal{H}(x,t),x\in\mathbb{R},t>0,$ is the height function and $\dot{\mathscr{W}}$ is the space-time gaussian white noise with covariance $\mathbb{E} \left( \dot{\mathscr{W}} (x,t) \dot{\mathscr{W}} (y,s) \right) = \delta_{x-y} \delta_{t-s}$. This equation is notoriously ill posed due to the non linear term $\left( \partial_x \mathcal{H} \right)^2$ and making sense of the solution of \eqref{eq:KPZ} has been one of the recent major accomplishments in the theory of stochastic analysis \cite{bertiniGiacomin1997stochastic,Hairer11,gubinelli_imkeller_perkowski_2015}. In this paper
    we consider the Cole-Hopf solutions of the KPZ equation defined as 
    \begin{equation} \label{eq:Cole_Hopf}
        \mathcal{H}(x,t) = \log \mathscr{Z}(x,t),
    \end{equation}
    where $\mathscr{Z}$ is a \emph{mild solution} of the stochastic heat equation with multiplicative noise (abbreviated as ``mSHE" in the text)
    \begin{equation} \label{eq:SHE}
        \partial_t \mathscr{Z} = \frac{1}{2} \partial_x^2 \mathscr{Z} + \mathscr{Z} \dot{\mathscr{W}}.
    \end{equation}
    The correct prescription to interpret the noise term in the right hand side is discussed \cite{BertiniCancrini1995}, whereas for an introduction to mild solutions see \cite{BertiniCancrini1995,quastel_introduction_to_KPZ} and references therein.
    Throughout the text we will mostly discuss the solution of the mSHE $\mathscr{Z}$ rather than $\mathcal{H}$, as this allows to phrase our statements rigorously. In all cases the use of the Cole-Hopf transformation \eqref{eq:Cole_Hopf} will be justified.

    Interest in the KPZ equation has, in the last decades, been motivated also by the fact that certain geometries and initial conditions allow to write (and sometimes prove) exact formulas for the distribution of the height function $\mathcal{H}$ \cite{AmirCorwinQuastel2011,Calabrese_LeDoussal_Rosso,Dotsenko,SasamotoSpohn2010,Le_Doussal_2012,Gueudre_Le_Doussal_2012,imamura_stationaryKPZ,BorodinCorwinFerrariVeto2013,borodin_bufetov_corwin_nested,barraquand2018,krejenbrink_le_doussal_KPZ_half_space,Barraquand_le_doussal_krejenbrink_2020,De_Nardis_et_al_2020,Barraquand_le_doussal_half_space_flat}. In this paper we deal with two such choices of solvable geometries and initial data, which we introduce next.  

    \subsection{The KPZ equation in full space with narrow wedge initial conditions}
    
    The proper formulation of the problem of this case is given in terms of the mSHE which reads
        \begin{equation} \label{eq:SHE_narrow wedge}
            \begin{cases}
                \partial_t \mathscr{Z} = \frac{1}{2} \partial_x^2 \mathscr{Z} + \mathscr{Z} \dot{\mathscr{W}},
                \qquad x\in \mathbb{R}, t\in \mathbb{R}_+,
                \\
                \mathscr{Z}(x,0) = \delta_0(x),
            \end{cases}
        \end{equation}
    where $\delta_0$ is the Dirac delta function.
    Heuristically, for the function $\mathcal{H}$, the initial data in \eqref{eq:SHE_narrow wedge} represents a situation where the random KPZ growth is started from a wedge $\mathcal{H}(x,0) ``=" - |x|/\varepsilon$, with $\varepsilon$ infinitesimally small, which justifies the name. Under delta initial conditions the function $\mathscr{Z}$ can be interpreted as the point-to-point partition function of the \emph{continuum directed random polymer} \cite{AmirCorwinQuastel2011,AlbertsKhaninQuastel2012}, which is a model of a polymer in 1+1 dimension propagating through a space-time in uncorrelated random environment. This is evident expressing the solution of \eqref{eq:SHE_narrow wedge} by the Feyman-Kac formula 
        \begin{equation} \label{eq:feynman_kac}
            \mathscr{Z}(x,t) = \mathbb{E}\left[ :\exp: \left\{ \int_0^t \dot{\mathscr{W}}(\mathscr{B}(s),s) \diff s \right\} \bigg| \mathscr{B}(0)=0, \, \mathscr{B}(t)=x  \right],
        \end{equation}
    where $:\exp:$ indicates the Wick ordered exponential and the expectation is taken only over the brownian bridge $\mathscr{B}$. In this way the integral in \eqref{eq:feynman_kac} has the natural interpretation of the random energy of the polymer $\mathscr{B}$.
    
    The first exact solution of the mSHE \eqref{eq:SHE_narrow wedge} became available in 2010 after the breakthrough result of four different groups \cite{AmirCorwinQuastel2011,Dotsenko,Calabrese_LeDoussal_Rosso,SasamotoSpohn2010}. Famously, this solution expresses the Laplace transform of the probability distribution of $\mathscr{Z}$ as a Fredholm determinant. 

    \begin{theorem} \label{thm:solution_mSHE}
        Consider $\mathscr{Z}(x,t)$ the solution of the stochastic heat equation \eqref{eq:SHE_narrow wedge}. Then, we have
        \begin{equation} \label{eq:fredholm_det_SHE}
            \mathbb{E} \left[ e^{ - e^{ -s + \log \mathscr{Z}(0,t) +t/24 } } \right] = \det (1 - \mathscr{K}_{\mathrm{mSHE}})_{\mathbb{L}^2(s,+\infty)},
        \end{equation}
        where $\mathscr{K}_{\mathrm{mSHE}}$ is the linear operator with integral kernel
        \begin{equation} \label{eq:K_SHE}
            \mathscr{K}_{\mathrm{mSHE}}(x,y) = \int_{ \mathrm{i} \mathbb{R} -d} \frac{\diff Z}{2 \pi \mathrm{i} } \int_{\mathrm{i} \mathbb{R} +d} \frac{\diff W}{ 2 \pi \mathrm{i} } e^{ -\frac{t}{2} \left( \frac{Z^3}{3} - \frac{W^3}{3} \right) +Zx - Wy } \frac{ \pi }{ \sin[\pi(W-Z)] }
        \end{equation}
        and for the integration contours we choose $d\in (0,1/2)$.
    \end{theorem}

    \begin{remark}    
    The kernel $\mathscr{K}_{\mathrm{mSHE}}$ 
    can be written as a ``positive temperature variant" of the Airy kernel,
    \begin{equation}\label{eq:Kff}
        \mathscr{K}_{\mathrm{mSHE}}(x,y)=
        \int_{\mathbb{R}} \frac{{\rm Ai}((x+\xi)/\gamma_t) {\rm Ai} ((y+\xi)/\gamma_t)}{1+e^{-\xi}} \diff \xi,
    \end{equation}
    with $\gamma_t=(t/2)^{1/3}$.
    \end{remark}

    From a physical standpoint the explicit solution of \eqref{eq:SHE_narrow wedge} is relevant as it has allowed to rigorously establish the large time asymptotic limit
    \begin{equation} \label{eq:asymptotics_KPZ}
            \mathcal{H}(0,t) \approx - \frac{t}{24} + \frac{t^{1/3}}{2^{1/3}} \chi_{\mathrm{GUE}},
    \end{equation}
    where $\chi_\mathrm{GUE}$ follows the GUE Tracy-Widom distrubution\cite{tracy_widom1994level_airy}. Historically, the anomalous scaling exponent $1/3$ had been predicted in \cite{KPZ1986,HuseHenleyPolymers1985}, whereas more precise conjectures about the universal fluctuations 
    for a large class of models (KPZ universality class)
    were formulated \cite{Praehofer2002} after the works \cite{baik1999distribution,johansson2000shape}. For the case of the KPZ equation \eqref{eq:asymptotics_KPZ} confirms the predictions from the KPZ scaling theory.
    
    In this paper we will present a derivation of the solution of \cref{thm:solution_mSHE} which is conceptually different from ideas used in literature since 2010. The advantage of our techniques, compared to pre-existing methods, is that they explain the emergence of Fredholm determinants, by establishing an a priori correspondence between the height function and the edge statistics of a free fermion at positive temperature.
    Moreover, methods developed in this paper can be applied not only to the study of the mSHE/KPZ equation, but also to a class of discrete models in the KPZ class. These are the ASEP, the $q$-PushTASEP, or the Log Gamma polymer and we will discuss this aspect more in \cref{subs:hieranchies}. More remarkably our methods can be naturally adapted to the more challenging task of treating the mSHE/KPZ equation in half space, as we discuss in the next subsection. See also \cref{subs:novelty} for more detailed account about novelty of our work.
    
        
        \subsection{The KPZ equation in half space with narrow wedge initial conditions.}
        
        In this case the problem is phrased in terms of the mSHE as
            \begin{equation} \label{eq:SHE_hs}
            \begin{cases}
                \partial_t \mathscr{Z} = \frac{1}{2} \partial_x^2 \mathscr{Z} + \mathscr{Z} \dot{\mathscr{W}},
                \qquad x\in \mathbb{R}_{\ge 0}, t\in \mathbb{R}_+,
                \\
                \mathscr{Z}(x,0) = \delta_0(x), \qquad
                \left(\partial_x - \omega\right) \mathscr{Z}(x,t)\bigr|_{x=0} = 0.
            \end{cases}
        \end{equation}
        The study of the well-posedness of \eqref{eq:SHE_hs} was done in \cite{Xuan_Wu_Intermediate_Disorder,Parekh2019} and in particular a mild solution of \eqref{eq:SHE_hs} exists, is unique and it is almost surely positive. This justifies the definition of $\mathcal{H}^\mathrm{hs} = \log \mathscr{Z}^\mathrm{hs}$ as a solution of the KPZ equation in half space with initial conditions consisting in a narrow wedge peaked at the origin. The Robin boundary conditions of \eqref{eq:SHE_hs} become, for the height function $\mathcal{H}^\mathrm{hs}$ the Neumann boundary conditions
        \begin{equation}
            \partial_x \mathcal{H}^\mathrm{hs}(x,t)\big|_{x=0} = \omega.
        \end{equation}
        A physical realization of such boundary condition was proposed in \cite{Takeuchi_ito_2018_half_space}.
        As in the full space case the solution $\mathscr{Z}^\mathrm{hs}$ of \eqref{eq:SHE_hs} can also be interpreted as a point-to-point partition function of the continuum random polymer propagating in the half line $[0,+\infty)$ with an interaction produced by a wall at the origin. Again, this is evident from the Feynman-Kac representation
        \begin{multline} \label{eq:feynman_kac_hs}
            \mathscr{Z}^{\mathrm{hs}}(x,t) \\= \mathbb{E} \left[ :\exp: \left\{ \int_0^t \left( \dot{\mathscr{W}}(\mathscr{B}^{\mathrm{hs}}(s),s) - \omega \delta_0(\mathscr{B}^{\mathrm{hs}}(s)) \right) \diff s \right\} \bigg| \mathscr{B}^{\mathrm{hs}}(0)=0, \, \mathscr{B}^{\mathrm{hs}}(t)=x \right],
        \end{multline}
        where the expectation is taken over the reflected Brownian bridge $\mathscr{B}^{\mathrm{hs}}$, i.e. a Brownian bridge conditioned to remain nonnegative at all times. From \eqref{eq:feynman_kac_hs} we observe that the interaction with the wall, modulated by the parameter $\omega$, is attractive when $\omega<0$ and repulsive for $\omega>0$.

        \subsubsection{Fredholm pfaffian solution.}
        
        The truly novel part of our results consists in the derivation of Fredholm pfaffian formulas for solvable models related to the half space KPZ equation. The typical approach to this problem has been that of leveraging integrable strucures of the continuum directed polymer model, which will be discussed in the next subsection, to obtain exact formulas for the partition funciton. This route involves the use of complicated and mathematically non-rigorous Bethe ansatz computations that only recently have led to manageable, although conjectural, pfaffian formulas; see \cite{Gueudre_Le_Doussal_2012,borodin_bufetov_corwin_nested,krejenbrink_le_doussal_KPZ_half_space,De_Nardis_et_al_2020}. 
        Alternatively, in the more general case of the Log Gamma polymer model, exact formulas for the distribution of the partition function have been found through combinatorial techniques in \cite{OSZ2012} or in \cite{barraquand_half_space_mac} using eigenrelations of Macdonald polynomials. These last results, although mathematically rigorous have the disadvantage to be hard to manage and did not lead to a proof of asymptotic results. The only mathematically rigorous derivation of a pfaffian solution of \eqref{eq:SHE_hs}, prior to this paper, was given in \cite{barraquand2018}, where the authors managed to study the case $\omega=-\frac{1}{2}$, through a matching between the Laplace transform of the KPZ equation and a nonlocal average of a pfaffian point process.
        In this paper we overcome these technical difficulties by taking a different point of view and establishing a direct and a priori corresondence between KPZ models in half space and explicit pfaffian point processes. This allows us to study the full range $\omega \in \mathbb{R}$.
    
    For the next theorem we recall the standard notation
    \begin{equation} \label{eq:J}
        J = \left( \begin{matrix} 0 & 1 \\ -1 & 0 \end{matrix} \right),
    \end{equation}
    and we use the convention that $J(x,y) = \mathbf{1}_{x=y} J$. 
    
    \begin{theorem} \label{thm:solution_mSHE_hs}
        Consider $\mathscr{Z}^{\mathrm{hs}}(x,t)$ the solution of the half space stochastic heat equation \eqref{eq:SHE_hs} with initial conditions $\mathscr{Z}^{\mathrm{hs}}(x,0) = \delta_{x}$ and boundary parameter $\omega > - 1/2$. Then, we have
        \begin{equation} \label{eq:fredholm_pfaffian_SHE}
            \mathbb{E} \left[ e^{ - e^{ -s +\log \mathscr{Z}^{\mathrm{hs}}(0,t) + t/24  }} \right] = \Pf \left[ J - \mathscr{L}_{\mathrm{mSHE}} \right]_{\mathbb{L}^2(s,+\infty)},
        \end{equation}
        where $\mathscr{L}_{\mathrm{mSHE}}$ is the $2 \times 2$ matrix linear operator with integral kernel
        \begin{equation}
            \mathscr{L}_{\mathrm{mSHE}}(X,Y) = \left( \begin{matrix} \mathscr{K}^\mathrm{hs} (X,Y) & -\partial_y \mathscr{K}^\mathrm{hs}(X,Y)
            \\
             - \partial_x \mathscr{K}^\mathrm{hs}(X,Y) & \partial_x \partial_y \mathscr{K}^\mathrm{hs}(X,Y)
            \end{matrix} \right),
        \end{equation}
        with
        \begin{equation} \label{eq:k_she_hs}
        \begin{split}
            \mathscr{K}^\mathrm{hs}(X,Y) = \int_{\mathrm{i} \mathbb{R} +d} \frac{\diff Z}{2 \pi \mathrm{i}} \int_{\mathrm{i} \mathbb{R} +d}
            \frac{\diff W}{2 \pi \mathrm{i}}
            &
            e^{\frac{t}{2} \left(\frac{Z^3}{3} +  \frac{W^3}{3} \right) - ZX - WY}
            \frac{\Gamma(\frac{1}{2} + \omega - Z)}{\Gamma (\frac{1}{2} + \omega+Z)} \frac{\Gamma(\frac{1}{2} + \omega -W)}{\Gamma (\frac{1}{2} + \omega+W)}
            \\
            &
            \times
            \Gamma(2Z) \Gamma(2W) \frac{\sin [\pi (Z-W)]}{\sin [\pi (Z+W)]}
        \end{split}
        \end{equation}
        and we assume $ 0 < d < \min( 1/2, 1/2+\omega)$.
    \end{theorem}
    
    A few remarks are in order.
    
    \begin{remark}
        In \cref{subs:krejenbrink_le_doussal} we report equivalent expressions for the Fredholm pfaffian in the right hand side of \eqref{eq:fredholm_pfaffian_SHE}. These had been previously derived in \cite{krejenbrink_le_doussal_KPZ_half_space}, although through non rigorous Bethe ansatz computations.
    \end{remark}
    
    \begin{remark}
        In \cref{sec:KPZpS} we will present pfaffian formulas analogous to that of \eqref{eq:fredholm_pfaffian_SHE} characterizing the distribution of the rightmost particle in a $q$-PushTASEP with particle creation and the free energy of the Log Gamma polymer model in half space. Unlike for the case of the mSHE, pfaffian formulas for these more general models had not appeared before in literature, not even at a conjectural level.
    \end{remark}
    
    \begin{remark}
        A theorem by Parekh \cite{Parekh_half_space_symmetry} relates the solution of the half space mSHE \eqref{eq:SHE_hs}, which we denote by $\mathscr{Z}^\mathrm{hs}$ and the solution of the half space mSHE with certain Brownian initial conditions and Dirichlet boundary conditions
        \begin{equation} \label{eq:mSHE_dirichlet}
            \begin{cases}
                \partial_t \mathscr{Z} = \frac{1}{2} \partial_x^2 \mathscr{Z} + \mathscr{Z} \dot{\mathscr{W}},
                \qquad x\in \mathbb{R}_{\ge 0}, t\in \mathbb{R}_+,
                \\
                \mathscr{Z}(x,0) = e^{\mathscr{B}(x) - (\omega + \frac{1}{2})x}, \qquad
                \mathscr{Z}(0,t) = 0,
            \end{cases}
        \end{equation}
        where $\mathscr{B}(x)$ is a standard Brownian motion. Denoting by $\bar{\mathscr{Z}}^\mathrm{hs}$ the solution of \eqref{eq:mSHE_dirichlet}, then \cite[Theorem 1.1]{Parekh_half_space_symmetry} states the equivalence in distribution
        \begin{equation}
            \mathscr{Z}^\mathrm{hs}(0,t) \stackrel{\mathcal{D}}{=} \lim_{x\to 0} \frac{1}{x} \bar{\mathscr{Z}}^\mathrm{hs}(x,t).
        \end{equation}
        By this equality, \cref{thm:solution_mSHE_hs} also provides an explicit formula for the probability distribution of $\bar{\mathscr{Z}}^\mathrm{hs}(x,t)$ around $x=0$, although we are not going to state such results in this paper.
    \end{remark}
    
    Although expression \eqref{eq:fredholm_pfaffian_SHE}, as stated, only holds for $\omega>-1/2$, Fredholm pfaffian expressions for the expectation $\mathbb{E} \left[ e^{ - e^{ -s  +\log \mathscr{Z}^{\mathrm{hs}}(0,t) + t/24  }} \right]$ are available for any $\omega \in \mathbb{R}$. They can be derived by analytic continuation, in the parameter $\omega$, although such procedure, especially for $\omega<-1/2$, turns out to be not completely trivial. Such results are presented in \cref{thm:SHE_gaussian}.

        \subsubsection{Limiting distribution and phase transition.}
        
        Characterizing the large scale behavior of the free energy of directed polymers in presence of a wall has been a problem considered in the physics literature since the work of Kardar \cite{kardar_depinning}. Varying the value of $\omega$, the partition function of the polymer is expected to undergo a phase transition, which needs to be characterized. This an example of the so called ``depinning transition"; see \cite[Section 7.3-7.4]{Praehofer2002}. Non-rigorous explicit computations \cite{krejenbrink_le_doussal_KPZ_half_space,barraquand_half_space_mac,De_Nardis_et_al_2020} have allowed to conjecture the large time behavior for the free energy of the polymer model, which stated in terms of the height function $\mathcal{H}^{\mathrm{hs}}$ reads
        \begin{equation} \label{eq:H_hs_asymptotics}
            \mathcal{H}^{\mathrm{hs}}(0,t) \approx - f_\omega t + \sigma_\omega t^{\zeta(\omega)} \chi_\omega,
        \end{equation}
        where $f_\omega, \sigma_\omega$ are explicit constants depending on $\omega$ and
        \begin{equation} \label{eq:phase_transition}
            \zeta(\omega) = 
            \begin{cases}
                \frac{1}{3} \qquad & \text{if } \omega \ge -\frac{1}{2},
                \\
                \frac{1}{2} \qquad & \text{if } \omega< -\frac{1}{2},
            \end{cases}
            \qquad
            \text{and}
            \qquad
            \chi_\omega = 
            \begin{cases}
                \chi_\mathrm{GSE} \qquad & \text{if } \omega > -\frac{1}{2},
                \\
                \chi_\mathrm{GOE} \qquad & \text{if } \omega = -\frac{1}{2},
                \\
                \mathcal{N}(0,1) \qquad & \text{if } \omega< -\frac{1}{2}.
            \end{cases} 
        \end{equation}
        Above $\chi_\mathrm{GSE}, \chi_\mathrm{GOE}$ follow respectively the GSE and the GOE Tracy-Widom distributions \cite{tracy1996orthogonal}, while $\mathcal{N}(0,1)$ is a standard Gaussian random variable. As mentioned above, the only rigorous previous study of this problem was carried out in \cite{barraquand2018} where the case $\omega=-\frac{1}{2}$ was confirmed. 
        
        In the next theorem we establish the phase transition \eqref{eq:H_hs_asymptotics}, \eqref{eq:phase_transition}. Scaling $\omega$ properly around the critical value $-\frac{1}{2}$ as a function of another parameter $\xi$ we find that $\log \mathscr{Z}^\mathrm{hs}(0,t)$ possesses fluctuations described by a crossover distribution $\mathrm{F}_\mathrm{cross}(\,\cdot\,|\xi)$. This interpolates between the GOE Tracy-Widom distribution, which is recovered setting $\xi=0$ and the GSE Tracy-Widom distribution, which is obtained as the limit of $F_\mathrm{cross}(\,\cdot\,|\xi)$ when $\xi\to +\infty$. The explicit expression of $F_\mathrm{cross}(\,\cdot\,|\xi)$ is given in \cref{subs:Tracy_Widom}, where we also prove, in \cref{cor:equivalence_Baik_Rains_crossover}, that it corresponds to a distribution described first by Baik and Rains in \cite{baik_rains2001asymptotics}.
    
    \begin{theorem} \label{thm:asymptotics_SHE}
        Let $\mathscr{Z}^{\mathrm{hs}}(x,t)$ be the solution of the stochastic heat equation in half space \eqref{eq:SHE_hs}. Then, the following limits hold:
        \begin{itemize}
            \item if $\omega \ge -1/2$, rescaling $\omega= -1/2 + \frac{\xi}{2^{-1/3} t^{1/3}}$ we have
            \begin{equation} \label{eq:SHE_crossover}
                 \lim_{t \to +\infty} \mathbb{P} \left[ \frac{\log \mathscr{Z}^{\mathrm{hs}}(0,t) +  t/24}{ 2^{-1/3} t^{1/3}} \le r \right] = F_{\mathrm{cross}}(r;\xi),
            \end{equation}
         
            where $F_{\mathrm{cross}}$ is the Baik-Rains crossover distribution of \cref{def:crossover};
            \item if $\omega<-1/2$, we have
            \begin{equation}
                \lim_{t \to +\infty} \mathbb{P} \left[ \frac{\log \mathscr{Z}^{\mathrm{hs}}(0, t) + f_\omega t}{ \sigma_\omega t^{1/2}} \le r \right] = \int_{-\infty}^r \frac{e^{-u^2/2}}{\sqrt{2 \pi}} \diff u,
            \end{equation}
            where $f_\omega=\frac{1}{6} + \frac{\omega}{2} + \frac{\omega^2}{2}$ and $\sigma_\omega =\sqrt{-\omega-\frac{1}{2}}$.
        \end{itemize}
    \end{theorem}

\subsection{Hierarchies of solvable models} \label{subs:hieranchies}

The solvability of the mSHE comes from the fact that the $n$ point average
    $$\mathbb{E}[\mathscr{Z}(x_1,t) \cdots \mathscr{Z}(x_n,t) ]$$ 
solves the imaginary time delta Bose gas with $n$ particles. This idea, first formulated in \cite{Kardar1987}, was used directly by \cite{Dotsenko,Calabrese_LeDoussal_Rosso}, to implement the celebrated ``replica method" and to derive an explicit formula for the moment generating function of $\mathscr{Z}(x,t)$. Since moments of $\mathscr{Z}(x,t)$ are highly diverging (see \cite[Section A.4]{BorodinCorwinSasamoto2012}), treating directly the mSHE through replica method yields expressions that cannot be justified mathematically and therefore other routes to produce explicit formulas are needed. The standard procedure to circumvent this problem is that of considering, in place of $\mathscr{Z}$ or $\mathcal{H}$, certain more regular discretizations which preserve integrability \cite{BorodinCorwinSasamoto2012}. Such regularized models usually possess a ``temperature" parameter\footnote{Here the use of the word ``temperature" is motivated by the appearance of the Fermi-Dirac distribution. See \cref{subs:alternative_det} The same Fermi-Dirac distribution is present already in the kernel \eqref{eq:Kff}.} $q$ and a height function $h^{(q)}$, which under a proper high temperature limit converges in distribution to the solution of the KPZ equation $h^{(q)} \to \mathcal{H}$.

The most well studied of these reguralizations of the KPZ equation is arguably the Asymmetric Simple Exclusion Process (ASEP), in which case the parameter $q$ modulates the asymmetry in the jump rate of particles; see \cref{subs:ASEP} for more details. In this case the height function $h^{(q)}$ is given by the integrated current through a location and as $q \to 1$ convergence to the Cole-Hopf solution of the KPZ equation was proven in \cite{bertiniGiacomin1997stochastic}. Considering such weak noise scaling was the approach adopted in \cite{AmirCorwinQuastel2011,SasamotoSpohn2010}, where authors were able to leverage Fredholm determinant formulas found earlier for the ASEP by Tracy and Widom \cite{TW_ASEP1,TW_ASEP2,TW_ASEP4}.

In the last decade, other discrete solvable models have been introduced in literature. The ones possessing the richest structure admit a formulation in terms of stochastic vertex models such as Corwin and Petrov's Higher Spin Vertex Model \cite{CorwinPetrov2015}. 
For all these models we can apply regularizations of the replica method, which consist in a combination of stochastic duality \cite{schutz1997dualityASEP,ImamuraSasamoto2011current,BorodinCorwinSasamoto2012,CorwinPetrov2015} and Bethe Ansatz \cite{BCG6V,BCPS2014_arXiv_v4,BorodinPetrov2016inhom}.
A more indirect approach is that of using Macdonald processes \cite{BorodinCorwin2011Macdonald}, where KPZ solvable models arise as certain marginals of a more general 2+1 growth process \cite{BorodinCorwin2011Macdonald,OConnellPei2012,BorodinPetrov2013NN,MatveevPetrov2014,BufetovPetrovYB2017,BufetovMucciconiPetrov2018,MucciconiPetrov2020,BorodinBufetovWheeler2016,barraquand_half_space_mac,barraquand2018,chen_ding_littlewood}. In this last case the common approach is that of leveraging combinatorial properties of Macdonald functions to derive informations about the model.

In this paper we will mainly consider a particular regularization of the KPZ equation, that is the $q$-PushTASEP \cite{BorodinPetrov2013NN,MatveevPetrov2014}. See Figure \ref{fig:q_PushTASEP} below for an illustration of the model in continuous time. In its most general formulation this is, admittedly, a rather artificial model, whose definition is postponed to \cref{subs:qPushTASEP}, which nevertheless turns out to be the right one to focus on for three key reasons enumerated below.

\medskip

    \emph{Multivariate model.} The $q$-PushTASEP depends on a number of parameters which in the text we denote by $a_1,a_2,\dots,b_1,b_2,\dots$. Considering proper scalings of $q,a_i,b_j$, we can transform the $q$-PushTASEP into several interesting other models, as schematically described in \cref{fig:hierarchy}. For instance taking a certain $q \to 1$ limit one obtains the Log Gamma polymer model introduced by Seppalainen \cite{Seppalainen2012}. This model is of interest on its own and in the last decade has received considerable attention; see \cite{zygouras_review} and references therein.
    
\medskip
    
    \emph{Half space variant.} The $q$-PushTASEP possesses a solvable variant that regularizes the mSHE in half space, called \emph{$q$-PushTASEP with particle creation}, introduced in \cite{barraquand_half_space_mac}. This is also a multivariate model depending on parameters denoted by $\gamma,a_1,a_2,\dots$, where $\gamma$ governs the creation of new particles and in the KPZ picture plays the role of boundary parameter $\omega$. As we will show in the text we will be able to solve, i.e. write down exact formulas, for this half space model, for general choices of these parameters. This is remarkable since exact formulas for the ASEP in half space, that also approximates the mSHE \eqref{eq:SHE_hs}, are only known for the particular choice $\omega=-\frac{1}{2}$ and were found in \cite{barraquand2018}.

\medskip
    
    \emph{Relation with $q$-Whittaker measures.} Both versions of the $q$-PushTASEP hinted above have joint law described respectively by the $q$-Whittaker measure \cite{BorodinCorwin2011Macdonald} and by the half space $q$-Whittaker measure \cite{barraquand_half_space_mac}, which are probability measures over partitions $\mu$ of the form respectively
    \begin{equation}
        \mathbb{P}(\mu) \propto \mathscr{P}_\mu(a;q) \mathscr{Q}_\mu(b;q),
        \qquad
        \mathbb{P}(\mu) \propto \mathbf{1}_{\mu' \, \text{even}} \mathscr{Q}_\mu(\gamma,a;q).
    \end{equation}
    Here $\mathscr{P}_\mu, \mathscr{Q}_\mu$ are the $q$-Whittaker polynomials;
    see \cref{sec:qW} for precise definitions. This is important since the $q$-Whittaker polynomials have a nice combinatorial structure. In particular adopting a bijective approach, we were recently able to discover and prove certain new summation identities \cite{IMS_skew_RSK} relating the $q$-Whittaker measures with the periodic Schur measure \cite{borodin2007periodic} and the Free boundary Schur measure \cite{Betea_et_al_free_boundary}; see \eqref{eq:qW_and_Schur_1}, \eqref{eq:qW_and_Schur_2} in the text. These are other measures over partitions $\lambda$ having respectively the form
    \begin{equation}
        \mathbb{P}(\lambda) \propto \sum_\rho q^{|\rho|} s_{\lambda/\rho}(a) s_{\lambda/\rho}(b),
        \qquad
        \mathbb{P}(\lambda) \propto \mathbf{1}_{ \lambda ' \, \text{even}} \sum_{\rho' \, \text{even}} q^{|\rho|/2} s_{\lambda/\rho}(a),
    \end{equation}
    where $s_{\lambda/\rho}$ are the skew Schur polynomials, defined below in \cref{sec:qW}. They have notable determinantal and pfaffian structures which we will recall in \cref{sec:pS}.

\begin{figure}
    \centering
    \includegraphics{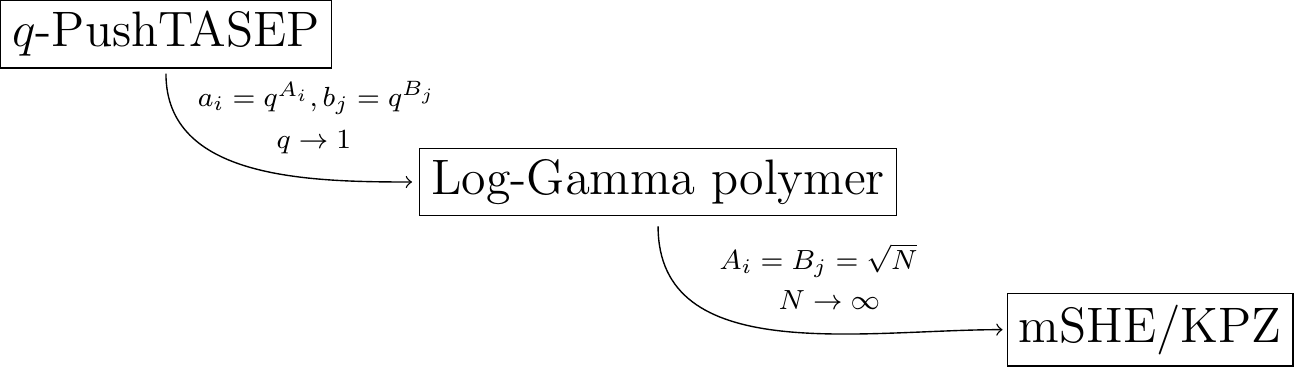}
    \caption{Schematic representation of hierarchy of models related to the $q$-PushTASEP.}
    \label{fig:hierarchy}
\end{figure}

\subsection{Methods}
    The location of a tagged particle $\mathsf{x}_N(T)$ in the $q$-PushTASEP is equivalent in distribution to the length of the first row of a random partition $\mu$ sampled according to the $q$-Whittaker measure \cite{MatveevPetrov2014}
    \begin{equation}
        \mathsf{x}_N(T) \stackrel{\mathcal{D}}{=} \mu_1 + N.
    \end{equation}
    A summation identity discovered in \cite{IMS_matching,IMS_skew_RSK} relates the first row in the $q$-Whittaker measure with the first row $\lambda_1$ in the periodic Schur measure, up to a shift
    \begin{equation}
        \mu_1+\chi \stackrel{\mathcal{D}}{=} \lambda_1.
    \end{equation}
    Here $\chi$ is an independent random variable with explicit distribution and this result is proven below in \cref{thm:matching_qW_ps}. The combination of the previous two relations immediately grants Fredholm determinant formulas describing the distribution of the tagged particle $\mathsf{x}_N(T)$, borrowing known results on the periodic Schur measure \cite{borodin2007periodic,betea_bouttier_periodic}.
        
    Once formulas are obtained for the most general model, i.e. the $q$-PushTASEP, we will consider various scaling limits. A change of specializations of symmetric functions in the $q$-Whittaker measure will give us Fredholm determinant formulas for the ASEP, thanks to a notable relation between the Hall-Littlewood measure (seen as a particular case of the $q$-Whittaker measure) and the ASEP (seen as a particular specialization of the Stochastic Six Vertex Model) discovered in \cite{BorodinBufetovWheeler2016}. A Fredholm determinant representation for the probability distribution of the free energy of the log Gamma polymer model will be derived through a $q\to 1$ limit of formulas for the $q$-PushTASEP. In doing so a number of delicate estimates of decay for (ratio of) $q$-Pochhammer symbols will be required and we will present them in \cref{sec:LG}. In literature, bounds for $q$-Pochhammer symbols are usually derived using certain complicated integral representations involving other special functions; see, for instance, the recent work \cite{corwin_knizel_open_KPZ} and references therein. In our case, all new results given in \cref{sec:LG} will have completely elementary and self contained proofs, which is also a novel, although somewhat technical, aspect of our work. Considering a scaling limit, called \emph{intermediate disorder regime} \cite{AlbertsKhaninQuastel2012}, of formula for the Log Gamma polymer model we will prove \cref{thm:solution_mSHE}. 
        
    \medskip
    
    As anticipated, we will also derive exact formulas for KPZ models in half space. Conceptually, the approach we follow mirrors that used to derive formulas in full space setting. Our starting point is the identity in distribution \cite{barraquand_half_space_mac}
        \begin{equation} \label{eq:intro_x_mu}
            \mathsf{x}^\mathrm{hs}_N(T) \stackrel{\mathcal{D}}{=} \mu_1^\mathrm{hs} + N,
        \end{equation}
        between the $q$-PushTASEP with particle creation on the left hand side and the half space $q$-Whittaker measure on the right hand side. We then relate, through a summation identity discovered in \cite{IMS_skew_RSK}, a particular case of the half space $q$-Whittaker measure with the free boundary Schur measure \cite{Betea_et_al_free_boundary}. In formulas our relation reads
        \begin{equation} \label{eq:intro_mu_lambda}
            \mu_1^\mathrm{hs} + \chi \stackrel{\mathcal{D}}{=} \lambda_1^\mathrm{hs}.
        \end{equation}
        Combining \eqref{eq:intro_x_mu} and \eqref{eq:intro_mu_lambda} we are able to relate the rightmost particle $\mathsf{x}^\mathrm{hs}_T(T)$ of the $q$-PushTASEP with particle creation with an explicit pfaffian point process. Strictly speaking such correspondence only holds when the parameter governing the creation of new particles, which plays the role of boundary parameter in the KPZ picture, is set to $\gamma=0$. To extend the correspondence to all non-negative values of $\gamma$ we will use a special symmetry of the half space $q$-Whittaker measure found in \cite{barraquand_half_space_mac}; see \cref{prop:symmmetry_half_space_qW} in the text. In this way we find a Fredholm pfaffian representation for the probability distribution of $\mathsf{x}_T^\mathrm{hs}(T)$ valid for a general choice of parameters. We then move to consider scaling limits of our formula for the $q$-PushTASEP. Taking the $q\to 1$ limit already considered in full space setting we derive a Fredholm pfaffian representation for the probability distribution of the point-to-point partition function of the Log Gamma polymer in half space. The asymptotic analysis in this case builds up on the study made for the full space model, although it ends up being slightly more involved; also in this case technical estimates are needed and they are reported in \cref{sec:LG}. Formulas we get for the Log Gamma polymer model are amenable to a subsequent scaling limit, considered in \cite{Xuan_Wu_Intermediate_Disorder}, under which the polymer partition function reduces to the solution $\mathscr{Z}^\mathrm{hs}$ of the half space mSHE \eqref{eq:SHE_hs} and in this way we prove \cref{thm:solution_mSHE_hs}.
        
        \medskip
        
        In \cref{sec:asymptotics} and \cref{sec:KPZ} we will perform the large time-space asymptotic limits of our Fredholm determinant and pfaffian formulas. In this way we recover GUE limiting fluctuations results for the Log Gamma polymer model in full space. The emphasis in the proof we provide should be on the simplicity of the asymptotic analysis, given by the ``free fermionic" nature of our Fredholm operator. This should be compared with other existing techniques (see e.g. \cite{Barraquand_Corwin_Dimitrov_log_gamma}), where a rigorous application of the saddle point method requires involved manipulations of integral contours of the Fredholm kernel. 
        
        The asymptotic analysis of our Fredholm pfaffian formulas will be performed employing a standard saddle point analysis. We will focus on three special cases. In \cref{subs:baik_rains_pfaffian_schur} we will consider the asymptotic limit of Fredholm pfaffian formulas for the half space $q$-Whittaker measure for $q=0$. Doing so we will lay out the generic structure of the asymptotic analysis of our pfaffian formulas and we will additionally prove a new representation of the Baik-Rains crossover distribution originally defined in \cite{baik_rains2001asymptotics}. Subsequently we will compute the thermodynamic limit of the polymer partition function for the half space Log Gamma polymer model, proving Baik-Rains phase transition. Analogous computations will be reported in \cref{subs:hsSHE_Baik_Rains}, where we report the proof of \cref{thm:asymptotics_SHE}.

        \subsection{Novelty of our approach.} \label{subs:novelty}
        Lastly, let us emphasize how our techniques differ from those available in the pre-existing literature and what are the advantages of our approach. In simple terms we have implemented a bijective approach to the study of the $q$-Whittaker measure, and considering symmetries of our construction we have established correspondences between solvable KPZ models in full and half space with explicit determinantal and pfaffian point processes of free fermionic origin.
        At the level of formulas,  correspondences between KPZ models and free fermions at positive temperature had been known for a long time, since the works of \cite{AmirCorwinQuastel2011,Dotsenko,Calabrese_LeDoussal_Rosso,SasamotoSpohn2010} in 2010, because the kernel \eqref{eq:Kff} for the solution of the KPZ equation \eqref{eq:SHE_narrow wedge} can be interpreted as the correlation kernel for a positive temperature free fermion (see for instance \cite{DLDMS_free_fermions_KPZ} for an early application).    
        The correspondence is formulated in a slightly different manner by Borodin and Gorin \cite{BG2016_Airy_moments}, who understood the Fredholm determinant solution of the KPZ equation as a \emph{multiplicative nonlocal statistics} of the Airy point process. 
        The latter holds more in general for solvable regularizations of the KPZ equation, as found by Borodin in \cite{borodin2016stochastic_MM} and was extended in half space setting, for the particular case of ASEP, by Barraquand, Borodin, Corwin, Wheeler \cite{barraquand2018} to study the KPZ equation on the half line with critical boundary parameter $\omega = -\frac{1}{2}$.
        
        
        The correspondence we establish and use in this paper are conceptually different. The bijection  connecting $q$-Whittaker measure and periodic Schur measure, which we found in \cite{IMS_skew_RSK}, translates to a connection between KPZ models and free fermion at positive tempereture, and gives precise relations between \emph{local observables} in both models. Once the connection is made, one immediately concludes the existence of Fredholm determinant and pfaffian formulas, without doing any further calculation. 
        
        Our new approach proposed in this paper to study KPZ models can be considered as ``positive temperature" version of the works of Johansson \cite{johansson2000shape} and of Baik and Rains \cite{baik_rains2001algebraic,baik_rains2001symmetrized,baik_rains2001asymptotics} around 2000. Johansson applied the Robinson-Schensted-Knuth (RSK) correspondence to map the problem of  directed polymer at zero temperature, or the totally asymmetric simple exclusion process (TASEP), in full space to a free fermion at zero temperature. Baik and Rains used the symmetries of RSK correspondence to study the above models in half space. The ideas and methods in these works are quite powerful and have been applied to many models in various settings but are not suited for studying positive temperature models, for which different methodologies using Bethe ansatz or Macdonald operators have been developed \cite{BorodinCorwin2011Macdonald,BorodinCorwinSasamoto2012}. In \cite{IMS_skew_RSK} we succeeded in constructing a generalization of the RSK algorithm,
        which allows us to 
        study positive temperature KPZ models by connecting them to free fermions at finite temperature, both in full and half space. Currently, our methods only work for the study of ``single point" functions, as in the works of \cite{johansson2000shape, baik_rains2001algebraic,baik_rains2001asymptotics,baik_rains2001symmetrized}. It would be of great interest to understand if these can be extended to extract information for ``multi point" observables. For this task a more careful analysis of our bijective construction seems to be required.

    \subsection{Outline}
    
    In \cref{sec:qW} we recall basic notions in combinatorics and symmetric functions. There, we also define the $q$-Whittaker measure and the half space $q$-Whittaker measure. In \cref{sec:KPZqW} we introduce the solvable models in the KPZ class we study in the rest of the paper, i.e. the $q$-PushTASEP, the ASEP and the Log Gamma polymer. In \cref{sec:pS} we define the periodic Schur measure and the free boundary Schur measure giving also Fredholm determinant and pfaffian formulas for them. In \cref{subs:matching_qW_schur} we state the fundamental identities relating $q$-Whittaker measures with periodic and free boundary Schur measure. In \cref{sec:KPZpS} we state Fredholm determinant and pfaffian formulas for all the models introduced in \cref{sec:KPZqW}. Their asymptotic limit is computed in \cref{sec:asymptotics} where we prove convergence to the GUE Tracy-Widom distribution and to the Baik-Rains crossover. In \cref{sec:KPZ} we will state and prove formulas for the solution of the KPZ/mSHE equation in full and half space. In \cref{subs:hsSHE_Baik_Rains} we will compute the large time fluctuations of the solution $\mathscr{Z}^{\mathrm{hs}}(0,t)$ of the half space mSHE \eqref{eq:SHE_hs} at the origin establishing Baik-Rains phase transition. In \cref{app:Fredholm} we recall basic definitions and properties of Fredholm determinants and pfaffians. Finally, in \cref{sec:LG} we will prove useful bounds needed to establish formulas for Log Gamma polymer model in full and half space.

    \subsection{Acknowledgments}
    The work of TS has been supported by JSPS KAKENHI Grant Nos. JP16H06338, JP18H03672, JP21H04432, No. JP22H01143. The work of TI has been supported by JSPS KAKENHI Grant Nos. JP16K05192, JP19H01793, JP20K03626, and JP22H01143. The work of MM has been supported by the European Union’s Horizon 2020 research and innovation programme under the Marie Skłodowska-Curie grant agreement No. 101030938.

\section{$q$-Whittaker Measure and its half space variant}
\label{sec:qW}

Here we define the $q$-Whittaker symmetric functions and recall some of their properties and notable summation identities. In \cref{subs:qW_measure,subs:hs_qW_measure} we define the $q$-Whittaker measures, which are important in the study of solvable models in the KPZ class.

\subsection{$q$-Whittaker and Schur polynomials}
    
    Symmetric functions we consider in this paper are labeled by partitions, i.e. weakly decreasing lists of nonnegative integers $\lambda = (\lambda_1 \ge \lambda_2 \ge \cdots \ge 0)$ such that  $\lambda_n=\lambda_{n+1} =\cdots =0$, for $n$ large enough. For a given partition $\lambda$ we define its transpose $\lambda' = (\lambda_1' \ge \lambda_2' \ge \cdots)$ setting $\lambda_i' = \# \{ j:\lambda_j\ge i \}$. Skew partitions, denoted as $\lambda/\rho$, are also used and these are pairs $\lambda,\rho$, satisfying $\lambda_i \ge \rho_i$ for all $i$. Two partitions $\lambda$ and $\rho$ are said to interlace if $\lambda_1 \ge \rho_1 \ge \lambda_2 \ge \rho_2 \ge \cdots$ and this condition is denoted by the notation $\rho \prec \lambda$.
    
    For the next definition we recall the notion of complete homogeneous symmetric polynomials in $n$ variables $x_1,\dots,x_n$,
    \begin{equation}
        h_\ell(x) = \sum_{1\le i_1 \le \cdots \le i_\ell \le n} x_{i_1} \cdots x_{i_\ell},
    \end{equation}
    where $\ell \in \mathbb{Z}_{\ge 0}$. By convention we set $h_0(x)=1$ and $h_{\ell}(x)=0$ for $\ell<0$.
    \begin{definition}[Schur polynomials]
        For any skew partition $\lambda/\rho$ and a set of $n$ variables $x=(x_1,\dots,x_n)$, we define the \emph{skew Schur polynomial}
    \begin{equation} \label{eq:Jacobi_Trudi}
        s_{\lambda/\rho}(x) = \det\left[ h_{\lambda_i - \rho_j - i+j}(x) \right]_{i,j=1}^N,
    \end{equation}
    where the number $N$ is large enough so that $\lambda_N=\rho_N=0$. Setting $\rho=(0,0,\dots)$ we obtain the Schur polynomial $s_\lambda(x)$.
    \end{definition}
    
    \medskip
    
    Next we present a notable $q$-deformation of the Schur polynomials. For a fixed parameter $q \in (0,1)$, we recall the notion of $q$-Pochhammer symbols
    \begin{equation}
        (z;q)_n = \prod_{i=0}^{n-1} (1-z q^i)
        \qquad
        \text{for } n=0,1,\dots,
        \qquad
        (z;q)_\infty = \prod_{i=0}^{\infty} (1-z q^i)
    \end{equation}
    and of $q$-binomials
    \begin{equation}
        \Qbinomial{n}{j}{q} = \frac{(q;q)_n}{(q;q)_j (q;q)_{n-j}}.
    \end{equation}
    
    \begin{definition}[$q$-Whittaker polynomials]
        Given a skew partition $\mu / \varkappa$, the \emph{skew $q$-Whittaker polynomial} in $n$ variables $\mathscr{P}_{\mu /\varkappa}(x_1,\dots,x_n;q)$ is defined by the recursive relation
        \begin{equation}
            \mathscr{P}_{\mu/\varkappa}(x_1,\dots,x_n;q) = \sum_{\eta} \mathscr{P}_{\eta / \varkappa}(x_1,\dots,x_{n-1};q) \mathscr{P}_{\mu/\eta}(x_n;q),
        \end{equation}
        where
        \begin{equation}
            \mathscr{P}_{\mu/\eta}(z;q) = \mathbf{1}_{\eta \prec \mu} \prod_{i \ge 1} z^{\mu_i - \eta_i} \Qbinomial{\mu_i - \mu_{i+1}}{\mu_i - \eta_i}{q}.
        \end{equation}
        We also define the \emph{dual $q$-Whittaker polynomials}
        \begin{equation}
            \mathscr{Q}_\mu(x;q) = \mathdutchcal{b}_\mu(q) \mathscr{P}_\mu(x;q) 
            \qquad
            \text{and}
            \qquad
            \mathscr{Q}_{\mu/\eta}(x;q) = \frac{\mathdutchcal{b}_\mu(q)}{\mathdutchcal{b}_\eta(q)} \mathscr{P}_{\mu/\eta}(x;q),
        \end{equation}
        where
        \begin{equation} \label{eq:b_mu}
            \mathdutchcal{b}_\mu(q) = \prod_{i\ge 1} \frac{1}{(q;q)_{\mu_i - \mu_{i+1}}}.
        \end{equation}    
    \end{definition}
    Setting $q=0$ both polynomials $\mathscr{P}_{\mu / \varkappa}$ and $\mathscr{Q}_{\mu / \varkappa}$ become the skew Schur polynomials $s_{\mu / \varkappa}$ and this gives a well known alternative definitions of $s_{\mu / \varkappa}$.
    
    The $q$-Whittaker polynomials can be seen as a special case of the more general Macdonald polynomials \cite[Section VI]{Macdonald1995}. These are special symmetric polynomials depending on two parameters $q,t$ and denoted commonly as $P_\mu(x;q,t)$. Setting $t=0$ we recover the $q$-Whittaker polynomial  $\mathscr{P}_\mu(x;q) = P_\mu(x;q,t=0)$. Another interesting particular case of the Macdonald polynomials is given setting $q=0$ and keeping $t$ positive. These are the Hall-Littlewood polynomials, which are also of interest in the context of integrable probability and we will introduce them in the next subsection, albeit through a different route. 

    \subsection{Specializations of symmetric functions}
    The $q$-Whittaker polynomials enjoy the stability property
    \begin{equation} \label{eq:stability_qW}
        \mathscr{P}_{\mu/\eta}(x_1,\dots,x_{n-1},x_n=0;q) = \mathscr{P}_{\mu/\eta}(x_1,\dots,x_{n-1};q)
    \end{equation}
    and this allows, through an inverse limit procedure, to define them in the algebra of symmetric functions $\mathbf{\Lambda}_\mathbb{C}$. In \cite[Section VI.4]{Macdonald1995}, this is explained in the more general case of Macdonald functions. Implications of this observation include the extension of the notion of ``variables" of a symmetric function, which then become evaluation homomorphisms of the algebra $\mathbf{\Lambda}_\mathbb{C}$.
        
    \begin{definition}
        A \emph{specialization} $\varrho$ of the algebra of symmetric functions is an algebra homomorphism $\varrho : \mathbf{\Lambda}_\mathbb{C} \mapsto \mathbb{C}$. If $f \in \mathbf{\Lambda}_\mathbb{C}$ we will use the notation $ f(\varrho) = \varrho(f)$.
    \end{definition}
    
    When dealing with explicit specializations it is convenient to define their action on the algebraic basis of $\mathbf{\Lambda}_\mathbb{C}$ of power sum symmetric functions $\{ p_n : n \in \mathbb{Z}_{\ge 0} \}$, setting the values $\varrho(p_n)$. Rules for addition and scalar multiplication of specializations are straightforward from the definition and we have
    \begin{equation}
        (\varrho_1+\varrho_2) (p_n) = \varrho_1 (p_n) + \varrho_2 (p_n),
        \qquad
        ( c \varrho) (p_n) = c^n \varrho(p_n),
    \end{equation}
    for all specializations $\varrho, \varrho_1, \varrho_2$ and $c\in \mathbb{C}$. Three types of specializations will be considered mainly in this paper. These are:
    \begin{enumerate}
        \item \emph{alpha specializations} of complex parameters $x=(x_1,x_2,\dots)$, where
        \begin{equation}
            \varrho: p_n \mapsto p_n(x) = x_1^n + x_2^n  + \cdots,   
        \end{equation}
        evaluates the power sum symmetric functions numerically at $x_1,x_2,\dots$ and we always assume that $\sum_{i\ge 1} |x_i|<\infty$;
        \item \emph{$q$-beta specializations} of complex parameters $x=(x_1,x_2,\dots)$, where one sets
        \begin{equation} \label{eq:q_beta_spec}
            \varrho(p_n) =(-1)^{n-1}  (1-q^n) p_n(x),
        \end{equation}
        assuming again $\sum_{i\ge 1} |x_i|<\infty$. For brevity we will denote the $q$-beta specialization of a symmetric function $f$ by $f(\widehat{x})$;
        \item \emph{Plancherel specializations} of parameter $\tau  \in \mathbb{C}$ defined by $\varrho(p_n) = \delta_{n,1} \tau$.
    \end{enumerate}
    
    Evaluating $q$-beta specializations of $q$-Whittaker functions yields another remarkable family of symmetric polynomials (or functions)
    \begin{equation} \label{eq:Hall_Littlewood}
        \mathscr{H}_{\mu' / \eta'}(x;q) =  \mathscr{Q}_{\mu / \eta}(\widehat{x};q),
    \end{equation}
    called \emph{Hall-Littlewood polynomials}; see \cite[Section VI eq. (5.1)]{Macdonald1995}. In literature the Hall-Littlewood polynomials usually depend on a parameter $t$, which is $q$ in the notation \eqref{eq:Hall_Littlewood}. Because of such simple relation between $\mathscr{P}$ and $\mathscr{H}$ functions, in this paper we will not emphasize their difference.
    
    \begin{remark}[Positive specializations] \label{rem:specializations}
        For probabilistic applications we will be interested in $q$-Whittaker-\emph{positive specializations} $\varrho$, which are defined by the property that $\varrho ( \mathscr{P}_{\mu} ) \ge 0$ for all partitions $\mu$. These were characterized by Matveev in \cite{Matveev_Kerov_conjecture} for the more general case of Macdonald functions. Adapting \cite[Theorem 1.4]{Matveev_Kerov_conjecture} to our case, we see that a specialization $\varrho$ is $q$-Whittaker-positive if and only if it can be expressed as
        \begin{equation} \label{eq:specializations}
            \varrho=\varrho_x + \varrho_{\widehat{y}} +\varrho^{\mathrm{pl}}_\tau
        \end{equation}
        where $\varrho_x,\varrho_{\widehat{y}},\varrho^{\mathrm{pl}}_\tau$ are respectively an alpha, $q$-beta and Plancherel specializations of parameters $x=(x_1,x_2,\dots)$, $y=(y_1,y_2,\dots)$ and $\tau$ such that $x_i, y_i, \tau \ge 0 $ for all $i$ and
        \begin{equation}
            \sum_{i=1}^\infty (x_i +y_i) <\infty.
        \end{equation}
        
        Analogously \emph{Schur-positive} specializations $\varrho$ are defined imposing $\varrho(s_\mu) \ge 0$ for all $\mu$. In this case the Edrei-Thoma's theorem \cite{ASW52,Edrei53,whitney_totally_positive}, says that a specialization is Schur positive if and only if it is of the form \eqref{eq:specializations}, setting $q=0$ in $\varrho_{\widehat{y}}$. 
    \end{remark}
    
    \subsection{Summation identities}
    
    The $q$-Whittaker functions possess a number of interesting properties and satisfy notable summation identities. Summing over products of two $q$-Whittaker functions one obtains the \emph{Cauchy identity} that reads
    \begin{equation} \label{eq:skew_CI}
        \sum_{\mu}  \mathscr{P}_{\mu/\nu}(a;q) \mathscr{Q}_{\mu/\varkappa}(b;q) = \Pi(a,b) \sum_{\eta}  \mathscr{P}_{\varkappa/\eta}(a;q) \mathscr{Q}_{\nu/\eta}(b;q),
    \end{equation}
    where
    \begin{equation} \label{eq:pi_and_H}
        \Pi(a,b) =  \prod_{k \ge 0} H(q^k a,b)
        \qquad
        \text{and}
        \qquad
        H(a,b) = \exp \left\{ \sum_{n \ge 1} \frac{ p_n(a) p_n(b) }{n} \right\}.
    \end{equation}
    Setting $\nu=\varkappa=(0,0,\dots)$, identity \eqref{eq:skew_CI} simplifies as
    \begin{equation} \label{eq:macdonald_CI}
        \sum_{\mu}  \mathscr{P}_\mu(a;q) \mathscr{Q}_\mu(b;q) = \Pi(a;b).
    \end{equation}
    Here $a$ and $b$ are generic specializations for which the function $H(a,b)$ is numerically convergent. In the simplest case, when $a$ and $b$ are alpha specializations of finitely many parameters $(a_1,\dots,a_n)$, $(b_1,\dots,b_t)$, such that $|a_i b_j|<1$ for all $i,j$, we have
    \begin{equation}
        \Pi(a,b) = \prod_{i=1}^n \prod_{j=1}^t \frac{1}{(a_i b_j;q)_\infty} .
    \end{equation}
        
    Summations over single $q$-Whittaker functions yield the \emph{Littlewood identity}
    \begin{equation}     \label{eq:macdonald_LI}
        \sum_{\mu} \mathdutchcal{b}^{\mathrm{el}}_\mu(q) \mathscr{P}_\mu(a;q) = \widetilde{\Pi} (a),
    \end{equation}
    where $\mathdutchcal{b}^{\mathrm{el}}_\mu$ is defined as
    \begin{equation}     \label{eq:b_q_z}
        \mathdutchcal{b}_\mu^{\mathrm{el}}(q) = \prod_{i=2,4,6,\dots} \frac{\mathbf{1}_{\mu_{i-1} = \mu_{i}}}{(q;q)_{\mu_{i} - \mu_{i+1}}},
    \end{equation}
    and
    \begin{equation} \label{eq:pi_tilde}
        \widetilde{\Pi}(a) = \prod_{k \ge 0} \widetilde{H}(q^{k} a),
        \qquad
        \qquad
        \widetilde{H}(a) = \exp \left\{ \sum_{n \ge 1} \frac{p_{2n-1}(a)}{2n-1} + \frac{p_{n}(a)^2}{2n} \right\}.
    \end{equation}
    In \eqref{eq:b_q_z} the superscript ${}^\mathrm{el}$ means ``even length" as $\mathdutchcal{b}^\mathrm{el}_\mu$ vanishes unless $\mu'$ is even, i.e. $\mu'_i$ is even for all $i$. 
    When $a$ is an alpha specialization of finitely many parameters $(a_1,\dots,a_n)$ we have
    \begin{equation}
        \widetilde{\Pi}(a) = 
        \prod_{1\le i < j \le n} \frac{1}{(a_i a_j ;q)_\infty},
    \end{equation}
    as it can be verified by evaluating the summations in the expression of $\widetilde{H}(a)$.
    
    \medskip
    
    Interestingly, using homogeneity of Schur functions and their skew Cauchy identities it is possible to reproduce $\Pi(a,b),\widetilde{\Pi}(a)$ appearing on the right hand side of \eqref{eq:macdonald_CI}, \eqref{eq:macdonald_LI}. We have \cite[Corollary 6.7]{sagan1990robinson}
    \begin{equation} \label{eq:skew_shur_CI}
        \sum_{\lambda,\rho} q^{|\rho|} s_{\lambda/\rho}(a) s_{\lambda/\rho}(b) = \frac{1}{(q;q)_\infty} \Pi(a;b) 
    \end{equation}
    and \cite[Corollary 6.10]{sagan1990robinson}
    \begin{equation} \label{eq:skew_shur_LI}
        \sum_{\substack{\lambda,\rho: \\ \lambda', \rho' \text{ even}}} q^{|\rho|/2} s_{\lambda/\rho}(a) = \frac{1}{(q;q)_\infty} \widetilde{\Pi}(a).
    \end{equation}
    Indeed, similarities between identities \eqref{eq:skew_shur_CI}, \eqref{eq:skew_shur_LI} and \eqref{eq:macdonald_CI}, \eqref{eq:macdonald_LI} are not merely a coincidence and their mutual relations were understood in \cite{IMS_matching,IMS_skew_RSK}. We recall these results in the following theorems.
    
    \begin{theorem} \label{thm:qW_and_Schur_1}
        Fix $|q|<1$ and specializations of symmetric functions $a$ and $b$. Then, for all $k=0,1,2,\dots$, we have
        \begin{equation} \label{eq:qW_and_Schur_1}
            \sum_{\ell=0}^k \frac{q^\ell}{(q;q)_\ell} \sum_{ \mu: \mu_1 = k - \ell} \mathscr{P}_\mu (a;q) \mathscr{Q}_\mu (b;q) = \sum_{\lambda,\rho : \lambda_1= k} q^{|\rho|} s_{\lambda / \rho}(a) s_{\lambda / \rho}(b).
        \end{equation}
    \end{theorem}

    Notice that equality \eqref{eq:qW_and_Schur_1} represents a mutual refinement of \eqref{eq:macdonald_CI} and \eqref{eq:skew_shur_CI}. A similar mutual refinement between \eqref{eq:macdonald_LI} and \eqref{eq:skew_shur_LI} is given below.

    \begin{theorem} \label{thm:qW_and_Schur_2}
        Fix $|q|<1$ and a specialization of symmetric functions $a$. Then, for all $k=0,1,2,\dots$, we have
        \begin{equation} \label{eq:qW_and_Schur_2}
            \sum_{\ell=0}^k \frac{q^{ \ell}}{(q;q)_\ell} \sum_{ \mu: \mu_1 = k - \ell} \mathdutchcal{b}_\mu^{\mathrm{el}}(q) \mathscr{P}_\mu (a;q) = \sum_{\substack{\lambda,\rho :\lambda',\rho' \, \mathrm{ even} \\ \lambda_1= k}} q^{|\rho|/2} s_{\lambda / \rho}(a).
        \end{equation}
    \end{theorem}

    In \cite{IMS_skew_RSK}, \cref{thm:qW_and_Schur_1,thm:qW_and_Schur_2} were proven bijectively only in the case where $a$ and $b$ are alpha specializations. Nevertheless, as pointed out in \cite[Remark 10.8 and Remark 10.14]{IMS_skew_RSK} the same identities hold for general specializations following simple stability considerations such as \eqref{eq:stability_qW}.

\subsection{$q$-Whittaker measure} \label{subs:qW_measure}
    The \emph{$q$-Whittaker measure} \cite{BorodinCorwin2011Macdonald} is the probability measure on the set of partitions given by  
    \begin{equation}
    \label{bqWM}
        \mathbb{W}^{(q)}_{a;b} (\mu) = \frac{1}{\Pi(a;b)} \mathdutchcal{b}_\mu(q) \mathscr{P}_\mu (a;q)\mathscr{P}_\mu (b;q),
    \end{equation}
    where $a, b$ are ($q$-Whittaker) positive specializations such that the function $H(a,b)$ is numerically convergent. In case $a$ and $b$ are $q$-beta specializations $\widehat{a}, \widehat{b}$ of parameters $(a_1,\dots ,a_n)$, $(b_1,\dots ,b_t)$, the $q$-Whittaker polynomials $\mathscr{P}_{\mu}(\widehat{a};q), \mathscr{P}_{\mu}(\widehat{b};q)$ turn into the \emph{dual} Hall-Littlewood polynomials labeled by the transposed partitions $\mu'$ and the measure $\mathbb{W}^{(q)}_{\widehat{a};\widehat{b}}$ turns into the Hall-Littlewood measure \cite[Chapter VI eq. (5.1)]{Macdonald1995}. This is another important measure in integrable probability and it is related to KPZ models as we will discuss below; for more on the subject see \cite{vuletic2009plane,BorodinBufetovWheeler2016,BufetovMatveev2017,barraquand2018,BufetovPetrovYB2017}. In this paper we will not introduce a different notation for the Hall-Littlewood measure and we will consider it as a particular specialization of the $q$-Whittaker measure.

\subsection{Half space $q$-Whittaker measure} \label{subs:hs_qW_measure}
    Introducing the function
    \begin{equation}
        \mathscr{E}_\mu(\gamma;q) = \sum_{\eta} \mathdutchcal{b}_\eta^{\mathrm{el}}(q) \mathscr{Q}_{\mu / \eta}(\gamma;q),
    \end{equation}
    we define, as in \cite{barraquand_half_space_mac}, the \emph{half space $q$-Whittaker measure}
    \begin{equation}
    \label{hqWM}
        \mathbb{HW}^{(q)}_{a;\gamma} (\mu) = \frac{1}{\widetilde{\Pi}(a) \Pi(a;\gamma)} \mathscr{E}_\mu (\gamma,q) \mathscr{P}_\mu (a,q).
    \end{equation}
    Here $a$ and $\gamma$ are $q$-Whittaker positive specializations such that functions $\widetilde{H}(a)$ and $H(a,\gamma)$ are numerically convergent. Combining Cauchy identities \eqref{eq:skew_CI} and \eqref{eq:macdonald_LI} one can verify that the normalization constant in the right hand side of \eqref{hqWM} is correct. 
    
    We will mostly be interested in the case when $\gamma$ is a single variable alpha specialization. In such situation, the half space $q$-Whittaker measure possesses a symmetry that allows to distribute the parameter $\gamma$ of the function $\mathscr{E}$ on the specialization of the $\mathscr{P}$ function.

    \begin{proposition}[\cite{barraquand_half_space_mac} Proposition 2.6] \label{prop:symmmetry_half_space_qW}
        Let $\boldsymbol{a}$ be a $q$-Whittaker positive specialization and $\gamma$ be a single variable alpha specialization. Consider a random partition $\mu$ distributed according to $\mathbb{HW}^{(q)}_{\boldsymbol{a};\gamma}$ and another random partition $\lambda$ distributed according to $\mathbb{HW}^{(q)}_{\boldsymbol{a}';0}$, where $\boldsymbol{a}'=(a,\gamma)$. Then $\{\mu_1,\mu_3,\dots\}$ and $\{ \lambda_1,\lambda_3,\dots \}$ have the same distribution.
    \end{proposition}

    We will be interested in the marginal $\mu_1$ of the half space $q$-Whittaker measure as it relates to solvable KPZ models; see \cref{sec:KPZqW}. For this purpose the symmetry reported in \cref{prop:symmmetry_half_space_qW} is very convenient as it reduces the study of the measure with specializations $\boldsymbol{a}$ and $\gamma$ to the simpler case when $\gamma=0$. 

\section{Models in the KPZ class and $q$-Whittaker measures }
\label{sec:KPZqW}
In this section, we explain the relation between KPZ models and $q$-Whittaker measure, both in the whole space case and in the half space. 

\subsection{$q$-PushTASEP} \label{subs:qPushTASEP}
    The \emph{geometric $q$-PushTASEP} is a stochastic interacting particle system introduced in \cite{MatveevPetrov2014}. In order to describe the random jumps of particles, we recall the \emph{$q$-deformed beta binomial distribution}
    \begin{equation}
        \boldsymbol{\varphi}_{q,\xi,\eta}( s | m) = \xi^s \frac{(\eta/\xi;q)_s (\xi;q)_{m-s}}{(\eta;q)_m} \frac{(q;q)_m}{(q;q)_s (q;q)_{m-s}},
    \end{equation}
    which is a probability distribution on $s \in \{0,\dots,m\}$, for any $m\in \mathbb{Z}_{\ge 0}$, whenever parameters $q,\xi,\eta$ make the above expression nonnegative. Two such choices are given by $q,\xi \in [0,1), \eta \in [0,\xi]$ and $q>1,\xi \in q^{\mathbb{Z}_{\le 0}},\eta=0$. Setting $\eta=0, m=\infty$ and taking $q,\xi \in [0,1)$, we obtain the \emph{$q$-Geometric} distribution of parameter $\xi$,
    \begin{equation}
        \boldsymbol{\varphi}_{q,\xi,0}( s | \infty) = \xi^s \frac{(\xi;q)_{\infty}}{(q;q)_{s}}, \qquad s\in \mathbb{Z}_{\ge 0},
    \end{equation}
    which we denote by $q$-Geo$(\xi)$.
    
    \medskip
    
    In the $q$-PushTASEP, a finite number $N$ of particles occupy, following the exclusion rule, sites of the lattice $\mathbb{Z}$ and their position is recorded in an array $\mathsf{X}(T) = (\mathsf{x}_1(T) < \cdots < \mathsf{x}_N(T) )$, where $T \in \mathbb{Z}_{\ge 0}$ denotes the time. During the time step $T \to T+1$, the array $\mathsf{X}(T)$ is updated sequentially from left to right and
    \begin{equation} \label{eq:q_PushTASEP}
        \mathsf{x}_k(T+1) = \mathsf{x}_k(T) + V_{k,T} + W_{k,T},
        \qquad 
        \text{for } k=1,\dots,N,
    \end{equation}
    where $V_{k,T} \sim q$-$\mathrm{Geo}(a_k b_{T+1})$ and 
    \begin{equation} \label{eq:W_kt}
        W_{k,T} \sim \boldsymbol{\varphi}_{q^{-1}, q^{\mathrm{gap}_k(T)},0} (\bullet \, | \, \mathsf{x}_{k-1}(T+1) - \mathsf{x}_{k-1}(T) ).
    \end{equation}
    In the previous expression we assume $\mathrm{gap}_j(T) = \mathsf{x}_j(T) - \mathsf{x}_{j-1}(T) -1$ and we set, by convention $\mathsf{x}_0(T)=-\infty$. In \eqref{eq:q_PushTASEP} we see that the random movement of a particle splits into two independent contributions: $W_{k,T}$ consists in a random pushing that the particle $\mathsf{x}_{k-1}$ impresses on $\mathsf{x}_k$, while $V_{k,T}$ is an independent jump of $\mathsf{x}_k$.
    
    Markov rules \eqref{eq:q_PushTASEP}, although explicit are far from intuitive or natural. An interesting simplification can be obtained considering a \emph{continuous time} limit, setting $b_i=\varepsilon$, $T=\tau/\varepsilon$ and $\varepsilon\to 0$. This procedure yields the continuous time $q$-PushTASEP first defined in \cite{BorodinPetrov2013NN}; for a visualization see \cref{fig:q_PushTASEP}.

    \medskip
    
    The $q$-PushTASEP is deeply connected to marginals of the $q$-Whittaker measure, as recalled in the next proposition.
    
    \begin{proposition}[\cite{MatveevPetrov2014}, Section 6.3] \label{prop:qPushTASEP_qW}
        Consider $\boldsymbol{a},\boldsymbol{b}$ alpha specializations of parameters respectively $(a_1,\dots,a_N) \in (0,1)^N$ and $(b_1,\dots,b_T) \in (0,1)^T$. Let $\mu \sim \mathbb{W}_{\boldsymbol{a};\boldsymbol{b}}^{(q)}$ and let $\mathsf{X}(T)$ be a $q$-PushTASEP under initial conditions $\mathsf{x}_k(0)=k$ for $k=1,\dots,N$. Then, we have the following equality in distribution
        \begin{equation}
            \mathsf{x}_N(T) \stackrel{\mathcal{D}}{=} \mu_1 + N.
        \end{equation}
    \end{proposition}
    
    \begin{figure}
        \centering
        \includegraphics{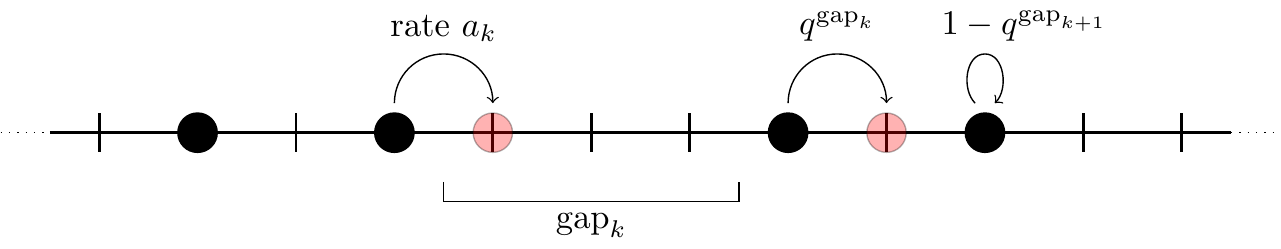}
        \caption{Visualization of the dynamics in the continuous time $q$-PushTASEP. Each particle $\mathsf{x}_k$ possesses an exponential clock with parameter $a_k$ and once the clock rings at time $\tau$ it performs a jump to its right $\mathsf{x}_k(\tau_+) = \mathsf{x}_k(\tau)+1$. The movement of the $k$-th particle triggers an instantaneous pushing mechanism and $\mathsf{x}_{j}(\tau_+) = \mathsf{x}_{j}(\tau) + 1$ with probability
        $q^{\mathrm{gap}_j(\tau)} \mathbf{1}_{\mathsf{x}_{j-1}(\tau_+) = \mathsf{x}_{j-1}(\tau) + 1}$, for all $j>k$.}
        \label{fig:q_PushTASEP}
    \end{figure}
    
\subsection{$q$-PushTASEP with particle creation} \label{subs:hs_qPushTASEP}

The $q$-PushTASEP admits a solvable variant where the stochastic dynamics 
do not preserve the number of particles, which are added to the system at each time step. This model was introduced in \cite{barraquand_half_space_mac} and we will refer to it as \emph{$q$-PushTASEP with particle creation}. At any time $T \in \mathbb{Z}_{\ge 0}$, the system consists of exactly $T+1$ particles located in the lattice $\mathbb{Z}_{\ge 0}$ and whose positions are
\begin{equation}
   0 = \mathsf{x}_0^{\mathrm{hs}}(T) < \mathsf{x}_1^{\mathrm{hs}}(T) < \cdots < \mathsf{x}_T^{\mathrm{hs}}(T) < \infty,
\end{equation}
recorded in the array $\mathsf{X}^{\mathrm{hs}}(T)$. During the time step from $T$ to  $T+1$, the configuration gets updated sequentially from the particle with label 1 to the right following the $q$-PushTASEP rules described by \eqref{eq:q_PushTASEP}, where we assume $V_{k,T} \sim q\text{-}\mathrm{Geo}(a_k a_{T+1})$. Once the position of the rightmost particle $\mathsf{x}_T^{\mathrm{hs}}(T+1)$ has been determined a new particle is added to its right and $\mathsf{x}_{T+1}^{\mathrm{hs}}(T+1) = \mathsf{x}_T^{\mathrm{hs}}(T+1) +1 + \widetilde{V}_{T+1}$, where $\widetilde{V}_{T+1} \sim q$-$\mathrm{Geo}(\gamma a_{T+1})$ is an independent increment. Notice that the $0$-th particle never moves. Here we have assumed that $a_i,\gamma$ are positive numbers such that $a_ia_j,\gamma a_i < 1$ for all $1 \le i<j$. By construction, we will only consider \emph{empty initial conditions}, that is $\mathsf{X}^{\mathrm{hs}}(0) = \{ \mathsf{x}_0(0)=0 \}$. 

\medskip 

The salient feature of this model is the fact that during the evolution, new particles are added at the right of the system with spacing modulated by the parameter $\gamma$. Such property, makes the $q$-pushTASEP with particle creation a regularization of the KPZ equation in half space, although to see this several non trivial scalings are necessary. As for the ``full space" version of the $q$-PushTASEP, this is, admittedly not necessarily the most intuitive model to consider. More canonical models of growth with boundary effects have been considered in the past. These include the Asymmetric Simple Exclusion Process \cite{barraquand2018,Tracy_Widom_half_space_asep,tracy_widom_open_asep_delta_bose} or other discrete polymer models \cite{OSZ2012}, but unfortunately their solvability, for generic values of boundary strength, does not lead directly to manageable or rigorous exact formulas. 

\medskip

The next proposition provides the relation between the $q$-PushTASEP with particle creation and the half space $q$-Whittaker measure.

\begin{proposition}[\cite{barraquand_half_space_mac} Proposition 4.24] \label{prop:matching_hs_qPushTASEP_qW}
    Consider parameters $a_1,\dots,a_N,\gamma \in (0,1)$ and let $\boldsymbol{a}$  be the alpha specialization with parameters $(a_1,\dots,a_N)$. Let $\mu \sim \mathbb{HW}^{(q)}_{\boldsymbol{a};\gamma}$ and let $\mathsf{X}^{\mathrm{hs}}(T)$ be the $q$-PushTASEP with particle creation and empty initial conditions with parameters $\boldsymbol{a},\gamma$. Then, we have the following equality in distribution
    \begin{equation}
        \mathsf{x}^{\mathrm{hs}}_N(N) = \mu_1 + N.
    \end{equation}
\end{proposition}

\subsection{ASEP} \label{subs:ASEP}
    The asymmetric simple exclusion process is a canonical model of transport of particles in one dimension. It was introduced in \cite{Spitzer1970} and its stochastic integrable stucture has been investigated intensively in the last two decades \cite{schutz1997dualityASEP,tracy2008fredholm,TW_ASEP1,TW_ASEP2,ImamuraSasamoto2011current,BorodinCorwinSasamoto2012,BO2016_ASEP}. On the lattice $\mathbb{Z}$ at time $\tau \ge 0$ we consider a set of particles, with locations\footnote{Notice that in this case the particle ordering is \emph{opposite} to the one adopted for the $q$-PushTASEP in \cref{subs:qPushTASEP,subs:hs_qPushTASEP}.} $\mathsf{x}_1(\tau) > \mathsf{x}_2(\tau) > \cdots$. Each particle possesses two exponential clocks of rate respectively $q \in [0,1) $ and $1$. When the first clock rings the particle attempts a jump to its left neighboring site and when the second one rings it attempts a jump to its right. Jumps are successful provided the target location is empty; see \cref{fig:ASEP}. Particle configurations we consider always possess a rightmost particle, whose location is $\mathsf{x}_1$. For this we can define the integrated current through a location $x$ as 
    \begin{equation}
        \mathsf{J}(x,\tau) = \sum_{k \ge 1} \mathbf{1}_{\mathsf{x}_k(\tau) > x} - \sum_{k \ge 1} \mathbf{1}_{\mathsf{x}_k(0) > x}.
    \end{equation}
    The ASEP is deeply related to the Hall-Littlewood measure. This was essentially discovered in \cite{BorodinBufetovWheeler2016}, where authors established a more general correspondence between a marginal of the Hall-Littlewood measure and the stochastic six vertex model. The convergence of the stochastic six vertex model to the ASEP was proven in \cite{Aggarwal_6v_to_ASEP}.
    
    \begin{proposition} \label{prop:ASEP_matching_mu_1}
        Consider the $q$-Whittaker measure $\mathbb{W}_{\widehat{a};\widehat{b}}^{(q)}$ where $\widehat{a},\widehat{b}$ are $q$-beta specializations of parameters $(a_1,\dots,a_N)$ and $(b_1,\dots,b_T)$. Let $\mu \sim \mathbb{W}_{\widehat{a};\widehat{b}}^{(q)}$ and consider the scaling 
        \begin{equation} \label{eq:ASEP_scaling}
            T= \frac{\tau}{\varepsilon},
            \qquad
            N= \frac{\tau}{\varepsilon} + x,
            \qquad
            a_i=b_j=1-\frac{(1-q)\varepsilon}{2},
            \qquad
            \mu_1 = N - \mathsf{J}^{(\varepsilon)}(x,\tau).
        \end{equation}
        Then, as $\varepsilon \to 0$, the random variable $\mathsf{J}^{(\varepsilon)}(x,\tau)$ converges in law to the integrated current $\mathsf{J}(x,\tau)$ of the ASEP with initial conditions $\mathsf{x}_k(0) = -k$ for $k=1,2\dots$.
    \end{proposition}

    \begin{figure}
        \centering
        \includegraphics{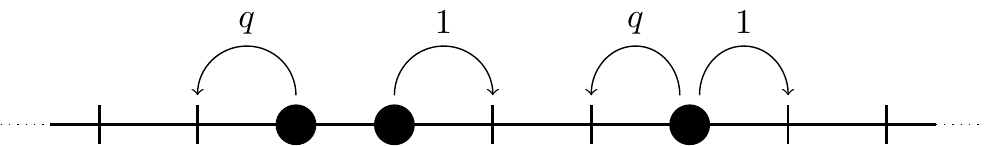}
        \caption{Visualization of the dynamics in the ASEP.}
        \label{fig:ASEP}
    \end{figure}
    
    As hinted earlier, the ASEP can be seen as a particular case of a more general stochastic solvable model, i.e. the stochastic six vertex model \cite{GwaSpohn1992,BCG6V}. In fact, combinatorial techniques allow to transform the stochastic six vertex model into many more solvable systems, including the $q$-PushTASEP; see \cite{BorodinPetrov2016inhom,CorwinPetrov2015,BufetovMucciconiPetrov2018}. Since the stochastic six vertex model is related to the Hall-Littlewood measure \cite{BorodinBufetovWheeler2016}, correspondences as those of \cref{prop:qPushTASEP_qW} or \cref{prop:ASEP_matching_mu_1}, leading to results we will elaborate in the following sections, could be given for a full hierarchy of solvable models. To keep our argument concise we will not report such results here.
    
\subsection{Log gamma polymer} \label{subs:log_gamma}
    The \emph{Log Gamma polymer} model is a discrete solvable random polymer model introduced by Seppal{\"a}inen in \cite{Seppalainen2012}.
    We recall that an \emph{inverse gamma random variable of parameter $\alpha$}, which we will denote as $\mathrm{Gamma}^{-1}(\alpha)$, has probability density function
    \begin{equation}
        p_{\mathrm{Gamma}^{-1}(\alpha)}(x) = \frac{x^{-\alpha-1} e^{-1/x}}{\Gamma(\alpha)}.
    \end{equation}
    On each site $(i,j)$ of the planar lattice $\Lambda=\mathbb{Z}_{\ge 0} \times \mathbb{Z}_{\ge 0}$ we place a random variable $w_{i,j} \sim \mathrm{Gamma}^{-1}(A_i + B_j)$, where $A_i,B_j$ are positive real numbers.
    For a fixed point $(T,N)$, we consider the set $\Pi_{(1,1)\to(T,N)}$ of up-right paths $\varpi = (\varpi_1,\dots,\varpi_{N+T-1})$ on $\Lambda$ with endpoints $\varpi_1=(1,1)$ and $\varpi_{N+T-1} = (T,N)$. Here by up-right path we mean that the increments $\varpi_{i+1} - \varpi_i$ are either $(1,0)$ or $(0,1)$ for any $i$; see \cref{fig:Log_Gamma_polymer}, left panel. To any path $\varpi$ we associate the random weight
    \begin{equation} \label{eq:weight_path}
        w(\varpi) = \prod_{\ell=1}^{N+T-1} w_{\varpi_{\ell}}
    \end{equation}
    and we define the partition function
    \begin{equation} \label{eq:Z(N,T)}
        Z(N,T) = \sum_{\varpi \in \Pi_{(1,1)\to(T,N)}} w(\varpi).
    \end{equation}
    
    \begin{remark}
        Let us explain the origin of the name ``Log Gamma" polymer.
        Canonically the partition function of a directed polymer with fixed endpoints at $(1,1)$ and $(T,N)$, has the form
    \begin{equation}
        Z(N,T) = \sum_{\varpi \in \Pi_{(1,1) \to (T,N)}} e^{\beta E(\varpi)},
    \end{equation}
    where $E$ is the energy of the polymer $\varpi$ and $\beta$ is the inverse temperature of the system. In case the energy is taken as a sum of independent terms $E(\varpi) = \sum_{\ell} E_{\varpi_\ell}$ and $E_{(i,j)} \sim - \log \mathrm{Gamma}(A_i+B_j)$ and the parameter $\beta=1$ we recover the partition function \eqref{eq:Z(N,T)}.  
    \end{remark}

\begin{figure}[ht]
    \centering
    \subfloat[]{ 
    \includegraphics[width=.35\linewidth]{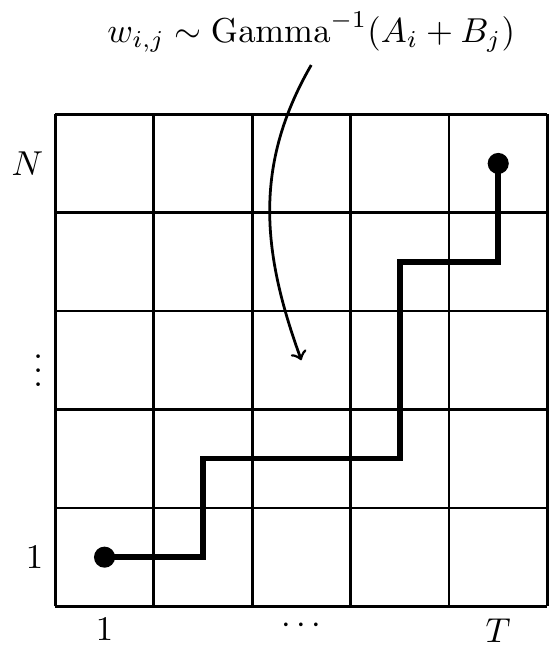} }%
    \hspace{.5cm}
    \subfloat[]{\includegraphics[width=.58\linewidth]{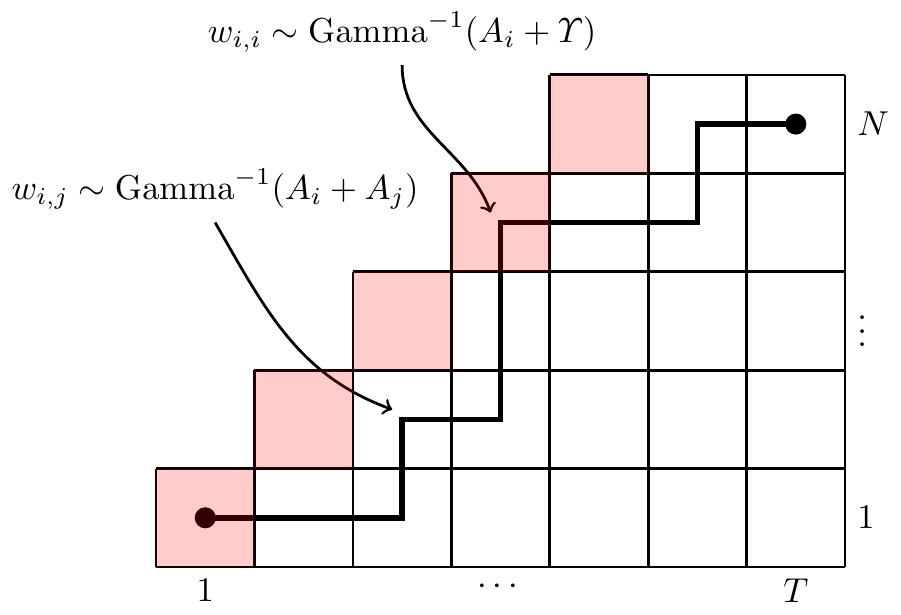} }%
    \caption{Visualizations of the Log Gamma polymer model in full space (a) and in half space (b).}
    \label{fig:Log_Gamma_polymer}
\end{figure}

    The Log Gamma polymer model arises as a scaling limit of the $q$-PushTASEP, as it was proven by Matveev and Petrov in \cite[Theorem 8.7]{MatveevPetrov2014}. A particular case of such result is stated in the next proposition.
    
    \begin{proposition} \label{prop:qPushTASEP_convergence}
        Let $\mathsf{X}(T)$ be a $q$-PushTASEP with initial conditions $\mathsf{x}_k(0)=k$ for $k=1,\dots,N$. Consider the scaling $q = e^{-\varepsilon}$, $a_i=e^{-\varepsilon A_i}$, $b_j=e^{-\varepsilon B_j}$ and 
        \begin{equation}
            \mathsf{x}_N(T) = (N+T-1) \varepsilon^{-1} \log \varepsilon^{-1} + \varepsilon^{-1} \log Z^{(\varepsilon)}(T,N).
        \end{equation}
        Then, as $\varepsilon \to 0$, the random variable $Z^{(\varepsilon)}(T,N)$ converges in distribution to $Z(T,N)$.
    \end{proposition}

\subsection{Half space Log Gamma Polymer}
    Here we consider a variant of the Log Gamma polymer model in the restricted lattice $\Lambda^{\mathrm{hs}} = \{ (i,j) \in \mathbb{Z}_{\ge 1} \times \mathbb{Z}_{\ge 1}: i \ge j  \}$. In this case to each site $(i,j)$ of $\Lambda^{\mathrm{hs}}$ we associate a random variable 
    \begin{equation} \label{eq:random_environment}
        w_{i,j} \sim 
        \begin{cases}
            \mathrm{Gamma}^{-1}(A_i + \varUpsilon) \qquad &\text{if } i=j,
            \\
            \mathrm{Gamma}^{-1}(A_i + A_j) \qquad &\text{if } i > j,
        \end{cases}
    \end{equation}
    where $\varUpsilon> - \min \{A_i :i=1,2,\dots \}$ is the \emph{boundary strength}. For any $N\le T$, define, as in the previous subsection, the set $\Pi^{\mathrm{hs}}_{(1,1) \to (T,N)}$ of up-right paths $\varpi = (\varpi_1 ,\dots , \varpi_{T+N-1})$ on $\Lambda^{\mathrm{hs}}$ with endpoints $\varpi_1=(1,1)$ and $\varpi_{T+N-1}=(T,N)$. The weight of a path $\varpi$ is defined as in \eqref{eq:weight_path} and the partition function of the half space Log Gamma polymer is the random function
    \begin{equation}
        Z^{\mathrm{hs}}(N,T) = \sum_{\varpi \in  \Pi^{\mathrm{hs}}_{(1,1) \to (T,N)} } w(\varpi).
    \end{equation}
    
    As in the whole space case, the random function $Z^\mathrm{hs}$ can be recovered as a scaling limit of the $q$-PushTASEP with particle creations, as recalled next.
    
\begin{proposition} \label{prop:hs_qPushTASEP_convergence}
    Let $\mathsf{X}^{\mathrm{hs}}(T)$ be a $q$-PushTASEP with particle creation and empty initial conditions. Consider the scaling $q=e^{-\varepsilon}$, $a_i = e^{-\varepsilon A_i}, \gamma = e^{-\varepsilon \varUpsilon}$ and
    \begin{equation} \label{eq:Z_hs_epsilon}
        \mathsf{x}^{\mathrm{hs}}_N(T) = (N+T-1) \varepsilon^{-1} \log \varepsilon^{-1} + \varepsilon^{-1} \log Z^{\mathrm{hs}}(T,N;\varepsilon),
    \end{equation}
    with $T \ge N$. Then, as $\varepsilon \to 0$, the random variable $Z^{\mathrm{hs}}(T,N;\varepsilon)$ converges in distribution to $Z^{\mathrm{hs}}(T,N)$.
\end{proposition}
\begin{proof}
    This is proven in \cite{barraquand_half_space_mac}, combining Proposition 6.27 and Proposition 4.24 of the same paper.
\end{proof}


\section{Periodic and free boundary Schur measures}
\label{sec:pS}

In this section we define two important objects: the periodic and the free boundary Schur measures. These are discrete models of free fermions at positive temperature and as such their law is described in terms of determinantal and pfaffian point processes. The fundamental identities \eqref{eq:qW_and_Schur_1} and \eqref{eq:qW_and_Schur_2} relate them precisely to the KPZ models discussed in the previous section. 

\subsection{Periodic Schur measure}
The \emph{periodic Schur measure} was first introduced by Borodin in \cite{borodin2007periodic} as a generalization of Okounkov's Schur measure \cite{okounkov2001infinite}. It is a measure on the set of partitions given by
\begin{equation}
\label{pS}
    \mathbb{PS}_{a,b}^{(q)}(\lambda) = \frac{1}{ (q;q)_\infty^{-1} \Pi(a;b) } \sum_{\rho} q^{|\rho|} s_{\lambda/\rho}(a) s_{\lambda/\rho}(b), 
\end{equation}
where $s_{\lambda/\rho}(x)$ is the skew Schur function \eqref{eq:Jacobi_Trudi}. Here $a,b$ are Schur-positive specializations and the normalization constant is a consequence of the Cauchy identity \eqref{eq:skew_shur_CI}. While the Schur measure is a model of free fermions at \emph{zero temperature}, the periodic Schur measure, after a simple trasformation describes the grand canonical ensemble of free fermions at \emph{positive temperature}, modulated by the parameter $q$; see \cite{betea_bouttier_periodic}.
As found in \cite{borodin2007periodic} a certain independent shift of coordinates of $\lambda$ transforms the periodic Schur measure into a determinantal point process. We say that an integer valued random variable $S$ has \emph{theta} distribution of parameters $q\in (0,1), \zeta>0$, in short $S \sim \mathrm{Theta}(\zeta;q)$, if
\begin{equation}
    \mathbb{P} (S = k)  = \frac{1}{\vartheta_3(\zeta; q)} q^{k^2/2} \zeta^k,
    \qquad
    k\in \mathbb{Z},
\end{equation}
where
\begin{equation} \label{eq:jacobi_theta}
    \vartheta_3(\zeta;q) = (q,- \sqrt{q} \zeta, -\sqrt{q}/\zeta ; q)_\infty    
\end{equation}
is the Jacobi theta function. The factorized form of the normalizaztion constant is the celebrated Jacobi triple product identity \cite[Theorem 12.3.2]{ismailBook}. We can now define the \emph{shift-mixed periodic Schur measure} as the point process
\begin{equation}
    \mathfrak{S}(\lambda, S) = \left( \lambda_i - i +\frac{1}{2} + S \right)_{i = 1,2,\dots } \subset \mathbb{Z}',
\end{equation}
where we used the notation $\mathbb{Z}' = \mathbb{Z} + \frac{1}{2}$, and $\lambda \sim \mathbb{PS}_{a,b}^{(q)}$, $S\sim \mathrm{Theta}(\zeta;q)$ are independent.
\begin{theorem}[\cite{borodin2007periodic}] \label{thm:periodic_Schur_det}
    The shift-mixed periodic Schur measure is a determinantal point process with correlation kernel
    \begin{equation} \label{eq:kernel_periodic_Schur}
        K(x,y) = \frac{1}{(2 \pi \mathrm{i})^2} \oint_{|z|=r} \frac{\diff z}{z^{x+1}} \oint_{|w|=r'} \frac{\diff w}{w^{-y+1}} \frac{F(z)}{F(w)} \kappa(z,w),
    \end{equation}
    where $x,y \in \mathbb{Z}'$ and 
    \begin{equation}
        F(z) = \prod_{\ell \ge 0} \frac{H(q^\ell a;z)}{H(q^\ell b;z^{-1})},
    \end{equation}
    \begin{equation}
        \kappa(z,w) = \sqrt{\frac{w}{z}} \frac{(q;q)_\infty^2}{(w/z,qz/w;q)_\infty} \frac{\vartheta_3 (\zeta z/w;q)}{\vartheta_3 (\zeta ;q)}.
    \end{equation}
    The function $H$ was defined in \eqref{eq:pi_and_H} and the integration contour's radii satisfy $1<r/r'<q^{-1}$ and $R^{-1} \le r,r' \le R$, where $R$ is the minimum of convergence radii of $H(a;z)$ and $H(b;z)$.
\end{theorem}

We will be interested in the rightmost coordinate of the shift-mixed periodic Schur measure $\lambda_1 + S$. Using basic notions of determinantal point processes its probability distribution can be written in terms of the Fredholm determinant of the kernel $K$ of \eqref{eq:kernel_periodic_Schur}. We have, for any $s\in \mathbb{R}$,
\begin{equation} \label{pSdet}
    \mathbb{P}[\lambda_1 + S \le s] = \det (1-K)_{\ell^2( \mathbb{Z}_{> s}' )},
\end{equation}
where $\mathbb{Z}_{> s}'$ denotes the set of half integers larger than to $s$.

\subsection{Alternative representation of the Fredholm determinant via Ramanujan's bilateral sum.} \label{subs:alternative_det}

To state the main result of the subsection we recall the identity
\begin{equation} \label{eq:ramanujan_psi}
    \sum_{m \in \mathbb{Z}'} \frac{z^m}{1+\zeta^{-1} q^{-m}} = z^{-1/2} \frac{(q;q)_\infty^2}{(1/z,qz;q)_\infty} \frac{\vartheta_3(\zeta z;q)}{\vartheta_3(\zeta;q)}, 
\end{equation}
which holds for $1<|z|<q^{-1}$ and it is a particular case of Ramanujan's ${}_1\psi_1$ sum \cite[Theorem 12.3.1]{ismailBook}.

\begin{proposition} \cite[Proposition 2.3]{IMS_matching} \label{prop:fredholm_det_M}
    Let $\lambda$ be distributed according to the periodic Schur measure \eqref{pS} and let $S\sim\mathrm{Theta}(\zeta;q)$. Then, for any $s \in \mathbb{Z}$, we have
    \begin{equation} \label{eq:fredholm_det_M}
        \mathbb{P}\left[ \lambda_1+S \le s \right] = \det\left[1 - f_{\zeta,s} \cdot M \right]_{\ell^2(\mathbb{Z}')},
    \end{equation}
    where $M$ is given by
    \begin{equation}
        M(m,n) = \oint_{|z|=r} \frac{\diff z}{2 \pi \mathrm{i}} \oint_{|w|=r'} \frac{\diff w}{2 \pi \mathrm{i}} \frac{F(z)}{F(w)} \frac{z^{n-1/2}}{w^{m+1/2}} \frac{1}{z-w},
    \end{equation}
    and $f_{\zeta,s}$ is the Fermi-Dirac distribution\footnote{Setting $q=e^{-\varepsilon/k_B T}$ and $\zeta = e^{\mu/k_B T}$, where $T$ is the temperature, $\varepsilon$ is the quantized energy, $k_B$ is Boltzmann constant and $\mu$ is the chemical potential, then $f_{\zeta,s}(m)$ assumes the canonical form of the Fermi-Dirac distribution.}
    \begin{equation} \label{eq:fermi_dirac}
        f_{\zeta,s} (m) = \frac{1}{1+\zeta^{-1}q^{-s-m}}.
    \end{equation}
\end{proposition}
\begin{proof}
    Using \eqref{eq:ramanujan_psi}, the kernel $K$ can be rewritten as
    \begin{equation}
        K (x,y) = \sum_{m \in \mathbb{Z}'} \frac{1}{1+\zeta^{-1} q^{-m}} \oint_{|z|=r} \frac{\diff z}{2\pi \mathrm{i}} \oint_{|w|=r'} \frac{\diff w}{2\pi \mathrm{i}} \frac{w^{y-1-m}}{z^{x+1-m}} \frac{F(z)}{F(w)}.
    \end{equation}
    Defining
    \begin{equation}
        A(x,m) = \oint_{|z|=r} \frac{\diff z}{2\pi \mathrm{i}} \frac{F(z)}{z^{x+1-m}}
        \qquad
        \text{and}
        \qquad
        B(n,y) = \oint_{|w|=r'} \frac{\diff w}{2\pi \mathrm{i}} \frac{w^{y-1-m}}{F(w)},
    \end{equation}
    we observe that
    \begin{equation}
        \begin{split}
            K(x+s,y+s) &= \sum_{m\in \mathbb{Z}'} \frac{1}{1+\zeta^{-1}q^{-m}} A(x+s,m) B(m,y+s)
            \\
            & = \sum_{m\in \mathbb{Z}'} \frac{1}{1+\zeta^{-1}q^{-m}} A(x,m-s) B(m-s,y)
            \\
            & = \sum_{m' \in \mathbb{Z}'} f_{\zeta,s} (m') A(x,m') B(m',y).
        \end{split}
    \end{equation}
    In this way, from the notable symmetry of the Fredholm determinant $\det[1-A f B] = \det[1-fBA]$, we find that
    \begin{equation}
        \det[1-K]_{\ell^2(\mathbb{Z}'_{> s})} = \det\left[1 - f_{\zeta,s} \cdot M \right]_{\ell^2(\mathbb{Z}')},
    \end{equation}
    where we noticed that
    \begin{equation}
        M(m,n) = \sum_{x = 0}^\infty B(m,x)A(x,n).
    \end{equation}
    Relation \eqref{pSdet} now concludes the proof.
\end{proof}

The determinantal expression stated in \cref{prop:fredholm_det_M} relates the probability distribution of the rightmost particle in the shift mixed periodic Schur measure to an average of a multiplicative observable of a determinantal point process with correlation kernel $M$.

\begin{corollary} \label{cor:multiplicative_expectation_det}
    Under the hypothesis of \cref{prop:fredholm_det_M}, we have
    \begin{equation} \label{eq:multiplicative_expectation_determinant}
        \mathbb{P}\left[ \lambda_1+ S \le s \right] = \mathbb{E}_M \left[ \prod_{i=1}^\infty \frac{1}{1+\zeta q^{s+\mathfrak{a}_i}} \right],
    \end{equation}
    where $\{ \mathfrak{a}_i:i \in \mathbb{Z}_{\ge 1}\}$ is a determinantal point process with correlation kernel $M$.
\end{corollary}

\begin{proof}
    Notable properties of determinantal point process imply that, for any function $g$ decaying ``sufficiently well", we have
    \begin{equation}
        \mathbb{E}_M\left[ \prod_{i=1}^\infty (1+g(\mathfrak{a}_i)) \right] = \det[1+g \cdot M]_{\ell^2(\mathbb{Z}')},
    \end{equation}
    so that setting $g=-f_{\zeta,s}$ and using \eqref{eq:fredholm_det_M} we prove relation \eqref{eq:multiplicative_expectation_determinant}.
\end{proof}

\begin{remark} \label{rem:multiplicative_expectation}
    Comparing kernels $M$ with $K$, evaluated at $q=0$, we see that $M$ is the correlation kernel of ``holes" in the Schur measure $\mathbb{PS}^{(q=0)}_{a,b}$ (see \cite{BO2016_ASEP}), with specializations $a,b$ consisting of geometric progressions of single parameter alpha specializations of ratio $q$, i.e. $a=(a_1,qa_1,q^2a_1,\dots, a_2, q a_2, q^2a_2, \dots)$, $b=(b_1,qb_1,q^2b_1,\dots, b_2, q b_2, q^2b_2, \dots)$.
\end{remark}

\subsection{Free boundary Schur measure}
The \emph{free boundary Schur measure} was introduced by Betea, Bouttier, Nejjar and Vuletic in \cite{Betea_et_al_free_boundary}. Rather than the general process considered in the original paper \cite{Betea_et_al_free_boundary}, we will only focus on a particular case that turns out to be very important in connection to KPZ models in half space. We define the probability measure on the set of partitions with \emph{even column} length
\begin{equation}
\label{fbS}
    \mathbb{FBS}_{a}^{(q)} (\lambda) 
    = \frac{
    \mathbf{1}_{\lambda' \text{ is even}}}{(q ; q)_\infty^{-1} \widetilde{\Pi}(a) } \sum_{\rho : \rho' \text{ is even}} q^{|\rho|/2} s_{\lambda/\rho}(a).
\end{equation}
Here $a$ is a Schur positive specialization and the normalization term $\widetilde{\Pi}(a)$ was introduced in \eqref{eq:pi_tilde}. As in the case of the periodic Schur measure, the free boundary Schur measure (\ref{fbS}) itself does not possess a nice determinantal nor pfaffian structure, but its shift-mixed version is shown to be a pfaffian point process. Let $S\sim \mathrm{Theta}(\zeta^2;q^2)$ be independent of $\lambda$. Then, the \emph{shift-mixed free boundary Schur measure} is the point process
\begin{equation}
    \mathfrak{S}(\lambda,2S) = \left( \lambda_i - i + \frac{1}{2} + 2S \right)_{i=1,2,\dots} \subset \mathbb{Z}'.
\end{equation}
In the next theorem we give a characterization of $\mathfrak{S}(\lambda,2S)$. For this we need to define the function $\Delta'$ as 
\begin{equation}
   \Delta'(x, y) = \mathbf{1}_{y=x+1} - \mathbf{1}_{y=x-1}
\end{equation}
and the difference operator $\nabla_x$ as
\begin{equation} \label{eq:nabla}
    \nabla_x f(x) = \frac{1}{2} [f(x+1) - f(x-1)].
\end{equation}

\begin{theorem}[\cite{Betea_et_al_free_boundary,Betea_etal2019}] \label{thm:free_boundary_pfaffian}
    The shift-mixed free boundary Schur measure is a pfaffian point process with $2 \times 2$ matrix correlation kernel
    \begin{equation} \label{eq:K_2x2_hs}
        K^{\mathrm{hs}}(x,y) = \left( \begin{matrix} k(x,y) & -2 \nabla_y k (x,y) \\ -2 \nabla_x k(y,x) & 4 \nabla_x \nabla_y k(x,y) + \Delta'(x,y) \end{matrix} \right),
    \end{equation}
    where
\begin{equation}
    k(x,y) 
    =
    \frac{1}{(2\pi \mathrm{i})^2}
    \oint_{|z|=r} \frac{\diff z}{z^{x+3/2}} \oint_{|w|=r} \frac{\diff w}{w^{y+5/2}} F(z)F(w) \kappa^{\mathrm{hs}}(z,w), 
    \label{eq:k_fbs}
\end{equation}
\begin{equation}
    F(z) 
    =
    \prod_{\ell \ge 0} \frac{ H(q^{ \ell} a ; z) }{  H(q^{ \ell} a ; z^{-1}) }
\end{equation}
and
\begin{equation} \label{eq:kappa_hs}
    \kappa^{\mathrm{hs}}(z,w) = \frac{(q,q,w/z,q z/w;q  )_\infty}{ (1/z^2, 1/w^2 , 1 / z w , q w z ;q)_\infty} \frac{\vartheta_3(\zeta^2 z^2 w^2 ; q^2)}{ \vartheta_3(\zeta^2; q^2) }.
\end{equation}
The function $H(a,z)$ was defined in \eqref{eq:pi_and_H}, while the integration contours are positively oriented and their radius satisfies $1 < r < \min(q^{-1/2},R)$, where $R$ is the convergence radius of $H(a;z)$. 
\end{theorem}

\begin{proof}
    Result of \cref{thm:free_boundary_pfaffian}, although implicitly derived in \cite{Betea_et_al_free_boundary}, is more explicitly stated in \cite{Betea_etal2019} and the expression of the $2\times 2$ matrix correlation kernel $K^{\mathrm{hs}}$ is a reformulation of a particular case of equations (4.2)-(4.4) of \cite{Betea_etal2019}, where in their notation we set $u=\sqrt{q},v=1,a_1=a_2=0$. Let us show how it is derived. From \cite{Betea_etal2019} eq. (4.2)-(4.4) one sees that the shift mixed free boundary Schur measure is a pfaffian point process with $2 \times 2$ matrix correlation kernel
    \begin{equation} \label{eq:K_hs_intermediate}
        \widetilde{K}(x,y) = \left( \begin{matrix} \widetilde{K}_{1,1}(x,y) & \widetilde{K}_{1,2}(x,y) \\ -\widetilde{K}_{1,2}(y,x) & \widetilde{K}_{2,2}(x,y) + \Delta'(x,y) \end{matrix} \right),
    \end{equation}
    defined as
\begin{align}
    \widetilde{K}_{1,1}(x,y) 
    &=
    \oint_{|z|=r} \frac{\diff z}{2 \pi \mathrm{i}} \oint_{|w|=r'} \frac{\diff w}{2 \pi \mathrm{i}} \frac{F(z)F(w)}{z^{x+3/2} w^{y+5/2}} \kappa^{\mathrm{hs}}(z,w), 
    \label{eq:k_11_fbs}
    \\
    \widetilde{K}_{1,2}(x,y) 
    &= 
    \oint_{|z|=r} \frac{\diff z}{2 \pi \mathrm{i}} \oint_{|w|=1/r'} \frac{\diff w}{2 \pi \mathrm{i}} \frac{F(z)F(1/w)}{z^{x+3/2} w^{-y+1/2}} \kappa^{\mathrm{hs}}(z,1/w) (1-w^2), 
    \label{eq:k_12_fbs}
    \\
    \widetilde{K}_{2,2}(x,y) 
    &=
    \oint_{|z|=1/r} \frac{\diff z}{2 \pi \mathrm{i}} \oint_{|w|=1/r'} \frac{\diff w}{2 \pi \mathrm{i}} \frac{F(1/z)F(1/w)}{{z^{-x+3/2} w^{-y+1/2}}} \kappa^{\mathrm{hs}}(1/z,1/w) (1-z^2)(1-w^2),
    \label{eq:k_22_fbs}
    \end{align}
    with $1<r,r'<\min(q^{-1},R)$.
    Notice that in \cite{Betea_etal2019} the element $\widetilde{K}_{2,2}$ has integration contours $|z|=r$ and $|w|=r'$ and the function $\Delta'$ is absent. To reach our expression we shift both contours and doing so the only pole one encounters is at $z=1/w$, whose contribution is
    \begin{equation}
    \begin{split}
        \oint_{|w|=1/r'} \frac{\diff w}{2 \pi \mathrm{i}} w^{y-x} \left( \frac{1}{w^2}-1 \right) &= \mathbf{1}_{y=x+1} - \mathbf{1}_{y=x-1} 
        = \Delta'(x,y).
    \end{split}
    \end{equation}
    With a change of variable $w\to 1/w$ in \eqref{eq:k_12_fbs} and $z\to 1/z, w\to 1/w$ in \eqref{eq:k_22_fbs} we can now make integration contours to be the same for each component $\widetilde{K}_{i,j}$. Finally we notice that, after such change of variables, the  $\widetilde{K}_{i,j}$'s differ from each other by the sole presence of factors $(w-1/w)$ and $(z-1/z)$ in the integrand, so that each term can be easily matched with discrete derivatives $\nabla_y, \nabla_x$ of the function $k(x,y) = \widetilde{K}_{1,1}(x,y)$ as in \eqref{eq:K_2x2_hs}, using the simple relation
    \begin{equation}
        \nabla_x \, \frac{1}{z^{x}} = - \frac{1}{2} \left(z-\frac{1}{z} \right) \frac{1}{z^x}.    
    \end{equation}
    This concludes the proof. 
\end{proof}

\begin{remark}
Setting $q=0$, the free boundary Schur measure becomes the \emph{Pfaffian Schur measure} defined in \cite{borodin2005eynard} (see also an earlier work \cite{imamura_sasamoto_half_space_png}) and studied in \cite{ghosal_pfaffian_schur, baik_barraquand_corwin_suidan_pfaffian, Betea_et_al_free_boundary}. Below in \cref{subs:baik_rains_pfaffian_schur} we will denote this measure by $\mathbb{FBS}^{(q=0)}$ and we will present an asymptotic analysis for the first row $\lambda_1$ of a random partition $\lambda$.
\end{remark}

As in case of the periodic Schur measure, we will be interested in the rightmost coordinate of the shift-mixed free boundary Schur measure. Its probability distribution function is explicitly written as the Fredholm pfaffian
\begin{equation}
    \mathbb{P}[\lambda_1 + 2 S  \le s] = \Pf (J - K^\mathrm{hs})_{\ell^2( \mathbb{Z}_{> s}' )}. 
\label{FBPf}
\end{equation}
Recall that $J$ is the matrix \eqref{eq:J}. In the following proposition we provide an alternative expression for the probability law of $\lambda_1$, which is slightly more convenient to consider degenerations and asymptotic limits.

\begin{proposition} \label{prop:fredholm_pfaffian_fbs}
    Let $\lambda$ be a random partition distributed according to the free boundary Schur measure and $S \sim \mathrm{Theta}(\zeta^2;q^2)$. Then the probablity distribution of $\lambda_1+2S$ admits the Fredholm pfaffian representation
    \begin{equation} \label{eq:fredholm_pfaffian_L}
        \mathbb{P}[\lambda_1 + 2S \le s] = \Pf(J-L)_{\ell^2(\mathbb{Z}'_{> s})},
    \end{equation}
    where $L$ is the $2 \times 2$ matrix kernel
    \begin{equation} \label{eq:L_fbs}
        L(x,y) = \left( \begin{matrix} k(x,y) & -\nabla_y k(x,y) \\ -\nabla_x k(x,y) & \nabla_x \nabla_y k(x,y) \end{matrix} \right),
    \end{equation}
    and the function $k$ is given by \eqref{eq:k_fbs}, while $\nabla_x,\nabla_y$ were defined in \eqref{eq:nabla}.
\end{proposition}

\begin{proof}
    This is a slight generalization of \cite[Proposition 5.2]{baik_barraquand_corwin_suidan_pfaffian}, where partial derivatives $\partial_x, \partial_y$ are replaced by discrete difference operators $\nabla_x,\nabla_y$. We find that the proof reported in \cite{baik_barraquand_corwin_suidan_pfaffian} adapts well to our situation and so here we are only going to sketch the argument. Notation we introduce below will be close to that of \cite{baik_barraquand_corwin_suidan_pfaffian}. Consider the expansion of Fredholm pfaffians of kernels $L$ and $K^{\mathrm{hs}}$ as
    \begin{equation}
        \Pf[J-L]_{\ell^2(\mathbb{Z}'_{> s})} = \sum_{\ell \ge 0} \frac{(-1)^\ell}{\ell!} a_\ell,
        \qquad
        \Pf\left[J-K^{\mathrm{hs}}\right]_{\ell^2(\mathbb{Z}'_{ > s})} = \sum_{\ell \ge 0} \frac{(-1)^\ell}{\ell!} b_\ell,
    \end{equation}
    where
    \begin{equation}
        a_\ell= \sum_{x_1,\dots,x_\ell > s} \Pf[L(x_i,x_j)]_{i,j=1}^\ell,
        \qquad
        b_\ell= \sum_{x_1,\dots,x_\ell > s} \Pf\left[K^{\mathrm{hs}}(x_i,x_j) \right]_{i,j=1}^\ell.
    \end{equation}
    The $2\times2$ matrix kernel $K^{\mathrm{hs}}$ can be written as a sum
    \begin{equation}
        K^{\mathrm{hs}}(x,y) = A(x,y) + \left( \begin{matrix} 0 & 0 \\ 0 & \Delta'(x,y) \end{matrix} \right),
    \end{equation}
    with $A_{1,1}(x,y)=k(x,y)$, $A_{1,2}(x,y)=-2\nabla_y k(x,y)$, $A_{1,2}(x,y)=-2\nabla_x k(x,y)$, $A_{2,2}(x,y) = 4\nabla_x \nabla_y k(x,y)$ . Using a notable minor expansion for pfaffians of sums of $n\times n$ matrices $N,M$ (see \cite[Lemma 4.2]{stembridge1990nonintersecting})
    \begin{equation} \label{eq:pfaffian_sum}
        \Pf[M+N] = \sum_{ \substack{I \subset \{1,\dots n\} \\ \#I \text{ even}} } (-1)^{|I| - \#I/2} \Pf[M_I] \Pf[N_{I^c}]
    \end{equation}
    where $\#I$ is the cardinality of $I$, while $|I|$ is the sum of its elements, we can express $b_\ell$ as
    \begin{equation} \label{eq:b_l_b_lm}
        \frac{(-1)^\ell}{\ell!} b_\ell = \sum_{m=\lceil \ell/2 \rceil }^\ell b_\ell^{(m)}.
    \end{equation}
    In the above expression we have condensated in the term $b_\ell^{(m)}$ all summations involving exactly a product of degree $\ell - m$ of $\Delta'$ functions originating from the evaluation of minors of the matrix $\left( \begin{smallmatrix} 0 & 0 \\ 0 & \Delta'(x_i,x_j) \end{smallmatrix} \right)_{i,j=1}^m$. A computation shows that terms $b_\ell^{(m)}$ have the explicit expression
    \begin{equation} \label{eq:b_l+i_l}
    \begin{split}
        b_{\ell+i}^{(\ell)} &= \frac{(-1)^{\ell+i}}{(\ell+i)!} \sum_{x_1,\dots,x_{\ell+i} > s} \left( \sum_{1\le j_1 < \cdots < j_{2i} \le \ell+i} \Pf[\Delta'(x_{j_r},x_{j_s})]_{r,s=1}^i (-1)^i \Pf\left [A_{\ell+i}^{\widehat{2j_1},\dots,\widehat{2j_{2i}} }\right] \right)
        \\
        &=
        \frac{(-1)^\ell}{(\ell+i)!} \binom{\ell+i}{2i} \sum_{x_1,\dots,x_{\ell+i} > s} \Pf[\Delta'(x_{r},x_{s})]_{r,s=\ell-i+1}^{\ell+i} \Pf\left [A_{\ell+i}^{\widehat{2(\ell - i+1)},\dots,\widehat{2(\ell+i)} }\right],
    \end{split}
    \end{equation}
    where we have denoted by $A_{n}^{\widehat{2j_1},\dots,\widehat{2j_r} }$ is the $2(n-r) \times 2(n-r)$ matrix resulting from removing rows and columns $2j_1 ,\dots, 2j_r$ from the matrix $(A(x_\alpha ,x_\beta))_{\alpha,\beta=1}^n$. Notice that the second equality in \eqref{eq:b_l+i_l} follows from the fact that, through a change of $x$-variables all terms in the sum inside the parentheses become equal, since the $x_i$ are all summed over the same domain. We now evaluate the sum in the right hand side of \eqref{eq:b_l+i_l}, expanding one of the Pfaffians as in \eqref{eq:pfaffian_Leibnitz}. We find
    \begin{equation} \label{eq:pfaffian_Delta_prime}
        \begin{split}
            &
            \sum_{x_1,\dots,x_{\ell+i}> s} \Pf[\Delta'(x_{r},x_{s})]_{r,s=\ell-i+1}^{\ell+i} \Pf\left [A_{\ell+i}^{\widehat{2(\ell - i+1)},\dots,\widehat{2(\ell+i)} }\right]
            \\
            & =
            \sum_{\substack{\alpha=\{(r_1,s_1),\dots,(r_i,s_i)\} \\
            r_j,s_j \in \{ \ell - i + 1,\dots, \ell+i \} \\
            r_1<\cdots <r_i, \,\, r_j<s_j}} \epsilon (\alpha) \sum_{x_1,\dots,x_{\ell+i} > s} \prod_{j=1}^i \Delta'(x_{r_j},x_{s_j}) \Pf\left [A_{\ell+i}^{\widehat{2(\ell - i+1)},\dots,\widehat{2(\ell+i)} }\right]
            \\
            & = 
            \frac{(2i)!}{2^i i!} \sum_{x_1,\dots,x_{\ell+i} > s} \prod_{j=1}^i \Delta'(x_{\ell - i +2j-1},x_{\ell-i+2j}) \Pf\left [A_{\ell+i}^{\widehat{2(\ell - i+1)},\dots,\widehat{2(\ell+i)} }\right]
        \end{split}
    \end{equation}
where in the third line we recognized that, changing $x$-variables and absorbing the sign $\varepsilon(\alpha)$ into the pfaffian all terms of the summation over permutations $\alpha$ in the second line are equal. We can now evaluate the action of operators $\Delta'$ on the pfaffian and, after taking summations over variables $x_{\ell+1},\dots x_{\ell+i}$, we find
\begin{equation}
\begin{split}
    &
    2^i\,
    \nabla_{x_{\ell}} \cdots \nabla_{x_{\ell+i}} \Pf\left [A_{\ell+i}^{\widehat{2(\ell - i+1)},\dots,\widehat{2(\ell+i)} } \right] \bigg|_{x_{\ell}=\cdots=x_{\ell+i}} 
    \\
    &= 2^i\, \Pf \left [ \nabla_{x_{\ell}} \cdots \nabla_{x_{\ell+i}} A_{\ell+i}^{\widehat{2(\ell - i+1)},\dots,\widehat{2(\ell+i)} } \bigg|_{x_{\ell}=\cdots=x_{\ell+i}} \right]
    \\
    & = (-1)^i 2^{\ell} \Pf[L(x_i,x_j)]_{i,j=1}^\ell.
\end{split}
\end{equation}
Substituting the last relation and \eqref{eq:pfaffian_Delta_prime} into \eqref{eq:b_l+i_l} we obtain
\begin{equation}
\begin{split}
    b_{\ell+i}^{(\ell)} 
    = \frac{(-1)^{\ell+i}}{(\ell+i)!} \binom{\ell+i}{2i} \frac{(2i)!}{2^i i!} 2^\ell a_\ell
     = \frac{(-1)^\ell a_\ell}{\ell!} \binom{\ell}{i} 2^{k-i} (-1)^i,
\end{split}
\end{equation}
so that
\begin{equation}
    \sum_{i=0}^\ell b_{\ell+i}^{(\ell)} =\frac{(-1)^\ell}{\ell!} a_\ell \sum_{i=0}^\ell \binom{\ell}{i} 2^{k-i} (-1)^i =  \frac{(-1)^\ell}{\ell!} a_\ell.
\end{equation}
We can now turn the Fredholm pfaffian expansion of $J-K^\mathrm{hs}$ into that of $J-L$ as
\begin{equation}
\begin{split}
    \Pf[J-K^{\mathrm{hs}}] &= \sum_{\ell \ge 0} \frac{(-1)^\ell}{\ell!} b_\ell
    = \sum_{m \ge 0} \sum_{i=0}^m b_{m+i}^{(m)}
    = \sum_{\ell \ge 0} \frac{(-1)^\ell}{\ell!} a_\ell = \Pf[J-L],
\end{split}
\end{equation}
where in the second equality we used \eqref{eq:b_l_b_lm} and operated an exchange of summation indices and in the third equality we renamed $m$ back to $\ell$. This manipulations are justified by absolute convergence of all terms involved, which is consequence of exponential decay of coefficients of matrix kernels $L(x,y)$ and $K^{\mathrm{hs}}(x,y)$ as $x,y \to \infty$. This
completes the proof.
\end{proof}

\subsection{Alternatve representation of Fredholm pfaffian via a $q$-trigonometric identity by Gosper}

Following Gosper \cite{gosper_q_trigonometric}, we define the $q$-analog of the sine and cosine functions as
\begin{equation}
    \sin_q(\pi z) = q^{\left( z - \frac{1}{2} \right)^2} \frac{(q^{2z},q^{2-2 z};q^2)_\infty}{(q,q;q^2)_\infty}
    \qquad
    \text{and}
    \qquad
    \cos_q(z) = \sin_q(\frac{\pi}{2}-z).
\end{equation}
In the same paper, Gosper conjectured, via numerical experiments a number of properties of his $q$-trigonometric functions. Among these we have the following $q$-analog of the addition formula for angles, proven more recently in \cite{ELBACHRAOUI_q_trig},
\begin{equation} \label{eq:addition_formula_sinq}
    \sin_q(x+y) = \frac{\sin_q(x-y)}{\sin_{q^2}(x-y)} \left[ \sin_{q^2}(x) \cos_{q^2}(y) + \sin_{q^2}(y) \cos_{q^2}(x) \right].
\end{equation}
Using this remarkable relation and Ramanujan's bilateral sum \eqref{eq:ramanujan_psi} we obtain an alternative representation of the Fredholm pfaffian in the right hand side of \eqref{eq:fredholm_pfaffian_L}.

\begin{proposition} \label{prop:q_krejenbrink_le_doussal}
    Let $\lambda$ be distributed according to the free boundary Schur measure \eqref{fbS} and $S \sim \mathrm{Theta}(\zeta^2;q^2)$. Then we have
    \begin{equation} \label{eq:q_krejenbrink_le_doussal}
        \mathbb{P}\left[ \lambda_1+2S \le s \right] = \Pf\left[ J - f_{\zeta^2,s} \cdot M^\mathrm{hs} \right]_{\ell^2(\mathbb{Z}')},
    \end{equation}
    where $M^\mathrm{hs}$ is the $2\times 2$ matrix kernel
    \begin{equation}
        M^\mathrm{hs}(m,n) = \left( \begin{matrix} M^\mathrm{hs}_{1,1}(m,n) & M^\mathrm{hs}_{1,2}(m,n) \\ -M^\mathrm{hs}_{1,2}(n,m) & M^\mathrm{hs}_{2,2}(m,n) \end{matrix} \right),
    \end{equation}
    with coefficients
    \begin{equation}
        \begin{split}
            M^\mathrm{hs}_{1,1}(m,n) = \oint_{|z|=r} \frac{\diff z}{2 \pi \mathrm{i}} \oint_{|z|=r} \frac{\diff w}{2 \pi \mathrm{i}} & \frac{F(z)}{z^{m+5/2}} \frac{F(w)}{w^{n+5/2}} \frac{z-w}{1-zw} \frac{1+zw}{2} 
            \\
            & \times \frac{z (1/z^2,q^2z^2;q^2)_\infty}{(1/z^2;q)_\infty} \frac{w (1/w^2,q^2w^2;q^2)_\infty}{(1/w^2;q)_\infty},
        \end{split}
    \end{equation}
    \begin{equation}
        \begin{split}
            M^\mathrm{hs}_{1,2}(m,n) = \oint_{|z|=r} \frac{\diff z}{2 \pi \mathrm{i}} \oint_{|z|=r} \frac{\diff w}{2 \pi \mathrm{i}} & \frac{F(z)}{z^{m+5/2}} \frac{F(w)}{w^{n+5/2}} \frac{z-w}{1-zw} \frac{1+zw}{2}
            \\ & \times \frac{z (1/z^2,q^2z^2;q^2)_\infty}{(1/z^2;q)_\infty} \frac{(q w^2,q/w^2;q^2)_\infty}{(q;q^2)_\infty^2(1/w^2;q)_\infty},
        \end{split}
    \end{equation}
    \begin{equation}
        \begin{split}
            M^\mathrm{hs}_{2,2}(m,n) = \oint_{|z|=r} \frac{\diff z}{2 \pi \mathrm{i}} \oint_{|z|=r} \frac{\diff w}{2 \pi \mathrm{i}} & \frac{F(z)}{z^{m+5/2}} \frac{F(w)}{w^{n+5/2}} \frac{z-w}{1-zw} \frac{1+zw}{2}
            \\ & \times \frac{(q z^2,q/z^2;q^2)_\infty}{(q;q^2)_\infty^2(1/z^2;q)_\infty} \frac{(q w^2,q/w^2;q^2)_\infty}{(q;q^2)_\infty^2(1/w^2;q)_\infty}
        \end{split}
    \end{equation}
    and $f_{\zeta^2,s}$ is given in \eqref{eq:fermi_dirac}.
    The notation adopted above is the same as in \cref{thm:free_boundary_pfaffian}.
\end{proposition}

\begin{proof}
    The addition formula \eqref{eq:addition_formula_sinq} implies, setting $z=q^{-2x},w=q^{2y}$ and rescaling $q \to q^{1/2}$, the identity
    \begin{equation} \label{eq:addition_sinq_simplified}
            (w/z,qz/w;q)_\infty = \frac{1}{z} \frac{(1/zw,qzw;q)_\infty}{(1/z^2w^2,q^2z^2w^2;q^2)_\infty} \left[ g_1(z) g_2(w) - g_1(w) g_2( z ) \right],
    \end{equation}
    where we defined
    \begin{equation}
        g_1(z) = z (1/z^2,q^2z^2;q^2)_\infty,
        \qquad
        g_2(w) = \frac{(q w^2,q/w^2;q^2)_\infty}{(q;q^2)^2_\infty}.
    \end{equation}
    Using \eqref{eq:addition_sinq_simplified} and \eqref{eq:ramanujan_psi} we can rewrite the function $\kappa^\mathrm{hs}$ of \eqref{eq:kappa_hs} as
    \begin{equation}
        \kappa^\mathrm{hs}(z,w) = w \frac{g_1(z) g_2(w) - g_1(w) g_2( z )}{(1/z^2,1/w^2;q)_\infty} \sum_{m\in \mathbb{Z}'} \frac{(zw)^{2m}}{1+\zeta^{-2}q^{-2m}}.
    \end{equation}
    Defining
    \begin{equation}
        a(x) = \oint_{|z|=r} \frac{\diff z}{2\pi \mathrm{i}} \frac{F(z)}{z^{x+3/2}} \frac{g_1(z)}{(1/z^2;q)_\infty},
        \qquad
        b(y) = \oint_{|w|=r} \frac{\diff w}{2\pi \mathrm{i}} \frac{F(w)}{w^{y+3/2}} \frac{g_2(w)}{(1/w^2;q)_\infty},
    \end{equation}
    we observe that
    \begin{equation}
        k(x,y) = \sum_{m \in \mathbb{Z}'} \frac{1}{1+\zeta^{-2}q^{-2m}} \left[ a(x-2m) b(y-2m) - a(y-2m) b(x-2m) \right],
    \end{equation}
    where $k(x,y)$ was given in \eqref{eq:k_fbs}. Now, defining
    \begin{equation}
        A(x,m) = \left( \begin{matrix} -\nabla_x a(x+2m) & \nabla_x b(x+2m) \\ - a(x+2m) &  b(x+2m) \end{matrix} \right),
        \qquad
        B(m,y) = \left( \begin{matrix} b(y+2m) & -\nabla_y b(y+2m) \\ a(y+2m) &  -\nabla_y a(y+2m) \end{matrix} \right),
    \end{equation}
    we can express the matrix kernel $L$ as
    \begin{equation}
        \begin{split}
            J L(x+s,y+s) &= \sum_{m \in \mathbb{Z}'} \frac{1}{1+\zeta^{-2}q^{-2m}} A(x+s,m) B(m,y+s)
            \\
            & = \sum_{m \in \mathbb{Z}'} \frac{1}{1+\zeta^{-2}q^{-2m}} A(x,m-\frac{s}{2}) B(m-\frac{s}{2},y)
            \\ 
            & = \sum_{m' \in \mathbb{Z}'} f_{\zeta^2,s}(m') A(x,\frac{m'}{2}) B(\frac{m'}{2},y).
        \end{split}
    \end{equation}
    Using the notable relation $\Pf[J-J^TAfB] = \Pf[J- fBAJ^T ]$, we find that
    \begin{equation}
        \Pf[J-L]_{\ell^2(\mathbb{Z}_{> s}')} = \Pf \left[J-f_{\zeta^2,s} \cdot M^\mathrm{hs} \right]_{\ell^2(\mathbb{Z}')},
    \end{equation}
    since 
    \begin{equation}
        M^\mathrm{hs}(m,n) = \sum_{x=1}^\infty B(\frac{m}{2},x) A(x,\frac{n}{2}) J^T.
    \end{equation}
    This completes the proof.
\end{proof}

The identity \eqref{eq:q_krejenbrink_le_doussal} implies that the probability distribution of the right edge particle of the shift mixed free boundary Schur measure is equivalent to the expectation of the multiplicative Fermi-Dirac factor over a pfaffian point process with correlation kernel $M^\mathrm{hs}$. 

\begin{corollary}
    Under the hypothesis of \cref{prop:q_krejenbrink_le_doussal}, we have
    \begin{equation} \label{eq:multiplicative_expectation_pfaffian}
        \mathbb{P}\left[ \lambda_1+2S \le s \right] = \mathbb{E}_{M^\mathrm{hs}} \left[ \prod_{i=1}^\infty \frac{1}{1+\zeta^2 q^{s+\mathfrak{a}_i}} \right],
    \end{equation}
    where $\{ \mathfrak{a}_i:i \in \mathbb{Z}_{\ge 1}\}$ is a pfaffian point process with correlation kernel $M^\mathrm{hs}$.
\end{corollary}

\begin{proof}
    Notable properties of pfaffian point process imply that, for any function $g$ with ``nice" decay properties, we have
    \begin{equation} \label{eq:multiplicative_expectation_pfaff_general}
        \mathbb{E}_{M^\mathrm{hs}}\left[ \prod_{i=1}^\infty (1+g(\mathfrak{a}_i)) \right] = \Pf \left[ J+g \cdot M^\mathrm{hs} \right]_{\ell^2(\mathbb{Z}')},
    \end{equation}
    so that setting $g=-f_{\zeta^2,s}$ and using \eqref{eq:q_krejenbrink_le_doussal} we prove relation \eqref{eq:multiplicative_expectation_pfaffian}.
\end{proof}

We think of \eqref{eq:multiplicative_expectation_pfaffian} as a pfaffian variant of \eqref{eq:multiplicative_expectation_determinant}. In \cref{rem:multiplicative_expectation} we pointed out the relation between the correlation kernel $M$ and the Schur measure and one could naively expect a similar relation between the correlation kernel $M^\mathrm{hs}$ and the Pfaffian Schur measure. However, from a simple inspection of formulas given in \cref{prop:q_krejenbrink_le_doussal}, this does not appear to be true. It is in fact an interesting problem to give a physical interpretation of the pfaffian point process associated to $M^\mathrm{hs}$.

\subsection{Equivalence with marginals of $q$-Whittaker measures} \label{subs:matching_qW_schur}

Here we state the fundamental connections between models of determinantal and pfaffian measures presented above in this section and discrete KPZ solvable models discussed in \cref{sec:KPZqW}.

\begin{theorem}[\cite{IMS_matching}] \label{thm:matching_qW_ps}
    Let $\mu \sim \mathbb{W}^{(q)}_{a;b}$ and $\chi \sim q\text{-}\mathrm{Geo}(q)$ be indipendent random variables. Consider also $\lambda \sim \mathbb{PS}^{(q)}_{a,b}$. Then the following equivalence in distribution holds
    \begin{equation}
        \mu_1 + \chi  \stackrel{\mathcal{D}}{=} \lambda_1.
    \end{equation}
\end{theorem}

\Cref{thm:matching_qW_ps} is an immediate consequence of the summation identity \eqref{eq:qW_and_Schur_1}. In \cite{IMS_matching} we proved the same theorem, along with identity \eqref{eq:qW_and_Schur_1}, by matching Fredholm determinant expressions for shift-mixed periodic Schur measure with those of $q$-Whittaker measure found in \cite{imamura2017fluctuations}. 

\medskip

The following theorem is new and it is a half space extension of \cref{thm:matching_qW_ps}.

\begin{theorem} \label{thm:matching_hs_qW_fbs}
    Let $\mu \sim \mathbb{HW}^{(q)}_{a;0}$ and $\chi \sim q\text{-}\mathrm{Geo}(q)$ be indipendent random variables. Consider also $\lambda \sim \mathbb{FBS}^{(q)}_{a}$. Then the following equivalence in distribution holds
    \begin{equation} \label{eq:hs_qW_free_boundary_schur}
        \mu_1 + \chi \stackrel{\mathcal{D}}{=} \lambda_1.
    \end{equation}
\end{theorem}

\begin{proof}
    This is straightforward consequence of the refined Littlewood identity \eqref{eq:qW_and_Schur_2}.
\end{proof}

Unlike for \cref{thm:matching_qW_ps}, a proof of equality \eqref{eq:hs_qW_free_boundary_schur} based on integrable probabilistic techniques is at the moment not available. This is because rigorous methods using Bethe ansatz have proven to be difficult to implement for  the study of half space $q$-Whittaker measure and through such techniques it remains a challenge to obtain explicit and manageable pfaffian formulas for interesting physical observables; see \cite{barraquand_half_space_mac}.

\section{Fredholm determinant and pfaffian formulas for KPZ models}
\label{sec:KPZpS}

Leveraging equalities in distribution established in \cref{thm:matching_qW_ps} and \cref{thm:matching_hs_qW_fbs} along with the simple structure of the periodic and free boundary Schur measures we are able to produce determinantal and pfaffian formulas for observables of KPZ solvable models. Determinantal formulas reported below, although new, add to a number of analogous results discovered in the last 15 years \cite{tracy2008fredholm,TW_ASEP2,BorodinCorwin2011Macdonald,BorodinCorwinSasamoto2012,BorodinCorwinRemenik,MatveevPetrov2014} and in many cases appear to be simpler. Pfaffian formulas, on the other hand, are completely original and have no analog in the existing mathematical literature.

\subsection{Fredholm determinant formulas for $q$-PushTASEP and ASEP in full space}

Let us start presenting a representation for the probability distribution of a tagged particle in $q$-PushTASEP in terms of a Fredholm determinant of a free-fermionic type kernel.

\begin{corollary} \label{cor:fredholm_det_q_PushTASEP}
    Let $\mathsf{X}(T)$ be the $q$-PushTASEP with initial conditions $\mathsf{x}_k(0)=k$, for $k=1,\dots,N$ and $\chi \sim q$-$\mathrm{Geo}(q)$, $S \sim \mathrm{Theta}(\zeta,q)$ independent random variables. Then we have
    \begin{equation} \label{eq:fredholm_det_q_PushTASEP}
        \mathbb{P} \left( \mathsf{x}_N(T) + \chi + S \le N + s \right)
        = \det \left( 1 - K_{\mathrm{push}} \right)_{\ell^2(\mathbb{Z}_{> s}' )},
    \end{equation}
    where the kernel $K_{\mathrm{push}}$ is obtained from \eqref{eq:kernel_periodic_Schur} setting
    \begin{equation}
        F(z) = \frac{\prod_{i=1}^T (b_i/z;q)_\infty}{\prod_{i=1}^N (a_iz;q)_\infty}.
    \end{equation}
\end{corollary}

\begin{proof}
    This is obtained combining \cref{thm:matching_qW_ps}, \cref{prop:qPushTASEP_qW} and \cref{thm:periodic_Schur_det}.
\end{proof}

\begin{remark} \label{rem:q_laplace_as_chi+S}
    The left hand side of \eqref{eq:fredholm_det_q_PushTASEP} admits an alternative expression as $q$-Laplace transform
    \begin{equation}
        \mathbb{E} \left( \frac{1}{(-\zeta q^{ \frac{1}{2} + s + N - \mathsf{x}_N(T)} ; q)_\infty} \right)
    \end{equation} 
    of the probability density function of $\mathsf{x}_N(T)$. This is a consequence of the identity
    \begin{equation} \label{eq:sum_chi_plus_S}
        \mathbb{P} (\chi + S \le n) = \frac{1}{(-\zeta q^{\frac{1}{2} + n} ;q)_\infty},
    \end{equation}
    where $\chi$ and $S$ are as in \cref{cor:fredholm_det_q_PushTASEP}. For a direct proof of \eqref{eq:sum_chi_plus_S} see \cite[Lemma 2.4]{IMS_matching}.
\end{remark}

The first instance of a determinantal formula for the $q$-PushTASEP was given in \cite[Theorem 3.3]{BorodinCorwinFerrariVeto2013} and it was based on clever manipulations of integral formulas for $q$-moments; see \cite{BorodinCorwin2011Macdonald}. The kernel presented in \cite{BorodinCorwinFerrariVeto2013}, as others obtained starting from $q$-moments formulas, does not appear to be related with determinantal point processes or free fermions. As such, asymptotic analysis is rather involved, although doable, as one has to deal with a complicated pole structure of the kernel. It would be very useful to obtain our expression \eqref{eq:fredholm_det_q_PushTASEP} by a proper summation of $q$-moments.

\medskip

Next we present a Fredholm determinant formula of free fermionic type for the ASEP.

\begin{corollary} \label{cor:ASEP_determinant}
    Let $\mathsf{J}(x,\tau)$ be the integrated current of the ASEP with initial conditions $\mathsf{x}_k(0)=-k$, for $k=1,2,\dots$ and let $\chi \sim q$-$\mathrm{Geo}(q)$ and $S \sim \mathrm{Theta}(\zeta,q)$ be independent random variables. Then, we have
    \begin{equation} \label{eq:ASEP_fredholm_det}
        \mathbb{P}\left(-\mathsf{J}(x,\tau) + \chi + S \le s \right) = \det(1-K_\mathrm{ASEP})_{\ell^2(\mathbb{Z}'_{> s})},
    \end{equation}
    where the kernel $K_{\mathrm{ASEP}}$ is obtained from \eqref{eq:kernel_periodic_Schur} setting
    \begin{equation}
        F(z) = e^{-(1-q) \frac{z \tau }{1+z}}  (1+1/z)^x . 
    \end{equation}
\end{corollary}
\begin{proof}
    Consider the $q$-Whittaker measure $\mathbb{W}^{(q)}_{a',b'}$, where $a',b'$ are $q$-beta specializations with parameters $(a_1,\dots,a_N)$, $(b_1,\dots,b_T)$. Then, combining \cref{thm:matching_qW_ps} and \cref{thm:periodic_Schur_det} we have
    \begin{equation}\label{eq:prob_s6vm}
        \mathbb{P}(\mu_1 + \chi + S \le r) = \det(1-\tilde{K})_{\ell^2(\mathbb{Z}'_{> r})},
    \end{equation}
    where $\tilde{K}$ is obtained from \eqref{eq:kernel_periodic_Schur} setting 
    \begin{equation}
        F(z) = \frac{\prod_{i=1}^N (1+z a_i)}{\prod_{i=1}^T (1+b_i/z)}.
    \end{equation}
    This is simply a result of the definition of $q$-beta specialization \eqref{eq:q_beta_spec}. Considering the scaling  \eqref{eq:ASEP_scaling}, under which $\mu_1-N$ approaches in law the current of the ASEP with step initial data, the left hand side of \eqref{eq:prob_s6vm} becomes
    \begin{equation}
        \mathbb{P}(-\mathsf{J}^{(\varepsilon)}(x,\tau) + \chi + S \le r - N),
    \end{equation}
    while the function $F(z)$ becomes
    \begin{equation}
        F(z) = z^N (1+1/z)^x e^{\tau \frac{(1-q)}{2}\left( 1- \frac{2z}{1+z} \right)  +O(\varepsilon) }.
    \end{equation}
    From the explicit expression \eqref{eq:kernel_periodic_Schur} we see that rescaling $\tilde{K}(x+N,y+N)$ and $r=s+N$ one can prove the convergence
    \begin{equation}
        \det(1-\tilde{K})_{\ell^2(\mathbb{Z}'_{> r})} \xrightarrow[\varepsilon \to 0]{} \det (1-K_\mathrm{ASEP})_{\ell^2(\mathbb{Z}'_{> s})},
    \end{equation}
    which combined with \cref{prop:ASEP_matching_mu_1}, proves \eqref{eq:ASEP_fredholm_det}.
\end{proof}

Using \cref{cor:ASEP_determinant}, along with \cref{cor:multiplicative_expectation_det} one can recover results found by Borodin and Olshanski in \cite[Theorem 10.2]{BO2016_ASEP} relating the integrated current of the ASEP with the discrete Laguerre ensemble. Similarly to the observation made at the end of \cref{subs:ASEP}, one could use the correspondence between the $q$-Whittaker measure and the periodic Schur measure to deduce determinantal formula for the full hierarchy of solvable models descending from the stochastic six vertex model. Nevertheless we will not take this route in this paper having decided to treat only the most representative cases.

\subsection{Fredholm determinant representation for Log Gamma polymer in full space}

We report a Fredholm determinant formula for the Laplace transform of the probability density function of the Log Gamma polymer partition function $Z(N,N)$. In literature a resembling expression was derived originally in \cite{BorodinCorwinRemenik}, where the integral kernel possessed a single contour integral representation; see also \cite{Barraquand_Corwin_Dimitrov_log_gamma}. Compared to such knows formulas, the one we present below is noticeably simpler and the kernel possesses a double contour integral representation. This feature is a result of the direct relation between Log Gamma polymers, seen as a scaling limit of $q$-PushTASEP as in \cref{prop:qPushTASEP_convergence}, with free fermions at positive temperature.

\begin{theorem} \label{thm:fredholm_det_log_gamma}
    Let $Z(N,N)$ be the partition function of the Log Gamma polymer model on the quadrant. Then we have
    \begin{equation} \label{eq:fredholm_det_log_gamma}
        \mathbb{E} \left[ e^{ - e^{-\varsigma + \log Z(N,N)}} \right] = \det (1 - \mathbf{K})_{\mathbb{L}^2(\varsigma,+\infty)},
    \end{equation}
    where the kernel is given by
    \begin{equation} \label{eq:K_LG}
        \mathbf{K} (X,Y) = \int_{\mathrm{i} \mathbb{R} - d} \frac{\diff Z}{2 \pi \mathrm{i}} \int_{\mathrm{i} \mathbb{R} + d'} \frac{\diff W}{2 \pi \mathrm{i}} \frac{\pi}{\sin (\pi (W-Z))} \frac{g_{A,B}(Z)}{ g_{A,B}(W)} e^{Z X- W Y},
    \end{equation}
    where $d,d'>0$, such that $\frac{1}{2 N}<d'+  d<1$ and
    \begin{equation}
        g_{A,B}(Z) = \prod_{i=1}^N \frac{\Gamma(A_i +Z)}{\Gamma(B_i - Z)}.
    \end{equation}
\end{theorem}

The proof of \cref{thm:fredholm_det_log_gamma} is based on a scaling limit of the Fredholm determinant formula \eqref{eq:fredholm_det_q_PushTASEP}. We will present it below in the subsection.

\begin{remark} \label{rem:Z(T_N)}
    A Fredholm determinant representation analogous to \eqref{eq:fredholm_det_log_gamma} should be also available for the partition function $Z(T,N)$ for generic $T,N$. In this case the kernel $\mathbf{K}$ gets modified by changing the factor $g_{A,B}(Z)$ into $\prod_{i=1}^N \Gamma(A_i+Z) \prod_{i=1}^T \Gamma(B_i-Z)^{-1}$. This claim should follow from an adaptation of arguments presented in the proof of \cref{thm:fredholm_det_log_gamma} and technical results developed in \cref{sec:LG}. We will not discuss any further such generalizations in this paper.
\end{remark}

\begin{remark} \label{rem:Laplace_Gumbel}
    The Laplace transform in the left hand side of \eqref{eq:fredholm_det_log_gamma} admits an alternative expression using the identity
    \begin{equation} \label{eq:Laplace_Gumbel}
        \mathbb{E}\left[ e^{-e^{ -\varsigma + U }} \right] = \mathbb{P} \left[ U+ \mathcal{G} \leq \varsigma \right],
    \end{equation}
    which holds for any random variable $U$ and $\mathcal{G}$ is an independent standard Gumbel random variable. This can be seen as a continuous variant of the identity shown in \cref{rem:q_laplace_as_chi+S}.
\end{remark}

\begin{lemma} \label{lem:K_tau_hilbert_schmidt}
    The Fredholm determinant of the kernel $\mathbf{K}$ is absolutely convergent.
\end{lemma}
\begin{proof}
    Using the parameterization of integration variables $Z=-d + \mathrm{i} u$, $W=d + \mathrm{i} v$, with $d\in(\frac{1}{4},\frac{1}{2})$, we easily find the estimate
    \begin{equation} \label{eq:K_tau}
        \begin{split}
            \left| \mathbf{K} (X,Y) \right| \le e^{-d (X + Y)} \int_{\mathbb{R}} \frac{\diff u}{2\pi} \int_{\mathbb{R}} \frac{\diff v}{2\pi} & \left| \frac{\pi}{ \sin\left[ \pi ( 2d + \mathrm{i}(v-u) ) \right] } \right|
            \\
            &
            \times
            \prod_{i=1}^N \left| \frac{\Gamma(A_i - d +\mathrm{i}u )  }{\Gamma(B_i + d -\mathrm{i}u )} \frac{\Gamma(B_i - d -\mathrm{i}v )  }{\Gamma(A_i + d +\mathrm{i}v )} \right|.
        \end{split}
    \end{equation}
    The integral in the previous expression is bounded by a constant. To show this we use the estimates
    \begin{equation}
        \left| \frac{1}{\sin [ \pi ( \alpha + \mathrm{i} u ) ]} \right| \le C_\alpha e^{-\pi|u|},
    \end{equation}
    and 
    \begin{equation} \label{eq:bound_ratio_gamma}
        \left| \frac{\Gamma(a+ \mathrm{i} u)}{\Gamma(b- \mathrm{i} u)} \right| \le C_{a,b} \left( 1+ |u| \right)^{a-b}.
    \end{equation}
    In the first inequelity $C_\alpha$ is a constant depending on $\alpha \notin \mathbb{Z}$, while the second one holds for $a \notin \mathbb{Z}_{\le 0}$ and is a consequence of the following adaptation of the Stirling's approximation
    \begin{equation} \label{eq:decay_Gamma}
        | \Gamma (a + \mathrm{i} u) | = \sqrt{2 \pi} \, | u |^{a-\frac{1}{2}} e^{-a - \frac{|u| \pi}{2}} \left(1+ O(|u|^{-1}) \right).
    \end{equation}
    Then, the integral in the right hand side of \eqref{eq:K_tau} is bounded by
    \begin{equation}
        C \int_\mathbb{R} \diff u \int_\mathbb{R} \diff v e^{-\pi |u-v|} \left(\frac{1+|u|}{1+|v|} \right)^{\sum_{i=1}^N A_i - B_i} (1+|u|)^{-2Nd} (1+|v|)^{-2Nd},
    \end{equation}
    which is again bounded by a constant since $4Nd>1$, as it can be clearly seen after a change of variables $u \mapsto \frac{\varphi + \theta}{2} , v \mapsto \frac{\varphi - \theta}{2}$. We have then shown that
    \begin{equation} \label{eq:exponential_bound_log_gamma_kernel}
        \left| \mathbf{K} (X,Y) \right| \le C e^{-d (X + Y)},
    \end{equation}
    where $C$ is a constant independent of $X,Y$. We can now estimate the series expansion of the Fredholm determinant of $\mathbf{K}$ as
    \begin{equation} \label{eq:fredholm_det_absolute_conv}
        \begin{split}
            &
            \sum_{\ell \ge 0} \frac{1}{\ell!} \int_{[\varsigma,+\infty)^\ell} \diff X_1 \cdots \diff X_\ell \left| \det[\mathbf{K}(X_i,X_j)]_{i,j=1}^\ell \right|
            \\
            & \le \sum_{\ell \ge 0} \frac{1}{\ell!} \int_{[\varsigma,+\infty)^\ell} \diff X_1 \cdots \diff X_\ell \, \ell^{\ell/2} C^\ell e^{-d(X_1+\cdots + X_\ell)}
            \\
            &= \sum_{\ell \ge 0} \frac{1}{\ell!} \ell^{\ell/2} \left( \frac{C e^{-\varsigma d} }{d} \right)^\ell < \infty,
        \end{split}
    \end{equation}
    where in the first inequality we used the Hadamard's bound \eqref{eq:Hadamard_bound}. This completes the proof.
\end{proof}

We now present the proof of \cref{thm:fredholm_det_log_gamma}. For this we introduce the \emph{$q$-Gamma function} of a complex variable $X\in \mathbb{C} \setminus \mathbb{Z}_{\le 0}$
\begin{equation} \label{eq:q_gamma}
    \Gamma_q(X) = (1-q)^{1-X} \frac{(q;q)_\infty}{(q^X;q)_\infty}.
\end{equation}
Arguments used to prove our formula \eqref{eq:fredholm_det_log_gamma} use a number of technical bounds on ratio of $q$-Pochhammer symbols and $q$-Gamma functions, which for the sake of readability are reported in \cref{sec:LG}.

\begin{proof}[Proof of \cref{thm:fredholm_det_log_gamma}]
    Our starting point is the Fredholm determinant representation for the probability distribution of a tagged particle $\mathsf{x}_N(T)$ in the $q$-PushTASEP \eqref{eq:fredholm_det_q_PushTASEP}. 
    Adopting the scaling 
    \begin{equation} \label{eq:scaling_convergence_log_full_space}
        q=e^{-\varepsilon}, 
        \qquad a_i=e^{-\varepsilon A_i}, \qquad b_i=e^{-\varepsilon B_i}, \qquad \zeta=(1-q)^{2N} e^{-\varsigma},
    \end{equation}
    we find that, assuming $\chi \sim q$-$\mathrm{Geo}(q)$ and $S \sim \mathrm{Theta}(\zeta;q)$ and using \cref{prop:qPushTASEP_convergence}, \cref{lem:convergence_gumbel,lem:convergence_theta},
    \begin{equation}
        \mathsf{x}_N(N) + \chi + S = \varepsilon^{-1} \log Z(N,N) + \varepsilon^{-1} \mathcal{G} - \varepsilon^{-1} \varsigma + \varepsilon^{-1/2} \mathcal{N}(0,1) + O(1),
    \end{equation}
    where $\mathcal{G} \sim \mathrm{Gumbel}$ and $\mathcal{N}(0,1)$ is a standard normal random variable.
    In this way, the left hand side of \eqref{eq:fredholm_det_q_PushTASEP} becomes
    \begin{equation}
    \begin{split}
        \lim_{\varepsilon\to 0} \mathbb{P}(\mathsf{x}_N(N) + \chi + S < N+s) &= \mathbb{P} ( \log Z(N,N) + \mathcal{G} < \varsigma ) 
        \\
        & = \mathbb{E}\left( e^{-e^{-\varsigma + \log Z(N,N)}} \right),
    \end{split}
    \end{equation}
    where the second equality was discussed in \cref{rem:Laplace_Gumbel} and comes from a direct evaluation of the probability in the middle term using the definition of the Gumbel distribution.
    
    We now move our attention to the right hand side of \eqref{eq:fredholm_det_q_PushTASEP}. Under the scaling \eqref{eq:scaling_convergence_log_full_space} the integral kernel $K_\mathrm{push}$, using the notion of $q$-Gamma function \eqref{eq:q_gamma}, can be written as
    \begin{equation} \label{eq:kernel_K_push_rescaled}
    \begin{split}
        K_\mathrm{push}( \varepsilon^{-1} X, \varepsilon^{-1} Y ) = \frac{\varepsilon}{(2\pi \mathrm{i})^2} & \int_{ \substack{ Z \in \mathcal{D}(-d,\varepsilon) \\ W \in \mathcal{D}(d,\varepsilon) }} \frac{\diff Z \diff W}{q^{(Z-W)/2}} e^{XZ-YW}
        \\
        &
        \times 
        \prod_{i=1}^N  \frac{\Gamma_q(A_i+Z) \Gamma_q(B_i - W)}{\Gamma_q(B_i-Z) \Gamma_q(A_i+W)}
        \Gamma_q(Z-W) \Gamma_q(1+W-Z) 
        \\
        &\times (1-q)^{2N(Z-W)} \frac{\vartheta_3(q^{Z-W}e^\varsigma (1-q)^{2N};q)}{\vartheta_3(e^\varsigma (1-q)^{2N};q)},
    \end{split}
    \end{equation}
    where, compared to \eqref{eq:kernel_periodic_Schur}, we operated a change of variables $z=q^Z, w=q^W$, taking as radii of the integration contours $r=q^{d}, r'=q^{-d}$. For the integration contours we have used the notation
    \begin{equation} \label{eq:integration_contours}
        \mathcal{D}(d,\varepsilon) = d+\mathrm{i} [-\frac{\pi}{\varepsilon}, \frac{\pi}{\varepsilon}].
    \end{equation}
    By virtue of \cref{lem:convergence_ratio_theta,lem:convergence_qGamma}, the integrand in \eqref{eq:kernel_K_push_rescaled} converges pointwise to the integrand of \eqref{eq:K_LG}. To prove the pointwise convergence $\varepsilon^{-1}K_\mathrm{push}(\varepsilon^{-1}X,\varepsilon^{-1}Y) \xrightarrow[]{\varepsilon \to 0} \mathbf{K}(X+\varsigma,Y+\varsigma)$ we notice that the integrand in \eqref{eq:kernel_K_push_rescaled} is bounded in absolute value by 
    \begin{equation} \label{eq:upper_bound_K_push}
        C_1 e^{-d (X + Y)} \prod_{i=1}^N (1+|u|)^{-(B_i-A_i)-2d} (1+|v|)^{-(A_i-B_i)-2d} e^{-C_2 |u-v|},
    \end{equation}
    where $u=\Im\{Z\},v=\Im\{W\}$ and $C_1,C_2$ are constants independent of $q,Z,W,X,Y$; this can be shown combining \cref{lem:convergence_ratio_theta,lem:decay_qGamma,lem:bound_ratio_qgamma}. The function in \eqref{eq:upper_bound_K_push} is summable, as can be proven following the same ideas used in \cref{lem:K_tau_hilbert_schmidt}. 
    Again bounding the integrand of \eqref{eq:kernel_K_push_rescaled} with \eqref{eq:upper_bound_K_push} we obtain the bound
    \begin{equation}
        \left| \varepsilon^{-1} K_\mathrm{push} (\varepsilon^{-1}X, \varepsilon^{-1}Y) \right| \le C e^{-d (X + Y)},
    \end{equation}
    which hold for a constant $C$ independent of $\varepsilon,X,Y$. This implies that
    \begin{equation}
        \lim_{\varepsilon \to 0} \det( 1 - K_\mathrm{push})_{\ell^{2}(\mathbb{Z}_{> s}')} = \det (1 - \mathbf{K})_{\mathbb{L}^2(\varsigma,+\infty)},
    \end{equation}
    by dominated convergence, concluding the proof.
\end{proof}

\subsection{Fredholm pfaffian representation for $q$-PushTASEP in half space}

Here we present a Fredholm pfaffian formula for the distribution of the rightmost particle in the $q$-PushTASEP with particle creation. With the exception of \cite{barraquand2018}, where a Pfaffian representation was found for the current distribution of the ASEP with an open boundary kept at critical density, \cref{cor:pfaffian_q_PushTASEP} below represents the first result of this kind for discrete solvable models and it holds for any admissible choice of parameters $\gamma,a_1,a_2,\dots$.

\begin{corollary} \label{cor:pfaffian_q_PushTASEP}
    Let $\mathsf{X}^{\mathrm{hs}}(T)$ be the geometric $q$-PushTASEP with particle creation with parameters $\gamma,a_1,a_2,\dots \in (0,1)$ and empty initial conditions. Consider also $\chi \sim q$-$\mathrm{Geo}(q)$ and $S\sim \mathrm{Theta}(\zeta^2,q^2)$ to be independent random variables. Then we have
    \begin{equation} \label{eq:qpush_TASEP_pfaffian}
        \mathbb{P} \left( \mathsf{x}^{\mathrm{hs}}_N(N) - N + \chi + 2 S \le s \right) = \Pf \left( J - L_\mathrm{push} \right)_{\ell^2 (\mathbb{Z}'_{> s} )},
    \end{equation}
    where $L_\mathrm{push}$ is the  $2 \times 2$ matrix kernel obtained from \eqref{eq:L_fbs}, setting
    \begin{equation}
        F(z)= \frac{(\gamma/z;q)_\infty}{(\gamma z;q)_\infty} \prod_{i=1}^N \frac{(a_i/z;q)_\infty}{(a_i z;q)_\infty}.
    \end{equation}
\end{corollary}

\begin{proof}
    As in \cref{subs:hs_qPushTASEP}, we call $\gamma,a_1,a_2,\dots$ the parameters describing the jump process in the $q$-PushTASEP with particle creation. Defining $a'=(a_1,\dots,a_N,\gamma)$ and taking $\tilde{\mu} \sim \mathbb{HW}^{(q)}_{a';0}$ we have the equality in distribution
    \begin{equation}
        \mathsf{x}^{\mathrm{hs}}_N(N) = \tilde{\mu}_1 + N,
    \end{equation}
    which follows combining \cref{prop:matching_hs_qPushTASEP_qW} with the symmetry of the half space $q$-Whittaker measure of \cref{prop:symmmetry_half_space_qW}. We can now use \cref{thm:matching_hs_qW_fbs} to express the law of $\tilde{\mu}_1+ \chi + 2S$ in terms of the rightmost particle in the Free boundary Schur measure whose Fredholm pfaffian representation is reported in \cref{prop:fredholm_pfaffian_fbs}, concluding the proof. 
\end{proof}

\begin{remark}
    Notice that, unlike in the case of $q$-PushTASEP in full space, the left hand side of \eqref{eq:qpush_TASEP_pfaffian} is not the $q$-Laplace transform of the probability distribution of $\mathsf{x}^{\mathrm{hs}}_N(N)$.
\end{remark}

The $q$-PushTASEP with particle creation was introduced and studied in \cite{barraquand_half_space_mac}. In the same paper authors, using Macdonald difference operators, recovered integral representations for joint $q$-moments of the model and a representation of the $q$-Laplace transform of the distribution of a tagged particle as a series of nested integrals; see \cite[Corollary 4.7]{barraquand_half_space_mac}. Due to the presence of certain cross terms it was not possible for the authors to rearrange their expression as the expansion of a Fredholm pfaffian or determinant. Our result \eqref{eq:qpush_TASEP_pfaffian} indicates that an observable different than the $q$-Laplace transform of $\mathsf{x}^{\mathrm{hs}}_N(N)$ has the potential to be studied through $q$-moment computations.

\subsection{Fredholm pfaffian representation for Log Gamma polymer in half space} \label{subs:pfaffian_log_gamma}

Here we consider the point-to-point partition function $Z^{\mathrm{hs}}(N,N)$ in the Log Gamma polymer in half space and we present Fredholm pfaffian representations for the Laplace transform of its probability distribution. First, in \cref{thm:fredholm_pfaff_log_gamma}, we will consider the case when the boundary parameter $\varUpsilon$ is positive. To extend such solution to negative values of $\varUpsilon$ requires an analytic continuation, which will lead to the result of \cref{thm:fredholm_pfaff_log_gamma_gaussian}, presented in the next subsection. Formulas reported below are ameanable of asymptotic analysis, which will be performed in \cref{sec:asymptotics}. 

\begin{theorem} \label{thm:fredholm_pfaff_log_gamma}
    Let $Z^{\mathrm{hs}}(N,N)$ be the partition function of the Log Gamma polymer model in half space with parameters $A_1,\dots,A_N>0$ and boundary strength $\varUpsilon>0$. Take a number $d$ 
    such that
    \begin{equation} \label{eq:conditions_delta}
        0 < d < \min \left\{ \frac{1}{2}, \varUpsilon,A_1,\dots,A_N \right\}
        \qquad
        \text{and}
        \qquad
        d >\frac{1}{2N}.
    \end{equation}
    Then, we have
    \begin{equation} \label{eq:fredholm_pfaff_log_gamma}
        \mathbb{E}\left[ e^{-e^{-\varsigma+ \log Z^{\mathrm{hs}}(N,N)} } \right] = \Pf (J - \mathbf{L})_{\mathbb{L}^2(\varsigma,+\infty)},
    \end{equation}
    where $\mathbf{L}$ is the $2 \times 2$ matrix kernel
    \begin{equation} \label{eq:L_log_gamma}
        \mathbf{L}(X,Y) = \left( \begin{matrix} \mathbf{k}(X,Y) & - \partial_Y \mathbf{k}(X,Y) \\ - \partial_X \mathbf{k}(X,Y) & \partial_X \partial_Y \mathbf{k}(X,Y) \end{matrix} \right)
    \end{equation} 
    with
    \begin{equation}
            \mathbf{k}(X,Y) = \int_{\mathrm{i} \mathbb{R} - d} \frac{\diff Z}{2 \pi \mathrm{i}} \int_{\mathrm{i} \mathbb{R} - d}
            \frac{\diff W}{2 \pi \mathrm{i}} e^{ X Z + Y W} g_{\varUpsilon;A}(Z) g_{\varUpsilon;A}(W) \mathsf{k}^{\mathrm{hs}}(Z,W),
            \label{eq:k_hs_log_gamma}
    \end{equation}
    \begin{equation}
        g_{\varUpsilon;A}(Z) =
        \frac{\Gamma(\varUpsilon + Z)}{ \Gamma(\varUpsilon - Z)}\prod_{i=1}^N \frac{\Gamma(A_i + Z)}{ \Gamma(A_i - Z)}
    \end{equation}
    and
    \begin{equation}
    \mathsf{k}^{\mathrm{hs}}(Z,W) = \Gamma(-2Z) \Gamma(-2W) \frac{\sin\left[\pi( Z-W )\right]}{\sin\left[\pi( Z+W )\right]}.
    \end{equation}
\end{theorem}

In literature explicit formulas for the expectation $\mathbb{E}\left[ e^{-e^{-\varsigma} Z^\mathrm{hs}(N,N)} \right]$ were derived in \cite[Corollary 5.6]{OSZ2012} and \cite[Corollary 6.40]{barraquand_half_space_mac}, although in none of these cases they were expressed as Fredholm pfaffians. It would be interesting to understand if it is possible to prove our formula \eqref{eq:pfaff_L_hat_log_gamma} starting from these other explicit results. Moreover, recently in \cite{barraquand_rychnovky_half_space} authors were able to prove, using Bethe Ansatz, pfaffian formulas for a polymer model in half space with certain Beta weights. It would be interesting to understand if our techniques can be extended to prove identities for the Beta polymer model in half space.

\medskip

The proof of \cref{thm:fredholm_pfaff_log_gamma} will be presented at the end of this subseciton. It is based on a scaling limit of the Fredholm pfaffian formula of \cref{cor:pfaffian_q_PushTASEP} to establish which technical results reported in \cref{sec:LG} will be used. In the next proposition, we confirm that the right hand side of \eqref{eq:fredholm_pfaff_log_gamma} defines a numerically convergent series and hence the relation is well posed.

\begin{proposition} \label{prop:fredholm_pfaffian_absolutely_convergent}
    The Fredholm pfaffian of the $2\times 2$ matrix kernel $\mathbf{L}$ is absolutely convergent.
\end{proposition}

\begin{proof}
    We make use of the asymptotic expansion \eqref{eq:decay_Gamma} to estimate the integrand in \eqref{eq:k_hs_log_gamma}. Following ideas already elaborated in the proof of \cref{lem:K_tau_hilbert_schmidt} we obtain the bound
    \begin{equation} \label{eq:bound_k}
    \begin{split}
        \left| \mathbf{k}(X,Y) \right| &\le C e^{-d (X + Y)} \int_\mathbb{R} \diff u \int_\mathbb{R} \diff v [(1+|u|) (1+|v|)]^{-2Nd -1/2} e^{-\pi|u+v|}
        \\
        &
        \le \frac{C}{2Nd} e^{-d(X+Y)},
    \end{split}
    \end{equation}
    where $C$ is a constant independent of $X,Y$.
    Conditions \eqref{eq:conditions_delta} on parameter $d$ imply that the double integral in the right hand side of the inequality above is convergent and bounded by a constant independent of $X,Y$. Through similar argument we easily get
    \begin{equation} \label{eq:exponential_estimates_k}
        | \mathbf{k}(X,Y) |, | \partial_X \mathbf{k}(X,Y) |,| \partial_Y \mathbf{k}(X,Y) |, | \partial_X \partial_Y \mathbf{k}(X,Y) | \le D e^{-d (X + Y)},
    \end{equation}
    where $D$ is another constant. We can now estimate each term of the expansion of the Fredholm pfaffian of $\mathbf{L}$ and, using the Hadamard's bound \eqref{eq:Hadamard_bound}, we have 
    \begin{equation} \label{eq:fredholm_pfaffian_absolute_conv}
        \begin{split}
            &
            \sum_{\ell \ge 0} \frac{1}{\ell!} \int_{[\varsigma,+\infty)^\ell} \diff X_1 \cdots \diff X_\ell \left| \Pf[\mathbf{L}(X_i,X_j)]_{i,j=1}^\ell \right|
            \\
            & \le \sum_{\ell \ge 0} \frac{1}{\ell!} \int_{[\varsigma,+\infty)^\ell} \diff X_1 \cdots \diff X_\ell (2\ell )^{\ell/2} D^\ell e^{-d(X_1+\cdots + X_\ell)}
            \\
            &= \sum_{\ell \ge 0} \frac{1}{\ell!} \ell^{\ell/2} \left( \frac{\sqrt{2} D e^{-\varsigma d} }{d} \right)^\ell < \infty,
        \end{split}
    \end{equation}
    which proves the result.
\end{proof}

\begin{proof}[Proof of \cref{thm:fredholm_pfaff_log_gamma}]
    We proceed as in \cref{thm:fredholm_det_log_gamma}. Our starting point in this case is the Fredholm pfaffian representation \eqref{eq:qpush_TASEP_pfaffian} of the probability distribution for the particle $\mathsf{x}^{\mathrm{hs}}_N(N)$ of a $q$-PushTASEP with particle creation. Using the convergence result of \cref{prop:hs_qPushTASEP_convergence}, along with \cref{lem:convergence_gumbel,lem:convergence_theta}, imposing the scaling,
    \begin{equation}
        q=e^{-\varepsilon}, 
        \qquad
        a_i=e^{-\varepsilon A_i}, 
        \qquad
        \gamma = e^{-\varepsilon \varUpsilon},
        \qquad
        \zeta= (1-q^2)^N e^{-\varsigma/2},
    \end{equation}
    we have
    \begin{equation}
        \mathsf{x}_N^{\mathrm{hs}}(N) - N + \chi +2S = \varepsilon^{-1} \left( \log Z^{\mathrm{hs}}(N,N) + \mathcal{G} -\varsigma + \sqrt{2 \varepsilon} \mathcal{N}(0,1) \right) + O(1).
    \end{equation}
    This shows that
    \begin{equation}
    \begin{split}
        \lim_{\varepsilon \to 0} \mathbb{P}(\mathsf{x}_N^{\mathrm{hs}}(N) - N + \chi +2S < s) &= \mathbb{P}(\log Z^{\mathrm{hs}}(N,N) + \mathcal{G} < \varsigma ) 
        \\
        & = \mathbb{E} \left( e^{-e^{ -\varsigma + \log Z^{\mathrm{hs}} (N,N)}} \right),
    \end{split}
    \end{equation}
    and it takes care of the left hand side of \eqref{eq:qpush_TASEP_pfaffian}. Let us now take the same limit on the right hand side. We can rewrite the $2\times 2$ matrix kernel prescribed by \cref{cor:pfaffian_q_PushTASEP} as
    \begin{equation}
      L_{\mathrm{push}} ( \varepsilon^{-1} X, \varepsilon^{-1} Y ) = \left( \begin{matrix} \varepsilon  & - \nabla_y \\ - \nabla_x & \varepsilon^{-1} \nabla_x \nabla_y  \end{matrix} \right) k( \varepsilon^{-1} X, \varepsilon^{-1} Y ),
    \end{equation}
    with
    \begin{equation}
        k ( \varepsilon^{-1} X, \varepsilon^{-1} Y ) = \frac{1}{(2\pi \mathrm{i})^2} \frac{\varepsilon^2}{(1-q)^2} \int_{ \substack{ Z \in \mathcal{D}(-d,\varepsilon) \\ W \in \mathcal{D}(-d,\varepsilon) }} \frac{\diff Z \diff W}{q^{(Z+3W)/2}} e^{XZ+YW} f(Z)f(W) \tilde{\kappa}^{\mathrm{hs}}(Z,W).     
    \end{equation}
    Here, with respect to \eqref{eq:k_fbs} we operated the change of variables $z=q^Z,w=q^W$ and we made use of auxiliary functions
    \begin{gather*}
        \tilde{\kappa}^{\mathrm{hs}}(Z,W) = \frac{\Gamma_q(-2Z)\Gamma_q(-2W) \Gamma_q(1+Z+W) \Gamma_q(-Z-W)}{\Gamma_q(W-Z) \Gamma_q(1+Z+W)} E_q(\varsigma,Z+W),
        \\
        f(Z) = \frac{\Gamma_q(\varUpsilon+Z)}{\Gamma_q(\varUpsilon-Z)} \prod_{i=1}^N \frac{\Gamma_q(A_i+Z)}{\Gamma_q(A_i-Z)},
        \\
        E_q(\varsigma,U) = (1-q)^{2N U} \frac{ \vartheta_3((1-q)^{2N} e^{-\varsigma} q^{2U}; q^2 )}{ \vartheta_3((1-q)^{2N} e^{-\varsigma}; q^2 ) }.
    \end{gather*}
    For the integration contours we used again the notation \eqref{eq:integration_contours}. Notice that the discrete derivatives $\nabla_x$ and $\nabla_y$ modify the function $k$ by inserting respectively factors $\frac{1}{2}(q^{-Z}-q^Z) \approx \varepsilon Z$ and $\frac{1}{2}(q^{-W}-q^W) \approx \varepsilon W$ in the integrand. By means of estimates reported in \cref{lem:convergence_ratio_theta,lem:bound_ratio_qgamma} and \cref{prop:bound_qgamma_k11} we can bound the absolute value of the integrand in the expression of $k$ with summable functions proving that
    \begin{equation}
        \lim_{\varepsilon \to 0} k(\varepsilon^{-1} X, \varepsilon^{-1}Y)  = \mathbf{k}(X+\varsigma,Y+ \varsigma ),
    \end{equation}
    by the bounded convergence theorem and moreover that 
    \begin{equation}
        |k(\varepsilon^{-1} X, \varepsilon^{-1}Y)| \le C e^{-d (X + Y)},
    \end{equation}
    for some constant $C$ independent of $\varepsilon,X,Y$. Analogous arguments are applicable to other elements of the matrix kernel $L_{\mathrm{push}}$ and we find
    \begin{multline*}
        \lim_{\varepsilon \to 0} \left( \varepsilon^{-1} \nabla_y k(\varepsilon^{-1} X, \varepsilon^{-1}Y), \varepsilon^{-1} \nabla_x k(\varepsilon^{-1} X, \varepsilon^{-1}Y) , \varepsilon^{-2} \nabla_x \nabla_y k(\varepsilon^{-1} X, \varepsilon^{-1}Y)  \right) ,
        \\ = \left( \partial_Y \mathbf{k}(X+ \varsigma,Y+ \varsigma ), \partial_X \mathbf{k}(X+ \varsigma ,Y+ \varsigma ), \partial_X \partial_Y \mathbf{k}(X+ \varsigma ,Y+ \varsigma) \right)
    \end{multline*}
    and moreover
    \begin{equation*}
        |\varepsilon^{-1} \nabla_y k(\varepsilon^{-1} X, \varepsilon^{-1}Y)|, |\varepsilon^{-1} \nabla_x k(\varepsilon^{-1} X, \varepsilon^{-1}Y)|, |\varepsilon^{-2} \nabla_x \nabla_y k(\varepsilon^{-1} X, \varepsilon^{-1}Y)| \le C e^{-d (X + Y)},
    \end{equation*}
    where the constant $C$ can be chosen to be the same for all terms. These results imply, by bounded convergence, that
    \begin{equation}
        \lim_{\varepsilon \downarrow 0 }\Pf(J-L_{\mathrm{push}})_{\ell^2(\mathbb{Z}'_{> s})} = \Pf(J - \mathbf{L})_{\mathbb{L}^2( \varsigma,+\infty)},    
    \end{equation}
    proving the desired identity \eqref{eq:fredholm_pfaff_log_gamma}.
\end{proof}

\subsection{Analytic extension of Fredholm pfaffian}

We now want to characterize the distribution of the half space polymer partition function $Z^{\mathrm{hs}}(N,N)$ for choices of boundary parameter $\varUpsilon \le 0$. This will require an analytic continuation of the relation \eqref{eq:fredholm_pfaff_log_gamma}. From the explicit representation of the function $\mathbf{k}$, we see that moving the parameter $\varUpsilon$ from the positive to the negative real semi-axis several singularities for the $Z$ and $W$ variables cross the integration contours corresponding to values $-\varUpsilon,-\varUpsilon-1,\dots,-\varUpsilon-m+1$, if we assume that $\varUpsilon$ ends up in the interval $(-m,-m+1]$. Defining functions
\begin{equation} \label{eq:p_j}
    \mathbf{A}_j(X) = \frac{(-1)^j}{j!} e^{-(j+\varUpsilon)X} \frac{\Gamma(2\varUpsilon+2j)}{\Gamma(2\varUpsilon+j) } \prod_{i=1}^N \frac{\Gamma(A_i-\varUpsilon-j)}{\Gamma(A_i + \varUpsilon +j)},
\end{equation}
\begin{equation} \label{eq:q}
        \mathbf{B}(Y) = \int_{\mathrm{i}\mathbb{R} - d} \frac{\diff W}{2\pi \mathrm{i}} e^{YW} \Gamma(-2W)
        \frac{\Gamma(1-\varUpsilon+W)}{\Gamma(1-\varUpsilon -W)}
        \prod_{i=1}^N \frac{\Gamma(A_i+W)}{\Gamma(A_i-W)},
\end{equation}
we evaluate the $Z$-residue at $-\varUpsilon-j$ of the integrand of $\mathbf{k}$ and from \eqref{eq:k_hs_log_gamma}, we find it to be
\begin{equation}
    \mathrm{Res}_{Z=-\varUpsilon-j} \left\{ \parbox{2.2cm}{\centering $Z$-integrand in \eqref{eq:k_hs_log_gamma}} \right\} =
    \mathbf{A}_j(X) \mathbf{B}(Y).
\end{equation}
Analogously the $W$-residue at $-\varUpsilon-j$ of the integrand of $\mathbf{k}$ is
\begin{equation}
    \mathrm{Res}_{W=-\varUpsilon-j} \left\{ \parbox{2.2cm}{\centering $W$-integrand in \eqref{eq:k_hs_log_gamma}} \right\} = - \mathbf{A}_j(Y) \mathbf{B}(X). 
\end{equation}
These computations allow us to identify the analytic continuation of the $2\times 2$ matrix kernel $\mathbf{L}$, by adding residues $\mathbf{A}_j(X) \mathbf{B}(Y) - \mathbf{A}_j(Y) \mathbf{B}(X)$ to the function $\mathbf{k}$ in each component of \eqref{eq:L_log_gamma}. However, when $\varUpsilon$ is negative and $j+\varUpsilon<0$ we see that functions $\mathbf{A}_j(X)$'s are exponentially diverging as $X$ grows. This means that the analytic continuation of $\mathbf{L}$ naively proposed above produces an operator which is not well defined on $\mathbb{L}^2(\varsigma,\infty) \times \mathbb{L}^2(\varsigma,\infty)$ and therefore we cannot immediately make sense of its Fredholm pfaffian. In the next theorem we resolve this issue and present the mathematically correct extension of identity \eqref{eq:fredholm_pfaff_log_gamma}. To state our result we introduce the scalar product $\langle (\cdot, \cdot) | (\cdot, \cdot) \rangle $ on the Hilbert space $\mathbb{L}^2(\varsigma,\infty) \times \mathbb{L}^2(\varsigma,\infty)$ as
\begin{equation}
    \langle (f_1,f_2) | (g_1,g_2) \rangle =\int_{\varsigma}^{+\infty} \left[ f_1(x)g_1(x) + f_2(x)g_2(x) \right]  \diff x.
\end{equation}

\begin{theorem} \label{thm:fredholm_pfaff_log_gamma_gaussian}
    Let $Z^{\mathrm{hs}}$ be the partition function of the Log Gamma polymer model in half space with parameters $A_1,\dots,A_N>0$ and boundary strength $\varUpsilon$ such that $-\min\{A_1,\dots,A_N\}<\varUpsilon \le 0$. Let $m\in\mathbb{Z}_{\ge 0}$ and $\varepsilon\in \{0,1\}$ be such that $\varUpsilon \in (-m+\frac{1-\varepsilon}{2},-m+\frac{2-\varepsilon}{2}]$ and take a number $d$ satisfying
    \begin{equation}
        0 < d < \min\left\{ \frac{1}{2}, \varUpsilon + m, A_1,\dots,A_N \right\}
        \qquad
        \text{and}
        \qquad
        d>\frac{1}{2N}.
    \end{equation}
    Then, there exists $N^*$ such that, for all $N \ge N^*$ we have
    \begin{equation} \label{eq:pfaff_L_hat_log_gamma}
    \begin{split}
        \mathbb{E}\left[ e^{-e^{-\varsigma + \log Z^{\mathrm{hs}}(N,N)}} \right] = & \Pf \left[ J - \widehat{\mathbf{L}} \right]_{\mathbb{L}^2(\varsigma,+\infty)}
        \\
        & \qquad
        \times \bigg( 1-  \langle (\mathbf{B},-\mathbf{B}') | \left(1-J^T \widehat{\mathbf{L}}\right)^{-1} | (\mathbf{A}',\mathbf{A} )\rangle \bigg),
    \end{split}
    \end{equation}
    where 
    \begin{equation} \label{eq:L_hat}
        \widehat{\mathbf{L}}(X,Y) = \mathbf{L}(X,Y) - \left( \begin{matrix} 1 & -\partial_Y \\ -\partial_X & \partial_X \partial_Y \end{matrix} \right) \sum_{j=m}^{2m+\varepsilon-2} \left[ \mathbf{A}_j(X) \mathbf{B}(Y) - \mathbf{A}_j(Y) \mathbf{B}(X) \right],
    \end{equation}
    \begin{equation} \label{eq:p}
        \mathbf{A}(X) = \sum_{j=0}^{2m+\varepsilon-2} \mathbf{A}_j(X).
    \end{equation}
    Notice that when $m=0$ we have $\widehat{\mathbf{L}} = \mathbf{L}$, since the summations in \eqref{eq:L_hat} is empty.
\end{theorem}

\begin{lemma} \label{lem:norm_L}
    Consider parameters $A_1,\dots, A_N > 0$ and $\varUpsilon > - \min\{A_1,\dots,A_N \}$. Then there exists $N^*$, depending on $\varUpsilon$ such that for each $N\ge N^*$ the operator $\mathbf{L}$ has norm $\| \mathbf{L} \|<1$.
\end{lemma}
\begin{proof}
    This is a straightforward consequence of bound \eqref{eq:bound_k} and of analogous estimates for partial derivatives of $\mathbf{k}$, which one can easily produce also when the parameter $\varUpsilon$ becomes negative.
\end{proof}

\begin{proof}[Proof of \cref{thm:fredholm_pfaff_log_gamma_gaussian}]
    As hinted above before the statement of the theorem, identity \eqref{eq:pfaff_L_hat_log_gamma} will follow from \eqref{eq:fredholm_pfaff_log_gamma} after an analytic continuation in the parameter $\varUpsilon$ to the region $\varUpsilon \le 0$. We first observe that the Laplace transform $\mathbb{E}\left( e^{-e^{-\varsigma} Z^{\mathrm{hs}}(N,N) } \right)$ is clearly an analytic function of the variable $\varUpsilon$ in the region $\varUpsilon > - \min\{A_1,\dots,A_N\}$, since the $\varUpsilon$ dependence is given by boundary random variables \eqref{eq:random_environment}. To show that the same is true for the Fredholm pfaffian in the right hand side of \eqref{eq:fredholm_pfaff_log_gamma}, we show that the right hand side of \eqref{eq:pfaff_L_hat_log_gamma} is an analytic function of the variable $\varUpsilon$ that is well defined for $\varUpsilon>-m$ and that it equals $\Pf[J-\mathbf{L}]$, of \eqref{eq:fredholm_pfaff_log_gamma} when $\Re\{\varUpsilon\}\in(0,1)$.
    
    Let us start by assuming that parameters $A_1,\dots,A_N>2m$. This hypothesis can be removed at any time through analytic continuation and it only serves to ensure that through the procedure we describe below we do not have to worry about possible singularities coming from factors $\Gamma(A_i+Z),\Gamma(A_i+W)$ in the function $\mathbf{k}$ in \eqref{eq:k_hs_log_gamma}.  
    As discussed above, the problem with the argument presented prior to the statement of the theorem is that perturbations of the kernel $\mathbf{L}$ given by combinations of functions $\mathbf{A}_j,\mathbf{B}$ and their derivatives define an unbounded operator when $\varUpsilon<0$. Let us then assume that $\varUpsilon>0$ and define the kernel $\mathbf{R}(X,Y)$
    \begin{equation}
        \mathbf{R} (X,Y) = \left( \begin{matrix} \mathbf{A}(X) \mathbf{B}(Y) - \mathbf{A}(Y) \mathbf{B}(X) & -\mathbf{A}(X) \mathbf{B}'(Y) + \mathbf{A}'(Y) \mathbf{B}(X) \\ -\mathbf{A}'(X) \mathbf{B}(Y) + \mathbf{A}(Y) \mathbf{B}'(X) & \mathbf{A}'(X) \mathbf{B}'(Y) - \mathbf{A}'(Y) \mathbf{B}'(X) \end{matrix} \right).
    \end{equation}
    By simple estimates 
    \begin{equation} \label{eq:bound_p,q_1}
        |\mathbf{A}(X)|,|\mathbf{A}'(X)| \le D e^{-\varUpsilon X},
        \qquad
        |\mathbf{B}(Y)|,|\mathbf{B}'(Y)| \le D e^{-d Y},
    \end{equation}
    holding for a constant $D>0$, which can be made arbitrarily small choosing $N$ large, we deduce that, in case $\varUpsilon>0$, $\mathbf{R}(X,Y)$ defines a rank 2 operator $\mathbf{R}:\mathbb{L}^2(-\varsigma,+\infty) \times \mathbb{L}^2(-\varsigma,+\infty) \to \mathbb{L}^2(-\varsigma,+\infty) \times \mathbb{L}^2(-\varsigma,+\infty)$. Applying \cref{prop:pfaffian_L+R} we produce the identity
    \begin{equation} \label{eq:pfaffian_L-R+R}
    \begin{split}
        \Pf(J-\mathbf{L}) & = \Pf(J-(\mathbf{L}-\mathbf{R}) -\mathbf{R})  
        \\
        & = \Pf(J-(\mathbf{L}-\mathbf{R}))\bigg(1-\langle(\mathbf{B},-\mathbf{B}') | (1-J^T(\mathbf{L} - \mathbf{R}))^{-1} | (\mathbf{A}',\mathbf{A}) \rangle \bigg).
    \end{split}
    \end{equation}
    In the right hand side the inverse operator is defined by the Neumann series $(1-J^T(\mathbf{L} - \mathbf{R}))^{-1} = \sum_{\ell \ge 0} \left(J^T(\mathbf{L} - \mathbf{R}) \right)^\ell $ and such series is absolutely convergent for $N$ large enough as a result of \cref{lem:norm_L} and \eqref{eq:bound_p,q_1}. We now show that the right hand side of \eqref{eq:pfaffian_L-R+R} remains well defined and analytic even when we drag the parameter $\varUpsilon$ in the region $\Re\{\varUpsilon\}<0$. First we focus on the kernel $\mathbf{L}-\mathbf{R}$. Moving $\varUpsilon$ inside the segment $(-m+\frac{1-\varepsilon}{2},-m+\frac{2-\varepsilon}{2}]$, poles at $-\varUpsilon,-\varUpsilon-1,\dots,-\varUpsilon-m+1$ cross the integration contours of the function $\mathbf{k}$ so that the kernel $\mathbf{L}$ becomes
    \begin{equation}
        \mathbf{L} + \left( \begin{matrix} 1 & -\partial_Y \\ -\partial_X & \partial_X \partial_Y \end{matrix} \right) \sum_{j=0}^{m-1} \left[ \mathbf{A}_j(X) \mathbf{B}(Y) - \mathbf{A}_j(Y) \mathbf{B}(X) \right]
    \end{equation}
    and hence the first $m$ terms in the sum over $j$ of $\mathbf{A}$ \eqref{eq:p} in $\mathbf{R}$ cancel and $\mathbf{L}-\mathbf{R}$ gets transformed into $\widehat{\mathbf{L}}$. This proves that,
    for the parameter $\varUpsilon$ in this parameter region, the right hand side of \eqref{eq:pfaffian_L-R+R} is represented as
    \begin{equation} \label{eq:pfaffian_L_hat+R}
      \Pf(J-\widehat{\mathbf{L}})\bigg(1-\langle(\mathbf{B},-\mathbf{B}') | (1-J^T \widehat{\mathbf{L}})^{-1} | (\mathbf{A}',\mathbf{A}) \rangle \bigg),
    \end{equation}
    as long as we can show that both factors in the right hand side are absolutely convergent. The kernel $\widehat{\mathbf{L}}$ can be assumed to have operator norm $\| \widehat{\mathbf{L}} \| <1$, by choosing $N$ large. This follows from \cref{lem:norm_L} and from estimates
    \begin{equation} \label{eq:bound_p,q_2}
        |\mathbf{A}_j(X)|,|\mathbf{A}_j'(X)| \le D' e^{-(\varUpsilon+j)X}, 
        \qquad
        |\mathbf{B}(Y)|,|\mathbf{B}'(Y)| \le D' e^{(\varUpsilon-\delta) Y},
    \end{equation}
    where $D'>0$ can be taken small if $N$ is large and the bound for functions $\mathbf{B},\mathbf{B}'$ is derived from \eqref{eq:q} choosing integration contour $\mathrm{i} \mathbb{R} +\varUpsilon -\delta$, with $\delta \in (0,1)$. Thus, the first factor $\Pf(J-\widehat{\mathbf{L}})$ is well defined. Showing that the scalar product $\langle(\mathbf{B},-\mathbf{B}') | (1-J^T \widehat{\mathbf{L}})^{-1} | (\mathbf{A}',\mathbf{A}) \rangle$ is well posed for $\varUpsilon<0$ is a slightly more laborious task. As noted in \eqref{eq:bound_p,q_1}, the functions $\mathbf{A},\mathbf{A}'$ diverge exponentially for $\varUpsilon<0$ and hence we have to prove that for any $\ell \ge 0$ the quantity
    \begin{equation} \label{eq:scalar_product_q_JL_p}
        \langle(\mathbf{B},-\mathbf{B}') | (J^T \widehat{\mathbf{L}})^\ell | (\mathbf{A}',\mathbf{A}) \rangle
    \end{equation}
    is finite and exponentially decaying in $\ell$. When $\ell=0$ the scalar product,
   $     \langle(\mathbf{B},-\mathbf{B}') | (\mathbf{A}',\mathbf{A}) \rangle, $   
    is convergent since 
    \begin{equation}
        \mathbf{B}(X) \mathbf{A}'(X) , \mathbf{B}'(X) \mathbf{A}(X) = O(e^{-\delta X}),
    \end{equation}
    which is a consequence of estimates \eqref{eq:bound_p,q_1}, \eqref{eq:bound_p,q_2} respectively for functions $\mathbf{A},\mathbf{A}'$ and $\mathbf{B},\mathbf{B}'$. Let us now consider the case $\ell=1$, which is illuminating to treat the general $\ell$ case. We need to evaluate the asymptotic behavior of the function
    \begin{equation} \label{eq:J^TL.P}
        \left( J^T \widehat{\mathbf{L}} \right) \left( \begin{matrix} \mathbf{A}' \\ \mathbf{A}  \end{matrix} \right) (X) = \int_{\varsigma}^{+\infty} J^T \widehat{\mathbf{L}}(X,Y) \left( \begin{matrix} \mathbf{A}' (Y) \\ \mathbf{A} (Y) \end{matrix} \right) \diff Y.
    \end{equation}
    For this we will deform the integration contours of function $\mathbf{k}$ and its derivatives, entering coefficients of $\widehat{\mathbf{L}}$. Let us focus on the (1,1) element,
    \begin{equation} \label{eq:L_11}
        \widehat{\mathbf{L}}_{1,1}(X,Y) = \mathbf{k}(X,Y) - \sum_{j=m}^{2m+\varepsilon-2} \left[ \mathbf{A}_j(X) \mathbf{B}(Y) - \mathbf{A}_j(Y) \mathbf{B}(X) \right].
    \end{equation}
    Moving the $Z$ and $W$ integration contours for the integral expression of $\mathbf{k}$ in \eqref{eq:k_hs_log_gamma} from $\mathrm{i}\mathbb{R}-d$ to $\mathrm{i} \mathbb{R} + \varUpsilon -\delta$ we cross a number of poles for the $Z$ and $W$ variables. These occur at points $-\varUpsilon-m,\dots, -\varUpsilon-2m+\varepsilon-2$ and at points $Z=-W-1,\dots-W-2m-\varepsilon+2$, as a result of presence of functions $\Gamma(\varUpsilon+Z),\Gamma(\varUpsilon+W)$ and $\frac{1}{\sin[\pi(W+Z)]}$. The crossing of the first set of poles produces residues that cancel out with summation in the right hand side of \eqref{eq:L_11} and this is the motivation for including such terms in the kernel $\widehat{\mathbf{L}}$ in the first place. Residues coming from poles $Z=-W-j$ for $j=1,\dots, 2m+\varepsilon-2$ do not cancel out and their value is
    \begin{equation}
        S_j(X,Y)=-\int_{\mathrm{i} \mathbb{R} + R } \frac{\diff W}{2\pi \mathrm{i}} 
        e^{YW - X(W+j)} g_{\varUpsilon;A}(W)g_{\varUpsilon;A}(-W-j) \frac{\Gamma(2W+2j)}{\Gamma(2W+1)},
    \end{equation}
    where in the integration contour $R$ can be taken in the interval $(\varUpsilon-1,-\varUpsilon-j+1).$
    We obtain the equality
    \begin{equation} \label{eq:L_11_2}
        \widehat{\mathbf{L}}_{1,1}(X,Y) =\mathbf{k}_{\varUpsilon-\delta}(X,Y) + \sum_{j=1}^{2m+\varepsilon-2} S_j(X,Y),
    \end{equation}
    where $\mathbf{k}_{\varUpsilon-\delta}$ denotes the function $\mathbf{k}$ of \eqref{eq:k_hs_log_gamma} where contours are taken to be $\mathrm{i} \mathbb{R} + \varUpsilon-\delta$. Straightforward arguments lead to the following bounds for each term appearing in the right hand side of \eqref{eq:L_11_2}
    \begin{equation}
        | \mathbf{k}_{\varUpsilon-\delta}(X,Y) | \le C e^{(\varUpsilon-\delta)(X+Y)},
        \qquad
        |S_j(X,Y)| \le C \, e^{ (-R-j)X + RY},
    \end{equation}
    which we can combine to obtain
    \begin{equation}
        | \widehat{\mathbf{L}}_{1,1}(X,Y) | \le C' e^{ (-R-1)X + RY}.
    \end{equation}
    Similar arguments apply to all other components of the kernel $\widehat{\mathbf{L}}$, so that we get the estimate
    \begin{equation} \label{eq:L_hat_exponential_decay_1}
        \widehat{\mathbf{L}}(X,Y) = O \left(e^{ (-R-1)X + R Y} \right),
    \end{equation}
    which holds for $\varUpsilon-\delta \le R \le -\varUpsilon$.
    This allows us to evaluate expression \eqref{eq:J^TL.P} and choosing $R=\varUpsilon -\delta$ we find
    \begin{equation}
        \left( J^T \widehat{\mathbf{L}} \right) \left( \begin{matrix} \mathbf{A}' \\ \mathbf{A}  \end{matrix} \right) (X) = O \left(e^{ (-\varUpsilon+\delta-1)X} \right).
    \end{equation}
    The argument produced above shows that the application of the operator $J^T \widehat{\mathbf{L}}$ to the pair of functions $\left( \begin{matrix} \mathbf{A}' \\ \mathbf{A}  \end{matrix} \right)$ lowers their exponential growth by a factor $e^{-(1-\delta)X}$. Iteratively, by properly choosing the value of $R$ from time to time, we can then show that
    \begin{equation} \label{eq:J^TL_r_times}
        \left( J^T \widehat{\mathbf{L}} \right)^r \left( \begin{matrix} \mathbf{A}' \\ \mathbf{A}  \end{matrix} \right) (X) = O \left(e^{ -\varUpsilon X - r(1-\delta)X} \right)
    \end{equation}
    for $r=1,2,\dots,M$, where $M$ is the smallest integer such that $-\varUpsilon-M(1-\delta)<0$. Thanks to estimates \eqref{eq:J^TL_r_times} we can now prove that the scalar product \eqref{eq:scalar_product_q_JL_p} consists of absolutely convergent integrals for any $\ell \ge 0$  and that
    \begin{equation}
        |\langle(\mathbf{B},-\mathbf{B}') | (J^T \widehat{\mathbf{L}})^\ell | (\mathbf{A}',\mathbf{A}) \rangle| \le \mathrm{const.} \|\widehat{\mathbf{L}}\|^{\ell-M}.
    \end{equation}
    Since $\| \widehat{\mathbf{L}} \|<1$ for $N$ chosen properly large we conclude that the scalar product in the right hand side of \eqref{eq:pfaffian_L_hat+R} is well defined, and analytic in the parameter $\varUpsilon$ whenever $\Re\{\varUpsilon\}>-m+\frac{1-\varepsilon}{2}$.
    To conclude the proof we now like to remove the assumption $A_1,\dots,A_N>2m$. This again can be done by analytic continuation of \eqref{eq:pfaff_L_hat_log_gamma}, noticing that, as long as $A_i +\varUpsilon > 0$ the expression in the right hand side does not possess any singularity. Taking parameters $A_i$ uniformly larger than $-\varUpsilon$, say $A_i+\varUpsilon>\tilde{\delta}$ for all $i\ge 1$ and $\tilde{\delta}$ fixed, we can choose $N$ large enough so that the norm of $\widehat{\mathbf{L}}$ stays smaller than one, condition necessary to make sense of terms in the right hand side of \eqref{eq:pfaff_L_hat_log_gamma}. This concludes the proof.
\end{proof}

\section{Asymptotic results} \label{sec:asymptotics}

In this section we establish asymptotic limits of (some of) the solvable models in the KPZ class discussed in \cref{sec:KPZqW}. Most of our attention will be on models in half space as such results are new.

\subsection{Tracy-Widom distributions} \label{subs:Tracy_Widom}
    
    Here we give the definition of Tracy-Widom distributions and of the so called Baik-Rains crossover distribution which will appear in subsequent subsections in the characterization of limiting fluctuations of KPZ models. Their expression is given in terms of Fredholm determinants or Fredholm pfaffians. 
    
    \begin{definition} \label{def:contour_C_p_theta}
        For any $p\in \mathbb{R}$ and $\theta\in[0,2\pi]$, we define the curve $C_p^{\theta} \subset \mathbb{C}$ as
        \begin{equation}
            C_p^{\theta} = \{ p + |r| e^{ \mathrm{sign}(r) \mathrm{i}  \theta  } : r\in \mathbb{R} \}.
        \end{equation}
        We will use the notation $C_p^\theta$ to denote integration contours and in these situations we will always assume the contour is positively oriented.
    \end{definition}    
    
    We recall the definition of the Airy function
    \begin{equation}
        \mathrm{Ai}(u) = \int_{C_1^{\pi/3}} e^{\frac{\alpha^3}{3} - u \alpha} \frac{\diff \alpha}{2 \pi \mathrm{i}},
    \end{equation}
    of which we make extensive use below.
    
    \begin{definition}
        The GUE Tracy-Widom distribution is given by
        \begin{equation}
            F_\mathrm{GUE}(s) = \det\left( 1 - \mathcal{K}_\mathrm{Airy} \right)_{\mathbb{L}^2(s,+\infty)},
        \end{equation}
        where $s\in \mathbb{R}$ and $\mathcal{K}_\mathrm{Airy}$ is the Airy kernel 
        \begin{equation} \label{eq:Airy_kernel}
            \mathcal{K}_\mathrm{Airy}(u,v) = \int_{C^{2\pi/3}_{-1} } \frac{\diff \alpha}{2\pi i} \int_{C^{\pi/3}_{1} } \frac{\diff \beta}{2\pi i} \frac{e^{-\frac{\alpha^3}{3} + \alpha u + \frac{\beta^3}{3} - \beta v } }{\beta - \alpha}.
        \end{equation}
        Notice that the integration contours do not intersect.
    \end{definition}
    
    Frequently in literature the Airy kernel is defined as $\mathcal{K}_\mathrm{Airy}(u,v) = \int_0^{\infty} \mathrm{Ai} (u+ \lambda) \mathrm{Ai}(v+\lambda) d \lambda$, but it is possible to see such representation is equivalent to \eqref{eq:Airy_kernel}.
    
    We now define a family of probability distributions parameterized by a parameter $\xi \in [0,+\infty]$. This is the crossover distribution $F_\mathrm{cross} (s;\xi)$ introduced for the first time in \cite[Definition 4]{baik_rains2001asymptotics}.
    
    \begin{definition} \label{def:crossover}
        The Baik-Rains crossover distribution is the family of probability distributions on $s\in \mathbb{R}$ given by 
        \begin{equation} \label{eq:crossover}
            F_\mathrm{cross} (s;\xi) = \Pf \left( J - \mathcal{K}_\mathrm{cross}^{(\xi)}\right)_{\mathbb{L}^2(s,+\infty)}. 
        \end{equation}
        Here we take $\xi>0$ and $\mathcal{K}_\mathrm{cross}^{(\xi)}$ is the $2\times 2$ matrix kernel
        \begin{equation} \label{eq:kernel_K_cross}
            \mathcal{K}_\mathrm{cross}^{(\xi)} (u,v) = \left( \begin{matrix} \mathpzc{k}^{(\xi)}(u,v) & -\partial_v \mathpzc{k}^{(\xi)}(u,v) \\ -\partial_u \mathpzc{k}^{(\xi)}(u,v) & \partial_u \partial_v \mathpzc{k}^{(\xi)}(u,v) \end{matrix} \right),
        \end{equation}
        where
        \begin{equation} \label{eq:k_xi}
            \mathpzc{k}^{(\xi)}(u,v)
            = \frac{1}{4} \int_{C_\delta^{\pi/3}} \frac{\diff \alpha}{2\pi \mathrm{i}} \int_{C_\delta^{\pi/3}} \frac{\diff\beta}{2\pi \mathrm{i}} \frac{(\xi+\alpha) (\xi+\beta)}{(\xi-\alpha)(\xi-\beta)} \frac{\alpha-\beta}{\alpha \beta (\alpha+\beta)} e^{\frac{\alpha^3}{3} - \alpha u + \frac{\beta^3}{3} - \beta v }
        \end{equation}
        and in the integration contour we chose $0<\delta<\xi$. 
    \end{definition}
    
    In literature, several equivalent representations for $F_\mathrm{cross} (s;\xi)$ have appeared \cite{baik_rains2001symmetrized,baik_barraquand_corwin_suidan_pfaffian}, although compared to these existing expressions \eqref{eq:crossover} is new. The matrix kernel $\mathcal{K}_{\mathrm{cross}}^{(\xi)}$ was first written in \cite[eq. (71)]{krejenbrink_le_doussal_KPZ_half_space}, where authors conjectured the Fredholm pfaffian in the right hand side of \eqref{eq:crossover} to be equal to the Baik-Rains crossover. Below, in \cref{cor:equivalence_Baik_Rains_crossover}, we will prove such conjecture, hence justifying the definition of $F_\mathrm{cross} (s;\xi)$ as Baik-Rains crossover.
    
    The distribution $F_\mathrm{cross} (s;\xi)$ is well defined also for $\xi=+\infty$ or $\xi=0$, where it becomes respectively $F_{\mathrm{GSE}}(s)$ and $F_{\mathrm{GOE}}(s)$, i.e. the Tracy-Widom distributions corresponding to the symplectic and orthogonal gaussian ensembles \cite{tracy1996orthogonal}. The limit $\xi \to + \infty$ can be taken  straightforwardly from \eqref{eq:crossover} since the limiting form of $\mathcal{K}_{\mathrm{cross}}^{(\xi)}$ is readily computed setting
    \begin{equation} \label{eq:k_infty}
            \mathpzc{k}^{(\infty)}(u,v) = \lim_{\xi\to \infty} \mathpzc{k}^{(\xi)}(u,v) = \frac{1}{4} \int_{C_1^{\pi/3}} \frac{\diff \alpha}{2\pi \mathrm{i}} \int_{C_1^{\pi/3}} \frac{\diff \beta}{2\pi \mathrm{i}} \frac{\alpha-\beta}{\alpha \beta (\alpha+\beta)} e^{\frac{\alpha^3}{3} - \alpha u + \frac{\beta^3}{3} - \beta v }.
        \end{equation}
        
    \begin{definition} \label{def:GSE}
        The GSE Tracy-Widom distribution has the expression
        \begin{equation}
            F_{\mathrm{GSE}}(s) = \Pf \left( J - \mathcal{K}_{\mathrm{GSE}} \right)_{\mathbb{L}^2(s,+\infty)},
        \end{equation}
        where $\mathcal{K}_{\mathrm{GSE}}(u,v)=\lim_{\xi \to \infty} \mathcal{K}_{\mathrm{cross}}^{(\xi)}(u,v)$.
    \end{definition}
    
    In order to specialize the Fredholm pfaffian expression of $F_{\mathrm{cross}}$ to $\xi=0$, it requires to perform an analytic continuation of the kernel $\mathcal{K}_{\mathrm{cross}}^{(\xi)}$. When $\xi \to 0$, from the integral formula of function $\mathpzc{k}^{(\xi)}$ we see that $\alpha$ and $\beta$ poles at $\xi$ cross the integration contour and we have
    \begin{equation} \label{eq:k_0_crossover}
            \mathpzc{k}^{(0)}(u,v) = \lim_{\xi\to 0} \mathpzc{k}^{(\xi)}(u,v)
            = \mathpzc{k}^{(\infty)}(u,v) + \mathpzc{B}(u)  - \mathpzc{B}(v),
    \end{equation}
    where 
    \begin{equation} \label{eq:q_Airy}
        \mathpzc{B}(u) =\int_{C_1^{\pi/3}} \frac{\diff \alpha}{2\pi \mathrm{i}} \frac{e^{\frac{\alpha}{3}-\alpha u}}{2\alpha}
        =
        \frac{1}{2} \int_{0}^{+\infty} \mathrm{Ai}(\lambda+u) \diff\lambda.
    \end{equation}
    
    \begin{definition} \label{def:GOE}
        The GOE Tracy Widom distributions is given by
        \begin{equation} \label{eq:GOE}
            F_{\mathrm{GOE}}(s) = \Pf(J-\mathcal{K}_{\mathrm{GOE}})_{\mathbb{L}^2(s,+\infty)},
        \end{equation}
        where $\mathcal{K}_{\mathrm{GOE}}(u,v)= \lim_{\xi \to 0} \mathcal{K}_{\mathrm{cross}}^{(0)}(u,v)$.
    \end{definition}
    
    The expression given in \cref{def:GSE} for $F_{\mathrm{GSE}}$ had appeared in \cite{baik_barraquand_corwin_suidan_pfaffian}, whereas the expression of \cref{def:GOE} for $F_{\mathrm{GOE}}$ is new. As for the case of $F_{\mathrm{cross}}^{(\xi)}$, in \cite{krejenbrink_le_doussal_KPZ_half_space} authors found the distribution in the right hand side of \eqref{eq:GOE},  conjecturing it to be the GOE Tracy-Widom distribution and we confirm this in \cref{cor:equivalence_Baik_Rains_crossover} below.

    \subsection{GUE Tracy-Widom asymptotics for models in full space}
    
    We recall the digamma function and the polygamma functions
    \begin{equation} \label{eq:digamma_polygamma}
        \psi(z) = \frac{1}{\Gamma(z)}\frac{\diff}{\diff z} \Gamma(z),
        \qquad
        \text{and}
        \qquad
        \psi^{(n)}(z) = \frac{\diff^n}{\diff z^n} \psi(z).
    \end{equation}
    
    \begin{theorem} \label{thm:log_gamma_GUE}        
    Let $Z(N,T)$ be the partition function of the homogeneous Log Gamma polymer model in full space with parameters $A_1=A_2=\cdots=B_1=B_2=\cdots=A \in (0,\infty)$. Then, we have
        \begin{equation} \label{eq:GUE_log_gamma}
            \lim_{N \to \infty} \mathbb{P}\left[ \frac{\log Z(N, N) + f N }{ \sigma N^{1/3} } \le s \right] =F_{\mathrm{GUE}} (s), 
        \end{equation}
        where $f = 2 \psi (A) $, $\sigma = -\left[ \psi^{(2)}(A) \right]^{1/3}$.
    \end{theorem}
    
    The asymptotic behavior of the free energy of the point-to-point Log Gamma polymer partition function was derived already in \cite{BorodinCorwinRemenik,krishnan_quastel_log_gamma,Barraquand_Corwin_Dimitrov_log_gamma}, although in these works a different representation of the Fredholm determinant was used. Notably, the integral kernel utilized in the cited works, which dates back to \cite{BorodinCorwin2011Macdonald}, has a complicated pole structure which makes the saddle point analysis rather involved. In fact in \cite{BorodinCorwinRemenik} the limit \eqref{eq:GUE_log_gamma} was established with a certain technical assumption on the parameter $A$, which was then removed in \cite{krishnan_quastel_log_gamma} through a perturbation argument. In \cite{Barraquand_Corwin_Dimitrov_log_gamma} authors could prove, without restrictions on $A$ the asymptotic limit directly from the explicit Fredholm determinant formula, though again through rather complicated contour deformations. We will see below that such issues do not arise using the formula derived in \cref{thm:fredholm_det_log_gamma} making the evaluation of the asymptotic limit rather straightforward. Notice moreover that in \cite{Barraquand_Corwin_Dimitrov_log_gamma} authors treat a more general asymptotic limit than \eqref{eq:GUE_log_gamma}, as they consider the partition function $Z(T,N)$ with $N \neq T$ potentially. As discussed in \cref{rem:Z(T_N)} a Fredholm determinant representation for the Laplace transform of the pdf of $Z(T,N)$ with integral kernel of ``free fermionic" type should be available and using that one should be able to simplify the asymptotic analysis of \cite{Barraquand_Corwin_Dimitrov_log_gamma}. We will not discuss such more general setting in this paper. 
    
    \begin{proof}
        We will compute the asymptotic limit of the exact formula given by \cref{thm:fredholm_det_log_gamma}. Considering the scaling
        \begin{equation}
            \varsigma = - f N + \sigma N^{1/3} s,
        \end{equation}
        we can compute the limit of the Laplace transform in the left hand side of \eqref{eq:fredholm_det_log_gamma}, using equality \eqref{eq:Laplace_Gumbel}, as
        \begin{equation} \label{eq:coovergence_laplace_to_probability}
            \lim_{N\to \infty} \mathbb{E} \left[ e^{ - e^{-\varsigma+ \log Z(N,N)}} \right] = \lim_{N\to \infty} \mathbb{P} \left[ \frac{\log Z(N,N) + f N}{\sigma N^{1/3}} \le s \right].
        \end{equation}
        We will now consider the asymptotic limit of the Fredholm determinant in the right hand side of \eqref{eq:fredholm_det_log_gamma}. After a change of variables $X= \varsigma + u \sigma N^{1/3}, Y= \varsigma + v \sigma N^{1/3}$ we can rewrite the integral kernel $\mathbf{K}$ as 
        \begin{equation} \label{eq:log_gamma_kernel_rescaled}
            \int_{ \mathrm{i} \mathbb{R} - d } \frac{\diff Z}{2 \pi \mathrm{i}} \int_{ \mathrm{i} \mathbb{R} + d} \frac{\diff W}{2 \pi \mathrm{i}} \frac{\pi}{\sin (\pi (W-Z))} \frac{e^{N h(Z) + Z \sigma N^{1/3} (u + s) }}{ e^{N h(W) + W \sigma N^{1/3} (v + s) } },
        \end{equation}
        where
        \begin{equation}
            h(Z) = \log \Gamma(A + Z) - \log \Gamma(A - Z) - fZ.
        \end{equation}
        Under our assignment of parameters $f,\sigma$ we see that $h(Z)$ has a doubly critical point at $Z=0$ and we have
        \begin{equation}\label{eq:derivatives_h_log_gamma}
            h(0)=h'(0)=h''(0)=0,
            \qquad
            h'''(0) = 2 \sigma^3 >0.
        \end{equation}
        Moreover, for any $d >0$, the function $\Re\{h(Z)\}$ attains a maximum at $Z=-d$ along the vertical line $\mathrm{i} \mathbb{R} -d$. This is a result of the evaluation of the derivative
        \begin{equation}
        \begin{split}
            \frac{\diff}{\diff y} \Re\left\{h(-d+\mathrm{i}y) \right\} &= \frac{\diff}{\diff y} \sum_{n=0}^{\infty} \left( - \log \left| n+ A -d +\mathrm{i} y \right| + \log \left| n+ A +d -\mathrm{i} y \right| \right)
            \\
            & = \sum_{n=0}^{\infty} \frac{-4 d (n+A) y }{[(n+A-d)^2 + y^2] [(n+A+d)^2 + y^2]},
        \end{split}
        \end{equation}
        which is clearly negative for $y>0$. Notice that in the first equality we made use of the product representation of the Gamma function $\Gamma(z) = z^{-1} \prod_{n \ge 1} (1+z/n)^{-1} (1+1/n)^z$. From these observations we easily see that, for $d$ sufficiently close to $0$, the integration contour
        \begin{equation} \label{eq:contour_log_gamma}
            \mathsf{C}^{2\pi/3}\left( - \frac{1}{\sigma N^{1/3}} ; d \right) \cup \mathsf{D} (-d, \sqrt{3}(d-\frac{1}{\sigma N^{1/3}}))
        \end{equation}
        where
        \begin{equation} \label{eq:notation_contours}
            \mathsf{C}^{\varphi}(a;d) = C_a^{\varphi} \cap \{ z: |\Re{z}| \le d \}
            \qquad
            \text{and}
            \qquad
            \mathsf{D} (a,b) = \{a + \mathrm{i} y : |y| >b\}  
        \end{equation}
        is steep descent for the function $h(Z)$. For a representation of \eqref{eq:contour_log_gamma} see \cref{fig:integration_contour}. This allows us to evaluate the $Z$ integral in the expression \eqref{eq:log_gamma_kernel_rescaled} via saddle point method. Clearly the same argument works for the evaluation of the $W$ integral, in which case the integration contour is taken to be
        \begin{equation} \label{eq:contour_log_gammma}
            \mathsf{C}^{\pi/3}\left( \frac{1}{\sigma N^{1/3}} ; d \right) \cup \mathsf{D} (d, \sqrt{3}(d-\frac{1}{\sigma N^{1/3}})),
        \end{equation}
        as shown in \cref{fig:integration_contour}. The saddle point method shows that $Z$ and $W$ integrals in \eqref{eq:log_gamma_kernel_rescaled} are dominated by the contribution along the $\mathsf{C}^{2\pi/3}\left(- \frac{1}{\sigma N^{1/3}} ; d \right)$ and $\mathsf{C}^{\pi/3}\left( \frac{1}{\sigma N^{1/3}} ; d \right)$ sections of the contours. We can therefore rescale integration variables as
        \begin{equation}
            Z=-\frac{\alpha}{\sigma N^{1/3}},
            \qquad
            W=\frac{\beta}{\sigma N^{1/3}},
        \end{equation}
        to compute the limiting form of the kernel $\mathbf{K}$ as
        \begin{equation}
            \mathbf{K}( \varsigma + u \sigma N^{1/3}, \varsigma + v \sigma N^{1/3})
            =\frac{1}{\sigma N^{1/3}} \mathcal{K}_{\mathrm{Airy}}(u+s,v+s)(1+O(N^{-1/3})).
        \end{equation}
        The steep descent analysis allows also to recover an exponential bound
        \begin{equation}
            \left|\sigma N^{1/3} \mathbf{K}( \varsigma + u \sigma N^{1/3}, \varsigma + v \sigma N^{1/3})\right|
            \le C e^{-c(u+v)},
        \end{equation}
        where $C,c$ are positive constants independent of $N,u,v$. The proof of the exponential inequality above is relatively straightforward and similar bounds were established, though in a slighlty different context, in \cite{BFP_2007}.
        We then conclude, by the bounded convergence theorem that
        \begin{equation}
            \lim_{N\to \infty} \det(1-\mathbf{K})_{\mathbb{L}^2(\varsigma,+\infty)} = \det(1-\mathcal{K}_\mathrm{Airy})_{\mathbb{L}^2(s,+\infty)}
        \end{equation}
        and this proves \eqref{eq:GUE_log_gamma}.
    \end{proof}
    
    \subsection{Baik-Rains phase transition for half space models: zero temperature case} \label{subs:baik_rains_pfaffian_schur}
    
    In this and the next subsections we undertake the evaluation of asymptotic limits of models in half space. In this subsection we focus on the simplest case of the Pfaffian Schur measure, i.e., the $q=0$ case of the free boundary Schur measure, and we consider asymptotics of $\lambda_1$ when $\lambda \sim \mathbb{FBS}^{(q=0)}_{(\boldsymbol{a},\gamma)}$. Through the Robinson-Schensted-Knuth algorithm it is known that $\lambda_1$ describes the last passage percolation through a symmetric environment of geometric random variables; see \cite{baik_barraquand_corwin_suidan_pfaffian,Betea_et_al_free_boundary}. Taking scaling limits or specializations of variables $\boldsymbol{a}$ it is also known that $\lambda_1$ describes the longest increasing subsequence of certain random involution \cite{baik_rains2001asymptotics} or the height of a polynuclear growth model in half space at the origin \cite{imamura_sasamoto_half_space_png}. The asymptotic limits of such models were all considered in the cited sources, although the precise form of the Fredholm pfaffian we consider below was not addressed. This represents a first motivation for us to address asymptotics of this classical model first. The new form of the matrix kernel $L$ of \cref{prop:fredholm_pfaffian_fbs}, combined with the use of the symmetry of the half space $q$-Whittaker measure of \cref{prop:symmmetry_half_space_qW} leads us to discover new formulas for the limiting fluctuations of $\lambda_1$. Another motivation for studying asymptotics in the $q=0$ case is that, since matrix kernels appearing in \cref{cor:pfaffian_q_PushTASEP}, \cref{thm:fredholm_pfaff_log_gamma} and \cref{thm:fredholm_pfaff_log_gamma_gaussian} all resemble each other, key ideas to compute their asymptotic limit can be given in the the simplest case. 
    
    \begin{theorem} \label{thm:asymptotic_pfaffian_schur}
        Let $\lambda$ be distributed according to the free boundary Schur measure $\mathbb{FBS}^{(q=0)}_{(\boldsymbol{a},\gamma)}$ with $q=0$. Here specialization $(\boldsymbol{a},\gamma)$, is $\boldsymbol{a}=(a_1,\dots,a_N)$ with $a_i=a \in (0,1)$ for $1\leq i\leq N$ and $a \gamma \in (0,1)$. Then, the following limits hold
        \begin{itemize}
            \item if $\gamma \in (0,1]$, rescaling $\gamma = 1-\frac{\xi}{\sigma N^{1/3}}$, we have
        \begin{equation} \label{eq:crossover_q=0}
            \lim_{N\to \infty} \mathbb{P} \left[ \frac{ \lambda_1 - N L}{ \sigma N^{1/3} } \le s \right] = F_{\mathrm{cross}}(s;\xi),
        \end{equation}
        where $L=\frac{2a}{1-a}$ and $\sigma = \frac{a^{1/3} (1+a)^{1/3}}{1-a}$;
        \item if $\gamma \in (1,1/a)$ then
        \begin{equation} \label{eq:gaussian_fluctuations_q=0}
            \lim_{N\to \infty} \mathbb{P}\left[ \frac{\lambda_1 - NL_\gamma}{\sigma_{\gamma} N^{1/2}} \le s \right] = \int_{-\infty}^s \frac{e^{-u^2/2}}{\sqrt{2 \pi}} \diff u,
        \end{equation}
        where $L_\gamma = \frac{a(1-2a \gamma + \gamma^2)}{(\gamma - a) (1-\gamma a)}$ and $\sigma_\gamma = \frac{\sqrt{(1-a^2)(\gamma^2-1) a \gamma^3 }}{(\gamma - a)(1-\gamma a)}$.
        \end{itemize}
    \end{theorem}

    \begin{proof}
        From \cref{prop:fredholm_pfaffian_fbs}, for any $t\in \mathbb{R}$, we write, 
        \begin{equation} \label{eq:prob_pfaffian_schur}
            \mathbb{P}(\lambda_1 \le t) = \Pf[J-\bar{L}]_{\ell^2 (\mathbb{Z}'_{ > t})}.
        \end{equation}
        The $2\times 2$ matrix kernel is given by
        \begin{equation}
            \bar{L}(x,y) = \left( \begin{matrix} \bar{k}(x,y) & -\nabla_y \bar{k}(x,y)
            \\
            - \nabla_x \bar{k}(y,x) & \nabla_x \nabla_y \bar{k}(x,y) \end{matrix} \right),
        \end{equation}
        where
        \begin{equation} \label{eq:K0_11}
        \begin{split}
            \bar{k}(x,y) = \oint_{|z|=r} \frac{\diff z}{2\pi \mathrm{i} \sqrt{z} } \oint_{|w|=r}
            &
            \frac{\diff w}{2\pi \mathrm{i}  \sqrt{w} }
            e^{N g(z) - x \log z } e^{N g(w) - y \log w }
            \\
            &
            \times
            \frac{(z-\gamma)(w-\gamma)}{(1-\gamma z)(1-\gamma w) } \frac{z-w}{(z^2-1)(w^2-1)(zw-1)},
        \end{split}
        \end{equation}
        and
        \begin{equation} \label{eq:g_scaling_function}
            g(z) = \log (1-a/z) - \log(1-az).
        \end{equation}
        We will prove large $N$ asymptotics of the probability distribution of $\lambda_1$ separately, depending on the value of $\gamma$.
        
        \medskip
        \emph{ $0 \le \gamma < 1$ : convergence to crossover distribution.}
        In order to prove \eqref{eq:crossover_q=0}, we consider in \eqref{eq:prob_pfaffian_schur} the scaling 
        \begin{equation} \label{eq:scaling_t}
            t=NL + \sigma N^{1/3} s,
        \end{equation} 
        so that variable of kernel $\bar{L}$ are rescaled as
        \begin{equation} \label{eq:scaling_x_y}
            x = NL + \sigma N^{1/3} u,
            \qquad
            y = NL + \sigma N^{1/3} v,
        \end{equation}
        for $u,v \in (s,\infty)$.
        Isolating the term of order $N$ in the exponent of integral expressions for $\bar{k}$, we define the function
        \begin{equation}
            h(z)=g(z) - L \log z.
        \end{equation}
        Simple computations show that $h$ has a double critical point at $z=1$ and we have
        \begin{equation}\label{eq:derivatives_h}
            h(1)=h'(1)=h''(1)=0,
            \qquad
            h'''(1) = 2 \sigma^3 >0.
        \end{equation}
        Moreover, for any $r>1$, the complex curve $|z|=r$ is steep descent for $\Re \{ h(z) \}$ since
        \begin{equation} \label{eq:derivative_re_h}
            \frac{\diff}{\diff \theta} \Re \{ h( r e^{\mathrm{i} \theta} ) \} =  \frac{a (1-a^2) \sin \theta}{ |1-\frac{a}{r} e^{-\mathrm{i}\theta} |^2 |1- a r e^{\mathrm{i}\theta} |^2 } \left( r^{-1} - r \right) < 0,
            \qquad
            \text{for } \theta \in (0,\pi).
        \end{equation}
        Define the contour
        \begin{equation} \label{eq:contour_C_N}
            C^{\pi/3} \left(1+ \frac{\delta}{\sigma N^{1/3}} ;r \right) \cup D(\bar{\theta},r),
        \end{equation}
        where
        \begin{equation}
            C^{\varphi}(a;r) = \{z=a+\rho e^{\mathrm{sign}(\rho) \mathrm{i} \varphi }:|z|<r\},
            \qquad
            D(\varphi,r) = \{r e^{i\theta}:\theta\in(\varphi,2\pi -\varphi)\},
        \end{equation}
        and $\bar{\theta}>0$ solves the relation $1+\frac{\delta}{\sigma N^{1/3}} + \rho e^{\mathrm{i} \frac{\pi}{3}} = r e^{\mathrm{i}\bar{\theta}}$ for some $\rho>0$; see \cref{fig:integration_contour}.
        Properties \eqref{eq:derivatives_h}, \eqref{eq:derivative_re_h} of function $h$, imply that for $r$ sufficiently close to 1, but still $r-1=O(1)$, the contour \eqref{eq:contour_C_N} is steep descent for the function $h$. We now deform integration contours in \eqref{eq:K0_11} to be that of \eqref{eq:contour_C_N}, both for the $z$ and $w$ variables. Here we assume that $\delta \in (0,\xi)$, so that the pole $\gamma^{-1} \approx 1+ \frac{\xi}{\sigma N^{1/3}}$ lies outside of the integration contour as in \cref{fig:integration_contour}. When $N$ grows such integrals are dominated by the contribution over the curve $C^{\pi/3} \left(1+ \frac{\delta}{\sigma N^{1/3}} ;r \right)$, in the vicinity of the critical points $z=1, w=1$. We can therefore rescale integration variables as
        \begin{equation} \label{eq:scaling_z_w}
            z=1+\frac{\alpha}{\sigma N^{1/3}},
            \qquad
            w=1+\frac{\beta}{\sigma N^{1/3}},
        \end{equation}
        so that, thanks to \eqref{eq:derivatives_h}, the exponents in integrands of $\bar{k}$ become
        \begin{equation}
            N g(z) - x \log z = \frac{\alpha^3}{3} - \alpha u + O(N^{-1/3}),
            \qquad
            N g(w) - y \log w =  \frac{\beta^3}{3} - \beta v + O(N^{-1/3}).
        \end{equation}
        Expanding in the new variables $\alpha,\beta$ the remaining factors in the integrands of \eqref{eq:K0_11} we find
        \begin{equation}
        \begin{split}
            &
            \frac{1}{\sqrt{z w}} \frac{(z-\gamma)(w-\gamma)}{(1-\gamma z)(1-\gamma w) } \frac{z-w}{(z^2-1)(w^2-1)(zw-1)} 
            \\
            &
            \hspace{6cm} =  \frac{(\xi + \alpha) (\xi + \beta)}{(\xi - \alpha) (\xi - \beta)} \frac{\alpha - \beta}{4 \alpha \beta (\alpha + \beta)} + O(\frac{1}{\sigma N^{1/3}}),
        \end{split}
        \end{equation}
        proving the pointwise convergence
        \begin{equation}
            \bar{k}(x(u),y(v)) \xrightarrow[N\to \infty]{} \mathpzc{k}^{(\xi)}(u,v).
        \end{equation}
        Such convergence result can be strengthened and it holds uniformly for $u,v>s$, thanks to the exponential decay
        \begin{equation}
            \left| \bar{k}(x(u),y(v)) \right| \le C e^{-\delta (u+v)},
        \end{equation}
        which follows again from the steep descent property of the integration contour.
        
        Analogous estimates, for the remaining entries of the kernel $\bar{L}$ can be readily produced to show the convergence, 
        \begin{equation} \label{eq:nabla_y_k_bar}
            -\sigma N^{1/3} \nabla_y \bar{k}(x(u),y(v)) \xrightarrow[N\to \infty]{} - \partial_v \mathpzc{k}^{(\xi)}(u,v),
        \end{equation}
        \begin{equation} \label{eq:nabla_x_k_bar}
            -\sigma N^{1/3} \nabla_x \bar{k}(x(u),y(v)) \xrightarrow[N\to \infty]{} - \partial_u \mathpzc{k}^{(\xi)}(u,v),
        \end{equation}
        \begin{equation} \label{eq:nabla_x_nabla_y_k_bar}
            \sigma^2 N^{2/3} \nabla_x \nabla_y \bar{k}(x(u),y(v)) \xrightarrow[N\to \infty]{} \partial_u \partial_v \mathpzc{k}^{(\xi)}(u,v),
        \end{equation}
        which again hold uniformly for $u,v>s$. Exponential decay for $u,v\to \infty$ of terms in the left hand side of \eqref{eq:nabla_y_k_bar}, \eqref{eq:nabla_x_k_bar}, \eqref{eq:nabla_x_nabla_y_k_bar} is also easily derived from the steep descent of integration contour. Such convergence results for the kernel $\bar{L}$ are sufficient to establish convergence of Fredholm pfaffians
        \begin{equation}
            \Pf \left[J - \bar{L}\right]_{\ell^2(\mathbb{Z}'_{\ge t})} \xrightarrow[N\to \infty]{} \Pf \left[J-\mathcal{K}^{(\xi)}_{\mathrm{cross}} \right]_{\mathbb{L}^2{(s,+\infty)}},
        \end{equation}
        proving \eqref{eq:crossover_q=0} for $\gamma \in [0,1)$.
        
        \begin{figure}
            \centering
            \subfloat[]{ 
                \includegraphics[width=.45\linewidth]{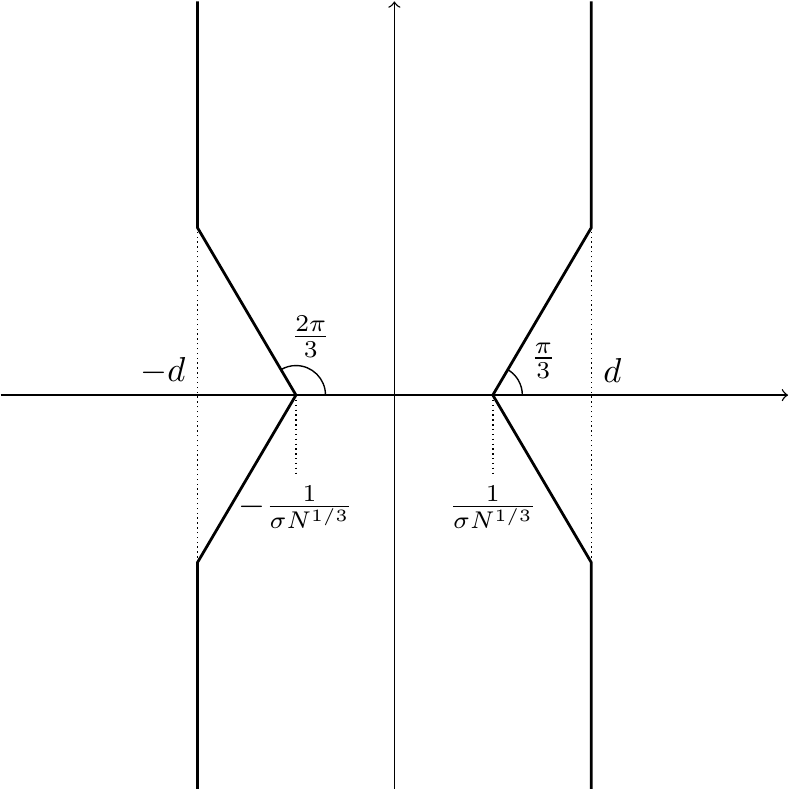} }%
                \hspace{.05cm}
                \subfloat[]{\includegraphics[width=.45\linewidth]{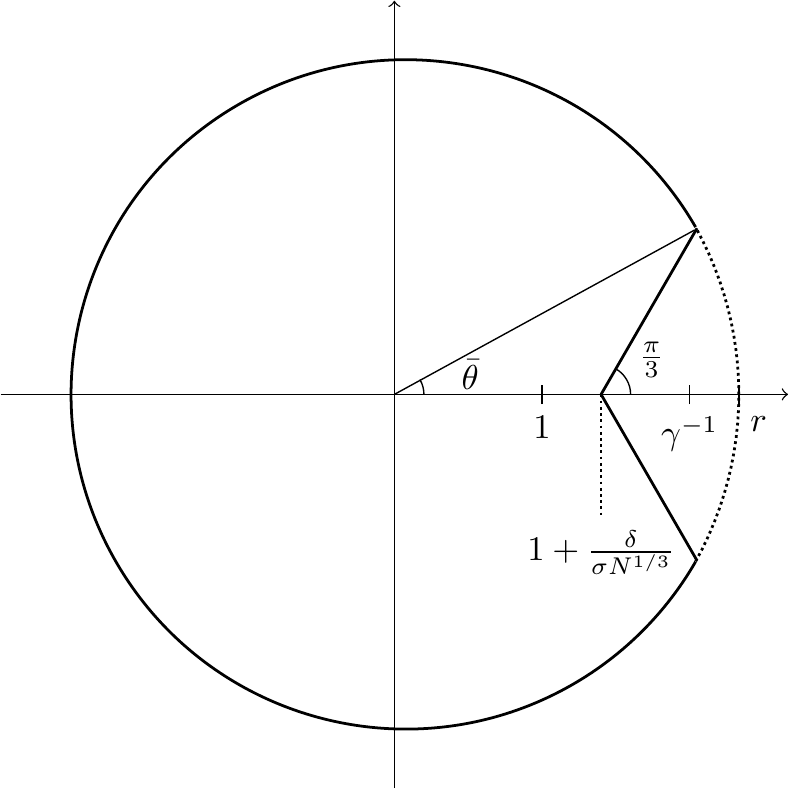} }%
                \caption{In the left panel the contours described in \eqref{eq:contour_log_gammma} are drawn. The right panel represents the conoturs of \eqref{eq:contour_C_N}.}
        \label{fig:integration_contour}
\end{figure}
        
        \medskip
        \emph{$\gamma=1$ : convergence to GOE Tracy-Widom distribution.} In order to set $\gamma=1$ in the Fredholm pfaffian formula \eqref{eq:prob_pfaffian_schur} for the probability distribution of $\lambda_1$ we need need to compute the  analytic extension of the function $\bar{k}$ in the the parameter $\gamma$, which we denote by $\bar{k}_{\mathrm{ext}}(x,y)$. From the integral expression \eqref{eq:K0_11} we see that $1/\gamma$ is a singularity for the integrand for both variables $z$ and $w$. Therefore when $1/\gamma<r$, the extended function $\bar{k}_\mathrm{ext}(x,y)$ equals $k(x,y)$ plus the residues of the integrand at poles $z=1/\gamma$ and $w=1/\gamma$, i.e.
        \begin{equation} \label{eq:k_bar_analytic_cont}
            \bar{k}_{\mathrm{ext}}(x,y) = \bar{k}(x,y) - A(x)B(y) + A(y)B(x),
        \end{equation}
        where
        \begin{equation}
            A(x) = e^{Ng(1/\gamma) - x \log 1/\gamma}
            \qquad
            \text{and} 
            \qquad
            B(y) = \sqrt{\gamma}  \oint_{|w|=r} \frac{\diff w}{2\pi \mathrm{i} \sqrt{w}} \frac{e^{Ng(w) - y \log w}}{w^2-1}.
        \end{equation}
        Denoting by $\bar{L}_{\mathrm{ext}}$ the analytic extension of the matrix kernel $\bar{L}$ in the parameter $\gamma$ we have
        \begin{equation} \label{eq:L_ext}
            \bar{L}_{\mathrm{ext}}(x,y) = \bar{L}(x,y) + \bar{R}(x,y),
        \end{equation}
        where
        \begin{equation}
             \bar{R}(x,y) = - \left( \begin{matrix} 1 & -\nabla_y \\ -\nabla_x & \nabla_x \nabla_y \end{matrix} \right) [A(x)B(y) - A(y) B(x)].
        \end{equation}
        When $\gamma<1$, functions $A,B$ are exponentially decaying and they belong to $\ell^2\left( \mathbb{Z}_{> t}' \right)$. Setting $\gamma=1$, we see that the function $A$ becomes constant $A\equiv 1$ and hence the kernel $\bar{R}(x,y)$ no longer defines an operator $\ell^2\left( \mathbb{Z}_{> t}' \right) \to \ell^2\left( \mathbb{Z}_{> t}' \right)$. Nevertheless we now show that the Fredholm pfaffian $\Pf[J-\bar{L}_{\mathrm{ext}}\big|_{\gamma=1}]$ still consists in an absolutely convergent series and it is therefore well defined. First let's compute, though saddle point method the asymptotic form of $\bar{k}$ and $B$. Using again the scaling \eqref{eq:scaling_t},  \eqref{eq:scaling_x_y} and following the same procedure described for the case $\gamma<1$, we find
        \begin{equation}
            \bar{k}(x(u),y(v)) \xrightarrow[N\to \infty]{} \mathpzc{k}^{(\infty)}(u,v)
        \end{equation}
        and
        \begin{equation}
            B(x(u)) \xrightarrow[N\to \infty]{} \mathpzc{B}(u),
        \end{equation}
        where the function $\mathpzc{B}$ was defined around \eqref{eq:k_0_crossover}. Such convergence results hold uniformly for $u>s$ as a result of exponential bounds
        \begin{equation} \label{eq:exponential_bound_k_gamma=1}
            |\bar{k}(x(u),y(v))| \le C e^{-\delta (u+v)}
            \qquad
            \text{and}
            \qquad
            |B(x(u))| \le C e^{-\delta u},
        \end{equation}
        which again follow by steep descent analysis. Analogous convergence statements hold also for discrete derivatives of $\bar{k}, q$, and in general we find
        \begin{equation}
            \left( \begin{matrix} \bar{k}_{\mathrm{ext}}(x(u),y(v)) & - \sigma N^{1/3} \nabla_y \bar{k}_{\mathrm{ext}}(x(u),y(v)) \\ - \sigma N^{1/3} \nabla_x \bar{k}_{\mathrm{ext}}(x(u),y(v)) & \sigma^2 N^{2/3} \nabla_x \nabla_y \bar{k}_{\mathrm{ext}}(x(u),y(v)) \end{matrix} \right) \xrightarrow[N\to \infty]{} \mathcal{K}_{\mathrm{GOE}} (u,v)
        \end{equation}
        Let us now prove that $\Pf[J-\bar{L}_{\mathrm{ext}}\big|_{\gamma=1}]$ is absolutely convergent for any $N$, so that the claim $\lim_{N\to \infty}\Pf[J-\bar{L}_{\mathrm{ext}}\big|_{\gamma=1}]  =  \Pf[J-\mathcal{K}_{\mathrm{GOE}}]$ will follow from the bounded convergence theorem. Observe that, for any $\ell \ge 0$, using \eqref{eq:pfaffian_sum}, we have
        \begin{equation} \label{eq:pfaffian_L_ext_gamma=1}
        \Pf[\bar{L}_{\mathrm{ext}}(x_i,x_j)\big|_{\gamma=1}]_{i,j=1}^\ell = \Pf[\bar{L}(x_i,x_j)]_{i,j=1}^\ell + \sum_{1\le a < b \le 2\ell} (-1)^{a+b-1} \Pf[\bar{R}^{a,b}] \Pf[\bar{L}^{\widehat{a},\widehat{b}}],
        \end{equation}
        where $R^{a,b}$ is the $2 \times 2$ minor of $[\bar{R}(x_i,x_j)]_{i,j=1}^\ell$ of elements at rows and columns $a,b$, while $\bar{L}^{\widehat{a},\widehat{b}}$ is the minor of $[\bar{L}(x_i,x_j)]_{i,j=1}^\ell$ obtained erasing rows and columns $a,b$. This follows from the fact that every minor of $[\bar{R}(x_i,x_j)]_{i,j=1}^\ell$ of size larger than 2 has vanishing pfaffian, for rank reasons. Setting $a=2(i-1) + \varepsilon_1$ and $b=2(j-1) + \varepsilon_2$ we see that $\Pf[\bar{R}^{a,b}]$ can only be equal to
        \begin{equation}
            -B(x_i) + B(x_j),
            \qquad
            \nabla B(x_j),
            \qquad
            - \nabla B(x_i),
            \qquad
            0
        \end{equation}
        respectively for $(\varepsilon_1,\varepsilon_2)=(1,1), (1,2),(2,1),(2,2)$. Using bounds \eqref{eq:exponential_bound_k_gamma=1} and Hadamard's inequality \eqref{eq:Hadamard_bound} we can estimate \eqref{eq:pfaffian_L_ext_gamma=1} as
        \begin{equation}
            \left| \Pf[\bar{L}_{\mathrm{ext}}(x(u_i),x(u_j))\big|_{\gamma=1} ]_{i,j=1}^\ell \right| \le \left[ (2\ell)^{\ell/2} + \binom{2 \ell}{2} 2 (2 \ell - 2)^{\frac{\ell -1}{2}} \right] C^\ell e^{-\delta (u_1+\cdots +u_\ell)},
        \end{equation}
        which proves that the series
        \begin{equation}
            \sum_{\ell \ge 0} \frac{1}{\ell!} \int_{[s,+\infty)^\ell} \diff u_1 \cdots \diff u_\ell  \left| \Pf[\bar{L}_{\mathrm{ext}}(x(u_i),x(u_j))\big|_{\gamma=1} ]_{i,j=1}^\ell \right|,
        \end{equation}
        converges. Since $\bar{L}_{\mathrm{ext}}\big|_{\gamma=1}$ converges uniformly to $\mathcal{K}_{\mathrm{GOE}}$, this proves that its Fredholm pfaffian, converges to $F_{\mathrm{GOE}}(s)$, proving \eqref{eq:crossover_q=0} for $\gamma=1$.
        
        \medskip
        
        \emph{ $\gamma>1$ : convergence to gaussian distribution.} We have seen in \eqref{eq:L_ext} that, in case $\gamma \ge 1$, the analytic extension of the matrix kernel $\bar{L}$ results in the addition of a rank 2 perturbation. Here we show that when $\gamma > 1$ such perturbation provides the dominant contribution to the Fredholm pfaffian of $\bar{L}_{\mathrm{ext}}$. Following ideas already elaborated in the proof of \cref{thm:fredholm_pfaff_log_gamma_gaussian}, we need first to make sense of the quantity $\Pf[J-\bar{L}_{\mathrm{ext}}]_{\ell^2(\mathbb{Z}_{\ge t})}$, since when $\gamma>1$, the functions $A(x)$ appearing in the rank 2 kernel $\bar{R}(x,y)$ are diverging. Consider the scaling
        \begin{equation} \label{eq:gaussian_scaling}
            t= L_\gamma N + \sigma_\gamma N^{1/2} s,
            \qquad
            x= L_\gamma N + \sigma_\gamma N^{1/2} u,
            \qquad
            y= L_\gamma N + \sigma_\gamma N^{1/2} v.
        \end{equation}
        Introducing the function
        \begin{equation}
            h_\gamma(z) = g(z) - L_\gamma \log z,
        \end{equation}
        we rewrite functions $\bar{k},A,B$ as
        \begin{equation} \label{eq:k_bar_gaussian_rescaled}
        \begin{split}
            \bar{k}(x(u),y(v)) = \oint_{|z|=r} \frac{\diff z}{2\pi \mathrm{i} \sqrt{z} } \oint_{|w|=r}
            &
            \frac{\diff w}{2\pi \mathrm{i}  \sqrt{w} }
            e^{N \left( h_\gamma(z) + h_\gamma(w) \right) - \sigma_{\gamma} N^{1/2} (u \log z + v \log w) }
            \\
            &
            \times
            \frac{(z-\gamma)(w-\gamma)}{(1-\gamma z)(1-\gamma w) } \frac{z-w}{(z^2-1)(w^2-1)(zw-1)},
        \end{split}
        \end{equation}
        \begin{equation} \label{eq:p_q_gaussian_rescaled}
            A(x(u)) = e^{Nh_{\gamma}(1/\gamma) - \sigma_\gamma N^{1/2} u \log 1/\gamma},
            \qquad
            B(y(v)) = \sqrt{\gamma}  \oint_{|w|=r} \frac{\diff w}{2\pi \mathrm{i} \sqrt{w}} \frac{e^{Nh_{\gamma}(w) - \sigma_\gamma N^{1/2} y \log w}}{w^2-1}.
        \end{equation}
        It is clear that, when $\gamma<1$ the kernel $\bar{L}_{\mathrm{ext}}$ with such scaling, defines a bounded operator as all functions are exponentially decaying in the variables $u,v>s$.  Using \cref{prop:pfaffian_L+R}, we write
        \begin{equation} \label{eq:pfaffian_L_ext}
            \Pf[J-\bar{L}_{\mathrm{ext}}] = \Pf[J-\bar{L}] \left( 1- \langle (B,-\nabla B) |(1-J^T \bar{L})^{-1}| (\nabla A,A) \rangle \right),
        \end{equation}
        where all operators and the scalar product are defined on the Hilbert space of pairs of square summable functions of variable $u$ such that $x > t$ under scaling \eqref{eq:gaussian_scaling}. 
        On the right hand side the inverse operator is defined by the Neumann series 
        \begin{equation} 
            (1-J^T \bar{L})^{-1} = \sum_{\ell \ge 0} (J^T \bar{L})^\ell.
        \end{equation} 
        This series is well defined as, for $N$ large enough the operator norm of $J^T \bar{L}$ can be made arbitrarily small, as a result of presence of factors $e^{N ( h_\gamma (z) +h_\gamma (w) )}$ for $|z|,|w|>1$ in the expression of function $\bar{k}$. We can now extend expression in the right hand side of \eqref{eq:pfaffian_L_ext} allowing $\gamma >1$. In this case we can evaluate the asymptotic limit of functions $\bar{k},A,B$ using steep descent analysis. Simple computations show that $h_\gamma$ satisfies some basic properties as
        \begin{equation}
            h_\gamma(1/z) = -h_\gamma(z),
            \qquad
            h_\gamma( \gamma ) < 0,
            \qquad
            h'_\gamma( \gamma) = 0,
            \qquad
            h''_\gamma( \gamma) = \frac{\sigma_\gamma^2}{\gamma^2} > 0.
        \end{equation}
        Unlike for the case $\gamma \le 1$, this time we choose as a integration contour the complex curve $|z|=\gamma - \frac{1}{\sigma_\gamma N^{1/2}}$ which is steep descent for the function $\Re\{ h_\gamma(z) \}$, following the same computation as in \eqref{eq:derivative_re_h}. This justifies the change of integration variables,
        \begin{equation}
            z= \gamma \left( 1 - \frac{\alpha}{\sigma_\gamma N^{1/2}} \right),
            \qquad
            w= \gamma \left( 1 - \frac{\beta}{\sigma_\gamma N^{1/2}}\right).
        \end{equation}
        With such rescaling in place, we see that the exponential terms in each integral in \eqref{eq:k_bar_gaussian_rescaled}, \eqref{eq:p_q_gaussian_rescaled} become
        \begin{equation}
            N h_\gamma (z) - \sigma N^{1/2} u \log z
            = N h_\gamma (\gamma) - \sigma_\gamma N^{1/2} u \log \gamma +\frac{1}{2} \alpha^2 + u \alpha
        \end{equation}
        and analogously
        \begin{equation}
            N h_\gamma (w) - \sigma N^{1/2} v \log w = N h_\gamma (\gamma) - \sigma_\gamma N^{1/2} v \log \gamma +\frac{1}{2} \beta^2 + v \beta.
        \end{equation}
        When $N$ is large we obtain
        \begin{equation} \label{eq:estimate_e}
            \bar{k}(x(u),y(v)) = O \left( N^{-5/2} e^{2N h_\gamma(\gamma) - \sigma_\gamma N^{1/2} (u+v) \log \gamma - c(u^2 + v^2)} \right),
        \end{equation}
        where $c \in (0,1/2)$ and
        \begin{equation} \label{eq:estimate_p}
            A(x(u)) = e^{-N h(\gamma) + \sigma N^{1/2} u \log \gamma},
        \end{equation}
        \begin{equation} \label{eq:estimate_q}
            B(y(v)) = \frac{1}{\sigma_\gamma N^{1/2}} \frac{1}{(\gamma - 1/\gamma) \sqrt{2\pi}} e^{ N h(\gamma) - \sigma N^{1/2} v \log \gamma - v^2/2}(1+O(N^{-1/2})).
        \end{equation}
        Notice that the effect of discrete derivatives in the limit reduces to a multiplication by a factor $\frac{1}{2}(\gamma-1/\gamma)$ and in particular we have
        \begin{equation} \label{eq:p_q_derivatives}
            \nabla_x A(x(u)) \approx \frac{1}{2}(\gamma-1/\gamma) A(x(u)),
            \qquad
            \nabla_y B(y(v)) \approx -\frac{1}{2}(\gamma-1/\gamma) B(y(v)),
        \end{equation}
        where $\approx$ denotes an approximation up to errors of order $N^{-1/2}$. We can now use estimates \eqref{eq:estimate_e}, \eqref{eq:estimate_p}, \eqref{eq:estimate_q} \eqref{eq:p_q_derivatives} to evaluate the right hand side of \eqref{eq:pfaffian_L_ext} obtaining
        \begin{equation}
            \lim_{N \to \infty}\Pf[J-\bar{L}] = 1
        \end{equation}
        and
        \begin{equation}
        \begin{split}
            \lim_{N \to \infty} \left( 1- \langle (B,-\nabla B) |(1-J^T \bar{L})^{-1}| (\nabla A,A) \rangle \right) & = \lim_{N \to \infty} \left( 1- \langle (B,-\nabla B) | (\nabla A,A) \rangle \right)
            \\
            & = 1 - \int_{s}^{\infty} \frac{e^{-v^2/2}}{\sqrt{2\pi}} \diff v.
        \end{split}
        \end{equation}
        This concludes the proof of \eqref{eq:gaussian_fluctuations_q=0} and hence of the theorem.
    \end{proof}
    
    \begin{corollary} \label{cor:equivalence_Baik_Rains_crossover}
        For any $\xi\in [0,+\infty]$, the definition of Baik-Rains crossover distribution $F_{\mathrm{cross}}(\,\cdot\,;\xi)$ given in \cref{def:crossover} is equivalent to that of $F^{\begin{tikzpicture} \draw[] (0,0) -- (.2,0) -- (.2,.2) -- (0,.2) -- (0,0);
        \draw[] (0,0) -- (.2,.2);
        \end{tikzpicture}}(\, \cdot \,;\xi)$ given in \cite[Definition 4]{baik_rains2001asymptotics}.
    \end{corollary}
    
    \begin{proof}
        In \cite{baik_rains2001asymptotics} authors proved that the crossover distribution arises as limiting distribution of the longest increasing subsequence of a random involution, there denoted by $L^{\begin{tikzpicture} \draw[] (0,0) -- (.2,0) -- (.2,.2) -- (0,.2) -- (0,0);
        \draw[] (0,0) -- (.2,.2);
        \end{tikzpicture}}$. The random variable $L^{\begin{tikzpicture} \draw[] (0,0) -- (.2,0) -- (.2,.2) -- (0,.2) -- (0,0);
        \draw[] (0,0) -- (.2,.2);
        \end{tikzpicture}}$ has the same law of $\lambda_1$, where $\lambda$ is distributed according to measure $\mathbb{HW}^{(q=0)}_{a;\gamma}$ and $a$ is a single Plancherel specialization; here the parameter $\gamma$ counts the number of fixed points of the involution. Notice that $\mathbb{HW}^{(q=0)}_{a;\gamma}$ is a Schur measure since $q=0$. From \cref{prop:symmmetry_half_space_qW} we see that it is equivalent to take $\lambda$ with law  $\mathbb{HW}^{(q=0)}_{(a,\gamma);0}$, i.e. the Pfaffian Schur measure $\mathbb{FBS}^{(q=0)}_{(a,\gamma)}$. This shows that the Baik-Rains crossover arises also as limiting law of $\lambda_1$, when $\lambda \sim \mathbb{FBS}^{(q=0)}_{(a,\gamma)}$ and $a$ is a single Plancherel specialization. 
        In \cref{thm:asymptotic_pfaffian_schur} we characterized the asymptotics of the first row in $\lambda \sim \mathbb{FBS}^{(q=0)}_{(a,\gamma)}$, when $a$ is a list of single variable specializations, but it is straightforward, following the proof to show that the same limit holds also when $a$ is a Plancherel specialization. This completes the proof.
    \end{proof}
    
    \subsection{Baik-Rains phase transition for half space models: positive temperature case} \label{subs:baik_rains_positive}
    
    Guided by the analysis of Pfaffian Schur measure described in the previous subsection, we can now establish asymptotic limits for the partition function of the half space log Gamma polymer or for the rightmost particle in a $q$-PushTASEP with particle creation. Similarities between matrix kernels in the Fredholm pfaffian expansions of both models allow us to present rather concise proofs, which mirror techniques explained in detail in \cref{thm:asymptotic_pfaffian_schur}. 

    \begin{theorem} \label{thm:asymptotic_hs_log_gamma}
        Let $Z^{\mathrm{hs}}(N,N)$ be the partition function of the half space log Gamma polymer with parameters $A_1=A_2=\cdots=A\in(0,+\infty)$ and boundary strength $\varUpsilon>-A$. Then, the following limits hold:
        \begin{itemize}
            \item if $\varUpsilon \ge 0$, rescaling $\varUpsilon = \frac{\xi}{ \sigma^{\mathrm{hs}} N^{1/3} }$ we have
                \begin{equation} \label{eq:convergence_crossover_log_gamma}
                    \lim_{N \to \infty} \mathbb{P} \left[ \frac{ \log Z^{\mathrm{hs}}(N,N) + f^{\mathrm{hs}} N }{ \sigma^{\mathrm{hs}} N^{1/3} } \le s \right] = F_{\mathrm{cross}} (s;\xi),
                \end{equation}
                where $f^{\mathrm{hs}} =  2 \psi(A)$ and $\sigma^{\mathrm{hs}} = \sqrt[3]{\psi^{(2)}(A)}$;
            \item if $-A<\varUpsilon<0$,
                \begin{equation} \label{eq:log_gamma_hs_gaussian_fluctuations}
                    \lim_{N \to \infty} \mathbb{P} \left[ \frac{ \log Z^{\mathrm{hs}}(N,N) - f^{\mathrm{hs}}_\varUpsilon N }{ \sigma^{\mathrm{hs}}_\varUpsilon N^{1/2} } \le s \right] = \int_{-\infty}^{s} \frac{e^{-u^2/2}}{\sqrt{2\pi}} \diff u,
                \end{equation}
                where $f^{\mathrm{hs}}_\varUpsilon = - \psi(A+\varUpsilon) - \psi(A - \varUpsilon)$ and $\sigma^{\mathrm{hs}}_\varUpsilon = \sqrt{\psi^{(1)}(A+\varUpsilon) - \psi^{(1)}(A-\varUpsilon)}$.
        \end{itemize}
    \end{theorem}
    
    \begin{proof}
        The proof is based on asymptotic analysis of Fredholm pfaffian formulas \eqref{eq:fredholm_pfaff_log_gamma} and \eqref{eq:pfaff_L_hat_log_gamma}. As observed in \eqref{eq:coovergence_laplace_to_probability}, following a scaling of the parameter $\varsigma$, the Laplace transform $\mathbb{E}\left[ e^{-e^{-\varsigma + \log Z^{\mathrm{hs}}(N,N)}} \right]$ identifies the probability distribution of asymptotic fluctuations of $\log Z^{\mathrm{hs}}(N,N)$. We distinguish three cases depending on the value of $\varUpsilon$.
        
        \medskip
        
        \emph{ $\varUpsilon > 0$ : convergence to crossover distribution.} The starting point is the exact formula \eqref{eq:fredholm_pfaff_log_gamma} and we consider the scaling
        \begin{equation} \label{eq:scaling_varsigma}
            -\varsigma = f_{\mathrm{hs}} N - \sigma_{\mathrm{hs}} N^{1/3} s,
        \end{equation}
        which induces in the kernel $\mathbf{L}(x,y)$ the change of variables
        \begin{equation} \label{eq:scaling_X_Y}
            X = \varsigma + \sigma_{\mathrm{hs}} N^{1/3} u,
            \qquad
            Y = \varsigma + \sigma_{\mathrm{hs}} N^{1/3} v,
        \end{equation}
        with $u,v>s$. Defining the function
        \begin{equation} \label{eq:h_hs_log_gamma}
            h(Z) = \log \Gamma (A+Z) - \log \Gamma (A-Z) - f_{\mathrm{hs}} Z
        \end{equation}
        we rewrite $\mathbf{k}$ of \eqref{eq:k_hs_log_gamma} as
        \begin{equation} \label{eq:kernel_log_gamma_hs_rescaled}
        \begin{split}
            \mathbf{k}(X(u),Y(v)) = \int_{\mathrm{i} \mathbb{R}-d} \frac{\diff Z}{2 \pi \mathrm{i}} \int_{\mathrm{i} \mathbb{R}-d} \frac{\diff W}{2 \pi \mathrm{i}} & e^{N (h(Z) +h(W)) + \sigma_{\mathrm{hs}} N^{1/3}(uZ + vW)}
            \\
            & \times \frac{\Gamma(\varUpsilon + Z)}{\Gamma(\varUpsilon - Z)} \frac{\Gamma(\varUpsilon + W)}{\Gamma(\varUpsilon - W)} \mathsf{k}^{\mathrm{hs}}(Z,W).
        \end{split}
        \end{equation}
        As already observed in the proof of \cref{thm:log_gamma_GUE}, $h(Z)$ possesses a doubly critical point at $Z=0$ and for $d$ sufficiently close to 0, the function $\Re\{h(Z)\}$, restricted on $Z \in \mathrm{i} \mathbb{R} - d$  attains a global maximum at $Z=-d$. We deform both $Z$ and $W$ integration contours to the curve
        \begin{equation} \label{eq:contour_hs_log_gamma}
            \mathsf{C}^{2 \pi/3}\left( -\frac{\delta}{\sigma N^{1/3}} ; d \right) \cup \mathsf{D} (-d, \sqrt{3}(d-\frac{\delta}{\sigma N^{1/3}})),
        \end{equation}
        where we recall the notation of \eqref{eq:notation_contours} and $\delta \in (0,\xi)$. We see that, by the properties of function $h$, the contour \eqref{eq:contour_hs_log_gamma} is steep descent and for large $N$, the dominant part of both $Z$ and $W$ integrals comes from the integration over the curve $\mathsf{C}^{2 \pi/3}\left( -\frac{\delta}{\sigma N^{1/3}} ; d \right)$. After a change of variables
        \begin{equation}
            Z=- \frac{\alpha}{ \sigma_{\mathrm{hs}} N^{1/3} },
            \qquad
            W = - \frac{\beta}{ \sigma_{\mathrm{hs}} N^{1/3} }
        \end{equation}
        from expression \eqref{eq:kernel_log_gamma_hs_rescaled}, we see that
        \begin{equation}
            \mathbf{k}(X(u),Y(v)) = \mathpzc{k}^{(\xi)}(u,v) (1+O(N^{-1/3}))
        \end{equation}
        and similarly for other elements of the matrix kernel, $\mathbf{L}$
        \begin{gather*}
            \partial_X \mathbf{k}(X(u),Y(v)) = \frac{1}{\sigma_{\mathrm{hs}}N^{1/3}} \partial_u\mathpzc{k}^{(\xi)}(u,v) (1+O(N^{-1/3})),
            \\
            \partial_Y \mathbf{k}(X(u),Y(v)) = \frac{1}{\sigma_{\mathrm{hs}}N^{1/3}} \partial_v\mathpzc{k}^{(\xi)}(u,v) (1+O(N^{-1/3})),
            \\
            \partial_X \partial_Y \mathbf{k}(X(u),Y(v)) =  \frac{1}{\sigma_{\mathrm{hs}}^2N^{2/3}} \partial_u \partial_v \mathpzc{k}^{(\xi)}(u,v) (1+O(N^{-1/3})).
        \end{gather*}
        Exponential bounds such as 
        \begin{multline}
            |\mathbf{k}(X(u),Y(v))|,|\sigma_{\mathrm{hs}}N^{1/3} \partial_Y \mathbf{k}(X(u),Y(v))|,
            \\
            |\sigma_{\mathrm{hs}}N^{1/3} \partial_X \mathbf{k}(X(u),Y(v))|,|\sigma_{\mathrm{hs}}^2N^{2/3} \partial_X\partial_Y \mathbf{k}(X(u),Y(v))| < \mathrm{const.} e^{-\delta(u+v)}
        \end{multline}
        are also straightforward from steep descent analysis and this proves, through bounded convergence that
        \begin{equation}
            \lim_{N \to \infty} \Pf \left[ J - \mathbf{L} \right]_{\mathbb{L}^2(\varsigma,+\infty)} = \Pf\left[ J - \mathcal{K}^{(\xi)}_{\mathrm{cross}} \right]_{\mathbb{L}^2(s,+\infty)}
        \end{equation}
        and hence \eqref{eq:convergence_crossover_log_gamma} for $\varUpsilon>0$.
        
        \medskip
        
        \emph{ $\varUpsilon = 0$ : convergence to GOE Tracy-Widom distribution.} In \cref{subs:pfaffian_log_gamma} right before \cref{thm:fredholm_pfaff_log_gamma_gaussian}, we have already observed that, when $\varUpsilon$ approaches $0$, for the representation \eqref{eq:k_hs_log_gamma} of the function $\mathbf{k}$, it becomes no longer possible to choose integration contours $\mathrm{i} \mathbb{R} -d$ such that $0
        <d<\varUpsilon$. Denoting by $\mathbf{k}_{\mathrm{ext}}(X,Y)$ the analytic extension of $\mathbf{k}$ in the parameter $\varUpsilon$ we have
        \begin{equation}
            \begin{split}
                \mathbf{k}_{\mathrm{ext}}(X,Y) \big|_{\varUpsilon=0} &= \left(\mathbf{k}(X,Y) - \mathbf{A}_0(X) \mathbf{B}(Y) +\mathbf{A}_0(Y) \mathbf{B}(X) \right)\big|_{\varUpsilon=0}
                \\ & =
                \left(\mathbf{k}(X,Y) - \mathbf{B}(Y) + \mathbf{B}(X) \right)\big|_{\varUpsilon=0}
            \end{split}
        \end{equation}
        where functions $\mathbf{A}_0,\mathbf{B}$ were defined in \eqref{eq:p_j} \eqref{eq:q} and we have used the fact that $\mathbf{A}_0(X)\big|_{\varUpsilon=0}=1$. We also define $\mathbf{L}_{\mathrm{ext}}$ to be the analytic extension of $\mathbf{L}$ in the parameter $\varUpsilon$ and we have
        \begin{equation}
            \mathbf{L}_{\mathrm{ext}}(X,Y)\big|_{\varUpsilon=0} = \mathbf{L}(X,Y)\big|_{\varUpsilon=0} - \left( \begin{matrix} 1 & -\partial_Y \\ -\partial_X & \partial_X \partial_Y \end{matrix} \right) (\mathbf{B}(Y) - \mathbf{B}(X))\big|_{\varUpsilon=0}.
        \end{equation}
        Considering the scaling \eqref{eq:scaling_varsigma}, \eqref{eq:scaling_X_Y} we can now compute through the steep descent method, the asymptotic form of the kernel $\mathbf{L}_{\mathrm{ext}}$, as done for the case $\varUpsilon>0$. We obtain 
        \begin{gather}
            \mathbf{k}(X(u),Y(v))\big|_{\varUpsilon=0} = \mathpzc{k}^{(\infty)}(u,v)(1+O(N^{-1/3})),
            \\
            \mathbf{B}(X(u)) = \mathpzc{B}(u) (1+O(N^{-1/3})),
        \end{gather}
        where the terms $O(N^{-1/3})$ can be shown to be absolutely bounded in $u,v>s$. This shows that 
        \begin{equation} \label{eq:L_log_gamma_converges_GOE}
            \left(\begin{matrix} 1 & 0 \\ 0 & \sigma_{\mathrm{hs}} N^{1/3} \end{matrix}\right) \cdot
            \mathbf{L}_{\mathrm{ext}}(X(u),Y(v)) \big|_{\varUpsilon=0} \cdot \left(\begin{matrix} 1 & 0 \\ 0 & \sigma_{\mathrm{hs}} N^{1/3} \end{matrix}\right) =  \mathcal{K}_{\mathrm{GOE}}(u,v) (1+O(N^{-1/3})).
        \end{equation}
        By arguments developed in the proof of the $\gamma=1$ case of \cref{thm:asymptotic_pfaffian_schur}, the Fredholm pfaffian of the matrix kernel $\mathbf{L}_{\mathrm{ext}}$ always defines an absolutely convergent expansion, even though it does not define a bounded operator on $\mathbb{L}^2(\varsigma,+\infty) \times \mathbb{L}^2(\varsigma,+\infty)$. This along with the convergence \eqref{eq:L_log_gamma_converges_GOE} imples that
        \begin{equation}
            \lim_{N \to \infty} \Pf \left[ J - \mathbf{L}_{\mathrm{ext}}\big|_{\varUpsilon=0} \right]_{\mathbb{L}^2(\varsigma,+\infty)} = \Pf\left[ J - \mathcal{K}_{\mathrm{GOE}} \right]_{\mathbb{L}^2(s,+\infty)},
        \end{equation}
        proving \eqref{eq:convergence_crossover_log_gamma} for $\varUpsilon=0$.
    
        \medskip
        
        \emph{ $\varUpsilon < 0$ : convergence to gaussian distribution.} We will follow the same ideas as in the case $\gamma>1$ of \cref{thm:asymptotic_pfaffian_schur}. Nevertheless, computations here become technically more involved due to the nontrivial pole structure of the kernel $\mathbf{L}$, compared to that of $\bar{L}$. The main issue we will overcome below is to control that, during the asymptotic limit, our pfaffian expression remains absolutely convergent. We consider the equality \eqref{eq:pfaff_L_hat_log_gamma} and the scaling
        \begin{equation} \label{eq:scaling_varsigma_log_gamma}
            -\varsigma = f_\varUpsilon^{\mathrm{hs}} N - \varsigma^{\mathrm{hs}}_\varUpsilon N^{1/2} s.
        \end{equation}
        With this choice it is clear that, in the left hand side of \eqref{eq:pfaff_L_hat_log_gamma}, 
        \begin{equation}
            \lim_{N\to \infty}\mathbb{E}\left[ e^{-e^{-\varsigma + \log Z^{\mathrm{hs}}(N,N)}} \right] = \lim_{N \to \infty} \mathbb{P} \left[ \frac{ \log Z^{\mathrm{hs}}(N,N) - f^{\mathrm{hs}}_\varUpsilon N }{ \sigma^{\mathrm{hs}}_\varUpsilon N^{1/2} } \le s \right]
        \end{equation}
        and therefore it remains to consider the scaling limit of the right hand side of \eqref{eq:pfaff_L_hat_log_gamma},
        \begin{equation} \label{eq:product_pfaff_L_hat_times_scalar_product}
            \Pf \left[ J - \widehat{\mathbf{L}} \right]_{\mathbb{L}^2(\varsigma,+\infty)} \times  \bigg( 1-  \langle (\mathbf{B},-\mathbf{B}') | \left(1-J^T \widehat{\mathbf{L}}\right)^{-1} | (\mathbf{A}',\mathbf{A} )\rangle \bigg).
        \end{equation}
        The scaling \eqref{eq:scaling_varsigma_log_gamma} of parameter $\varsigma$ induces on the functions $\widehat{\mathbf{L}}, \mathbf{B},\mathbf{A}$ the change of variable
        \begin{equation}
            X= \varsigma +  \sigma^{\mathrm{hs}}_\varUpsilon N^{1/2} u,
            \qquad
            Y= \varsigma +  \sigma^{\mathrm{hs}}_\varUpsilon N^{1/2} v.
        \end{equation}
        Defining the function
        \begin{equation}
            h_\varUpsilon(Z) = \log \Gamma(A+Z) - \log \Gamma (A-Z) - f_\varUpsilon^{\mathrm{hs}} Z,
        \end{equation}
        we rewrite $\mathbf{A}_j$ (which compose $\mathbf{A}$) and $\mathbf{B}$ appearing in terms in the right hand side of \eqref{eq:product_pfaff_L_hat_times_scalar_product} as
        \begin{equation}
            \mathbf{A}_j(X(u)) = \frac{(-1)^j}{j!} e^{N h_\varUpsilon(-\varUpsilon - j)   -(j+\varUpsilon) \sigma_\varUpsilon^{\mathrm{hs}} N^{1/2} u } \frac{\Gamma(2\varUpsilon+2j)}{\Gamma(2\varUpsilon+j) },
        \end{equation}
        \begin{equation}
            \mathbf{B}(Y(v)) = \int_{\mathrm{i}\mathbb{R} - d} \frac{\diff W}{2\pi \mathrm{i}} e^{N h_\varUpsilon(W) + \sigma_\varUpsilon^{\mathrm{hs}} N^{1/2} v W} \Gamma(-2W)
            \frac{\Gamma(1-\varUpsilon+W)}{\Gamma(1-\varUpsilon -W)}.
        \end{equation}
        For the matrix kernel $\widehat{\mathbf{L}}$ we are going to use the expression \begin{equation}
            \left( \begin{matrix} 1 & - \partial_Y \\ - \partial_X & \partial_X \partial_Y \end{matrix} \right) \widehat{\mathbf{L}}_{1,1}(X,Y),
        \end{equation}
        where $\widehat{\mathbf{L}}_{1,1}(X,Y)$ is given by \eqref{eq:L_11_2}. We write
        \begin{equation}
        \begin{split}
            S_j(X(u),Y(v))=-\int_{\mathrm{i} \mathbb{R} + R } \frac{\diff W}{2\pi \mathrm{i}} 
            &
            e^{N\left( h_\varUpsilon(W) + h_\varUpsilon(-W-j) \right) + \sigma_\varUpsilon^{\mathrm{hs}} N^{1/2} (vW - u(W+j)) } \\&\times
            \frac{\Gamma(\varUpsilon+W)}{\Gamma(\varUpsilon-W)} \frac{\Gamma(\varUpsilon-W-j)}{\Gamma(\varUpsilon+W+j)} \frac{\Gamma(2W+2j)}{\Gamma(2W+1)},
        \end{split}
        \end{equation}
        where $R\in (\varUpsilon-1,-\varUpsilon-j+1)$ and
        \begin{equation}
        \begin{split}
            \mathbf{k}_{\varUpsilon-\delta}(X(u),Y(v)) = \int_{\mathrm{i} \mathbb{R}+\varUpsilon -\delta} \frac{\diff Z}{2 \pi \mathrm{i}} \int_{\mathrm{i} \mathbb{R} +\varUpsilon -\delta} \frac{\diff W}{2 \pi \mathrm{i}} & e^{N (h_\varUpsilon(Z) +h_\varUpsilon(W)) + \sigma_\varUpsilon^{\mathrm{hs}} N^{1/2}(uZ + vW)}
            \\
            & \times \frac{\Gamma(\varUpsilon + Z)}{\Gamma(\varUpsilon - Z)} \frac{\Gamma(\varUpsilon + W)}{\Gamma(\varUpsilon - W)} \mathsf{k}^{\mathrm{hs}}(Z,W).
        \end{split}
        \end{equation}
        Simple computations show that $h_\varUpsilon$ has the following properties,
        \begin{gather*}
            h_\varUpsilon(Z) < 0 \quad \text{for }Z \in [\varUpsilon,0), \qquad h_\varUpsilon'(Z) > 0 \quad \text{for }Z \in (\varUpsilon,-\varUpsilon),
            \\
            h_\varUpsilon(-Z) = - h_\varUpsilon(Z), \qquad h_\varUpsilon'(\varUpsilon) = 0, \qquad h_\varUpsilon''(\varUpsilon) = (\sigma_\varUpsilon^{\mathrm{hs}})^2  > 0.
        \end{gather*}
        We evaluate the function $\mathbf{k}_{\varUpsilon-\delta}$ through saddle point method setting $\delta=0$, noticing that the integration contour passes through the saddle point $W=\varUpsilon$ of $h_\varUpsilon$. After a change of variables,
        \begin{equation}
            Z = \varUpsilon + \mathrm{i} \frac{\alpha}{\sigma_\varUpsilon^{\mathrm{hs}} N^{1/2} },
            \qquad
            W = \varUpsilon + \mathrm{i} \frac{\beta}{\sigma_\varUpsilon^{\mathrm{hs}} N^{1/2} },
        \end{equation}
        we have
        \begin{equation} \label{eq:estimate_K_varUpsilon}
        \begin{split}
            \mathbf{k}_{\varUpsilon}(X(u),Y(v)) = & e^{2N h_\varUpsilon(\varUpsilon) + \sigma_\varUpsilon^{\mathrm{hs}} N^{1/2} \varUpsilon (u+v) } \frac{\Gamma(2\varUpsilon)^2 \Gamma(-2\varUpsilon)^2}{\sin(2\pi \varUpsilon)} \left(\frac{1}{\sigma_\varUpsilon^\mathrm{hs} N^{1/2}} \right)^5
            \\
            & \times
            \int_{-\infty}^{\infty} \frac{\diff \alpha}{2\pi} \int_{-\infty}^{\infty} \frac{\diff \beta}{2\pi} e^{-(\alpha^2+\beta^2)/2 + \mathrm{i} (u\alpha +v\beta )} \alpha \beta (\alpha-\beta)
            \\
            & = O\left( N^{-5/2} e^{2N h_\varUpsilon(\varUpsilon) + \sigma_\varUpsilon^{\mathrm{hs}} N^{1/2} \varUpsilon (u+v) - c (u^2+v^2) } \right),
        \end{split}
        \end{equation}
        where $c \in (0,1/2)$. The previous expression is apparently divergent when $\varUpsilon=-1,-2,\dots$, but one can easily see that this is not the case as such singularities can be avoided shifting integration contours by $O(N^{-1/2})$ to the left of $\varUpsilon$ at the cost of adding a factor $N^{3/2}$. A similar estimate can be obtained for the functions $S_j$, although in that case we opt for a less sharp bound keeping the integration contour to be $\mathrm{i} \mathbb{R} +R$, which is not necessarily steep descent for $h_\varUpsilon(W) +h_\varUpsilon(-W-j)$. We have
        \begin{equation} \label{eq:bound_S_j}
            S_j(X(u),Y(v)) = O\left(N^{-1} e^{N(h_\varUpsilon(R) +h_\varUpsilon(-R-j)) + \sigma_\varUpsilon^{\mathrm{hs}} N^{1/2}(vR -u(R+j)) } \right).
        \end{equation}
        Notice that whenever $R\in(\varUpsilon-\tilde{\delta}, \tilde{\delta})$, for $\tilde{\delta}>0$ small enough and $j<-R$, then $h_\varUpsilon(R) +h_\varUpsilon(-R-j)<0$ and hence the right hand side of \eqref{eq:bound_S_j} is exponentially decaying in $N$. Combining \eqref{eq:estimate_K_varUpsilon} and \eqref{eq:bound_S_j} we obtain, recalling \eqref{eq:L_11_2},
        \begin{equation} \label{eq:bound_L_hat}
            \widehat{\mathbf{L}}(X(u),Y(v)) = O \left( N^{-1} e^{-CN} e^{\sigma_{\varUpsilon}^\mathrm{hs} N^{1/2} [(-R-1)u+Rv] } \right),
        \end{equation}
        for some $C>0$ fixed. Choosing a suitable $R \in (-1,0)$, such bound clearly implies the estimate
        \begin{equation}
            \Pf[J - \widehat{\mathbf{L}}]_{\mathbb{L}(\varsigma,+\infty)} = 1 + O(e^{-C' N}),
        \end{equation}
        where $0 < C'<C$, settling the convergence of the first factor in \eqref{eq:product_pfaff_L_hat_times_scalar_product}. 
        
        We now come to estimate the scalar product appearing in the second factor of \eqref{eq:product_pfaff_L_hat_times_scalar_product}. In order to estimate the function $\mathbf{B}$ we move the integration contour to the line $\mathrm{i} \mathbb{R}+\varUpsilon$ and by saddle point method we obtain
        \begin{equation} \label{eq:q_asymptotics}
        \begin{split}
            \mathbf{B}(Y(v)) & = e^{N h_\varUpsilon(\varUpsilon) + \sigma_{\varUpsilon}^{\mathrm{hs}} N^{1/2} v \varUpsilon } \frac{\Gamma(-2\varUpsilon)}{\Gamma(1-2\varUpsilon)} \int_{-\infty}^{+\infty} \frac{\diff \beta}{ 2 \pi } e^{-\beta^2/2 + \mathrm{i} v \beta} (1+O(N^{-1/2})) \\
            & = e^{N h_\varUpsilon(\varUpsilon) + \sigma_{\varUpsilon}^{\mathrm{hs}} N^{1/2} v \varUpsilon } \frac{1}{-2 \varUpsilon} \frac{1}{\sqrt{2 \pi}} e^{-v^2/2} (1+O(N^{-1/2})).
        \end{split}
        \end{equation}
        On the other hand for the function $\mathbf{A}$, we observe that
        \begin{equation}\label{eq:p_asymptotics}
            \mathbf{A}(X(u)) = \mathbf{A}_0(X(u))(1+o(1)) = e^{-Nh_\varUpsilon(\varUpsilon)-\varUpsilon\sigma_\varUpsilon^{\mathrm{hs}}N^{1/2} u}(1+o(1)),
        \end{equation}
        where the term $o(1)$ is bounded in $u$ and uniformly exponentially decaying in $N$. By virtue of estimate \eqref{eq:bound_L_hat} the operator $J^T \widehat{\mathbf{L}}$ has norm arbitrarily small, for $N$ large and hence the resolvent $(1-J^T \widehat{\mathbf{L}})^{-1}$ can be expanded in Neumann series and we have
        \begin{equation}
            \langle (\mathbf{B},-\mathbf{B}') | \left(1-J^T \widehat{\mathbf{L}}\right)^{-1} | (\mathbf{A}',\mathbf{A} )\rangle = \sum_{\ell \ge 0} \langle (\mathbf{B},-\mathbf{B}') | \left(J^T \widehat{\mathbf{L}}\right)^\ell | (\mathbf{A}',\mathbf{A} )\rangle.
        \end{equation}
        In the proof of \cref{thm:fredholm_pfaff_log_gamma_gaussian}, using estimates \eqref{eq:bound_p,q_2}, \eqref{eq:L_hat_exponential_decay_1} and producing the iterated bound \eqref{eq:J^TL_r_times}, we established the exponential decay of terms $\langle (\mathbf{B},-\mathbf{B}') | \left(J^T \widehat{\mathbf{L}}\right)^\ell | (\mathbf{A}',\mathbf{A} )\rangle$.  Here we can refine such arguments using the improved bounds \eqref{eq:bound_L_hat}, \eqref{eq:q_asymptotics}, \eqref{eq:p_asymptotics} and we obtain, for any $\ell \ge 1$,
        \begin{equation}
            \left|\langle (\mathbf{B},-\mathbf{B}') | \left(J^T \widehat{\mathbf{L}}\right)^\ell | (\mathbf{A}',\mathbf{A} )\rangle \right| = O \left( e^{-C'N \ell} \right).
        \end{equation}
        For $\ell=0$ we have 
        \begin{equation}
            \begin{split}
                \langle (\mathbf{B},-\mathbf{B}') | (\mathbf{A}',\mathbf{A} )\rangle = \int_{s}^{\infty} \frac{\diff u}{\sqrt{2\pi}} e^{-u^2/2}  (1+O(N^{-1/2})),
            \end{split}
        \end{equation}
    which proves that
    \begin{equation*} 
        \lim_{N\to \infty} \Pf \left[ J - \widehat{\mathbf{L}} \right]_{\mathbb{L}^2(\varsigma,+\infty)} \times  \bigg( 1-  \langle (\mathbf{B},-\mathbf{B}') | \left(1-J^T \widehat{\mathbf{L}}\right)^{-1} | (\mathbf{A}',\mathbf{A} )\rangle \bigg) = \int^{s}_{-\infty} \frac{\diff u}{\sqrt{2\pi}} e^{-u^2/2}.
    \end{equation*}
    This shows \eqref{eq:log_gamma_hs_gaussian_fluctuations} and completes the proof of the theorem.
    
    \end{proof}
    
    \begin{remark}
        The one presented above is the first instance of a proof of the Baik-Rains phase transition for a ``positive temperature" KPZ model. Recently, \cite{barraquand_wang_2021} proved a similar ``depinning transition" for the point-to-half-line partition function of the half space Log Gamma polymer using a matching in distribution with the point-to-point polymer partition function in full space.
    \end{remark}
    
    We also present asymptotic results for the $q$-PushTASEP in half space. The proof follows the same ideas as those presented in \cref{thm:asymptotic_hs_log_gamma} and therefore we are only going to sketch it.
    
    For the next theorem we recall the $q$-digamma and $q$-polygamma functions
    \begin{equation}
        \psi_q(z) = -\log(1-q) + \log q \sum_{\ell \ge 0} \frac{q^{z+\ell}}{1-q^{z+\ell}}
        \qquad
        \text{and}
        \qquad
        \psi^{(n)}_q(z) = \frac{\diff^n}{\diff z^n} \psi_q(z)
    \end{equation}
    which are well defined for $q\in(0,1)$ and $z \in \mathbb{C}\setminus\mathbb{Z}_{\le 0}$. Instead of the series representation given above, the $q$-digamma function can be defined in terms of the $q$-Gamma function $\psi_q(z) = \frac{\diff}{\diff z}\log \Gamma_q(z)$, where $\Gamma_q(z)$ is defined in \eqref{eq:q_gamma}.
    
    \begin{theorem}
        Let $\mathsf{X}^{\mathrm{hs}}(T)$ be the geometric $q$-PushTASEP with particle creation and assume that parameters are $a_1=a_2=\cdots=a\in(0,1)$ and $\gamma\in(0,1/a)$. Then, the following limits hold:
        \begin{itemize}
            \item if $\gamma \in (0,1]$, rescaling $\gamma = 1- \frac{\xi}{ \sigma^{\mathrm{hs}} N^{1/3} }$ we have
                \begin{equation} \label{eq:crossover_q_pushTASEP}
                    \lim_{N \to \infty} \mathbb{P} \left[ \frac{ \mathsf{x}^{\mathrm{hs}}_N(N) - (p^{\mathrm{hs}}+1) N }{ \sigma^{\mathrm{hs}} N^{1/3} } \le s' \right] = F_{\mathrm{cross}} (s';\xi),
                \end{equation}
                where $p^{\mathrm{hs}} =  \frac{2}{\log q} \left[ \log(1-q) + \psi_q\left( \log_q a \right) \right]$ and $\sigma^{\mathrm{hs}} = \frac{1}{\log q} \sqrt[3]{\psi_q^{(2)}\left( \log_q a \right) }$;
            \item if $\gamma\in(1,1/a)$,
                \begin{equation} \label{eq:q_pushTASEP_gaussian}
                    \lim_{N \to \infty} \mathbb{P} \left[ \frac{ \mathsf{x}^{\mathrm{hs}}_N(N) - (p^{\mathrm{hs}}_\gamma +1) N }{ \sigma^{\mathrm{hs}}_\varUpsilon N^{1/2} } \le s' \right] = \int_{-\infty}^{s'} \frac{e^{-u^2/2}}{\sqrt{2\pi}} \diff u,
                \end{equation}
                where 
                $$
                p^{\mathrm{hs}}_\gamma = \frac{1}{\log q} \left[ 2\log(1-q) + \psi_q\left( \log_q(a/\gamma) \right) + \psi_q\left( \log_q(a \gamma) \right) \right]
                $$ 
                and 
                $$
                \sigma^{\mathrm{hs}}_\gamma = \frac{\gamma}{\log q} \sqrt{\psi_q^{(1)}\left(\log_q(a \gamma)\right) - \psi_q^{(1)}\left(\log_q(a /\gamma)\right) }.
                $$
        \end{itemize}
    \end{theorem}
    
    \begin{proof}[Sketch of proof] 
    
    The starting point is the Fredholm pfaffian formula \eqref{eq:qpush_TASEP_pfaffian}. We consider different scalings of the parameter $s$ depending on the value of $\gamma$.
    
    \medskip
    
    \emph{$\gamma<1$: convergence to crossover distribution.} Under the scaling
    \begin{equation}
        s = (p^\mathrm{hs}-1)N + \sigma^\mathrm{hs}N^{1/3} s', \qquad x = s + \sigma^\mathrm{hs} N^{1/3} u, \qquad y = s + \sigma^\mathrm{hs} N^{1/3} v, 
    \end{equation}
    the function $k(x,y)$, which defines the kernel $L_\mathrm{push}$ can be written as
    \begin{equation} \label{eq:k_q_rescaled}
    \begin{split}
        k(x(u),y(v)) = \oint_{|z|=r}& \frac{\diff z}{2\pi \mathrm{i}z^{3/2}} \oint_{|w|=r'} \frac{\diff w}{2\pi \mathrm{i} w^{5/2}} \frac{(\gamma/z,\gamma/w;q)_\infty}{(\gamma z,\gamma w;q)_\infty} \kappa^\mathrm{hs}(z,w) \\
        & \times 
        e^{N h_q(z) - \sigma^\mathrm{hs} N^{1/3} (u+s') \log z} e^{N h_q(w) - \sigma^\mathrm{hs} N^{1/3} (u+s') \log w}.
    \end{split}
    \end{equation}
    Here we assume that $u,v>s'$ and the function $h_q$ in the exponents is defined as
    \begin{equation}
        h_q(z) = \log (a/z;q)_\infty - \log (a z;q)_\infty - p^\mathrm{hs} \log z.
    \end{equation}
    A calculation shows that $h_q(1)=h_q'(1)=h_q''(1)=0$ and $h_q'''(1)= 2 (\sigma^\mathrm{hs})^3>0$, so that, employing a saddle point analysis, around critical points $z=1,w=1$ we can prove that
    \begin{equation}
        \lim_{N\to \infty} k(x(u),y(v)) = \mathpzc{k}^{(\xi)} (u,v)
        \qquad
        \text{and}
        \qquad
        \lim_{N\to \infty} L_\mathrm{push}(x(u),y(v)) = \mathcal{K}_\mathrm{cross}^{(\xi)} (u,v).
    \end{equation}
    Considering the scaling of $s$ adopted in the left hand side of \eqref{eq:qpush_TASEP_pfaffian}, this shows \eqref{eq:crossover_q_pushTASEP} for $\gamma<1$.
    
    \medskip
    
    \emph{$\gamma=0$ : convergence to GOE distribution.} 
    For the rest of the proof we will express the residues of the integrand of \eqref{eq:k_q_rescaled} at poles $z,w=\gamma^{-1} q^{-j}$, for $j=0,1,2,\dots$, as
    \begin{equation*}
        \mathrm{Res}_{z=\frac{1}{\gamma q^j}} \left\{ \parbox{2.1cm}{\centering $z$-integrand in \eqref{eq:k_q_rescaled}} \right\} = A^{(q)}_j(x) B^{(q)}(y),
        \qquad
        \mathrm{Res}_{w=\frac{1}{\gamma q^j}} \left\{ \parbox{2.1cm}{\centering $w$-integrand in \eqref{eq:k_q_rescaled}} \right\} = - A^{(q)}_j(y) B^{(q)}(x),
    \end{equation*}
    where
    \begin{equation}
        A^{(q)}_j(x) = (\gamma q^j)^{x+1/2} (-1)^j \frac{q^{\binom{j+1}{2}}}{(q;q)_j} \frac{( \gamma^2 q^j;q)_\infty}{(\gamma^2 q^{2j};q)_\infty} \left( \frac{ (a\gamma q^j ;q)_\infty }{ (a /\gamma q^j ;q)_\infty } \right)^N,
    \end{equation}
    \begin{equation}
        B^{(q)}(y) = \frac{1}{2 \pi \mathrm{i}} \oint_{|w|=r'} \frac{\diff w}{w^{y+5/2}} F(w) \frac{(q,w \gamma q^j,1/\gamma q^{j-1}w;q)_\infty}{(1/w^2,q^j \gamma/w,z/\gamma q^{j-1};q)_\infty} \frac{\vartheta_3(\zeta^2w^2/\gamma^2 q^{2j};q^2)}{ \vartheta_3(\zeta^2;q^2) }.
    \end{equation}
    These should be seen as $q$-deformations of functions \eqref{eq:p_j}, \eqref{eq:q}. Then, we can proceed with the analytic continuation of formula \eqref{eq:qpush_TASEP_pfaffian} in the variable $\gamma$. In particular as $\gamma$ approaches $1$ from the left we see that it is no longer possible to select radii $r,r'$ of the integration contours of the function $k(x,y)$ such that $1<r,r'<1/\gamma$ and as the singularities at $z,w=1/\gamma$ cross the contours the analytic extension of the operator $L_\mathrm{push}$ reads
    \begin{equation}
        L_\mathrm{push,ext}\big|_{\gamma = 1} (x,y) = L_\mathrm{push}\big|_{\gamma = 1}(x,y) - \left( \begin{matrix} 1 & - \nabla_y \\ - \nabla_x & \nabla_x \nabla_y \end{matrix} \right) \left[ B^{(q)}(y) - B^{(q)}(x) \right]\bigg|_{\gamma = 1},
    \end{equation}
    since $A^{(q)}_0(x)\big|_{\gamma = 1} = 1$. Using ideas elaborated in the case $\varUpsilon=0$ of \cref{thm:asymptotic_hs_log_gamma} we can now prove that the Fredholm pfaffian of the operator $L_\mathrm{push,ext}\big|_{\gamma = 1}$ defines an absolutely convergent series and its limit for $N\to \infty$ can be computed using saddle point method. The result is 
    \begin{equation}
        \lim_{N \to \infty} \left(\begin{matrix} 1 & 0 \\ 0 & \sigma^{\mathrm{hs}} N^{1/3} \end{matrix}\right) \cdot L_\mathrm{push,ext}\big|_{\gamma = 1} \cdot \left(\begin{matrix} 1 & 0 \\ 0 & \sigma^{\mathrm{hs}} N^{1/3} \end{matrix}\right) = \mathcal{K}_\mathrm{GOE}.
    \end{equation}
    One can now show, as in \cref{thm:asymptotic_hs_log_gamma}, that such limit implies the convergence of Fredholm pfaffians, proving the case $\gamma =1$ of \eqref{eq:crossover_q_pushTASEP}.
    
    \medskip
    
    \emph{$\gamma > 1$ : convergence to gaussian distribution} We will argue in the same way as in the case $\varUpsilon<0$ of \cref{thm:asymptotic_hs_log_gamma}. First we need to provide an analytic extension of the Fredholm pfaffian in the right hand side of \eqref{eq:qpush_TASEP_pfaffian} valid for $\gamma >1$. Reasoning as in \cref{thm:fredholm_pfaff_log_gamma_gaussian} one can show that, assuming that $\gamma \in [q^{-m+\frac{2-\varepsilon}{2}},q^{-m+\frac{1-\varepsilon}{2}})$ for some $m\in \mathbb{Z}_{\ge 0}$ and $\varepsilon\in\{0,1\}$, then
    \begin{equation}
        \Pf\left[ J - \widehat{L}_\mathrm{push} \right]_{\ell^2(\mathbb{Z}'_{> s})} \left( 1- \langle (B^{(q)},-\nabla B^{(q)}) | \left( 1-\widehat{L}_\mathrm{push} \right)^{-1} | (\nabla A^{(q)},A^{(q)}) \right),
    \end{equation}
    is the analytic extension of $\Pf\left[ J - L_\mathrm{push} \right]_{\ell^2(\mathbb{Z}'_{\ge s})}$, where
    \begin{equation} \label{eq:L_hat_push}
        \widehat{L}_\mathrm{push}(x,y) = L_\mathrm{push}(x,y) - \left( \begin{matrix} 1 & - \nabla_y \\ - \nabla_x & \nabla_x \nabla_y \end{matrix} \right) \sum_{j=m}^{2m+\varepsilon-2} \left[A_j^{(q)}(x) B^{(q)}(y) - A_j^{(q)}(y) B^{(q)}(x)\right],
    \end{equation}
    and $A^{(q)} = \sum_{j=0}^{2m+\varepsilon-2} A_j^{(q)}$. Here, we are assuming that $\langle (\, \cdot \, , \, \cdot \,) | (\, \cdot \, , \, \cdot \,) \rangle $ is the natural scalar product on $\ell^2(\mathbb{Z}'_{>s}) \times \ell^2(\mathbb{Z}'_{>s})$. Then, considering the scaling
    \begin{equation}
        s = (p^\mathrm{hs}_\gamma -1) N + \sigma^\mathrm{hs}_\gamma N^{1/2} s'
    \end{equation}
    one can show, following ideas developed in \cref{thm:asymptotic_hs_log_gamma}, that
    \begin{equation}
        \lim_{N\to \infty} \Pf\left[ J - \widehat{L}_\mathrm{push} \right]_{\ell^2(\mathbb{Z}'_{\ge s})} = 1
    \end{equation}
    and
    \begin{equation}
        \lim_{N\to \infty} \left( 1- \langle (B^{(q)},-\nabla B^{(q)}) | \left( 1-\widehat{L}_\mathrm{push} \right)^{-1} | (\nabla A^{(q)},A^{(q)}) \right) = \int_{-\infty}^{s'} \frac{e^{-u^2/2}}{\sqrt{2\pi}} \diff u.
    \end{equation}
    Taking the same scaling limit on the left hand side of \eqref{eq:qpush_TASEP_pfaffian} completes the proof of \eqref{eq:q_pushTASEP_gaussian}.
\end{proof}

\section{KPZ equation} 
\label{sec:KPZ}

Throughout this section we will treat the Cole-Hopf solution $\mathcal{H}$ of the KPZ equation, defined as $\mathcal{H}=\log \mathscr{Z}$ where $\mathscr{Z}$ is the mild solution of the stochastic heat equation with multiplicative noise \eqref{eq:SHE}. For a more details on mild solutions we consider below the reader can consult \cite{quastel_introduction_to_KPZ,Parekh2019}, where proofs of their existence uniqueness and positivity are provided.

\subsection{Convergence of Log Gamma polymers to stochastic heat equation with multiplicative noise} \label{subs:convergence_log_gamma_SHE}

Certain classes of models of directed polymer in random media are known to converge, under a weak noise scaling to interesting limits. In particular, in \cite{AlbertsKhaninQuastel2012}, authors introduced the so-called \emph{intermediate disorder regime}, under which the noise of a space-time uncorrelated random environment tends to zero as the length of a polymer diverge. It was there proved that, under certain assumptions, the point-to-point partition function of the polymer model converges, after proper scaling, to the solution of the stochastic heat equation with multiplicative noise. We recall this result below, adapting their general result to the particular case of the Log Gamma polymer model.

\begin{proposition}[\cite{AlbertsKhaninQuastel2012},Theorem 2.2] \label{prop:Alberts_khanin_quastel}
    Let $Z(T,N)$ be the partition function of the Log Gamma polymer with parameters $A_1=A_2=\cdots=B_1=B_2=\cdots=A$. Consider the scaling $A=\sqrt{N}$, and
    \begin{equation}
        \log Z( \frac{tN}{2} + xN^{1/2}, \frac{tN}{2}) = \log \mathscr{Z}_N(x,t) + \frac{t}{8} + \frac{x}{2} - \frac{t}{2}N \log N + \frac{1}{2}(t - x \log N) N^{1/2}.
    \end{equation}
    Then, as $N\to +\infty$, the process $\mathscr{Z}_N(x,t)$ converges in distribution to $\mathscr{Z}(x,t)$, the mild solution of the stochastic heat equation with multiplicative noise \eqref{eq:SHE} on $x \in \mathbb{R}$ and initial conditions $\mathscr{Z}(x,0)=\delta_0(x)$.
\end{proposition}

The intermediate disorder regime was considered also for polymer models in restricted geometries. Below we recall a result of \cite{Xuan_Wu_Intermediate_Disorder}, which provides the half space variant of \cref{prop:Alberts_khanin_quastel}.

\begin{proposition}[\cite{Xuan_Wu_Intermediate_Disorder}, Theorem 5.1] \label{prop:xuan_wu_SHE_hs}
    Let $Z^{\mathrm{hs}}(T,N)$ be the partition function of the homogeneous Log Gamma polymer in half space with parameters $A_1=A_2=\cdots=A$ and boundary strength $\varUpsilon$. Consider the scaling $A=\sqrt{N}$, $\varUpsilon = \omega + \frac{1}{2} \in \mathbb{R}$ and
    \begin{equation}
        \log Z^{\mathrm{hs}} ( \frac{t N}{2} + x \sqrt{N} , \frac{t N}{2}) = \log \mathscr{Z}_N^{\mathrm{hs}} (x,t) +\frac{t}{8} +\frac{x}{2} -\frac{t}{2} N \log N + \frac{1}{2}(t - x \log N) \sqrt{N}.
    \end{equation}
    Then, as $N\to + \infty$ the process $\mathscr{Z}_N(x,t)$ converges in distribution to $\mathscr{Z}(x,t)$ the mild solution of the stochastic heat equation with multiplicative noise \eqref{eq:SHE_hs} on $x \in \mathbb{R}_+$ and Robin boundary conditions $\left(\partial_x - \omega\right) \mathscr{Z}(x,t)\big|_{x=0} = 0 $ and initial conditions $\mathscr{Z}(x,0)=\delta_0(x)$.
\end{proposition}

\subsection{Proofs of \cref{thm:solution_mSHE} and of \cref{thm:solution_mSHE_hs}}

Taking appropriate limits of Fredholm determinant and pfaffian formulas for partition functions of Log Gamma polymer models given in \cref{thm:fredholm_det_log_gamma},  \cref{thm:fredholm_pfaff_log_gamma} and \cref{thm:fredholm_pfaff_log_gamma_gaussian} we prove here exact formulas for the Hopf-Cole solution of the KPZ equation.

\begin{proof}[Proof of \cref{thm:solution_mSHE}]
    We are going to consider the intermediate disorder scaling, prescribed by \cref{prop:Alberts_khanin_quastel}, in both sides of relation \eqref{eq:fredholm_det_log_gamma}. Setting
    \begin{equation} \label{eq:scaling_varsigma_SHE}
        \varsigma = -\frac{t}{2} N \log N + \frac{t}{2} \sqrt{N} + \frac{t}{12} +s,
    \end{equation}
    by \cref{prop:Alberts_khanin_quastel} the left hand side of \eqref{eq:fredholm_det_log_gamma} becomes the left hand side of \eqref{eq:fredholm_det_SHE}, that is
    \begin{equation} \label{eq:laplace_log_gamma_converges_laplace_SHE_hs}
        \lim_{N \to \infty} \mathbb{E} \left[ e^{ - e^{ -\varsigma + \log Z(N,N)}} \right] = \mathbb{E} \left[ e^{ - e^{ -s + \log \mathscr{Z}(0,t) +t/24}} \right].
    \end{equation}
    We now show that the Fredholm determinant in the right hand side of the kernel $\mathbf{K}$ converges to that of $\mathscr{K}_{\mathrm{mSHE}}$. First let us deform the integration contours in the integral expression \eqref{eq:K_LG} of $\mathbf{K}$ to be $C_{-1}^{3\pi/4}$ and $C_{1}^{\pi/4}$ respectively for the $Z$ and $W$ variable. By the Stirling's approximation
    \begin{equation} \label{eq:stirling_approx}
        \log \Gamma(z) = \left(z- \frac{1}{2} \right) \log z - z + \frac{1}{2} \log (2 \pi) + \frac{1}{12} \frac{1}{z} + O \left(\frac{1}{z^3} \right),
    \end{equation}
    which holds as long as $\arg\{z\}$ stays bounded away from $\pm \pi$, we can estimate the function $g_{A,A}$ appearing in the integral representation \eqref{eq:K_LG} as
    \begin{equation} \label{eq:log_gAA}
        \log g_{A,A} (Z) +Z \varsigma =  -\frac{t}{2} \frac{Z^3}{3} + Z s + O(ZN^{-1/2}) .
    \end{equation} 
    Always using the Stirling's approximation, this time applied to the polygamma function $\psi^{(1)}$, one can show that the function
    \begin{equation} \label{eq:Re_g_AA}
        \Re\left\{ \log g_{A,A} (Z) + Z \varsigma \right\}
    \end{equation}
    is concave along the half line $y \mapsto Z= -1 + y e^{\mathrm{i}3\pi/4} $ for $y>0$ and $N$ sufficiently large. For this compute
    \begin{equation}
    \begin{split}
        &\frac{\diff^2}{\diff y^2}\Re\left\{ \log g_{A,A} (Z) + Z \varsigma \right\}\big|_{Z=Z= -1 + y e^{\mathrm{i}3\pi/4}} 
        \\
        &=  N \Im \left\{ \psi^{(1)} (\sqrt{N} + Z  ) - \psi^{(1)}(\sqrt{N} - Z) \right\} \big|_{Z=Z= -1 + y e^{\mathrm{i}3\pi/4}}
        \\
        &\approx N\Im \left\{ \frac{1}{\sqrt{N} + Z} + \frac{1}{2(\sqrt{N}+Z)^2} - \frac{1}{\sqrt{N} - Z} - \frac{1}{2(\sqrt{N}-Z)^2}  \right\}\big|_{Z=Z= -1 + y e^{\mathrm{i}3\pi/4}},
    \end{split}
    \end{equation}
    then one can explictly evaluate the last line and confirm that for $N$ large enough it is a strictly negative function for all $y>0$.
    The explicit evaluation is not too much appealing and therefore we do not report it. This, along with estimate \eqref{eq:log_gAA} implies that there exists an exponentially decaying function $\tilde{g}$, independent of $N$, such that
    \begin{equation} \label{eq:bound_g_AA}
        \left| g_{A,A}(Z) e^{Z \varsigma } \right|\bigg|_{Z=Z= -1 + |y| e^{\mathrm{sign}(y)\mathrm{i}3\pi/4}} \le \tilde{g}(y),
    \end{equation}
    for all $y\in \mathbb{R}$. Similar arguments applied to terms $g_{A,A} (W)$ show that
    \begin{equation} \label{eq:K_log_gamma_converges_to_K_SHE}
        \lim_{N \to \infty}\mathbf{K}(X+ \varsigma ,Y+ \varsigma ) = \mathscr{K}_{\mathrm{mSHE}}(X +s,Y +s),
    \end{equation}
    by bounded convergence. Moreover the inequality
    \begin{equation}
        | \mathbf{K}(X+\varsigma,Y+\varsigma) | \le \mathrm{const.} e^{-(X+Y)},
    \end{equation}
    in which the constant is independent of $N$, combined with the limit \eqref{eq:K_log_gamma_converges_to_K_SHE} implies that
    \begin{equation}
        \lim_{N \to \infty} \det [1-\mathbf{K}]_{\mathbb{L}^2(\varsigma,+\infty)} = \det \left[1-\mathscr{K}_{\mathrm{mSHE}} \right]_{\mathbb{L}^2(s,+\infty)},
    \end{equation}
    again by bounded convergence. Notice that, so far we carried out the limit assuming that integration contours in the integral expression of $\mathscr{K}_{\mathrm{mSHE}}$ were $C_{-1}^{3\pi/4}$ and $C_{1}^{\pi/4}$ respectively for the $Z$ and $W$ variable. Once the convergence result is established we can deform such contours to be $\mathrm{i} \mathbb{R}-d$ and $\mathrm{i} \mathbb{R}+d$ as prescribed by \eqref{eq:K_SHE}. This concludes the proof.
\end{proof}

Next we prove the Fredholm pfaffian formula \eqref{eq:fredholm_pfaffian_SHE}. 

\begin{proof}[Proof of \cref{thm:solution_mSHE_hs}]
    We follow the same arguments presented in the proof of \cref{thm:solution_mSHE} above. Consider the Fredholm pfaffian formula for the half space polymer partition function \eqref{eq:fredholm_pfaff_log_gamma}. Under the scaling prescribed by \cref{prop:xuan_wu_SHE_hs} we have
    \begin{equation}
        \lim_{N\to \infty} \mathbb{E}\left[ e^{-e^{-\varsigma + \log Z^{\mathrm{hs}}(N,N)} } \right] = \mathbb{E}\left[ e^{-e^{-s+ \log \mathscr{Z}^{\mathrm{hs}}(N,N)+t/24} } \right],
    \end{equation}
    where $\varsigma$ is taken as in \eqref{eq:scaling_varsigma_SHE}. Moving to the right hand side of \eqref{eq:fredholm_pfaff_log_gamma}, using the Stirling's approximation we evaluate the function $g_{\varUpsilon;A}$ appearing in the integrand of function $\mathbf{k}$ as
    \begin{equation}
        g_{\varUpsilon;A} (Z) e^{Z \varsigma} = \frac{\Gamma (\frac{1}{2} + v + Z)}{\Gamma (\frac{1}{2} + v - Z)}  e^{ -\frac{t}{2} \frac{Z^3}{3} + Z s + O(Z N^{-1/2}) }.
    \end{equation} 
    Uniform bounds in $N$ on the absolute value of $g_{\varUpsilon;A} (Z) e^{Z \varsigma }$, for $Z\in C_{d}^{\pi/4}$ can be also derived as described around \eqref{eq:Re_g_AA}. This shows that
    \begin{equation}
        \lim_{N\to\infty} \mathbf{k}(X+ \varsigma ,Y+\varsigma ) = \mathscr{K}^\mathrm{hs}(X+s,Y+s)  
    \end{equation}
    and hence 
    \begin{equation} \label{eq:convergence_L_log_gamma_to_L_SHE}
        \lim_{N\to\infty} \mathbf{L}(X+\varsigma,Y+\varsigma) = \mathscr{L}_\mathrm{mSHE}(X+s,Y+s).
    \end{equation}
    Exponential bounds $\mathbf{k}(X+\varsigma,Y+\varsigma) = O\left( e^{-d(X+Y)} \right)$ now allow to prove the convergence of the Fredholm pfaffian
    \begin{equation}
        \lim_{N\to \infty}\Pf[J - \mathbf{L}]_{\mathbb{L}^2(\varsigma,+\infty)} = \Pf[J - \mathscr{L}_{\mathrm{mSHE}}]_{\mathbb{L}^2(s,+\infty)}.
    \end{equation}
\end{proof}

Next we provide an explicit expression for the probability distribution of the solution of the half space stochastic heat equation $\mathscr{Z}^\mathrm{hs}(0,t)$ for the boundary parameter $\omega \le-1/2$. To derive such formula we are simply going to take the scaling limit of results of \eqref{thm:fredholm_pfaff_log_gamma_gaussian}. Alternatively we could produce an analytic continuation, in the parameter $\omega$, of relation \eqref{eq:fredholm_pfaffian_SHE}. The two procedure, as one can observe from the statement of \cref{thm:SHE_gaussian}, yield the same result. We define the functions
\begin{equation} 
    \mathdutchcal{A}_j(X) = \frac{(-1)^j}{j!} e^{\frac{t}{2} \frac{(1/2 + \omega +j)^3}{3}-(1/2+\omega+j)X} \frac{\Gamma(1+2\omega+2j)}{\Gamma(1+2\omega+j) },
\end{equation}
\begin{equation} 
        \mathdutchcal{B}(Y) = \int_{\mathrm{i}\mathbb{R}+d} \frac{\diff W}{2\pi \mathrm{i}} e^{ \frac{t}{2} \frac{W^3}{3}-YW} \Gamma(2W)
        \frac{\Gamma(\frac{1}{2}-\omega-W)}{\Gamma( \frac{1}{2}-\omega +W)}.
\end{equation}

\begin{theorem} \label{thm:SHE_gaussian}
    Consider $\mathscr{Z}^{\mathrm{hs}}(x,t)$ the solution of the stochastic heat equation \eqref{eq:SHE_hs} in half space with initial conditions $\mathscr{Z}^{\mathrm{hs}}(x,0)=\delta_x$ and boundary parameter $\omega \le -1/2$. Let $m\in\mathbb{Z}_{\ge 0}$ and $\varepsilon\in \{0,1\}$ be such that $\frac{1}{2}+\omega \in (-m+\frac{1-\varepsilon}{2},-m+\frac{2-\varepsilon}{2}]$ and take a number $d$ satisfying
    \begin{equation}
        0 < d < \min\left\{ \frac{1}{2}, \frac{1}{2} + \omega + m \right\}.
    \end{equation}
    Then, there exists $s^*$ such that, for all $s>s^*$ we have
    \begin{equation} \label{eq:fredholm_pfaffian_SHE_gaussian}
    \begin{split}
        \mathbb{E}\left[ e^{-e^{-s + \log \mathscr{Z}^{\mathrm{hs}}(0,t) +t/24 }} \right] = & \Pf \left[ J - \widehat{\mathscr{L}}_{\mathrm{mSHE}} \right]_{\mathbb{L}^2(s,+\infty)}
        \\
        & \qquad
        \times \bigg( 1-  \langle (\mathdutchcal{B},-\mathdutchcal{B}') | \left(1-J^T \widehat{\mathscr{L}}_{\mathrm{mSHE}} \right)^{-1} | (\mathdutchcal{A}',\mathdutchcal{A} )\rangle \bigg),
    \end{split}
    \end{equation}
    where 
    \begin{equation} 
        \widehat{\mathscr{L}}_{\mathrm{mSHE}}(X,Y) = \mathscr{L}_\mathrm{mSHE}(X,Y) - \left( \begin{matrix} 1 & -\partial_Y \\ -\partial_X & \partial_X \partial_Y \end{matrix} \right) \sum_{j=m}^{2m+\varepsilon-2} \left[ \mathdutchcal{A}_j(X) \mathdutchcal{B}(Y) - \mathdutchcal{A}_j(Y) \mathdutchcal{B}(X) \right],
    \end{equation}
    \begin{equation} 
        \mathdutchcal{A}(X) = \sum_{j=0}^{2m+\varepsilon-2} \mathdutchcal{A}_j(X).
    \end{equation}
\end{theorem}

\begin{proof}
    As in the proof of \cref{thm:solution_mSHE_hs} we consider the scaling indicated in \cref{prop:xuan_wu_SHE_hs} and compute the limit of relation \eqref{eq:pfaff_L_hat_log_gamma}. Setting the parameter $\varsigma$ as in \eqref{eq:scaling_varsigma_SHE}, relation \eqref{eq:laplace_log_gamma_converges_laplace_SHE_hs}, stating the convergence of the Laplace transform of $Z^{\mathrm{hs}}(N,N)$ to that of $\mathscr{Z}^\mathrm{hs}(0,t)$, still holds, as this is independent of the value of the boundary parameter $v$. To evaluate the limit of the Fredholm pfaffian in the right hand side of \eqref{eq:pfaff_L_hat_log_gamma}, we recall that such expression is well defined only when $\| \widehat{\mathbf{L}} \| < 1$, condition that in \cref{thm:fredholm_pfaff_log_gamma_gaussian} can be achieved taking $N$ sufficiently large, while leaving $\varsigma$ fixed. Given the scaling of the parameter $\varsigma$ we have adopted, we can no longer control the norm of $\widehat{\mathbf{L}}$ in such way and for this reason we are now forced to take $s$ to be sufficiently large, but still fixed.  As already explained in the proof of \cref{thm:solution_mSHE} and \cref{thm:solution_mSHE_hs} above in the subsection the function
    \begin{equation}
        g_{\varUpsilon;A}(Z) e^{Z \varsigma }
    \end{equation}
    is absolutely bounded for $Z\in C_{-d}^{3\pi/4}$ by a summable function on the same curve and it converges, for $N$ large, to $e^{-\frac{t}{2} \frac{Z^3}{3}+ s Z}$. We can refine the estimate \eqref{eq:bound_g_AA} as
    \begin{equation}
        \left| g_{\varUpsilon,A}(Z) e^{Z \varsigma } \right|\bigg|_{Z=Z= -d+ |y| e^{\mathrm{sign}(y)\mathrm{i}3\pi/4}} \le e^{-s d} \tilde{g}(y),
    \end{equation}
    where $\tilde{g} \in \mathbb{L}^1(\mathbb{R})$ and produce the bound
    \begin{equation}
        \widehat{\mathbf{L}}(X,Y) = O( e^{-s d} e^{-d(X+Y)}),
    \end{equation}
    which holds independently of $N$. It is now clear that we can fix an $s^*$ large enough so that, for any $s>s^*$ the norm of $\widehat{\mathbf{L}}$ can be made arbitrarily small uniformly in $N$.
    Similar arguments allow to produce bounds
    \begin{equation}
        |\mathbf{A}(X+\varsigma)|,|\mathbf{A}'(X+\varsigma)|\le \mathrm{const.}e^{ -(1/2+\omega) X},
        \qquad
        |\mathbf{B}(X+\varsigma)|,|\mathbf{B}'(X+\varsigma)|< \mathrm{const.} e^{-d X},
    \end{equation}
    which also hold uniformly in $N$. Thanks to this observation and to arguments detailed in the proof of \cref{thm:fredholm_pfaff_log_gamma_gaussian}, the expression in the right hand side of \eqref{eq:pfaff_L_hat_log_gamma} is well defined for any $N$. We can now proceed to evaluate the limit of functions $\mathbf{A}_j, \mathbf{B}$ and we obtain
    \begin{equation}
        \lim_{N\to \infty} \mathbf{A}_j(X+\varsigma) = \mathdutchcal{A}_j(X+s),
        \qquad
        \lim_{N\to \infty} \mathbf{B}(X+\varsigma) = \mathdutchcal{B}(X+s),
    \end{equation}
    which combined with \eqref{eq:convergence_L_log_gamma_to_L_SHE} imply, by bounded convergence, that
    \begin{multline*}
        \lim_{N \to \infty}
        \Pf \left[ J - \widehat{\mathbf{L}} \right]_{\mathbb{L}^2(\varsigma,+\infty)}
        \bigg( 1-  \langle (\mathbf{B},-\mathbf{B}') | \left(1-J^T \widehat{\mathbf{L}}\right)^{-1} | (\mathbf{A}',\mathbf{A} )\rangle \bigg),
        \\
        =
        \Pf \left[ J - \widehat{\mathscr{L}}_{\mathrm{mSHE}} \right]_{\mathbb{L}^2(s,+\infty)} \bigg( 1-  \langle (\mathdutchcal{B},-\mathdutchcal{B}') | \left(1-J^T \widehat{\mathscr{L}}_{\mathrm{mSHE}} \right)^{-1} | (\mathdutchcal{A}',\mathdutchcal{A} )\rangle \bigg),
    \end{multline*}
    completing the proof.
\end{proof}

\subsection{Alternative representation of kernel and matching with Krejenbrink-Le Doussal kernel} \label{subs:krejenbrink_le_doussal}

Here we recover, through simple manipulations of the matrix kernel $\mathscr{L}_{\mathrm{mSHE}}$, a Fredholm pfaffian formula for the solution of the mSHE in half space $\mathscr{Z}^\mathrm{hs}(0,t)$, previously found by Krejenbrink and Le Doussal in \cite{krejenbrink_le_doussal_KPZ_half_space}. In this last work, authors derived their expression for the Laplace transform $\mathbb{E} \left[ e^{ - e^{ -s  +\log \mathscr{Z}^{\mathrm{hs}}(0,t) + t/24  }} \right]$  via Bethe ansatz, although the procedure involved the summation over moments of $\mathscr{Z}^\mathrm{hs}(0,t)$, only a finite number of which are finite. Nevertheless, the solution proposed in \cite{krejenbrink_le_doussal_KPZ_half_space}, although not mathematically justified, turns out to be correct as we confirm in \cref{cor:krejenbrink_le_doussal} below.

\medskip

We introduce the function
\begin{equation}\label{eq:fstheta}
    \mathscr{F}_s(\theta) = \frac{1}{1+e^{-\theta +s}}
\end{equation}
and the $2 \times 2$ matrix kernel
\begin{equation}
    \mathscr{M} (\theta,\varphi) =\left( \begin{matrix} \mathscr{M}_{1,1}(\theta,\varphi) & \mathscr{M}_{1,2}(\theta,\varphi)
    \\
    - \mathscr{M}_{1,2}(\varphi, \theta) & \mathscr{M}_{2,2}(\theta,\varphi)
    \end{matrix} \right).
\end{equation}
Its coefficients are given by
    \begin{equation}
    \begin{split}
        \mathscr{M}_{1,1}(\theta,\varphi) = \int_{\mathrm{i}\mathbb{R}+d} \frac{\diff Z}{2\pi \mathrm{i}} \int_{\mathrm{i}\mathbb{R}+d} \frac{\diff W}{2\pi \mathrm{i}} & e^{\frac{t}{2}\left( \frac{Z^3}{3} + \frac{W^3}{3} \right) -\theta Z - \varphi W } \frac{Z-W}{Z+W} \Gamma(2Z) \Gamma(2W)
        \\
        & \times \frac{\Gamma(\frac{1}{2}+\omega-Z)}{\Gamma(\frac{1}{2}+\omega+Z) }\frac{ \Gamma(\frac{1}{2}+\omega-W)}{ \Gamma(\frac{1}{2}+\omega+W)} \frac{\sin(\pi Z) }{\pi} \frac{\sin(\pi W) }{\pi},
    \end{split}
    \end{equation}
    \begin{equation}
    \begin{split}
        \mathscr{M}_{1,2}(\theta,\varphi) = \int_{\mathrm{i}\mathbb{R}+d} \frac{\diff Z}{2\pi \mathrm{i}} \int_{\mathrm{i}\mathbb{R}+d} \frac{\diff W}{2\pi \mathrm{i}} & e^{\frac{t}{2}\left( \frac{Z^3}{3} + \frac{W^3}{3} \right) - \theta Z -\varphi W } \frac{Z-W}{Z+W} \Gamma(2Z) \Gamma(2W)
        \\
        & \times \frac{\Gamma(\frac{1}{2}+\omega-Z)}{\Gamma(\frac{1}{2}+\omega+Z) }\frac{ \Gamma(\frac{1}{2}+\omega-W)}{ \Gamma(\frac{1}{2}+\omega+W)} \frac{\sin(\pi Z) }{\pi} \cos(\pi W),
    \end{split}
    \end{equation}
    \begin{equation}
    \begin{split}
        \mathscr{M}_{2,2}(\theta,\varphi) = \int_{\mathrm{i}\mathbb{R}+d} \frac{\diff Z}{2\pi \mathrm{i}} \int_{\mathrm{i}\mathbb{R}+d} \frac{\diff W}{2\pi \mathrm{i}} & e^{\frac{t}{2}\left( \frac{Z^3}{3} + \frac{W^3}{3} \right) - \theta Z -\varphi W } \frac{Z-W}{Z+W} \Gamma(2Z) \Gamma(2W)
        \\
        & \times \frac{\Gamma(\frac{1}{2}+\omega-Z)}{\Gamma(\frac{1}{2}+\omega+Z) }\frac{ \Gamma(\frac{1}{2}+\omega-W)}{ \Gamma(\frac{1}{2}+\omega+W)} \cos(\pi Z) \cos(\pi W).
    \end{split}
    \end{equation}

We are now ready to match our result of \cref{thm:solution_mSHE_hs} with Krejenbrink and Le Doussal's formulas \cite[eq. (6), (7)]{krejenbrink_le_doussal_KPZ_half_space}.
    
\begin{corollary} \label{cor:krejenbrink_le_doussal}
    Under the hypothesis of \cref{thm:solution_mSHE_hs}, we have
    \begin{equation} \label{eq:krejenbrink_le_doussal}
        \mathbb{E} \left[ e^{ - e^{ -s  +\log \mathscr{Z}^{\mathrm{hs}}(0,t) + t/24  }} \right] =
        \Pf \left[J-\mathscr{F}_s \cdot \mathscr{M} \right]_{\mathbb{L}^2(\mathbb{R})}.
    \end{equation}
\end{corollary}

\begin{proof}
    One can adapt the proof of \cref{prop:q_krejenbrink_le_doussal} and write 
    \begin{equation}
        J \mathscr{L}_{\mathrm{mSHE}}(X+s,Y+s) = \int_{-\infty}^{\infty} \mathscr{F}_s(\theta) A(X,\theta) B(\theta,Y) \diff \theta,
    \end{equation}
    for some explicit operators $A(X,\theta), B(\theta,Y)$, which for convenience we do not report. To do so, rather than \eqref{eq:ramanujan_psi}, \eqref{eq:addition_formula_sinq} one needs to use the more canonical notable formulas
    \begin{equation}
        \sin[\pi(Z-W)] = \sin(\pi Z) \cos(\pi W) - \cos(\pi Z) \sin(\pi W)
    \end{equation}
    and
    \begin{equation}
        \frac{e^{-s(Z+W)}}{\sin[\pi (Z+W)]} = \frac{1}{\pi} \int_{-\infty}^{\infty} \frac{e^{-(Z+W) \theta}}{1+e^{-\theta+s}} \diff \theta,
    \end{equation}
    where the second one is well posed only for $\Re\{Z+W\} \in (0,1)$. Then \eqref{eq:krejenbrink_le_doussal} follows from the symmetry $\Pf[J-J^TA \mathscr{F}_s B] = \Pf[J- \mathscr{F}_sBAJ^T]$ of Fredholm pfaffians. We do not report the details of the argument any further.

\end{proof}

An immediate corollary of relation \eqref{eq:krejenbrink_le_doussal} is that the Laplace transform of the density function of $\mathscr{Z}^\mathrm{hs}(0,t)$ is equal to the expectation of a multiplicative function of a pfaffian point process with correlation kernel $\mathscr{M}$. This had been observed in \cite[eq. (8)]{krejenbrink_le_doussal_KPZ_half_space} and we report the same result below for completeness.

\begin{corollary}
    Under the hypothesis of \cref{thm:solution_mSHE_hs} we have
    \begin{equation}
        \mathbb{E}\left[ e^{-e^{-s+\log \mathscr{Z}^{\mathrm{hs}}(0,t) +t/24}} \right] = \mathbb{E}_{\mathscr{M}} \left[ \prod_{i=1}^\infty \frac{1}{1+e^{-s+\mathfrak{a}_i}} \right],
    \end{equation}
    where in the right hand side $\{\mathfrak{a}_i:i\in \mathbb{Z}_{\ge 1} \}$ is a pfaffian point process with correlation kernel $\mathscr{M}$.
\end{corollary}

\begin{proof}
    This is a straightforward consequence of \eqref{eq:krejenbrink_le_doussal} from the basic property \eqref{eq:multiplicative_expectation_pfaff_general} of pfaffian point processes.
\end{proof}
In \cite{krejenbrink_le_doussal_KPZ_half_space} authors discovered how to rewrite the Fredholm pfaffian of the $2\times 2$ matrix kernel $\mathscr{M}$ as a Fredholm determinant of a scalar kernel. We define 
\begin{equation}
        \mathscr{D}_s(x,y) = 2 \partial_x \int_{-\infty}^\infty \diff \theta \, \mathscr{F}_s(\theta) \left[ \mathsf{f}_{\mathrm{even}}(x+\theta) \mathsf{f}_{\mathrm{odd}}(y+\theta) - \mathsf{f}_{\mathrm{odd}}(x+\theta) \mathsf{f}_{\mathrm{even}}(y+\theta) \right],
    \end{equation}
    where
    \begin{equation}
        \mathsf{f}_{\mathrm{odd}}(\theta) = \int_{ \mathrm{i} \mathbb{R} + d } \frac{\diff Z}{2 \pi \mathrm{i}} e^{\frac{t}{2}\frac{Z^3}{3} - \theta Z} \frac{\Gamma(\frac{1}{2} + \omega -Z)}{\Gamma(\frac{1}{2} + \omega +Z)} \Gamma(2Z) \frac{\sin(\pi Z)}{\pi},
    \end{equation}
    \begin{equation}
        \mathsf{f}_{\mathrm{even}}(\theta) = \int_{\mathrm{i} \mathbb{R} + d} \frac{\diff Z}{2 \pi \mathrm{i}} e^{\frac{t}{2}\frac{Z^3}{3} - \theta Z} \frac{\Gamma(\frac{1}{2} + \omega -Z)}{\Gamma(\frac{1}{2} + \omega +Z)} \Gamma(2Z) \cos(\pi Z).
    \end{equation}
    \begin{corollary}
        Under the assumptions of \cref{thm:solution_mSHE_hs} we have
        \begin{equation}
            \mathbb{E}\left[ e^{-e^{ -s +\log \mathscr{Z}^{\mathrm{hs}}(0,t) +t/24}} \right] = \sqrt{\det[1-\mathscr{D}_s]_{\mathbb{L}^2(\mathbb{R})}}.
        \end{equation}
    \end{corollary}
    \begin{proof}
        It was shown in \cite[Proposition B.2]{krejenbrink_le_doussal_KPZ_half_space}, that
        \begin{equation}
            \Pf[J-\mathscr{F}_s \cdot \mathscr{M}]_{\mathbb{L}^2(\mathbb{R})} = \sqrt{\det[1-\mathscr{D}_s]_{\mathbb{L}^2(\mathbb{R})}},
        \end{equation}    
        so that \cref{cor:krejenbrink_le_doussal} yields the proof.
    \end{proof}

\subsection{Asymptotic fluctuations of $\mathscr{Z}^\mathrm{hs}(0,t)$ and Baik-Rains phase transition: proof of \cref{thm:asymptotics_SHE}} \label{subs:hsSHE_Baik_Rains}

In this section we consider large $t$ asymptotics of our Fredholm pfaffian expressions \eqref{eq:fredholm_pfaffian_SHE}, \eqref{eq:fredholm_pfaffian_SHE_gaussian} for the distribution of $\mathscr{Z}^\mathrm{hs}(0,t)$. Such asymptotic analysis is fundamentally equivalent to those considered in \cref{subs:baik_rains_positive} and therefore in the discussion below we will at times omit the details of our arguments.

\begin{proof}[Proof of \cref{thm:asymptotics_SHE}] As usual we will proceed case by case, depending on the value of $\omega$.

\medskip
\emph{$\omega>-\frac{1}{2}$ : convergence to crossover distribution.} Consider the exact formula \eqref{eq:fredholm_pfaffian_SHE} and set
\begin{equation}
    s=  r \left( \frac{t}{2} \right)^{1/3} .
\end{equation}
Then, using relation \eqref{eq:Laplace_Gumbel} we have
\begin{equation}
    \lim_{t \to \infty} \mathbb{E} \left[ e^{ - e^{ -s  +\log \mathscr{Z}^{\mathrm{hs}}(0,t) + t/24  }} \right] = \lim_{t \to +\infty} \mathbb{P} \left[ \frac{\log \mathscr{Z}^{\mathrm{hs}}(0,t) +  t/24}{ 2^{-1/3} t^{1/3}} \le r \right].
\end{equation}
We now move to the right hand side of \eqref{eq:fredholm_pfaffian_SHE}. After a change of variables
\begin{equation} \label{eq:change_of_variables_SHE_asymptotics}
    X = u \left( \frac{t}{2} \right)^{1/3} , \qquad Y = v  \left( \frac{t}{2} \right)^{1/3} , \qquad Z =  \alpha \left( \frac{2}{t} \right)^{1/3} , \qquad W = \beta  \left( \frac{2}{t} \right)^{1/3} 
\end{equation}
and a repeated use of approximations $\Gamma(z) \approx \frac{1}{z}$, $\sin z \approx z$
we compute the scaling limit of function $\mathscr{K}^\mathrm{hs}$ as
\begin{equation} 
    \begin{split}
     \lim_{t\to \infty} \mathscr{K}^\mathrm{hs}(X(u),Y(v)) = \mathpzc{k}^{(\xi)} (u,v),
    \end{split}
\end{equation}
where $\mathpzc{k}^{(\xi)}$ is defined in \eqref{eq:k_xi}. Notice that the integration contours were rescaled too after the change of variables and they have been set to $C^{\pi/3}_\delta$. Analogous limits for derivatives of $\mathscr{K}^\mathrm{hs}$ can be established and we have
\begin{equation}
    \lim_{t \to \infty} \left( \begin{matrix} 1 & 0 \\ 0 & \left( \frac{2}{t} \right)^{1/3} \end{matrix} \right) \mathscr{L}_{\mathrm{mSHE}}(X(u),Y(v)) \left( \begin{matrix} 1 & 0 \\ 0 & \left( \frac{2}{t} \right)^{1/3} \end{matrix} \right) =  \mathcal{K}^{(\xi)}_{\mathrm{cross}}(u,v).
\end{equation}
Exponential bounds, in the variables $u,v$, for the rescaled kernel $\mathscr{L}_{\mathrm{mSHE}}(X(u),Y(v))$ now allow to conclude that
\begin{equation}
    \lim_{t \to \infty} \Pf\left[ J - \mathscr{L}_\mathrm{mSHE} \right]_{\mathbb{L}^2(s,+\infty)} = \Pf\left[ J - \mathcal{K}_\mathrm{cross}^{(\xi)} \right]_{\mathbb{L}^2(r,+\infty)}.
\end{equation}

\medskip
\emph{$\omega=-\frac{1}{2}$ : convergence to GOE Tracy-Widom distribution.} When $\omega$ tends to $-1/2$ we can no longer choose, in the integral expression of $\mathscr{K}^\mathrm{hs}(X,Y)$, contours $C^{\pi/3}_d$ with $0<d<\omega+1/2$. Let $\mathscr{K}^\mathrm{hs}_\mathrm{ext}(X,Y)$ be the analytic extension of $\mathscr{K}^\mathrm{hs}(X,Y)$ in the parameter $\omega$. Then it is clear that
\begin{equation}
    \mathscr{K}^\mathrm{hs}_\mathrm{ext}(X,Y)\big|_{\omega=-\frac{1}{2}} = \mathscr{K}^\mathrm{hs}(X,Y)\big|_{\omega=-\frac{1}{2}} - ( \mathdutchcal{B}(Y) - \mathdutchcal{B}(X) )\big|_{\omega=-\frac{1}{2}}
\end{equation}
and hence
\begin{equation}
    \left(\mathscr{L}_{\mathrm{mSHE}}\right)_\mathrm{ext}(X,Y)\big|_{\omega=-\frac{1}{2}} = \mathscr{L}_{\mathrm{mSHE}}(X,Y)\big|_{\omega=-\frac{1}{2}} - \left( \begin{matrix} 1 & - \partial_Y \\ -\partial_X & \partial_X \partial_Y \end{matrix} \right) (\mathdutchcal{B}(Y) -\mathdutchcal{B}(X) )\big|_{\omega=-\frac{1}{2}},
\end{equation}
where, naturally $\left(\mathscr{L}_{\mathrm{mSHE}}\right)_\mathrm{ext}$ denotes the analytic extension of $\mathscr{L}_{\mathrm{mSHE}}$. After the change of variables \eqref{eq:change_of_variables_SHE_asymptotics} it is simple to compute the limit
\begin{equation}
    \lim_{t\to \infty} \mathscr{K}^\mathrm{hs}_\mathrm{ext}(X(u),Y(v))\big|_{\omega=-\frac{1}{2}} = \mathpzc{k}^{(\infty)}(u,v) - \mathpzc{B}(v) + \mathpzc{B}(u)
\end{equation}
and similarly
\begin{equation}
    \lim_{t\to \infty} \left( \begin{matrix} 1 & 0 \\ 0 & \left( \frac{2}{t} \right)^{1/3} \end{matrix} \right) \left(\mathscr{L}_{\mathrm{mSHE}}\right)_\mathrm{ext}(X(u),Y(v))\big|_{\omega=-\frac{1}{2}} \left( \begin{matrix} 1 & 0 \\ 0 & \left( \frac{2}{t} \right)^{1/3} \end{matrix} \right) = \mathcal{K}_{\mathrm{GOE}}(u,v).
\end{equation}
Recall that $\mathcal{K}_\mathrm{GOE}, \mathpzc{k}^{(\infty)}$ and $\mathpzc{B}$ were given in \cref{def:GOE}, \eqref{eq:k_infty} and \eqref{eq:q_Airy}. Although the kernel $\left(\mathscr{L}_{\mathrm{mSHE}}\right)_\mathrm{ext}(X(u),Y(v))\big|_{\omega=-\frac{1}{2}}$ does not define a bounded operator on $\mathbb{L}^2(s,+\infty)\times \mathbb{L}^2(s,+\infty)$, its Fredholm pfaffian defines an absolutely convergent series, uniformly in $u,v>r$. This can be verified as in the proof of the case $\gamma=1$ of \cref{thm:asymptotic_pfaffian_schur}. Due to the exponential decay of terms of functions $\mathscr{K}^\mathrm{hs}(X(u),Y(v))\big|_{\omega=-\frac{1}{2}}, \mathdutchcal{B}(X(u)) \big|_{\omega=-\frac{1}{2}}$ we are able to pass the limit symbol inside the Fredholm pfaffian and establish the limit
\begin{equation}
    \lim_{t \to \infty} \Pf\left[ J - \left(\mathscr{L}_\mathrm{mSHE}\right)_\mathrm{ext} \big|_{\omega=-\frac{1}{2}} \right]_{\mathbb{L}^2(s,+\infty)} = \Pf\left[ J - \mathcal{K}_\mathrm{GOE} \right]_{\mathbb{L}^2(r,+\infty)},
\end{equation}
which proves \eqref{eq:SHE_crossover} for $\omega=-\frac{1}{2}$.

\medskip
\emph{$\omega<-\frac{1}{2}$ : convergence to gaussian distribution.} In this case we consider the exact formula \eqref{eq:fredholm_pfaffian_SHE_gaussian} and we set
\begin{equation}
    s =  \frac{t}{2} \left( \frac{1}{2} + \omega \right)^2 + \sigma_\omega t^{1/2} r.
\end{equation}
Notice that $ t f_\omega = \frac{t}{24} + \frac{t}{2} \left( \frac{1}{2} + \omega \right)^2$. Taking the limit of the Laplace transform we get
\begin{equation}
    \lim_{t \to + \infty} \mathbb{E} \left[ e^{ - e^{ -s  +\log \mathscr{Z}^{\mathrm{hs}}(0,t) + t/24  }} \right] = \lim_{t \to +\infty} \mathbb{P} \left[ \frac{\log \mathscr{Z}^{\mathrm{hs}}(0,t) +  tf_\omega }{ \sigma_\omega t^{1/2}} \le r \right].
\end{equation}
We move to the asymptotic limit of the right hand side of \eqref{eq:fredholm_pfaffian_SHE_gaussian}. For the analysis that follows we define the function
\begin{equation}
    h_\omega(Z) = \frac{Z^3}{3} - \left( \frac{1}{2} + \omega \right)^2 Z.
\end{equation}
Let us first evaluate the limit of the matrix kernel $\widehat{\mathscr{L}}_\mathrm{mSHE}$ which we write as
\begin{equation} \label{eq:L_hat_SHE}
    \widehat{\mathscr{L}}_\mathrm{mSHE}(X,Y)= \left( \begin{matrix} 1 & - \partial_Y \\ - \partial_X & \partial_X \partial_Y \end{matrix} \right) \left(  \widehat{\mathscr{L}}_\mathrm{mSHE}\right)_{1,1}(X,Y),
\end{equation}
where
\begin{equation} \label{eq:L_hat_SHE_11}
    \left(  \widehat{\mathscr{L}}_\mathrm{mSHE}\right)_{1,1}(X,Y) = \mathscr{K}^\mathrm{hs}_{\delta-\frac{1}{2}-\omega} (X,Y) + \sum_{j=1}^{2m+\varepsilon-2} \mathscr{S}_j(X,Y).
\end{equation}
In the previous expression $\mathscr{K}^\mathrm{hs}_{\delta-\frac{1}{2}-\omega}$ denotes the function $\mathscr{K}^\mathrm{hs}$ of \eqref{eq:k_she_hs} where the integration contours for the $Z$ and $W$ variables are taken to be $\mathrm{i}\mathbb{R}+\delta-\frac{1}{2} -\omega $, while the $\mathscr{S}_j$ are the residues at the $Z$ pole at $-W+j$ of the $Z$-integrand of the $\mathscr{K}^\mathrm{hs}$. Here we are assuming that $\delta>0$ is sufficiently small, but fixed. For more details on this identity see the passage around \eqref{eq:L_11_2}, where an analogous formula was derived in the study of the half space Log Gamma polymer partition function. After a change of variables
\begin{equation}
    X=\frac{t}{2} \left( \frac{1}{2} + \omega \right)^2 + \sigma_\omega t^{1/2} u, \qquad Y=\frac{t}{2} \left( \frac{1}{2} + \omega \right)^2 + \sigma_\omega t^{1/2} v,
\end{equation}
where we take $u,v>r$, we can write the expression for functions $\mathscr{S}_j$
\begin{equation}
\begin{split}
    \mathscr{S}_j(X(u),Y(v)) = (-1)^{j+1} \int_{\mathrm{i} \mathbb{R} + R} & \frac{\diff W}{2 \pi \mathrm{i}} e^{\frac{t}{2} \left( h_\omega(W) + h_\omega(j-W) \right) -\sigma_\omega t^{1/2} (Wv + (j-W)u) } \\ & \times \frac{\Gamma(\frac{1}{2}+\omega+W-j)}{\Gamma(\frac{1}{2}+\omega-W+j)}\frac{\Gamma(\frac{1}{2}+\omega-W)}{\Gamma(\frac{1}{2}+\omega+W)}\frac{\Gamma(2j-2W)}{\Gamma(1-2W)},
\end{split}
\end{equation}
where $R$ in the integration contours can be taken in the interval $(\frac{1}{2}+\omega-1,-\frac{1}{2}-\omega-j+1)$. We obtain the estitmate
\begin{equation} \label{eq:bound_scr_S_j}
    \left| \mathscr{S}_j(X(u),Y(v)) \right| \le \mathrm{const.} e^{\frac{t}{2}\left[ h_\omega(R) + h_\omega(j-R) \right] -\sigma_\omega t^{1/2} \left[ Rv + (j-R)u \right]},
\end{equation}
where the constant can be made independent of $t,R$. Similarly we have
\begin{equation}
\begin{split}
    \mathscr{K}^\mathrm{hs}_{-\frac{1}{2}-\omega+\delta} (X(u),Y(v)) = &
             \int_{\mathrm{i}\mathbb{R} +\delta-\frac{1}{2}-\omega} \frac{\diff Z}{2 \pi \mathrm{i}} \int_{\mathrm{i}\mathbb{R}+\delta-\frac{1}{2}-\omega}
            \frac{\diff W}{2 \pi \mathrm{i}}
            e^{\frac{t}{2} \left[h_\omega(Z) + h_\omega(W) \right] - \sigma_\omega t^{1/2} (Zu - Wv)}
            \\
            &
            \times
            \frac{\Gamma(\frac{1}{2} + \omega - Z)}{\Gamma (\frac{1}{2} + \omega+Z)} \frac{\Gamma(\frac{1}{2} + \omega -W)}{\Gamma (\frac{1}{2} + \omega+W)} \Gamma(2Z) \Gamma(2W) \frac{\sin [\pi (Z-W)]}{\sin [\pi (Z+W)]},
    \end{split}
\end{equation}
from where we get
\begin{equation} \label{eq:bound_scr_K}
    \left| \mathscr{K}^\mathrm{hs}_{-\frac{1}{2}-\omega+\delta} (X(u),Y(v)) \right| \le \mathrm{const.} e^{t \, h_\omega\left( \delta - \frac{1}{2} - \omega \right) - \sigma_\omega t^{1/2} \left(\delta-\frac{1}{2}-\omega \right) (u + v)}.
\end{equation}
Combining estimates \eqref{eq:bound_scr_S_j}, where we might set $R=\frac{j}{2}$, \eqref{eq:bound_scr_K} with \eqref{eq:L_hat_SHE}, \eqref{eq:L_hat_SHE_11} we see that
\begin{equation}
    \lim_{t\to \infty} \Pf\left[J-\widehat{\mathscr{L}}_{\mathrm{mSHE}} \right]_{\mathbb{L}( s ,+\infty)} = 1,
\end{equation}
since $h_\omega(\delta - \frac{1}{2}-\omega),h_\omega(\frac{j}{2})<0$, respectively for $\delta>0$ small and $1\le j \le 2\left( \frac{1}{2} +\omega \right)^2$ and all terms of kernel $\widehat{\mathscr{L}}_{\mathrm{mSHE}}$ are exponentially decaying in $t$.

We also write the scaling form of functions $\mathdutchcal{A}_j$ as
\begin{equation}
    \mathdutchcal{A}_j(X(u)) = \frac{(-1)^j}{j!} e^{\frac{t}{2} h(\frac{1}{2}+\omega+j)-\sigma_v t^{1/2}(1/2+\omega+j)u} \frac{\Gamma(1+2\omega+2j)}{\Gamma(1+2\omega+j) }
\end{equation}
and using saddle point method we compute the large $t$ limit of function $\mathdutchcal{B}$, to find
\begin{equation}
\begin{split}
    \mathdutchcal{B}(X(u)) &= \int_{\mathrm{i} \mathbb{R}-\frac{1}{2}-\omega} \frac{\diff W}{2 \pi \mathrm{i}} e^{\frac{t}{2}h_\omega(W)-\sigma_\omega t^{1/2} uW} \Gamma(2W) \frac{\Gamma(\frac{1}{2}-\omega-W)}{\Gamma(\frac{1}{2}-\omega+W)}
    \\
    &
    = \frac{1}{\sigma_\omega t^{1/2}} e^{\frac{t}{2}h_\omega(-\frac{1}{2}-\omega)-\sigma_\omega t^{1/2} (-\frac{1}{2}-\omega)u} \frac{1}{-1-2\omega} \frac{e^{-u^2/2}}{\sqrt{2 \pi}} (1+o(1)).
\end{split}
\end{equation}
Since the operator $\widehat{\mathscr{L}}_\mathrm{mSHE}$ is arbitrarily small in norm we can expand the resolvent $(1-J^T \widehat{\mathscr{L}}_\mathrm{mSHE})$ in Neumann series and the scalar product in the right hand side of \eqref{eq:fredholm_pfaffian_SHE_gaussian} becomes
\begin{equation}
\begin{split}
    \langle (\mathdutchcal{B},-\mathdutchcal{B}') | \left(1-J^T \widehat{\mathscr{L}}_{\mathrm{mSHE}} \right)^{-1} | (\mathdutchcal{A}',\mathdutchcal{A} )\rangle &= \sum_{\ell \ge 0} \langle (\mathdutchcal{B},-\mathdutchcal{B}') | \left(J^T \widehat{\mathscr{L}}_{\mathrm{mSHE}} \right)^\ell| (\mathdutchcal{A}',\mathdutchcal{A} )\rangle.
\end{split}
\end{equation}
As observed several times in different places in the text, functions $\mathdutchcal{A}_j$ are divergent, yet by means of estimates \eqref{eq:bound_scr_S_j}, \eqref{eq:bound_scr_K} we can show that all scalar products $\langle (\mathdutchcal{B},-\mathdutchcal{B}') | \left(J^T \widehat{\mathscr{L}}_{\mathrm{mSHE}} \right)^\ell| (\mathdutchcal{A}',\mathdutchcal{A} )\rangle$ are numerically convergent and they are of order $O(e^{-ct \ell})$, for some constant $c>0$. More details on this argument can be found in the proof of the case $\varUpsilon<0$ of \cref{thm:asymptotic_hs_log_gamma} above. We obtain
\begin{equation}
    \sum_{\ell \ge 0} \langle (\mathdutchcal{B},-\mathdutchcal{B}') | \left(J^T \widehat{\mathscr{L}}_{\mathrm{mSHE}} \right)^\ell| (\mathdutchcal{A}',\mathdutchcal{A} )\rangle
    = \langle (\mathdutchcal{B},-\mathdutchcal{B}') | (\mathdutchcal{A}',\mathdutchcal{A} )\rangle + o(1)
    = \int_r^\infty \frac{e^{-u^2/2}}{\sqrt{2 \pi}} \diff u + o(1).
\end{equation}
Finally we evaluate the right hand side of \eqref{eq:fredholm_pfaffian_SHE_gaussian} to find
\begin{multline*}
    \lim_{t\to \infty}
    \Pf \left[ J - \widehat{\mathscr{L}}_{\mathrm{mSHE}} \right]_{\mathbb{L}^2(s,+\infty)} \bigg( 1-  \langle (\mathdutchcal{B},-\mathdutchcal{B}') | \left(1-J^T \widehat{\mathscr{L}}_{\mathrm{mSHE}} \right)^{-1} | (\mathdutchcal{A}',\mathdutchcal{A} )\rangle \bigg) \\ = \int^r_{-\infty} \frac{e^{-u^2/2}}{\sqrt{2 \pi}} \diff u,
\end{multline*}
which concludes the proof.
\end{proof}

\appendix

\section{Fredholm determinant and Fredholm pfaffian} \label{app:Fredholm}

Given a kernel $K(x,y) \in \mathbb{L}^2(\Omega \times \Omega, \nu \otimes \nu) $ we define its \emph{Fredholm determinant} as
    \begin{equation}
        \det(1-K)_{\mathbb{L}^2(\Omega)} = 1 + \sum_{\ell \ge 1} \frac{(-1)^\ell}{\ell!} \int_{\Omega^\ell} \det [K(x_i,x_j)]_{i,j=1}^\ell \nu(\diff x_1) \cdots \nu(\diff x_\ell),
    \end{equation}
    whenever the right hand side is convergent.
    
    Consider a $2 \times 2$ matrix kernel 
    \begin{equation}
        k(x,y) = \left( \begin{matrix} k_{1,1}(x,y) & k_{1,2}(x,y) \\ -k_{1,2}(y,x) & k_{2,2} (x,y) \end{matrix} \right),
    \end{equation}
    where $k_{i,j} \in \mathbb{L}^2(\Omega \times \Omega, \nu \otimes \nu)$ and moreover $k_{i,i}(y,x) = - k_{i,i}(x,y)$ for all $x,y\in \Omega$ and $i=1,2$. The \emph{Fredholm pfaffian} of $k$ is defined as
    \begin{equation}
        \Pf(J-k)_{\mathbb{L}^2(\Omega)} = 1 + \sum_{\ell \ge 1} \frac{(-1)^\ell}{\ell!} \int_{\Omega^\ell} \Pf [k(x_i,x_j)]_{i,j=1}^\ell \nu(\diff x_1) \cdots \nu(\diff x_\ell).
    \end{equation}
    For instance the first few terms of the expansion above are
    \begin{multline*}
        1- \int_\Omega \Pf\left( \begin{matrix} 0 & k_{1,2}(x,x) \\ -k_{1,2}(x,x) & 0 \end{matrix} \right) \nu(\diff x) 
        \\
        + \frac{1}{2} \int_{\Omega^2} \Pf \left( \begin{matrix} 0 & k_{1,2}(x,x)  & k_{1,1}(x,y) & k_{1,2}(x,y) \\
        -k_{1,2}(x,x) & 0 & -k_{1,2}(y,x) & k_{2,2}(x,y)
        \\
        k_{1,1}(y,x) & k_{1,2}(y,x) & 0 & k_{1,2}(y,y)
        \\
        -k_{1,2}(x,y) & k_{2,2}(y,x) & -k_{1,2}(y,y) & 0
        \end{matrix} \right) \nu(\diff x)\nu(\diff y) + \cdots.
    \end{multline*}
    The Fredholm pfaffian of a skew-symmetric matrix kernel $k$ is related to its Fredholm determinants via the relation \cite[Lemma 8.1]{Rains2000}
    \begin{equation} \label{eq:pfaff_det}
        \Pf[J-k]_{\mathbb{L}^2(\Omega)}^2 = \det[1-J^Tk]_{\mathbb{L}^2(\Omega)}
    \end{equation}
    
    For any non vanishing function $f:\Omega \to \mathbb{R}$ we can define a rescaled kernel
    \begin{equation}
        \tilde{k}(x,y) = \left( \begin{matrix} f(x) f(y) k_{1,1}(x,y) & \frac{f(x)}{f(y)}k_{1,2}(x,y) \\ - \frac{f(y)}{f(x)} k_{1,2}(y,x) & \frac{1}{f(x) f(y)} k_{2,2} (x,y) \end{matrix} \right),
    \end{equation}
    so that by notable properties of determinants we have
    \begin{equation}
        \Pf(J-k)_{\mathbb{L}^2(\Omega)} = \Pf(J-\tilde{k})_{\mathbb{L}^2(\Omega)}
    \end{equation}

\medskip 

The pfaffian of a skew-symmetric matrix $M$ of size $2n \times 2n$ can be expressed through the expansion
    \begin{equation}\label{eq:pfaffian_Leibnitz}
        \Pf[M] = \sum_{ \alpha = \{(i_1,j_1),\dots,(i_n,j_n) \} }
        \epsilon (\alpha) M_{i_1,j_1} \cdots M_{i_n,j_n},
    \end{equation}
    where the summation runs over choices of $\alpha$ such that $i_1,j_1,\dots,i_n,j_n$ is a permutation of $\{1,\dots,2n\}$ with $i_1 < \cdots < i_n$ and $i_r < j_r$ for $r=1,\dots,n$. Finally $\epsilon(\alpha)$ is the sign of the permutation $\left(\begin{smallmatrix} 1 & 2 & \cdots & 2n-1 & 2n \\ i_1 & j_1 & \cdots & i_n & j_n \end{smallmatrix}\right)$.

A useful tool to establish convergence of Fredholm determinants and pfaffians is the Hadamard's bound. Assume that a matrix $M = (M_{i,j})_{i,j=1}^\ell$ has coeffients bounded as $|M_{i,j}|\le a_i b_j$ for some numbers $a_i,b_i>0, i=1,\dots,\ell$. Then we have
    \begin{equation} \label{eq:Hadamard_bound}
        \left| \det(M) \right| \le \ell^{\ell/2} \prod_{i=1}^\ell a_i b_i.
    \end{equation}
    Such bound is easily adapded to pfaffians of anti-symmetric matrices $M$ using the relation $\Pf[M] = \sqrt{\det(M)}$.

\begin{proposition} \label{prop:rank_2}
    Let $a,b \in \mathbb{L}^2(\Omega,\nu)$ and $T:\mathbb{L}^2(\Omega,\nu) \to \mathbb{L}^2(\Omega,\nu)$ be a linear operator.
    Define the $2\times 2$ matrix kernel
    \begin{equation} \label{eq:R_rank_2}
        R(x,y) = \left( \begin{matrix} a(x)b(y) - a(y)b(x) & -a(x)b'(y) + a'(y)b(x) \\ - a'(x)b(y) + a(y)b'(x) & a'(x)b'(y) - a'(y)b'(x) \end{matrix} \right),
    \end{equation}
    where $a'=T(a), b'=T(b)$. Then the matrix $(R(x_i,x_j))_{i,j=1}^n$ has rank at most 2 for any $n$. 
\end{proposition}

\begin{proof}
        Let $u,v \in \mathbb{L}^2(\Omega,\nu)$. Then we have
        \begin{equation}
            \int_\Omega R(x,y) \cdot \binom{u(y)}{v(y)} \nu(dy) = (\langle b | u \rangle - \langle b' | v \rangle) \binom{a(x)}{-a'(x)} + (\langle a | u \rangle - \langle a' | v \rangle) \binom{b(x)}{b'(x)},
        \end{equation}
        which shows that the image of the operator defined by $R(x,y)$ is a subspace of $\mathbb{L}^2(\Omega,\nu)\times \mathbb{L}^2(\Omega,\nu)$ of dimension at most 2. 
\end{proof}

\begin{proposition} \label{prop:pfaffian_L+R}
    Let $L$ be a $2 \times 2$ matrix kernel of trace class, with norm $\| L \|<1$ and let $R$ be a rank 2 perturbation as in \eqref{eq:R_rank_2}. Then we have
    \begin{equation}
        \Pf[J - (L + R)] = \Pf[J - L] (1-\langle B | \widetilde{A} \rangle),
    \end{equation}
    where $B = (b,-b')$ and $\widetilde{A} = (1-J^TL)^{-1} \binom{a'}{a}$.
\end{proposition}
\begin{proof}
    Define $G=(1-J^T L)^{-1}$. Using relation \eqref{eq:pfaff_det} and the multiplicative property of determinants we have
    \begin{equation} \label{eq:pfaffian_L+R}
        \Pf[J - (L + R)] = \sqrt{\det[1-J^T(L+R)]} = \sqrt{\det[1-J^TL]} \sqrt{\det[1-GJ^TR]}.
    \end{equation}
    We rewrite the matrix $J^TR$ in the more convenient notation
    \begin{equation}
        J^TR =  \left( \begin{matrix} a' \\ a \end{matrix} \right) \cdot (b,-b') - \left( \begin{matrix} b' \\ b \end{matrix} \right) \cdot (a,-a').
    \end{equation}
    Denoting $\widetilde{A}$ as in the statement of the proposition and $\widetilde{B}=G\binom{q'}{q}$ we have
    \begin{equation}
        GJ^TR= \widetilde{A} \cdot (b,-b') - \widetilde{B} \cdot (a,-a')
    \end{equation}
    so that
    \begin{equation}
        \det[1-GJ^TR] = \det \left[ 1 - \left( \begin{matrix} \langle B | \widetilde{A} \rangle &  - \langle B | \widetilde{B} \rangle \\ \langle A | \widetilde{A} \rangle &  - \langle A | \widetilde{B} \rangle \end{matrix}  \right) \right],
    \end{equation}
    where $A=(a,-a')$. We will deduce, by the skew-symmetric property of the kernel $L$, that $\langle B | \widetilde{B} \rangle=\langle A | \widetilde{A} \rangle = 0$ and $\langle B | \widetilde{A} \rangle= - \langle A | \widetilde{B} \rangle$. We only show that $\langle A | \widetilde{A} \rangle = 0$ as the identity $\langle B | \widetilde{B} \rangle = 0$ can be proven in the same way. 
    We consider the representation of $G$ as a Neumann series of the operator $ J^T L $ and we write
    \begin{equation}
    \begin{split}
        \langle A | \widetilde{A} \rangle &= (a,-a') \cdot  \sum_{\ell \ge 0} \left( J^T L \right)^\ell \cdot \left( \begin{matrix} a' \\ a \end{matrix} \right)
        \\
        & = - \sum_{\ell \ge 0} (a',a) \cdot J \cdot \left( J^T L \right)^\ell \cdot \left( \begin{matrix} a' \\ a \end{matrix} \right).
    \end{split}
    \end{equation}
    In the right hand side we recognize that for each $\ell$ the operator $J \cdot (J^TL)^\ell$ is skew symmetric and hence each term in the summation is identically zero. Let us now apply the same idea to show that $\langle B | \widetilde{A} \rangle= - \langle A | \widetilde{B} \rangle$. We have
    \begin{equation}
    \begin{split}
        \langle B | \widetilde{A} \rangle &= - \sum_{\ell \ge 0} (b',b) \cdot J \cdot (J^TL)^\ell \cdot \left( \begin{matrix} a' \\ a \end{matrix} \right)
        \\
        &= - \sum_{\ell \ge 0} (a',a) \cdot  (L^TJ)^\ell \cdot J^T \cdot \left( \begin{matrix} b' \\ b \end{matrix} \right)
        \\
        &= \sum_{\ell \ge 0} (a',a) \cdot J \cdot (J^T L)^\ell \cdot \left( \begin{matrix} b' \\ b \end{matrix} \right)
        = -\langle A | \widetilde{B} \rangle.
    \end{split}
    \end{equation}
    These computations prove that
    \begin{equation}
        \det[1-GJ^TR] = (1-\langle B| \widetilde{A } \rangle)^2
    \end{equation}
    and along with \eqref{eq:pfaffian_L+R} complete the proof.
\end{proof}

\section{Estimates on $q$-Pochhammer symbols}
\label{sec:LG}

In this section we will prove a number of bounds on $q$-Pochhammer symbols necessary to establish Fredholm determinant and pfaffian formulas for the distribution of the Log Gamma polymer reported in \cref{thm:fredholm_det_log_gamma} and \cref{thm:fredholm_pfaff_log_gamma}. Together with the classical and well known results of \cref{lem:convergence_gumbel} and \cref{lem:convergence_qGamma}, we will present additional refined estimates as those of \cref{lem:decay_qGamma}, \cref{lem:bound_ratio_qgamma}, which appear to be new. In deriving these new bounds we will only use elementary techniques.

\subsection{Useful bounds for the proof of \cref{thm:fredholm_det_log_gamma}}

We will start stating a number of preliminary lemmas concerning the convergence of $q$-Pochhammer symbols and Theta functions. 

\begin{lemma} \label{lem:convergence_gumbel}
    Let $\chi \sim q$-$\mathrm{Geo}(q)$ and consider the scaling $q=e^{-\varepsilon}$. Then, we have the following asymptotic equivalence in distribution
    \begin{equation}
        \chi = \varepsilon^{-1} \log \varepsilon^{-1} + \varepsilon^{-1} \mathcal{G} + O(1),
    \end{equation}
    where $\mathcal{G}$ is a Gumbel random variable.
\end{lemma}

\begin{proof}
    The unnormalized density function of the random variable $\chi$ is $k\mapsto q^k/(q;q)_k$ for $k\in \mathbb{N}_0$. After scaling and change of variables $k=\varepsilon^{-1} \log \varepsilon^{-1} + \varepsilon^{-1} y$, it becomes, up to constant terms in $y$
    \begin{equation}
        e^{-y - e^{-y} +o(1)},
    \end{equation}
    where the error $o(1)$ is uniformly bounded in $y$ and tends to zero, for $\varepsilon$ small, uniformly in $y$ ranging in any half line $[M,\infty)$. This follows from the asymptotic expansion (see \cite[Proposition 4.1.9]{BorodinCorwin2011Macdonald})
    \begin{equation}
        \log (q;q)_{\varepsilon^{-1} \log\varepsilon^{-1} + \varepsilon^{-1} y} = -\varepsilon^{-1} \frac{\pi^2}{6} - \frac{1}{2} \log \frac{\varepsilon}{ 2 \pi} + e^{-y} + o(1).
    \end{equation}
    From the expression of the unnormalized density function after rescaling we conclude that, in distribution $\varepsilon \chi - \log\varepsilon^{-1}$ converges to a Gumbel random variable. 
\end{proof}

\begin{lemma} \label{lem:convergence_theta}
    Let $\theta \sim \mathrm{Theta}(\zeta;q)$ and consider the scaling $\zeta=(1-q)^n e^{-\varsigma}$ $q=e^{-\varepsilon}$. Then, we have the following asymptotic equivalence in distribution
    \begin{equation}
        \theta =  -\varepsilon^{-1} n \log \varepsilon^{-1} - \varepsilon^{-1} \varsigma + \varepsilon^{-1/2} \mathcal{N}(0,1) + O(1),
    \end{equation}
    where $\mathcal{N}(0,1)$ is a centered gaussian random variable of variance 1. 
\end{lemma}
\begin{proof}
    The unnormalized density function of the random variable $\theta$ is, after the scaling,
    \begin{equation}
        e^{- \frac{\varepsilon k^2}{2} + k (-\varsigma + n \log(1-e^{-\varepsilon})) }
    \end{equation}
    with $k\in \mathbb{Z}$. After a change of variable $k= - \varepsilon^{-1} \varsigma + \varepsilon^{-1} n \log(1-e^{-\varepsilon}) + \varepsilon^{-1/2} x$ it becomes, modulo normalization, equal to $e^{-x^2/2}$, completing the proof.
\end{proof}

\begin{lemma} \label{lem:convergence_ratio_theta}
    Set $q=e^{-\varepsilon}$, with $\varepsilon>0$. Let $\varsigma\in \mathbb{R}, n \in \mathbb{R}_+$ and $\alpha = -d + \mathrm{i} y $ with $d \in (0,1)$ and $y\in [ -\varepsilon^{-1} \pi , \varepsilon^{-1} \pi]$. Then, we have
    \begin{equation} \label{eq:theta_3_convergence}
        \lim_{\varepsilon \downarrow 0}(1-q)^{\alpha n} \frac{\vartheta_3(q^\alpha e^{-\varsigma} (1-q)^n;q)}{\vartheta_3( e^{-\varsigma} (1-q)^n;q)} = e^{- \alpha \varsigma},
    \end{equation}
    where the convergence holds locally uniformly in $y$.
    Moreover, there exists a constant $C$ depending on $\varsigma,d,n$ such that
    \begin{equation} \label{eq:theta_3_bound}
        \left| (1-q)^{\alpha n} \frac{\vartheta_3(q^\alpha e^{-\varsigma} (1-q)^n;q)}{\vartheta_3( e^{-\varsigma} (1-q)^n;q)}  \right| \le C,
    \end{equation}
    for all $\varepsilon > 0, y \in [ -\varepsilon^{-1} \pi , \varepsilon^{-1} \pi]$.
\end{lemma}

\begin{proof}
    We will first prove the convergence result \eqref{eq:theta_3_convergence} and later we will establish the bound \eqref{eq:theta_3_bound}.
    From the Jacobi triple product identity we have
    \begin{equation}
        \vartheta_3(e^{\eta} , e^{-\varepsilon}) = \sum_{k \in \mathbb{Z}} e^{-\frac{\varepsilon k^2}{2} + \eta k},
    \end{equation}
    which, after completing the square in the exponent and operating the change of variable $x=\sqrt{\varepsilon} k - \frac{\eta}{\sqrt{\varepsilon}}$, becomes
    \begin{equation}
        \frac{e^{\frac{\eta^2}{2 \varepsilon}}}{\sqrt{\varepsilon}} \sum_{x \in \sqrt{\varepsilon} \mathbb{Z} + \frac{\eta}{\sqrt{\varepsilon}}} e^{-\frac{x^2}{2}} \sqrt{\varepsilon}.
    \end{equation}
    This proves the asymptotic expansions
    \begin{equation} \label{eq:theta_3_expansion}
        \vartheta_3(e^\eta ;e^{-\varepsilon} ) = \sqrt{\frac{2 \pi}{\varepsilon}} e^{\frac{\eta^2}{2 \varepsilon}}  \left(1+ O(\sqrt{\varepsilon}) \right),
    \end{equation}
    which holds when $\Im (\eta) = o(\sqrt{\varepsilon})$. Applying \eqref{eq:theta_3_expansion} to the cases $\eta=-\varsigma + n \log(1-e^{-\varepsilon})$ and $\eta= -\varsigma + n \log(1-e^{-\varepsilon}) - \varepsilon \alpha$ we can easily verify \eqref{eq:theta_3_convergence}.
    
    In order to show the bound \eqref{eq:theta_3_bound}, we will prove that the absolute value of the theta function $\vartheta_3( e^{a + \mathrm{i}y}, e^{-\varepsilon})$, with $a \in \mathbb{R}$ and $y\in[-\varepsilon^{-1} \pi , \varepsilon^{-1} \pi]$ has a maximum at $y=0$. To prove this we show that $|(-qe^{a+\mathrm{i}y};q)_\infty|$ is a decreasing function for $y\in [0,\varepsilon^{-1} \pi)$. For this we compute the derivative
    \begin{equation}
    \begin{split}
        \frac{\diff}{\diff y} \Re \left\{ \log (-qe^{a+\mathrm{i}y};q)_\infty \right\}
        &
        =
        \frac{\diff}{\diff y} \sum_{ k \ge 0} \Re \left\{ \log (1+ q^k e^{a+\mathrm{i}y} ) \right\}
        \\
        & =
        \frac{\diff}{\diff y} \sum_{ k \ge 0} \frac{1}{2} \log (1+ q^{2k} e^{2a} +2 q^k e^{a} \cos y) 
        \\
        & =
        \frac{1}{2}
        \sum_{k \ge 0} \frac{- q^k e^a \sin y }{1+ q^{2k} e^{2a} + 2 q^k e^{a} \cos y},
    \end{split}
    \end{equation}
    which clearly is negative if $y\in [0,\varepsilon^{-1} \pi)$. Using the product expansion of the Jacobi theta function we can prove that
    \begin{equation}
        \left| (1-q)^{-n(d-\mathrm{i}y)} \frac{\vartheta_3 (q^{-d+\mathrm{i}y} e^{-\varsigma} (1-q)^n ;q )}{\vartheta_3 ( e^{-\varsigma} (1-q)^n ;q )} \right| \le (1-q)^{-dn} \frac{\vartheta_3 (q^{-d} e^{-\varsigma} (1-q)^n ;q )}{\vartheta_3 ( e^{-\varsigma} (1-q)^n ;q )}
    \end{equation}
    and since the right hand side converges to $e^{- d \varsigma}$ as $\varepsilon \to 0$, it can certainly be bounded by a constant independent of $\varepsilon$.
\end{proof}

The following convergence result is classical.

\begin{lemma}[\cite{andrews1986q}, Appendix A] \label{lem:convergence_qGamma}
    The $q$-Gamma function $\Gamma_q(X)$ converges, as $q\to 1$, to $\Gamma(X)$ locally uniformly in $X\in \mathbb{C} \setminus \mathbb{Z}_{\le 0}$.
\end{lemma}

We recall that, similarly to the classical Gamma function, the $q$-Gamma function $\Gamma_q(X)$, introduced in \eqref{eq:q_gamma} satisfies the functional equation
\begin{equation} \label{eq:q_gamma_functional_eq}
    [X]_q \Gamma_q(X) = \Gamma_q(X+1),
\end{equation}
where $[X]_q=\frac{1-q^X}{1-q}$ are the $q$-numbers.

\medskip 

Below we will establish some bounds for the absolute value of $q$-Gamma function $\Gamma_q(z)$, which are uniform in $q\in(1/2,1)$ and $z$ with bounded real part. Techniques we will employ are elementary. The following lemma will be useful.

\begin{lemma} \label{lem:inequalities_sum_integrals}
    Consider real numbers $a<b$ and continuous functions $f,g:[a-1,b+1] \to \mathbb{R}$ such that $f$ is increasing, while $g$ is decreasing. Then, we have
    \begin{equation}
        \int_{a-1}^b f(k) \diff k \le \sum_{k\in \mathbb{Z} \cap [a,b]} f(k) \le \int_a^{b+1} f(k) \diff k
    \end{equation}
    and
    \begin{equation}
        \int_{a}^{b-1} g(k) \diff k \le \sum_{k\in \mathbb{Z} \cap [a,b]} g(k) \le \int_{a-1}^{b} g(k) \diff k.
    \end{equation}
\end{lemma}

\begin{lemma} \label{lem:decay_qGamma}
    Fix a real number $d \notin \mathbb{Z}_{\le 0}$ and let $q=e^{-\varepsilon}$. Then, there exists constants $C_1,C_2$ depending only on $d$, such that
    \begin{equation}
        \left| \Gamma_q (d+\mathrm{i}y) \right| \le C_1 e^{-C_2 |y|},
    \end{equation}
    for all $\varepsilon >0$ and $y \in [-\varepsilon^{-1} \pi, \varepsilon^{-1} \pi]$.
\end{lemma}

\begin{proof}
    We prove the result for $d > 1$ as this also implies the general case by means of the functional relation for the $q$-Gamma function \eqref{eq:q_gamma_functional_eq}. We also assume that $y\ge 0$ with no loss of generality, since $\overline{\Gamma_q(z)} = \Gamma_q(\overline{z})$.
    
    Define the function
    \begin{equation} \label{eq:phi}
    \begin{split}
        \phi_\varepsilon(d,y) &\coloneqq - \frac{\diff}{\diff y} \log \left| 1-e^{-\varepsilon (d+ \mathrm{i}y)} \right|
        \\
        & = - \frac{ \varepsilon e^{-\varepsilon d} \sin (\varepsilon y)  }{1+e^{-2 \varepsilon d} - 2e^{-\varepsilon d} \cos(\varepsilon y)},
    \end{split}
    \end{equation}
    so that
    \begin{equation} \label{eq:d_dy_log_qGamma}
        \frac{\diff}{\diff y} \log \left| \Gamma_q(d+\mathrm{i}y) \right| = \sum_{k\ge 0} \phi_\varepsilon(d+k,y). 
    \end{equation}
    A simple inspection of the derivative of $\phi_\varepsilon(d,y)$ shows that it is a strictly decreasing function in the $d$ variable for $d>0$, whenever $y\in(0,\varepsilon^{-1} \pi)$. We can therefore compare the summation in the right hand side of \eqref{eq:d_dy_log_qGamma} with the integral of $\phi_\varepsilon(d+k,y)$ over $k$ using \cref{lem:inequalities_sum_integrals}, obtaining inequalities
    \begin{equation} \label{eq:inequalities_G}
        G_\varepsilon(d-1,y)
        \le
        \sum_{k\ge 0} \phi_\varepsilon(d+k,y)  \le G_\varepsilon(d,y),
    \end{equation}
    where
    \begin{equation} \label{eq:G}
    \begin{split}
        G_\varepsilon(d,y) &\coloneqq \int_0^\infty \phi_\varepsilon(d+k,y) \diff k
        \\
        & =
        \arctan \left( \cot (\varepsilon y) - \frac{e^{-\varepsilon d}}{\sin (\varepsilon y)} \right) - \arctan (\cot (\varepsilon y) ).
    \end{split}
    \end{equation}
    Here the equality in the second line follows from an exact evaluation of the integral. A display of inequalities \eqref{eq:inequalities_G} is given in \cref{fig:plot_G} (\rm{b}), where we set $d=1/2$. When $d > 0$, the function $y\mapsto G_\varepsilon(d,y)$ is convex in the region $y \in [0,\varepsilon^{-1} \pi]$, it reaches the value 0 only at $y=0$ and $y=\varepsilon^{-1} \pi$ and it has a unique minimum at $y_*=\varepsilon^{-1} \arccos(e^{-\varepsilon d})$, where it attains the value $-\arctan\left( e^{-\varepsilon d} (1-e^{-2 \varepsilon d} )^{1/2} \right) \sim -\pi/2$. Moreover, always from the explicit expression, we see that, at $y=0$, $G_\varepsilon$ possesses the Taylor expansion $G_\varepsilon(d,y) = -y \frac{\varepsilon}{1-e^{-\varepsilon d}} + o(y)$. All these properties imply that, when $d>0$, there exist constants $c_1,c_2$, possibly dependent on $d$ but independent of $\varepsilon$, such that
    \begin{equation}
        \int_0^y G_\varepsilon (d,s) \diff s \le  c_1 - c_2 y,
    \end{equation}
    for all $y \in [0,\varepsilon^{-1} \pi]$. We finally come to estimate the absolute value of the $q$-Gamma function as
    \begin{equation}
    \begin{split}
        \left| \Gamma_q(d+\mathrm{i}y) \right| &= \Gamma_q(d) e^{ \int_0^y \frac{\diff}{\diff s} \log  \left|  \Gamma_q(d+\mathrm{i}s) \right| \diff s  }
        \\
        &
        \le 
        \Gamma_q(d) e^{\int_0^y G_\varepsilon(d,s) \diff s} 
        \\
        &
        \le \Gamma_q(d) e^{c_1-c_2y} 
    \end{split}
    \end{equation}
    and this completes the proof.
\end{proof}

\begin{lemma} \label{lem:bound_ratio_qgamma}
    Let $\varepsilon > 0$, $q=e^{-\varepsilon}$ and fix $a,b \in \mathbb{R}$, with $a\notin \mathbb{Z}_{\ge 0}$. Then, there exists a constant $C$ depending only on $a,b$, such that
    \begin{equation} \label{eq:bound_ratio_qgamma}
        \left|\frac{\Gamma_q(a + \mathrm{i}y)}{\Gamma_q(b + \mathrm{i} y)}\right| \le C (1+|y|)^{a-b},
    \end{equation}
     for all $\varepsilon>0$ and $y \in [-\varepsilon^{-1} \pi, \varepsilon^{-1} \pi]$.
\end{lemma}

\begin{proof}
    It suffices to prove the inequality \eqref{eq:bound_ratio_qgamma} for $1<a<b$, as the general case can be easily recovered using the functional equation \eqref{eq:q_gamma_functional_eq}.
    
    Assume $1<a<b$. By Lagrange's mean value theorem we have
    \begin{equation} \label{eq:difference_log_qGamma}
    \begin{split}
        \log\left| \Gamma_q (a+\mathrm{i} y) \right| - \log\left| \Gamma_q (b+\mathrm{i} y) \right| &
        = (a-b) \frac{\diff}{\diff a} \log\left| \Gamma_q (a+\mathrm{i} y) \right| \bigg|_{a=a_*}
        \\
        &
        =
        (a - b) \left( - \log \left( 1-e^{-\varepsilon} \right) + \sum_{k\ge 0} \varphi_\varepsilon (a_*+k,y) \right),
    \end{split}
    \end{equation}
    for some $a_* \in [a,b]$, where we have defined the functions 
    \begin{equation} \label{eq:varphi}
        \varphi_\varepsilon(x,y) \coloneqq -\frac{\diff}{\diff x} \log \left| 1-e^{-\varepsilon (x+\mathrm{i} y)} \right|
        =
        \frac{\varepsilon (1-e^{\varepsilon x} \cos(\varepsilon y))}{1+ e^{2\varepsilon x} - 2 e^{\varepsilon x} \cos(\varepsilon y)}.
    \end{equation}
    Computing the derivative
    \begin{equation}
        \frac{\diff}{\diff x} \varphi_\varepsilon(x,y) = \frac{\varepsilon^2 e^{2 \varepsilon x} ( ( e^{\varepsilon x} + e^{-\varepsilon x}) \cos (\varepsilon y) - 2 )  }{(1+ e^{2\varepsilon x} - 2 e^{\varepsilon x} \cos(\varepsilon y))^2},
    \end{equation}
    we find that the function $k\mapsto \varphi_\varepsilon (x+k ,y)$ has critical point at
    \begin{equation}
        k_* = -x+ \varepsilon^{-1} \arccosh\left( \frac{1}{\cos(\varepsilon y)} \right),
    \end{equation}
    where by convention we fix $k_*=-\infty$ if $\cos(\varepsilon y) <0$. Restricting to values $k>0$ and $x>1$, we see that $\varphi_\varepsilon(x+k,y)$ is increasing in the interval $k\in [0,k_*\vee 0)$ and decreasing for $k \ge k_*\vee 0$, so that we can estimate the summation in the second line of \eqref{eq:difference_log_qGamma}, using \cref{lem:inequalities_sum_integrals} as
    \begin{equation}
    \begin{split}
        \sum_k \varphi_\varepsilon (x+k,y) &\ge \left( \int_{-1}^{[k_*]} + \int_{[k_*] +1}^{\infty} \right) \varphi_\varepsilon (x+k,y) \diff k + \varphi_\varepsilon (x+[k_*],y)
        \\
        &
        =
        \left(\int_{-1}^{\infty} - \int_{[k_*]}^{[k_*]+1} \right) \varphi_\varepsilon (x+k,y) \diff k + \varphi_\varepsilon (x+[k_*],y)
        \\
        &
        =
        \log \left| 1-e^{-\varepsilon (x-1 + \mathrm{i}y) } \right| -\int_{k_*}^{k_*+1} \varphi_\varepsilon (x+k,y) \diff k + \varphi_\varepsilon (x+[k_*],y),
    \end{split}
    \end{equation}
    if $k_*>0$ and
    \begin{equation}
        \sum_{k=0}^\infty \varphi_\varepsilon (x+k,y) \ge \int_{0}^\infty \varphi_\varepsilon (x+k,y) \diff k
        = \log \left| 1-e^{-\varepsilon (x + \mathrm{i}y) } \right|,
    \end{equation}
    if $k_*\le0$. Since the function $\varphi_\varepsilon(x,y)$ is bounded in absolute value for $x>1$ and any $y$, as it can be seen from the explicit expression \eqref{eq:varphi}, we can conclude that
    \begin{equation}
        \sum_{k=0}^\infty \varphi_\varepsilon (x+k,y) \ge \log \left| 1-e^{-\varepsilon (x-1 + \mathrm{i}y) } \right| + C,
    \end{equation}
    for some constant $C$ independent of $\varepsilon,y$. Substituting this bound in \eqref{eq:difference_log_qGamma} we obtain
    \begin{equation}
    \begin{split}
        \log\left| \frac{\Gamma_q (a+\mathrm{i} y)}{\Gamma_q (b+\mathrm{i} y)} \right| &
        \le
        (a - b) \left( \log \left| \frac{ 1-e^{-\varepsilon (a_*-1 + \mathrm{i}y) }}{1-e^{-\varepsilon}} \right| + C \right) \le C' + (a-b) \log( 1 + |y|),
    \end{split}
    \end{equation}
    for some other constant $C'$ independent of $\varepsilon,y$, which proves \eqref{eq:bound_ratio_qgamma}.
\end{proof}

Making use of the estimates developed throughout this subsection we are finally able to prove \cref{thm:fredholm_det_log_gamma}.

\subsection{Useful bounds for the proof of \cref{thm:fredholm_pfaff_log_gamma}}

In order to prove \cref{thm:fredholm_pfaff_log_gamma} we need a number of bounds on ratio of $q$-Gamma function that add to results already established in the previous subsection. The main technical result is given by the following proposition.

\begin{proposition} \label{prop:bound_qgamma_k11}
    Let $q=e^{-\varepsilon}$, for $\varepsilon > 0$ and consider complex numbers $Z=-d+\mathrm{i}u,W=-d+\mathrm{i}v$ with $d \in (0,1/2)$. Then, there exists a $C$ independent of $u,v,\varepsilon$, such that we have
    \begin{multline} \label{eq:bound_ratio_many_qGamma}
        \left| \frac{\Gamma_q(-2Z) \Gamma_q(-2W) \Gamma_q(1+Z+W) \Gamma_q(-W-Z)}{\Gamma_q(W-Z) \Gamma_q(1+Z-W)} \right|
        \\
        \le 
        C
        (1+2|u|)^{-2d -1/2} (1+2|v|)^{-2d -1/2},
    \end{multline}
    for all $u,v \in [-\varepsilon^{-1} \pi, \varepsilon^{-1} \pi]$.
\end{proposition}

We prove \eqref{eq:bound_ratio_many_qGamma} by combining a number of bounds on ratio of $q$-Gamma functions, that we show separately in the following lemmas.

\begin{lemma} \label{lem:bound_ratio_gamma_a+b}
    Let $\varepsilon> 0$, $q=e^{-\varepsilon}$. Then, we have
    \begin{equation} \label{eq:bound_ratio_gamma_a+b}
        \left| \frac{\Gamma_q(\frac{1}{2} + \mathrm{i} (a+b)) \Gamma_q(\frac{1}{2} + \mathrm{i} a)}{\Gamma_q(\frac{1}{2} + \mathrm{i} b)} \right| \le C
    \end{equation}
    for all $a,b \in [0,\varepsilon^{-1}\pi]$, where $C$ is a constant independent of $a,b,\varepsilon$. 
\end{lemma}
\begin{proof}
Based on the value of $a,b$ we distinguish five regions, visualized in \cref{fig:plot_G} (a) and we prove the desired inequality in each of these regions.

\medskip

(\colorbox{gray!50}{\textcolor{gray!50}{$\cdot$}}) When $0 \le a \le b \le \frac{\pi}{\varepsilon}$, we have $|\Gamma_{q}(\frac{1}{2}+ \mathrm{i}a)| \le |\Gamma_{q}(\frac{1}{2}+ \mathrm{i}b)|$ and \eqref{eq:bound_ratio_gamma_a+b} holds since is $\Gamma_{q}(\frac{1}{2}+ \mathrm{i}(a+b))$ bounded; this case correspond to the gray area in \cref{fig:plot_G} (a). 

\medskip

(\colorbox{yellow}{\textcolor{yellow}{$\cdot$}}) When $a<b$ and $ a+b \le \frac{\pi}{\varepsilon}$, corresponding to the yellow region, we have $|\Gamma_{q}(\frac{1}{2}+ \mathrm{i}(a+b))| \le |\Gamma_{q}(\frac{1}{2}+ \mathrm{i}b)|$ and again \eqref{eq:bound_ratio_gamma_a+b} holds.

\medskip

(\colorbox{violet!50}{\textcolor{violet!50}{$\cdot$}}) The case $a<b, a+b > \frac{\pi}{\varepsilon}$ and $ b \le \frac{2\pi}{\varepsilon} -a -b$, corresponds to the violet region. Since
    \begin{equation} \label{eq:q_gamma_reflection_pi}
        \Gamma_q(\frac{1}{2} + \mathrm{i}(a+b)) = \Gamma_q(\frac{1}{2} + \mathrm{i}(\frac{2\pi}{\varepsilon} -a-b))
    \end{equation}
    as in the previous case we have $|\Gamma_q(\frac{1}{2} + \mathrm{i}(a+b))| \le |\Gamma_q(\frac{1}{2} + \mathrm{i}b)|$ and \eqref{eq:bound_ratio_gamma_a+b} holds.

\medskip

(\colorbox{red!50}{\textcolor{red!50}{$\cdot$}}) The red region corresponds to choice $a+b > \frac{\pi}{\varepsilon}$  and $a<\frac{2\pi}{\varepsilon} -a-b<b $. Here we find that the denominator in \eqref{eq:bound_ratio_gamma_a+b} is smaller than either factors in the numerator and in this case the desired inequality will be proven by carefully taking into account each term in the expression. Using again \eqref{eq:q_gamma_reflection_pi}, we write
    \begin{equation} \label{eq:log_qgamma_a+b_integral}
        \log \left| \frac{\Gamma_q(\frac{1}{2} + \mathrm{i} (a+b)) \Gamma_q(\frac{1}{2} + \mathrm{i} a)}{\Gamma_q(\frac{1}{2} + \mathrm{i} b)} \right| = \log | \Gamma_q(\frac{1}{2}) | + \left( \int_0^a - \int_{\frac{2\pi}{\varepsilon} - a-b}^b \right) \Phi_\varepsilon(y)\diff y,
    \end{equation}
    where
    \begin{equation}
        \Phi_\varepsilon(y) = \frac{\diff}{\diff y} \log|\Gamma_q(\frac{1}{2} + \mathrm{i} y)|.
    \end{equation}
    Since $\Phi_\varepsilon(y)$ is negative for $y\in(0,\frac{\pi}{\varepsilon})$ we can estimate the second integral in the right hand side of \eqref{eq:log_qgamma_a+b_integral} by enlarging the integration domain as
    \begin{equation}
        \int_{\frac{2\pi}{\varepsilon} - a-b}^b \Phi_\varepsilon(y) \diff y \ge \int_{\frac{\pi}{\varepsilon} - a}^{\frac{\pi}{\varepsilon}} \Phi_\varepsilon(y)\diff y.
    \end{equation}
    Plugging the previous inequality in \eqref{eq:log_qgamma_a+b_integral}, we find that \eqref{eq:bound_ratio_gamma_a+b} holds, for instance, if
    \begin{equation} \label{eq:inequality_integral_Phi}
        \int_0^a \Phi_\varepsilon (y) \diff y \le \int_{\frac{\pi}{\varepsilon} - a}^\frac{\pi}{\varepsilon} \Phi_\varepsilon (y) \diff y,
        \qquad
        \text{for all } a\in[0,\frac{2\pi}{3\varepsilon}]
    \end{equation}
    and this is what we aim to show. We make use of the estimate \eqref{eq:inequalities_G} to bound integrands in both sides of \eqref{eq:inequality_integral_Phi}. In particular, as shown in \cref{fig:plot_G} (b), we see that
    \begin{equation} \label{eq:G_upper_bound}
        \Phi_\varepsilon(y) \le G_\varepsilon(1/2,y_*)  \left[ \mathbf{1}_{[0,y_*]}(y) \frac{y}{y_*} + \mathbf{1}_{(y_*,\frac{\pi}{\varepsilon}]}(y) \frac{\frac{\pi}{\varepsilon} -y}{\frac{\pi}{\varepsilon} - y_*}  \right],
    \end{equation}
    where $y_*=\varepsilon^{-1} \arccos(e^{-\varepsilon/2})$ is the critical point of $G_\varepsilon(1/2,y)$ and 
    \begin{equation}\label{eq:G_lower_bound}
        \Phi_\varepsilon(y) \ge G_\varepsilon(-1/2,y) \ge -\varepsilon \frac{\frac{\pi}{\varepsilon}-y}{2e^{-\varepsilon/2}},
        \qquad
        \text{for } y\in[\frac{\pi}{3\varepsilon},\frac{\pi}{\varepsilon}].
    \end{equation}
    The first inequality follows from convexity of $G_\varepsilon(1/2,y)$, while the second can be derived bounding the derivative $\frac{\diff}{\diff y}G_\varepsilon(1/2,y)$ in the indicated region of $y$.
    In order to prove \eqref{eq:inequality_integral_Phi}, we can improve such inequality by replacing the integrands with the right hand sides of \eqref{eq:G_upper_bound} and \eqref{eq:G_lower_bound}. Computing the integrals we find that \eqref{eq:inequality_integral_Phi} holds, if
    \begin{equation}
        G_\varepsilon (1/2,y_*) \left[\mathbf{1}_{[0,y_*]}(a) \frac{a^2}{2y_*} + \mathbf{1}_{(y_*,\frac{\pi}{\varepsilon}]}(a) \frac{\pi(2a-y_*)-\varepsilon a^2}{2(\pi-\varepsilon y_*)} \right] \le -\frac{a^2 \varepsilon e^{-\varepsilon/2}}{2},
    \end{equation}
    for all $a\in[0,\frac{2\pi}{3 \varepsilon}]$. The previous polynomial inequality can at this point easily verified for $\varepsilon$ small enough, noticing that $y_*=\varepsilon^{-1/2} + O(\varepsilon^{1/2})$ and $G_\varepsilon(1/2,y_*)=-\frac{\pi}{2} + O(\varepsilon^{1/2})$. This proves \eqref{eq:inequality_integral_Phi} and hence \eqref{eq:bound_ratio_gamma_a+b} for the region $a+b>\frac{\pi}{\varepsilon}$ and $a <\frac{2\pi}{\varepsilon} -a-b < b$.

    \medskip
    
    (\colorbox{blue!50}{\textcolor{blue!50}{$\cdot$}}) Lastly we consider the blue region in \cref{fig:plot_G} (a), corresponding to choices $a+b>\frac{\pi}{\varepsilon}$ and $\frac{2 \pi}{\varepsilon} -a-b < a<b$. We use once again the integral expression \eqref{eq:log_qgamma_a+b_integral} and given our conditions on $a,b$ we need to estimate the difference
    \begin{equation} \label{eq:difference_integral_Phi}
        \left( \int_0^{\frac{2\pi}{\varepsilon} - a- b} - \int_a^b \right) \Phi_\varepsilon(y) \diff y.
    \end{equation}
    As a function of $b$, \eqref{eq:difference_integral_Phi} is an increasing function in the region $b\in[\frac{2\pi}{3\varepsilon},\frac{\pi}{\varepsilon}]$ and therefore we have
    \begin{equation}
        \left( \int_0^{\frac{2\pi}{\varepsilon} - a- b} - \int_a^b \right) \Phi_\varepsilon(y) \diff y \le \left( \int_0^{\frac{\pi}{\varepsilon} - a} - \int_a^{\frac{\pi}{\varepsilon}} \right) \Phi_\varepsilon(y) \diff y.
    \end{equation}
    By a change of variable $a'=\frac{\pi}{\varepsilon} - a \in [0,\frac{\pi}{2 \varepsilon}]$, we recognize that the term in the right hand side is negative as a result of \eqref{eq:inequality_integral_Phi} and this shows \eqref{eq:bound_ratio_gamma_a+b} for the region of parameters $a,b$ under consideration, concluding the proof.
\end{proof}

\begin{figure}[ht]
        \centering
        \subfloat[]{{\includegraphics[scale=.7]{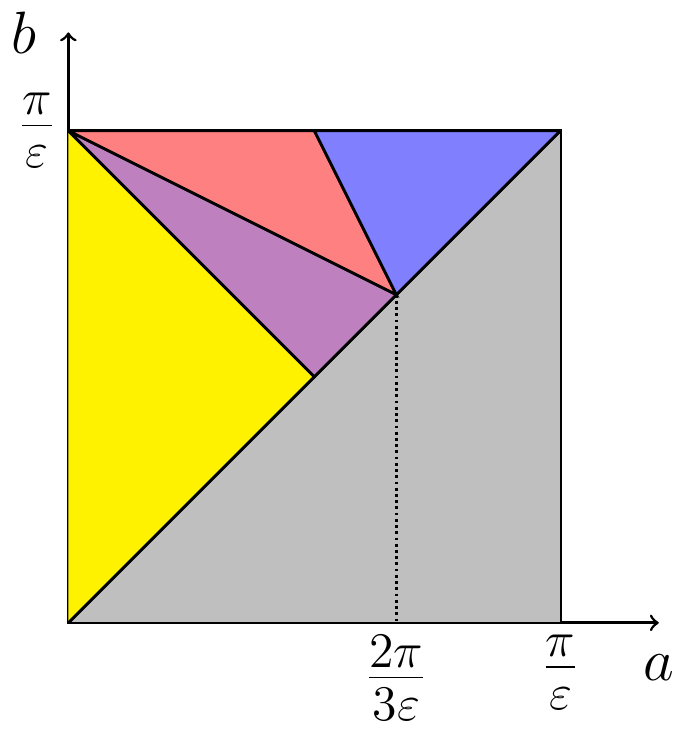}}}%
        \hspace{.7cm}
        \subfloat[]{{\includegraphics[scale=.6]{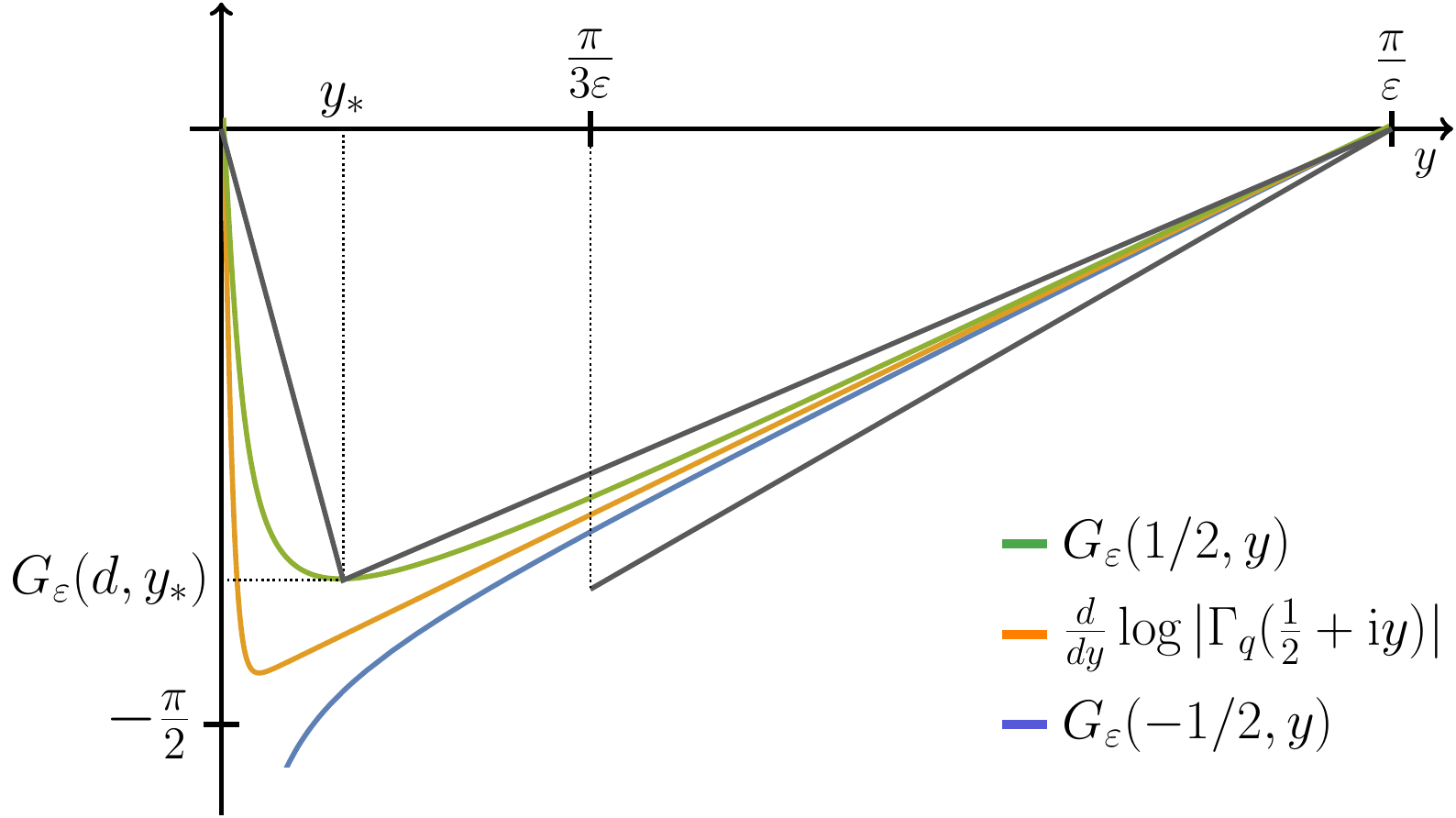} }}%
        \caption{On the left panel we see a breakdown of regions for parameters $a,b$ considered in the proof of \cref{lem:bound_ratio_gamma_a+b}. On the right panel we see a plot of the function $G_\varepsilon(d,y)$ defined in \eqref{eq:G} for $d=\pm1/2$ and a display of inequalities \eqref{eq:inequalities_G}.}
        \label{fig:plot_G}
    \end{figure}

\begin{lemma} \label{lem:ratio_qgamma_q2gamma}
    Let $\varepsilon>0, q=e^{-\varepsilon}$. Then, we have
    \begin{equation}
        \left| \frac{\Gamma_q(1+\mathrm{i} a)}{\Gamma_{q^2}(1+\mathrm{i} a)} \right| \le 1,
    \end{equation}
    for all $a\in[0,\varepsilon^{-1} \pi]$.
\end{lemma}

\begin{proof}
    We write
    \begin{equation}
        \begin{split}
            \log \left| \frac{\Gamma_q(1+\mathrm{i} a)}{\Gamma_{q^2}(1+\mathrm{i} a)} \right| &= \int_0^a  \frac{\diff}{\diff y} \log \left| \frac{\Gamma_q(1+\mathrm{i} y)}{\Gamma_{q^2}(1+\mathrm{i} y)} \right| \diff y
            \\
            &=
            \int_0^a \sum_{k\ge 1} \left( \phi_\varepsilon (k,y) - \phi_{2\varepsilon} (k,y) \right) \diff y,
        \end{split}
    \end{equation}
    where the function $\phi_\varepsilon$ was introduced in \eqref{eq:phi}. Evaluating explicitly the difference of the two $\phi_\varepsilon$ functions in the integrand we find
    \begin{equation}
        \phi_\varepsilon (k,y) - \phi_{2\varepsilon}  (k,y) = - \frac{\varepsilon e^{- \varepsilon k} \sin (\varepsilon y) }{1+ e^{-2 \varepsilon k} + 2 e^{- \varepsilon k} \cos(\varepsilon y)},
    \end{equation}
    which is negative for $0<y<\varepsilon^{-1} \pi$. This proves the desired inequality.
\end{proof}

\begin{lemma}\label{lem:ratio_qgamma_a}
    Let $\varepsilon>0, q=e^{-\varepsilon}$. Then, we have
    \begin{equation} \label{eq:ratio_qgamma_a}
        \left| \frac{\Gamma_q(\frac{1}{2}+\mathrm{i} a)^2}{\Gamma_{q}(\frac{1}{2} + 2\mathrm{i}  a)} \right| \le C,
    \end{equation}
    for all $a\in[0,\varepsilon^{-1} \pi /2]$, where $C$ is a constant independent of $a,\varepsilon$.
\end{lemma}
\begin{proof}
    We use the duplication formula for the $q$-Gamma function
    \begin{equation}
        \Gamma_q(2z) = \mathsf{c}_{q,z} \Gamma_{q^2} (z) \Gamma_{q^2} (\frac{1}{2} + z),
    \end{equation}
    where $\mathsf{c}_{q,z}=\frac{(1+q)^{2z -1}}{(1-q^2)^{1/2} (q^2,-q;q)_\infty}$ and we have $\mathsf{c}_{q,z} \xrightarrow[q\to1]{} \frac{2^{2z-1}}{\sqrt{\pi}} $.
    Through simple algebraic manipulations we rewrite the ratio in the left hand side of \eqref{eq:ratio_qgamma_a} as
    \begin{equation}
        \begin{split}
            \frac{\Gamma_q(\frac{1}{2}+\mathrm{i} a)^2}{\Gamma_{q}(\frac{1}{2} + 2\mathrm{i}  a)} &= C_{q,a} \frac{\Gamma_{q}(\frac{1}{2}+\mathrm{i}a)^2}{ \Gamma_{q^2}(\frac{1}{4}+\mathrm{i}a) \Gamma_{q^2}(\frac{3}{4}+\mathrm{i}a) } 
            \\
            &
            = C_{q,a}
            \left( \frac{\Gamma_q(\frac{1}{2}+\mathrm{i}a)}{\Gamma_q(1+\mathrm{i}a)} \right)^2  \left( \frac{\Gamma_{q}(1+\mathrm{i}a)}{\Gamma_{q^2}(1+\mathrm{i}a)} \right)^2 \frac{\Gamma_{q^2}(1+\mathrm{i}a)}{\Gamma_{q^2}(\frac{1}{4}+\mathrm{i}a)} \frac{\Gamma_{q^2}(1+\mathrm{i}a)}{\Gamma_{q^2}(\frac{3}{4}+\mathrm{i}a)},
        \end{split}
    \end{equation}
    for some constant $C_{q,a}$, which is absolutely bounded in $q$. Each term in the right hand side can be bounded in absolute value through \cref{lem:bound_ratio_qgamma} and \cref{lem:ratio_qgamma_q2gamma} and the resulting inequality yields \eqref{eq:ratio_qgamma_a}.
\end{proof}

\begin{lemma} \label{lem:bound_ratio_gamma_b-a}
    Let $\varepsilon> 0$, $q=e^{-\varepsilon}$. Then, we have
    \begin{equation} \label{eq:bound_ratio_gamma_b-a}
        \left| \frac{\Gamma_q(\frac{1}{2} + \mathrm{i} (b-a)) \Gamma_q(\frac{1}{2} + \mathrm{i} a)}{\Gamma_q(\frac{1}{2} + \mathrm{i} b)} \right| \le C
    \end{equation}
    for all $a,b \in [0,\varepsilon^{-1}\pi]$, where $C$ is a constant independent of $a,b,\varepsilon$. 
\end{lemma}

\begin{proof}
    When $b<a$ \eqref{eq:bound_ratio_gamma_b-a} is straightforward, as we have $|\Gamma_q(\frac{1}{2} + \mathrm{i}a)| < |\Gamma_q(\frac{1}{2} + \mathrm{i}b)|$ and $\Gamma_q(\frac{1}{2} + \mathrm{i}(b-a))$ is bounded. Therefore, from now on, we assume $b>a$. Using the same notation introduced in the proof of \cref{lem:bound_ratio_gamma_a+b}, we write
    \begin{equation} \label{eq:integal_Phi_b-a}
        \log \left| \frac{\Gamma_q(\frac{1}{2} + \mathrm{i} (b-a)) \Gamma_q(\frac{1}{2} + \mathrm{i} a)}{\Gamma_q(\frac{1}{2} + \mathrm{i} b)} \right| = \log | \Gamma_q(\frac{1}{2}) | + \left( \int_0^a - \int_{b-a}^b \right) \Phi_\varepsilon(y) \diff y.
    \end{equation}
    A study of the derivatives of the right hand side shows that that \eqref{eq:integal_Phi_b-a} has a global maximum along the line $b=2a$. This shows that
    \begin{equation}
        \left| \frac{\Gamma_q(\frac{1}{2} + \mathrm{i} (b-a)) \Gamma_q(\frac{1}{2} + \mathrm{i} a)}{\Gamma_q(\frac{1}{2} + \mathrm{i} b)} \right| \le \left| \frac{ \Gamma_q(\frac{1}{2} + \mathrm{i} a)^2 }{\Gamma_q(\frac{1}{2} + \mathrm{i} 2a )} \right|,
    \end{equation}
    for all $0\le a < b < \frac{\pi}{\varepsilon}$ and invoking \cref{lem:ratio_qgamma_a} we conclude the proof.
\end{proof}

We are finally ready to prove the main technical bound of the subsection.

\begin{proof}[Proof of \cref{prop:bound_qgamma_k11}]
    We rewrite the left hand side of \eqref{eq:bound_ratio_many_qGamma} as
    \begin{multline} \label{eq:ratio_many_qGamma_rewritten}
        \left|
        \frac{\Gamma_q(-2Z)}{\Gamma_q(\frac{1}{2}-2\mathrm{i}u)} \frac{\Gamma_q(-2W)}{\Gamma_q(\frac{1}{2}-2\mathrm{i}v)}
        \frac{\Gamma_q(1+Z+W)}{\Gamma_q(\frac{1}{2}+\mathrm{i}(u+v))}
        \frac{\Gamma_q(-Z-W)}{\Gamma_q(\frac{1}{2}-\mathrm{i}(u+v))}
        \frac{\Gamma_q(\frac{1}{2}-\mathrm{i}(u-v)) }{\Gamma_q(W-Z)}\frac{\Gamma_q(\frac{1}{2}+\mathrm{i}(u-v))}{\Gamma_q(1+Z-W)} \right|
        \\
        \times
        \left|
        \frac{\Gamma_q(\frac{1}{2}-2\mathrm{i}u) \Gamma_q(\frac{1}{2}-2\mathrm{i}v) \Gamma_q(\frac{1}{2}+\mathrm{i}(u+v)) \Gamma_q(\frac{1}{2}-\mathrm{i}(u+v))}{\Gamma_q(\frac{1}{2}-\mathrm{i}(u-v)) \Gamma_q(\frac{1}{2}+\mathrm{i}(u-v))} \right|.
    \end{multline}
    This manipulation is convenient, since, through \cref{lem:bound_ratio_qgamma} we can estimate the ratio of $q$-Gamma functions with arguments having equal imaginary parts. Using \eqref{eq:bound_ratio_qgamma}  repeatedly, we obtain
    \begin{multline*}
        \left|
        \frac{\Gamma_q(-2Z)}{\Gamma_q(\frac{1}{2}-2\mathrm{i}u)} \frac{\Gamma_q(-2W)}{\Gamma_q(\frac{1}{2}-2\mathrm{i}v)}
        \frac{\Gamma_q(1+Z+W)}{\Gamma_q(\frac{1}{2}+\mathrm{i}(u+v))}
        \frac{\Gamma_q(-Z-W)}{\Gamma_q(\frac{1}{2}-\mathrm{i}(u+v))}
        \frac{\Gamma_q(\frac{1}{2}-\mathrm{i}(u-v)) }{\Gamma_q(W-Z)}\frac{\Gamma_q(\frac{1}{2}+\mathrm{i}(u-v))}{\Gamma_q(1+Z-W)} \right|
        \\
        \le
        C
        (1+2|u|)^{-2d -1/2} (1+2|v|)^{-2d -1/2}.
    \end{multline*}
    In order to estimate the factor in the second line of \eqref{eq:ratio_many_qGamma_rewritten}, we utilize the change of variables $a=u+v, b=u-v$ and write it in the slighly simpler expression
    \begin{equation}
        \left|
        \frac{\Gamma_q\left(\frac{1}{2} + \mathrm{i} ( a+b) \right) \Gamma_q \left(\frac{1}{2} + \mathrm{i}a \right)^2  \Gamma_q \left(\frac{1}{2} + \mathrm{i}(b-a) \right)}{\Gamma_q \left(\frac{1}{2} +\mathrm{i} b\right)^2 } \right|.    
    \end{equation}
    Such function is symmetric with respect to the transformations $(a,b) \mapsto (\pm a, \pm b)$, so that we only need to show that it is bounded for $a,b \in [0,\varepsilon^{-1}\pi]$, in which case such property follows from \cref{lem:bound_ratio_gamma_a+b,lem:bound_ratio_gamma_b-a}. This completes the proof.
\end{proof}

\bibliographystyle{alpha}
\bibliography{bib}

\begin{thebibliography}{DNKLDT20}

\bibitem[ACQ11]{AmirCorwinQuastel2011}
G.~Amir, I.~Corwin, and J.~Quastel.
\newblock {Probability distribution of the free energy of the continuum
  directed random polymer in 1+ 1 dimensions}.
\newblock {\em Communications on Pure and Applied Mathematics}, 64(4):466--537,
  2011.

\bibitem[Agg16]{Aggarwal_6v_to_ASEP}
A~Aggarwal.
\newblock {Convergence of the Stochastic Six-Vertex Model to the ASEP}.
\newblock {\em Mathematical Physics, Analysis and Geometry}, 20, Dec 2016.

\bibitem[AKQ14]{AlbertsKhaninQuastel2012}
T.~Alberts, K.~Khanin, and J.~Quastel.
\newblock {Intermediate disorder regime for 1+ 1 dimensional directed
  polymers}.
\newblock {\em {Annals of Probability}}, 42(3):1212--1256, 2014.

\bibitem[And86]{andrews1986q}
G.~Andrews.
\newblock {\em {$ q $-Series: Their Development and Application in Analysis,
  Number Theory, Combinatorics, Physics and Computer Algebra}}, volume~66.
\newblock AMS, 1986.

\bibitem[ASW52]{ASW52}
M.~Aissen, I.~J. Schoenberg, and A.~Whitney.
\newblock {On the generating functions of totally positive sequences I}.
\newblock {\em Journal d'Analyse Mathématique}, 2:93--103, 1952.

\bibitem[BB19]{betea_bouttier_periodic}
D.~Betea and J.~Bouttier.
\newblock {The Periodic Schur Process and Free Fermions at Finite Temperature}.
\newblock {\em Mathematical Physics, Analysis and Geometry}, 22:3, 1 2019.

\bibitem[BBC16]{borodin_bufetov_corwin_nested}
A.~Borodin, A.~Bufetov, and I.~Corwin.
\newblock Directed random polymers via nested contour integrals.
\newblock {\em Annals of Physics}, 368:191--247, 2016.

\bibitem[BBC20]{barraquand_half_space_mac}
G.~Barraquand, A.~Borodin, and I.~Corwin.
\newblock {Half-Space Macdonald Processes}.
\newblock {\em Forum of Mathematics, Pi}, 8:e11, 2020.

\bibitem[BBCS18]{baik_barraquand_corwin_suidan_pfaffian}
J.~Baik, G.~Barraquand, I.~Corwin, and T.~Suidan.
\newblock {Pfaffian Schur processes and last passage percolation in a
  half-quadrant}.
\newblock {\em The Annals of Probability}, 46(6):3015 -- 3089, 2018.

\bibitem[BBCW18]{barraquand2018}
G.~Barraquand, A.~Borodin, I.~Corwin, and M.~Wheeler.
\newblock {Stochastic six-vertex model in a half-quadrant and half-line open
  asymmetric simple exclusion process}.
\newblock {\em Duke Mathematical Journal}, 167(13):2457--2529, 09 2018.

\bibitem[BBNV18]{Betea_et_al_free_boundary}
D.~Betea, J.~Bouttier, P.~Nejjar, and M.~Vuletić.
\newblock {The Free Boundary Schur Process and Applications I}.
\newblock {\em Annales Henri Poincaré}, 19:3663--3742, 12 2018.

\bibitem[BBNV19]{Betea_etal2019}
D.~Betea, J.~Bouttier, P.~Nejjar, and M.~Vuleti{\'c}.
\newblock {New edge asymptotics of skew Young diagrams via free boundaries}.
\newblock {\em arXiv preprint}, 2019.
\newblock arXiv:1902.08750 [math.CO].

\bibitem[BBW18]{BorodinBufetovWheeler2016}
A.~Borodin, A.~Bufetov, and M.~Wheeler.
\newblock {Between the stochastic six vertex model and Hall-Littlewood
  processes}.
\newblock {\em Duke Mathematical Journal}, 167(13):2457--2529, 2018.

\bibitem[BC95]{BertiniCancrini1995}
L.~Bertini and N.~Cancrini.
\newblock {The stochastic heat equation: Feynman-Kac formula and
  intermittence}.
\newblock {\em Journal of Statistical Physics}, 78(5-6):1377--1401, 1995.

\bibitem[BC14]{BorodinCorwin2011Macdonald}
A.~Borodin and I.~Corwin.
\newblock Macdonald processes.
\newblock {\em Probability Theory and Related Fields}, 158:225--400, 2014.

\bibitem[BCD21]{Barraquand_Corwin_Dimitrov_log_gamma}
G.~Barraquand, I.~Corwin, and E.~Dimitrov.
\newblock Fluctuations of the log-gamma polymer free energy with general
  parameters and slopes.
\newblock {\em Probability Theory and Related Fields}, 181:113--195, 11 2021.

\bibitem[BCFV15]{BorodinCorwinFerrariVeto2013}
A.~Borodin, I.~Corwin, P.~Ferrari, and B.~Veto.
\newblock {Height fluctuations for the stationary KPZ equation}.
\newblock {\em Mathematical Physics, Analysis and Geometry}, 18(1):1--95, 2015.

\bibitem[BCG16]{BCG6V}
A.~Borodin, I.~Corwin, and V.~Gorin.
\newblock Stochastic six-vertex model.
\newblock {\em Duke Mathematical Journal}, 165(3):563--624, 2016.

\bibitem[BCPS15]{BCPS2014_arXiv_v4}
A.~Borodin, I.~Corwin, L.~Petrov, and T.~Sasamoto.
\newblock Spectral theory for interacting particle systems solvable by
  coordinate bethe ansatz.
\newblock {\em Communications in Mathematical Physics}, 339(3):1167--1245,
  2015.
\newblock Updated version including erratum. Available at
  \url{https://arxiv.org/abs/1407.8534v4}.

\bibitem[BCR13]{BorodinCorwinRemenik}
A.~Borodin, I.~Corwin, and D.~Remenik.
\newblock {Log-Gamma polymer free energy fluctuations via a Fredholm
  determinant identity}.
\newblock {\em Communications in Mathematical Physics}, 324(1):215--232, 2013.

\bibitem[BCS14]{BorodinCorwinSasamoto2012}
A.~Borodin, I.~Corwin, and T.~Sasamoto.
\newblock {From duality to determinants for q-TASEP and ASEP}.
\newblock {\em Annals of Probability}, 42(6):2314--2382, 2014.

\bibitem[BDJ99]{baik1999distribution}
J.~Baik, P.~Deift, and K.~Johansson.
\newblock {On the distribution of the length of the longest increasing
  subsequence of random permutations}.
\newblock {\em Jour. AMS}, 12(4):1119--1178, 1999.

\bibitem[BFP07]{BFP_2007}
A.~Borodin, P.~L. Ferrari, and M.~Prähofer.
\newblock {Fluctuations in the Discrete TASEP with Periodic Initial
  Configurations and the Airy1 Process}.
\newblock {\em International Mathematics Research Papers}, 2007, 01 2007.
\newblock rpm002.

\bibitem[BG97]{bertiniGiacomin1997stochastic}
L.~Bertini and G.~Giacomin.
\newblock {Stochastic Burgers and KPZ equations from particle systems}.
\newblock {\em Communications in Mathematical Physics}, 183(3):571--607, 1997.

\bibitem[BG16]{BG2016_Airy_moments}
A.~Borodin and V.~Gorin.
\newblock {Moments match between the KPZ equation and the Airy point process}.
\newblock {\em arXiv preprint}, 2016.
\newblock arXiv:1608.01557 [math-ph].

\bibitem[BKLD20]{Barraquand_le_doussal_krejenbrink_2020}
G.~Barraquand, A.~Krajenbrink, and P.~Le~Doussal.
\newblock {Half-Space Stationary Kardar{\textendash}Parisi{\textendash}Zhang
  Equation}.
\newblock {\em Journal of Statistical Physics}, 181(4):1149--1203, aug 2020.

\bibitem[BLD21]{Barraquand_le_doussal_half_space_flat}
G.~Barraquand and P.~Le~Doussal.
\newblock {Kardar-Parisi-Zhang equation in a half space with flat initial
  condition and the unbinding of a directed polymer from an attractive wall}.
\newblock {\em Physical Review E}, 104:024502, Aug 2021.

\bibitem[BM18]{BufetovMatveev2017}
A.~Bufetov and K.~Matveev.
\newblock {Hall-Littlewood RSK field}.
\newblock {\em Selecta Mathematica}, 24(5):4839--4884, 2018.

\bibitem[BMP21]{BufetovMucciconiPetrov2018}
A.~Bufetov, M.~Mucciconi, and L.~Petrov.
\newblock {Yang-Baxter random fields and stochastic vertex models}.
\newblock {\em Advances in Mathematics}, 388:107865, 2021.

\bibitem[BO17]{BO2016_ASEP}
A.~Borodin and G.~Olshanski.
\newblock {The ASEP and determinantal point processes}.
\newblock {\em Communications in Mathematical Physics}, 353(2):853--903, 2017.

\bibitem[Bor07]{borodin2007periodic}
A.~Borodin.
\newblock {Periodic Schur process and cylindric partitions}.
\newblock {\em Duke Mathematical Journal}, 140(3):391--468, 2007.

\bibitem[Bor18]{borodin2016stochastic_MM}
A.~Borodin.
\newblock {Stochastic higher spin six vertex model and Macdonald measures}.
\newblock {\em Journal of Mathematical Physics}, 59(2):023301, 2018.

\bibitem[BP16]{BorodinPetrov2013NN}
A.~Borodin and L.~Petrov.
\newblock {Nearest neighbor Markov dynamics on Macdonald processes}.
\newblock {\em Advances in Mathematics}, 300:71--155, 2016.

\bibitem[BP18]{BorodinPetrov2016inhom}
A.~Borodin and L.~Petrov.
\newblock {Higher spin six vertex model and symmetric rational functions}.
\newblock {\em Selecta Mathematica}, 24(2):751--874, 2018.

\bibitem[BP19]{BufetovPetrovYB2017}
A.~Bufetov and L.~Petrov.
\newblock {Yang-Baxter field for spin Hall-Littlewood symmetric functions}.
\newblock {\em Forum of Mathematics, Sigma}, 7, 2019.

\bibitem[BR01a]{baik_rains2001algebraic}
J.~Baik and E.M. Rains.
\newblock {Algebraic aspects of increasing subsequences}.
\newblock {\em Duke Mathematical Journal}, 109(1):1--66, 2001.

\bibitem[BR01b]{baik_rains2001symmetrized}
J.~Baik and E.M. Rains.
\newblock {Symmetrized random permutations}.
\newblock {\em Random matrix models and their applications}, pages 1--29, 2001.

\bibitem[BR01c]{baik_rains2001asymptotics}
J.~Baik and E.M. Rains.
\newblock {The asymptotics of monotone subsequences of involutions}.
\newblock {\em Duke Mathematical Journal}, 109(2):205--282, 2001.

\bibitem[BR05]{borodin2005eynard}
A.~Borodin and E.M. Rains.
\newblock {Eynard--Mehta theorem, Schur process, and their Pfaffian analogs}.
\newblock {\em Journal of Statistical Physics}, 121(3):291--317, 2005.

\bibitem[BR22]{barraquand_rychnovky_half_space}
G.~Barraquand and M.~Rychnovsky.
\newblock {Random walk on nonnegative integers in beta distributed random
  environment}.
\newblock {\em arXiv preprint}, Jan 2022.
\newblock arXiv:2201.07270 [math.PR].

\bibitem[BW21]{barraquand_wang_2021}
G.~Barraquand and S.~Wang.
\newblock An identity in distribution between full-space and half-space
  log-gamma polymers.
\newblock {\em arXiv preprint}, 2021.
\newblock arXiv:2108.08737 [math.PR].

\bibitem[CD21]{chen_ding_littlewood}
K.~Chen and X.~Ding.
\newblock {Stable spin Hall-Littlewood symmetric functions, combinatorial
  identities, and half-space Yang-Baxter random field}.
\newblock {\em arXiv preprint}, 6 2021.
\newblock arXiv:2106.12557 [math-ph].

\bibitem[CK21]{corwin_knizel_open_KPZ}
I.~Corwin and A.~Knizel.
\newblock {Stationary measure for the open KPZ equation}.
\newblock {\em arXiv preprint}, Mar 2021.
\newblock arXiv:2103.12253 [math.PR].

\bibitem[CLDR10]{Calabrese_LeDoussal_Rosso}
P.~Calabrese, P.~Le~Doussal, and A.~Rosso.
\newblock {Free-energy distribution of the directed polymer at high
  temperature}.
\newblock {\em Europhysics Letters}, 90(2), 2010.

\bibitem[CP16]{CorwinPetrov2015}
I.~Corwin and L.~Petrov.
\newblock Stochastic higher spin vertex models on the line.
\newblock {\em Communications in Mathematical Physics}, 343(2):651--700, 2016.
\newblock Erratum available at
  \url{https://storage.lpetrov.cc/research_files/Petrov-publ/erratum_1502.pdf}.

\bibitem[CSV11]{vuletic2009plane}
S.~Corteel, C.~Savelief, and M.~Vuletić.
\newblock Plane overpartitions and cylindric partitions.
\newblock {\em Journal of Combinatorial Theory, Series A}, 118(4):1239--1269,
  2011.

\bibitem[DLDMS15]{DLDMS_free_fermions_KPZ}
D.~S. Dean, P.~Le~Doussal, S.~N. Majumdar, and G.~Schehr.
\newblock {Finite-Temperature Free Fermions and the Kardar-Parisi-Zhang
  Equation at Finite Time}.
\newblock {\em Physical Review Letters}, 114:110402, Mar 2015.

\bibitem[DNKLDT20]{De_Nardis_et_al_2020}
J.~De~Nardis, A.~Krajenbrink, P.~Le~Doussal, and T.~Thiery.
\newblock Delta-bose gas on a half-line and the
  kardar{\textendash}parisi{\textendash}zhang equation: boundary bound states
  and unbinding transitions.
\newblock {\em Journal of Statistical Mechanics: Theory and Experiment},
  2020(4):043207, apr 2020.

\bibitem[Dot10]{Dotsenko}
Victor Dotsenko.
\newblock Replica bethe ansatz derivation of the tracy{\textendash}widom
  distribution of the free energy fluctuations in one-dimensional directed
  polymers.
\newblock {\em Journal of Statistical Mechanics: Theory and Experiment},
  2010(07):P07010, Jul 2010.

\bibitem[EB18]{ELBACHRAOUI_q_trig}
M.~El~Bachraoui.
\newblock {Solving some q-trigonometric conjectures of Gosper}.
\newblock {\em Journal of Mathematical Analysis and Applications},
  460(2):610--617, 2018.

\bibitem[Edr53]{Edrei53}
A.~Edrei.
\newblock On the generating function of a doubly infinite, totally positive
  sequence.
\newblock {\em Transactions of the American Mathematical Society}, 74:367--383,
  1953.

\bibitem[Gho19]{ghosal_pfaffian_schur}
P.~Ghosal.
\newblock {Correlation Functions of the Pfaffian Schur Process Using Macdonald
  Difference Operators}.
\newblock {\em SIGMA}, 092(12):37, 2019.

\bibitem[GIP15]{gubinelli_imkeller_perkowski_2015}
M.~Gubinelli, P.~Imkeller, and N.~Perkowski.
\newblock {Paracontrolled distributions and singular PDEs}.
\newblock {\em Forum of Mathematics, Pi}, 3:e6, 2015.

\bibitem[GLD12]{Gueudre_Le_Doussal_2012}
T.~Gueudr{\'{e}} and P.~Le~Doussal.
\newblock Directed polymer near a hard wall and {KPZ} equation in the
  half-space.
\newblock {\em Europhysics Letters}, 100(2):26006, oct 2012.

\bibitem[Gos01]{gosper_q_trigonometric}
R.W. Gosper.
\newblock Experiments and discoveries in q-trigonometry.
\newblock {\em Symbolic Computation, Number Theory, Special Functions, Physics
  and Combinatorics. F. G. Garvan and M. E. H. Ismail. Kluwer, Dordrecht,
  Netherlands}, pages 79 -- 105, 2001.

\bibitem[GS92]{GwaSpohn1992}
L.-H. Gwa and H.~Spohn.
\newblock Six-vertex model, roughened surfaces, and an asymmetric spin
  {H}amiltonian.
\newblock {\em Physical Review Letters}, 68(6):725--728, 1992.

\bibitem[Hai14]{Hairer11}
M.~Hairer.
\newblock Solving the {KPZ} equation.
\newblock {\em Annals of Mathematics}, 178(2):559--664, 2014.

\bibitem[HH85]{HuseHenleyPolymers1985}
D.~Huse and C.~Henley.
\newblock Pinning and roughening of domain wall in ising systems due to random
  impurities.
\newblock {\em Physical Review Letters}, 54(25):2708, 1985.

\bibitem[IMS21a]{IMS_matching}
T.~Imamura, M.~Mucciconi, and T.~Sasamoto.
\newblock {Identity between restricted Cauchy sums for the $q$-Whittaker and
  skew Schur polynomials}.
\newblock {\em arXiv preprint}, 2021.
\newblock arXiv:2106.11913 [math.CO].

\bibitem[IMS21b]{IMS_skew_RSK}
T.~Imamura, M.~Mucciconi, and T.~Sasamoto.
\newblock {Skew RSK dynamics: invariant, affine crystals and $q$-Whittaker
  polynomials}.
\newblock {\em arXiv preprint}, 2021.
\newblock arXiv:2106.11922 [math.CO].

\bibitem[IS11]{ImamuraSasamoto2011current}
T.~Imamura and T.~Sasamoto.
\newblock {Current moments of 1D ASEP by duality}.
\newblock {\em Journal of Statistical Physics}, 142(5):919--930, 2011.

\bibitem[IS13]{imamura_stationaryKPZ}
T.~Imamura and T.~Sasamoto.
\newblock Stationary correlations for the 1d kpz equation.
\newblock {\em Journal of Statistical Physics}, 150:908--939, 03 2013.

\bibitem[IS19]{imamura2017fluctuations}
T.~Imamura and T.~Sasamoto.
\newblock Fluctuations for stationary q-tasep.
\newblock {\em Probability Theory and Related Fields}, 174:647--730, 2019.

\bibitem[Ism05]{ismailBook}
M.E.H. Ismail.
\newblock {\em {Classical and Quantum Orthogonal Polynomials in One Variable}}.
\newblock Cambridge University Press, 2005.

\bibitem[IT18]{Takeuchi_ito_2018_half_space}
Y.~Ito and K.A. Takeuchi.
\newblock {When fast and slow interfaces grow together: Connection to the
  half-space problem of the Kardar-Parisi-Zhang class}.
\newblock {\em Physical Review E}, 97:040103, Apr 2018.

\bibitem[Joh00]{johansson2000shape}
K.~Johansson.
\newblock {Shape fluctuations and random matrices}.
\newblock {\em Communications in Mathematical Physics}, 209(2):437--476, 2000.
\newblock arXiv:math/9903134 [math.CO].

\bibitem[Kar85]{kardar_depinning}
M.~Kardar.
\newblock {Depinning by Quenched Randomness}.
\newblock {\em Physical Review Letters}, 55:2235--2238, Nov 1985.

\bibitem[Kar87]{Kardar1987}
M.~Kardar.
\newblock {Replica Bethe ansatz studies of two-dimensional interfaces with
  quenched random impurities}.
\newblock {\em Nuclear Physics B}, 290:582--602, 1987.

\bibitem[KLD20]{krejenbrink_le_doussal_KPZ_half_space}
A.~Krajenbrink and P.~Le~Doussal.
\newblock {Replica Bethe Ansatz solution to the Kardar-Parisi-Zhang equation on
  the half-line}.
\newblock {\em SciPost Physics}, 8:35, 2020.

\bibitem[KPZ86]{KPZ1986}
M.~Kardar, G.~Parisi, and Y.~Zhang.
\newblock Dynamic scaling of growing interfaces.
\newblock {\em Physical Review Letters}, 56(9):889, 1986.

\bibitem[KQ18]{krishnan_quastel_log_gamma}
A.~Krishnan and J.~Quastel.
\newblock {Tracy–Widom fluctuations for perturbations of the log-gamma
  polymer in intermediate disorder}.
\newblock {\em The Annals of Applied Probability}, 28(6):3736 -- 3764, 2018.

\bibitem[LDC12]{Le_Doussal_2012}
P.~Le~Doussal and P.~Calabrese.
\newblock The {KPZ} equation with flat initial condition and the directed
  polymer with one free end.
\newblock {\em Journal of Statistical Mechanics: Theory and Experiment},
  2012(06), 06 2012.

\bibitem[Mac95]{Macdonald1995}
I.G. Macdonald.
\newblock {\em Symmetric functions and {H}all polynomials}.
\newblock Oxford University Press, 2nd edition, 1995.

\bibitem[Mat19]{Matveev_Kerov_conjecture}
K.~Matveev.
\newblock Macdonald-positive specializations of the algebra of symmetric
  functions: Proof of the kerov conjecture.
\newblock {\em Annals of Mathematics}, 189(1):277--316, 2019.

\bibitem[MP17]{MatveevPetrov2014}
K.~Matveev and L.~Petrov.
\newblock {$q$-randomized Robinson--Schensted--Knuth correspondences and random
  polymers}.
\newblock {\em Annales de l'IHP D}, 4(1):1--123, 2017.

\bibitem[MP22]{MucciconiPetrov2020}
M.~Mucciconi and L.~Petrov.
\newblock {Spin $q$-Whitaker polynomials and deformed quantum Toda}.
\newblock {\em Communications in Mathematical Physics}, 389:1331--1416, Feb
  2022.

\bibitem[Oko01]{okounkov2001infinite}
A.~Okounkov.
\newblock {Infinite wedge and random partitions}.
\newblock {\em Selecta Mathematica}, 7(1):57--81, 2001.

\bibitem[OP13]{OConnellPei2012}
N.~O'Connell and Y.~Pei.
\newblock {A q-weighted version of the Robinson-Schensted algorithm}.
\newblock {\em Electronic Journal of Probability}, 18(95):1--25, 2013.

\bibitem[OSZ14]{OSZ2012}
N.~O'Connell, T.~Sepp{\"a}l{\"a}inen, and N.~Zygouras.
\newblock {Geometric RSK correspondence, Whittaker functions and symmetrized
  random polymers}.
\newblock {\em Inventiones Mathematicae}, 197:361--416, 2014.

\bibitem[Par19a]{Parekh_half_space_symmetry}
S.~Parekh.
\newblock {Positive random walks and an identity for half-space SPDEs}.
\newblock {\em arXiv preprint}, 01 2019.
\newblock arXiv:1901.09449 [math.PR].

\bibitem[Par19b]{Parekh2019}
S~Parekh.
\newblock {The KPZ Limit of ASEP with Boundary}.
\newblock {\em Communications in Mathematical Physics}, 365:569--649, 01 2019.

\bibitem[PS02]{Praehofer2002}
M.~Pr{\"a}hofer and H.~Spohn.
\newblock {Current Fluctuations for the Totally Asymmetric Simple Exclusion
  Process}.
\newblock {\em In and Out of Equilibrium: Probability with a Physics Flavor
  (Progress in Probability)}, 51:185--204, 2002.

\bibitem[Qua13]{quastel_introduction_to_KPZ}
J.~Quastel.
\newblock {Introduction to KPZ}.
\newblock {\em {Current Developments in Mathematics}}, 2011, 03 2013.

\bibitem[Rai00]{Rains2000}
E.M. Rains.
\newblock Correlation functions for symmetrized increasing subsequences.
\newblock {\em arXiv preprint}, 2000.
\newblock arXiv:math/0006097 [math.CO].

\bibitem[Sch97]{schutz1997dualityASEP}
G.~Sch{\"u}tz.
\newblock Duality relations for asymmetric exclusion processes.
\newblock {\em Journal of Statistical Physics}, 86(5-6):1265--1287, 1997.

\bibitem[Sep12]{Seppalainen2012}
T.~Sepp{\"a}l{\"a}inen.
\newblock {Scaling for a one-dimensional directed polymer with boundary
  conditions}.
\newblock {\em {Annals of Probability}}, 40(1):19--73, 2012.

\bibitem[SI04]{imamura_sasamoto_half_space_png}
T.~Sasamoto and T.~Imamura.
\newblock Fluctuations of the one-dimensional polynuclear growth model in
  half-space.
\newblock {\em Journal of Statistical Physics}, 115:749--803, May 2004.

\bibitem[Spi70]{Spitzer1970}
F.~Spitzer.
\newblock {Interaction of Markov processes}.
\newblock {\em Advances in Mathematics}, 5(2):246--290, 1970.

\bibitem[SS90]{sagan1990robinson}
B.~Sagan and R.~Stanley.
\newblock {Robinson-Schensted algorithms for skew tableaux}.
\newblock {\em Journal of Combinatorial Theory, Series A}, 55(2):161--193,
  1990.

\bibitem[SS10]{SasamotoSpohn2010}
T.~Sasamoto and H.~Spohn.
\newblock {Exact height distributions for the KPZ equation with narrow wedge
  initial condition}.
\newblock {\em Nuclear Physics B}, 834(3):523--542, 2010.

\bibitem[Ste90]{stembridge1990nonintersecting}
J.R. Stembridge.
\newblock {Nonintersecting paths, pfaffians and plane partitions}.
\newblock {\em Advances in Mathematics}, 83(1):96--131, 1990.

\bibitem[TW94]{tracy_widom1994level_airy}
C.~Tracy and H.~Widom.
\newblock {Level-spacing distributions and the Airy kernel}.
\newblock {\em Communications in Mathematical Physics}, 159(1):151--174, 1994.
\newblock arXiv:hep-th/9211141.

\bibitem[TW96]{tracy1996orthogonal}
C.~Tracy and H.~Widom.
\newblock {On orthogonal and symplectic matrix ensembles}.
\newblock {\em Communications in Mathematical Physics}, 177(3):727--754, 1996.

\bibitem[TW08a]{tracy2008fredholm}
C.~Tracy and H.~Widom.
\newblock {A Fredholm determinant representation in ASEP}.
\newblock {\em Journal of Statistical Physics}, 132(2):291--300, 2008.

\bibitem[TW08b]{TW_ASEP1}
C.~Tracy and H.~Widom.
\newblock {Integral formulas for the asymmetric simple exclusion process}.
\newblock {\em Communications in Mathematical Physics}, 279:815--844, 2008.
\newblock Erratum: Communications in Mathematical Physics, 304:875--878, 2011.

\bibitem[TW09a]{TW_ASEP2}
C.~Tracy and H.~Widom.
\newblock {Asymptotics in ASEP with step initial condition}.
\newblock {\em Communications in Mathematical Physics}, 290:129--154, 2009.

\bibitem[TW09b]{TW_ASEP4}
C.~Tracy and H.~Widom.
\newblock {On ASEP with step Bernoulli initial condition}.
\newblock {\em Journal of Statistical Physics}, 137:825--838, 2009.

\bibitem[TW13a]{Tracy_Widom_half_space_asep}
C.~Tracy and H.~Widom.
\newblock The asymmetric simple exclusion process with an open boundary.
\newblock {\em Journal of Mathematical Physics}, 54(10):103301, 2013.

\bibitem[TW13b]{tracy_widom_open_asep_delta_bose}
C.~Tracy and H.~Widom.
\newblock {The Bose Gas and Asymmetric Simple Exclusion Process on the
  Half-Line}.
\newblock {\em Journal of Statistical Physics}, 150:1--12, 01 2013.

\bibitem[Whi52]{whitney_totally_positive}
A.M. Whitney.
\newblock A reduction theorem for totally positive matrices.
\newblock {\em Journal d’Analyse Mathématique}, 2:88--92, 12 1952.

\bibitem[Wu20]{Xuan_Wu_Intermediate_Disorder}
X.~Wu.
\newblock Intermediate disorder regime for half-space directed polymers.
\newblock {\em Journal of Statistical Physics}, 181:2372--2403, 12 2020.

\bibitem[Zyg18]{zygouras_review}
N.~Zygouras.
\newblock {Some algebraic structures in the KPZ universality}.
\newblock {\em arXiv preprint}, 2018.
\newblock arXiv:1812.07204v3 [math.PR].

\end{thebibliography}
\end{document}